\documentclass[11pt,a4paper,leqno]{amsart}
\usepackage{mathrsfs}

\usepackage{xspace}
\usepackage{graphics}
\usepackage{amsfonts,amssymb}
\usepackage{amsthm}
\usepackage[all]{xy}
\usepackage{stmaryrd}
\usepackage{color}
\usepackage{titletoc}

\newcommand{\be}{\begin{otherlanguage}{english}}
\newcommand{\ee}{\end{otherlanguage}}

\theoremstyle{definition}
\newtheorem{defn}{Definition}[section]
\newtheorem{rem}[defn]{Remark}
\theoremstyle{plain}
\newtheorem{lem}[defn]{Lemma}
\newtheorem{prop}[defn]{Proposition}
\newtheorem{thm}[defn]{Theorem}
\newtheorem{cor}[defn]{Corollary}

\newtheorem*{thm*}{Theorem}
\newtheorem*{prop*}{Proposition}
\newtheorem*{clm*}{Claim}
\theoremstyle{remark}
\newtheorem{exem}[defn]{\textbf{Example}}
\numberwithin{equation}{subsection}

\newcommand{\beq}{\begin{equation}}
\newcommand{\eeq}{\end{equation}}

\newcommand{\G}{\mathbb{G}}

\begin{document}
\title[]{A generalization of classical action of Hamiltonian
diffeomorphisms to Hamiltonian homeomorphisms on fixed points}

%    Information for first author
\author{Jian Wang}
\address{Chern Institute of Mathematics, Nankai University, Tianjin 300071,
P.R.China.}\email{wangjian@nankai.edu.cn}
%\curraddr{LAGA UMR 7539
%CNRS, Universit\'e Paris 13, 93430 Villetaneuse, France.}
%\email{wangjian@math.univ-paris13.fr}
%\thanks{Support information for the second author.}

%    General info
%\subjclass[2000]{37E30, 37E40, 37J10}

%\author{Jian Wang}
\date{September 9, 2012}
\maketitle

\begin{abstract}
We define boundedness properties on the contractible fixed points
set of the time-one map of an identity isotopy on a closed oriented
surface with genus $g\geq1$. In symplectic geometry, a classical
object is the notion of action function, defined on the set of
contractible fixed points of the time-one map of a Hamiltonian
isotopy. We give a dynamical interpretation of this function that
permits us to generalize it in the case of a homeomorphism isotopic
to identity that preserves a Borel finite measure of rotation vector
zero, provided that a boundedness condition is satisfied. We give
some properties of the generalized action. In particular, we
generalize a result of Schwarz \cite{SM} about the action function
being non-constant which has been proved by using Floer homology. As
applications, we generalize some results of Polterovich \cite{P}
about the symplectic and Hamiltonian diffeomorphisms groups on
closed oriented  surfaces being distortion free, which allows us to
give an alternative proof of the $C^1$-version of the Zimmer
conjecture on closed oriented surfaces.
\end{abstract}\bigskip

\setcounter{tocdepth}{1}\tableofcontents

\setcounter{section}{-1}
\section{Introduction} Suppose that $(M,\omega)$ is a symplectic manifold with
$\pi_2(M)=0$. Let $I=(F_t)_{t\in[0,1]}$ be a Hamiltonian flow  on
$M$ with $F_0=\mathrm{Id}_M$ and $F_1=F$. Suppose that the function
$H_t$ is the Hamiltonian function generating the flow $I$. Denote by
$\mathrm{Fix}_{\mathrm{Cont},I}(F)$ the set of contractible fixed
points of $F$, that is, $x$ is a fixed point of $F$ and the oriented
loop $I(x): t\mapsto F_t(x)$ defined on $[0,1]$ can be contractible
on $M$. The classical action function is defined, up to an additive
constant, on $\mathrm{Fix}_{\mathrm{Cont},I}(F)$  as follows (see
Section \ref{sec:action function} for the details)
\begin{equation}\label{eq:action functional}
\mathcal{A}_{H}(x)=\int_{D_x}\omega-\int_0^1
H_t(F_t(x))\,\mathrm{d}t\nonumber,
\end{equation}
where $x\in \mathrm{Fix}_{\mathrm{Cont},I}(F)$ and $D_x\subset M$ is
any $2$-simplex with $\partial D_x=I(x)$.\smallskip

%Let us consider a time dependent Hamiltonian vector field on a
%symplectic manifold $(M,\omega)$ such that the Hamiltonian is
%periodic in time, of period one. There is a natural function
%$\mathcal{A}_{H}(x)$, ``the action'' (see Section \ref{sec:action
%function}) that is defined, up to an additive constant, on the set
%of 1-periodic solutions that are contractible loops of $M$, in other
%words it is defined on the set of contractible fixed points of the
%time-one map $F$ of a Hamiltonian isotopy
%$I=(F_t)_{t\in[0,1]}$.

When $M$ is compact, among the properties of $F$, one may notice the
fact that it preserves the volume form
$\omega^n=\omega\wedge\cdots\wedge\omega$ and that the ``rotation
vector'' $\rho_{M,I}(\mu)\in H_1(M,\mathbb{R})$ (see Section
\ref{subsec:rotation vector}) of the finite measure $\mu$ induced by
$\omega^n$ vanishes. In the case of a closed symplectic surface, the
fact that a diffeomorphism isotopic to identity preserves a volume
form $\omega$ whose rotation vector is zero characterizes the fact
that it is the time-one map of a 1-periodic Hamiltonian isotopy (see
Section \ref{sec:action function}).\smallskip

The goal of this article is to give a precise dynamical explanation
of the action function defined on the set of contractible fixed
points in the case of surface. Through defining a weak boundedness
property (for example, $F$ ia a diffeomorphism or the set
$\mathrm{Fix}_{\mathrm{Cont},I}(F)$ is finite), written WB-property
for short, which is a certain boundedness condition about linking
numbers of contractible fixed points (see Section
\ref{subsec:boundedness}), we define a new action function with the
following desired properties and prove that it is a generalization
of the classical function:

\begin{itemize}
  \item It can be naturally generalized for any diffeomorphism (not necessarily $C^1$)
isotopic to the identity that preserves a finite Borel measure of
rotation vector zero with no atoms on the contractible fixed points
set.
  \item It can be naturally generalized for
any homeomorphism isotopic to the identity that preserves a finite
Borel measure  of rotation vector zero with total support and no
atoms on the contractible fixed points set, provided that the
WB-property is satisfied.
\item It can be naturally generalized for
any homeomorphism isotopic to the identity that preserves a finite
ergodic Borel measure $\mu$ of rotation vector zero with no atoms on
the contractible fixed points set, provided that the WB-property is
satisfied.
\end{itemize}

In addition, we investigate some properties of the new action
function: such as the boundedness (Proposition
\ref{prop:boundedness}) and the continuity (Proposition
\ref{prop:the continuity of lmu}). Interestingly, we furthermore
prove that, in the closed oriented surface case, the new action
function is not constant when the measure has total support, which
has been proved in \cite{SM} by Floer homology for the case where
the isotopy is a Hamiltonian flow. In \cite{SM}, the time-one map
$F$ requires to be at least $C^1$-smooth and the contractible fixed
point set of $F$ to satisfy certain non-degeneracy. In constract, we
only demand the isotopy to satisfy a much weaker property, the
proposed WB-property.

Moreover, we are in particular interested in the conservative
diffeomorphism groups of closed oriented  surfaces. By applying the
new action function to the groups of conservative diffeomorphisms,
we generalize some results of Polterovich \cite{P} about the absence
of distortion in the symplectic and Hamiltonian diffeomorphisms
groups on closed oriented surface. We proved that the
$C^1$-conservative diffeomorphism groups have no distortion on
closed oriented surfaces, which links to Zimmer conjecture on closed
oriented surfaces.\bigskip

The main results of this article can be summarized as
follows.\smallskip

Let $M$ be a smooth manifold (with boundary or not) and let $F$ be a
homeomorphism on $M$. Denote by $\mathrm{Diff}(M)$ (resp.
$\mathrm{Diff}^1(M)$) the group of all diffeomorphisms (resp.
$C^1$-diffeomorphisms) on $M$ and by $\mathcal {M}(F)$ the set of
Borel finite measures on $M$ whose elements are invariant by $F$. We
say that an isotopy $I=(F_t)_{t\in[0,1]}$ on $M$ is an
\emph{identity isotopy} if $F_0=\mathrm{Id}_{M}$. Our main results
are following:

\begin{thm}\label{thm:PW}
Let $F$ be the time-one map of an identity isotopy $I$ on a closed
oriented surface $M$ with $g\geq 1$. Suppose that $\mu\in\mathcal
{M}(F)$ has no atoms on $\mathrm{Fix}_{\mathrm{Cont},I}(F)$ and that
$\rho_{M,I}(\mu)=0$. In all of the following cases
\begin{itemize}
  \item $F\in\mathrm{Diff}(M)$ (not necessarily $C^1$);
  \item $I$ satisfies the WB-property, the measure $\mu$ has total
  support;
  \item $I$ satisfies the WB-property, the measure $\mu$ is ergodic,
\end{itemize}
an action function can be defined which generalizes the classical
case.
\end{thm}

On closed oriented surfaces, we get the following Proposition
\ref{prop:F is not constant if the contractible fixed points is
finite} and Corollary \ref{prop:F is not constant if F is not Id}
that are generalizations of Lemma 2.8 that is proved in \cite{SM} by
using Floer homology.

\begin{prop}\label{prop:F is not
constant if the contractible fixed points is finite} Let $F$ be the
time-one map of an identity isotopy $I$ on a closed oriented surface
$M$ with $g\geq 1$. Suppose that $\mu\in\mathcal {M}(F)$ has total
support, no atoms on $\mathrm{Fix}_{\mathrm{Cont},I}(F)$, and
$\rho_{M,I}(\mu)=0$. If $I$ satisfies the WB-property and $F\in
\mathrm{Homeo}_*(M)\setminus\{\,\mathrm{Id}_{M}\}$ where
$\mathrm{Homeo}_*(M)$ is the group of all homeomorphisms isotopic to
$\mathrm{Id}_{M}$ on $M$, then the action function we defined in
Theorem \ref{thm:PW} is not constant.
\end{prop}

Let $\widetilde{M}$ be the universal covering space and let
$\widetilde{F}$ be the time-one map of the lifted identity isotopy
of $I$ to $\widetilde{M}$. In fact, in order to generalize the
classical action of $F$, we first define the action function of $F$
on the fixed point set of $\widetilde{F}$ (see Section
\ref{sec:definition of action function}). And then, we define the
action spectrum $\sigma(\widetilde{F})$ which is the range of the
action function of $F$ (whose domain is the fixed points set of
$\widetilde{F}$), and define the action width
$\mathrm{width}(\widetilde{F})=\sup_{x,y\in\sigma(\widetilde{F})}|x-y|$
(it may be infinite, see Section \ref{sec:action spectrum}). Base on
Proposition \ref{prop:F is not constant if the contractible fixed
points is finite}, we can get the following Corollary \ref{cor:the
symplectic action when M with genus bigger 1}, Proposition
\ref{prop:ggeq1 no torsion} and Proposition \ref{prop:g>1 no
torsion} which are generalizations of Theorem 2.1.\,C, Proposition
2.6.\,A in \cite{P}.

\begin{cor}\label{cor:the symplectic action when M with genus bigger
1}Let $F$ be the time-one map of an identity isotopy $I$ on a closed
oriented surface $M$ with $g>1$. If $I$ satisfies the WB-property
and $F\in \mathrm{Homeo}_*(M)\setminus\{\,\mathrm{Id}_{M}\}$,
$\mu\in\mathcal {M}(F)$ has total support and no atoms on
$\mathrm{Fix}_{\mathrm{Cont},I}(F)$,
 then $\sharp\sigma(\widetilde{F})\geq2$.
%there exist two distinct fixed points $\widetilde{a}$ and
%$\widetilde{b}$ of $\widetilde{F}$ such that
%$i_\mu(\widetilde{F};\widetilde{a},\widetilde{b})\neq0$.
\end{cor}

We remark here that  Proposition \ref{prop:F is not constant if the
contractible fixed points is finite} and Corollary \ref{cor:the
symplectic action when M with genus bigger 1} are not valid when the
measure has not total support as shown by Example \ref{exem:T2 and
supp(u)notM} and Example \ref{exem:M and supp(u)notM}.

We extend the identity isotopy $I=(F_t)_{t\in[0,1]}$ to $\mathbb{R}$
by writing $F_{t+1}=F_t\circ F_1$. We have the following
conclusions:

\begin{prop}\label{prop:ggeq1 no torsion}Under the same hypotheses
as Proposition \ref{prop:F is not constant if the contractible fixed
points is finite}, % or Proposition \ref{prop:F is not constant if F is
%not Id},
there exists a constant $C>0$ such that
$\mathrm{width}(\widetilde{F}^n)\geq C\cdot n$ for every $n\geq1$
where $\widetilde{F}^n$ is the time-one map of the lifted identity
isotopy of $I^n=(F_t)_{t\in[0,n]}$ to $\widetilde{M}$.
\end{prop}

\begin{prop}\label{prop:g>1 no torsion}
Under the same hypotheses as Corollary \ref{cor:the symplectic
action when M with genus bigger 1}, there exists a constant $C>0$
such that $\mathrm{width}(\widetilde{F}^n)\geq C\cdot n$ for every
$n\geq1$.
\end{prop}\smallskip

Fix a Borel finite measure $\mu$ whose support is the whole space.
Denote by $\mathrm{Homeo}_*(M,\mu)$ the subgroup of
$\mathrm{Homeo}_*(M)$ whose element preserves the measure $\mu$, and
by $\mathrm{Hameo}(M,\mu)$ the subgroup of $\mathrm{Homeo}_*(M,\mu)$
whose elements satisfy furthermore that $\rho_{M,I}(\mu)=0$. For
convenience, we write $M_g$ the closed oriented surface with the
genus $g\geq1$.\bigskip

Based on the previous result listed above, we can study the periodic
homeomorphisms of surfaces. Applying Proposition \ref{prop:ggeq1 no
torsion}, Proposition \ref{prop:g>1 no torsion} and a result of
Fathi \cite{Fathi}(see Section \ref{subsec:application to
conservative diffeomorphism}), we can get the following corollary :
\begin{cor}\label{cor:no torsion}
The groups $\mathrm{Hameo}(\mathbb{T}^2,\mu)$,
$\mathrm{Homeo}_*(M_g)$ ($g>1$) are torsion free.
\end{cor}\bigskip

Let us now recall the definition of distortion (see \cite{P}). If
$\mathscr{G}$ is a finitely generated group with generators
$\{g_1,\ldots,g_s\}$, then
$f\in\mathscr{G}$ is said to be a \emph{distortion element} %(resp.
%\emph{U-element})
of $\mathscr{G}$ provided that $f$ has infinite order and
$$\liminf_{n\rightarrow+\infty}\frac{\|f^n\|_{\mathscr{G}}}{n}=0,$$
%(\mathrm{resp.
%\quad\liminf_{n\rightarrow+\infty}\frac{\log\|f^n\|_{\mathscr{G}}}{\log
%n}=0})
where $\|f^n\|_{\mathscr{G}}$ is the word length of $f^n$ in the
generators $\{g_1,\ldots,g_s\}$. If $\mathscr{G}$ is not finitely
generated, then we say that $f\in\mathscr{G}$ is distorted in
$\mathscr{G}$ if it is distorted in some finitely generated subgroup
of $\mathscr{G}$. %Obviously, a U-element of $\mathscr{G}$ is a
%distortion element of $\mathscr{G}$.
\smallskip

Given two positive sequences $\{a_n\}$ and $\{b_n\}$, we write
$a_n\succeq b_n$ if there is $c>0$ such that $a_n\geq cb_n$ for all
$n\in\mathbb{N}$, and $a_n\sim b_n$ if $a_n\succeq b_n$ and
$a_n\preceq b_n$.\smallskip

Denote by $\mathrm{Ham}^1(M,\mu)$ the group
$$\mathrm{Hameo}(M,\mu)\cap \mathrm{Diff}^1(M)$$ and by
$\mathrm{Diff}^1_*(M,\mu)$ the group $$\mathrm{Homeo}_*(M,\mu)\cap
\mathrm{Diff}^1(M).$$

Finally, we study the existence of distortion of the groups
$\mathrm{Ham}^1(\mathbb{T}^2,\mu)$ and $\mathrm{Diff}^1_*(M_g,\mu)$
and get the following result:

\begin{thm}\label{thm:Fn thicksim n}
Let $F\in\mathrm{Diff}^1_*(M_g,\mu)\setminus\{\mathrm{Id}_{M_g}\}$
($g>1$) (resp.
$F\in\mathrm{Ham}^1(\mathbb{T}^2,\mu)\setminus\{\mathrm{Id}_{\mathbb{T}^2}\}$),
and $\mathscr{G}\subset\mathrm{Diff}^1_*(M_g,\mu)$ ($g>1$) (resp.
$\mathscr{G}\subset\mathrm{Ham}^1(\mathbb{T}^2,\mu)$) be a finitely
generated subgroup containing $F$, then
$$\|F^n\|_{\mathscr{G}}\thicksim n.$$
As a consequence, the groups $\mathrm{Diff}^1_*(M_g,\mu)$ ($g>1$)
and $\mathrm{Ham}^1(\mathbb{T}^2,\mu)$ have no distortion.\end{thm}

Applying Theorem \ref{thm:Fn thicksim n}, some algebraic properties
of the lattice $\mathrm{SL}(n,\mathbb{Z})$ ($n\geq3$) and mapping
class group (see Section \ref{sec:distortion of group} for the
details), we get the following theorems.

\begin{thm}\label{thm:zimmer2}Every homomorphism from
$\mathrm{SL}(n,\mathbb{Z})$ ($n\geq3$) to
$\mathrm{Ham}^1(\mathbb{T}^2,\mu)$ or $\mathrm{Diff}^1_*(M_g,\mu)$
($g>1$) is trivial.
\end{thm}

\begin{thm}\label{thm:zimmer program}Every  homomorphism from
$\mathrm{SL}(n,\mathbb{Z})$ ($n\geq3$) to $\mathrm{Diff}^1(M_g,\mu)$
($g>1$) has only finite images.
\end{thm}

Remark here that Theorem \ref{thm:zimmer program} is a more general
conjecture of Zimmer (see \cite{Zim,P2,F4}) in the special surfaces
case. The reader can find more information about Zimmer conjecture
in Section \ref{sec:distortion of group}.
\bigskip

The article is organized as follows. In Section 1, we will first
introduce some notations and recall the precise definitions of some
important mathematical objects. And then we will define the linking
number on
 contractible fixed points and the boundedness properties. Finally, we
will study some conditions for these properties to hold. In Section
2, we will recall the classical action function in symplectic
geometry and analyze how to generalize the action function to a more
general case on closed oriented surfaces. In the end of this
section,  our main theorem is stated. In Section 3, we will recall
some well known results about the plan and the open annulus, and
extend some results of Franks to serve as the technical
preliminaries of this article. In Section 4, we will first extend
the definition of the linking number defined in Section 1 to
positively recurrent points, which is one of the main ingredients of
this article, and then we will give some elementary properties of
the extended linking number. In Section 5, we will first study the
boundedness of the extended linking number when it exists, and then
study the existence and the boundedness of the linking number in the
conservative case. In Section 6, based on the extended linking
number and its properties studied in Sections 4 and 5, we will
define a new action function and prove that it is a generalization
of the classical one, which is our main theorem. We study the
properties of the new action function, including boundedness and
continuity. Furthermore, we prove that, on the closed oriented
surface case, the new action function is not constant when the
measure has total support whose smooth case has been proved in
\cite{SM} by using Floer homology. We also give examples to
illuminate that this result is not true any more when the measure
has not total support. In Section 7, by applying our generalized
action function to the groups of conservative diffeomorphisms, we
generalize the results of Polterovich about the absence of
distortion in the symplectic and Hamiltonian diffeomorphisms groups
on closed oriented surface in \cite{P}. We proved that the
$C^1$-conservative diffeomorphism groups have no distortion on
closed oriented surfaces, which permits us to give an alternative
proof of the $C^1$-version of the Zimmer conjecture on oriented
closed surfaces when the measure is a Borel finite measure with full
surport. In Section 8, we provide a proof of a well known fact
required in this article but which we have not found in the
literature. In addition, we construct some examples to further
complete our results.\bigskip

\noindent\textbf{Acknowledgements.} %This work is a partial part of
%my PhD thesis of Tsinghua University of China.
I would like to thank my advisors, Patrice Le Calvez and Zhiying
Wen, for many helpful discussions and suggestions. I also thank
Olivier Jaulent and Juliana Xavier for their explanations of their
results to me. I am grateful to Bassam Fayad and Lucien Guillou for
careful reading the manuscript and many useful remarks.
\bigskip

\section{Notation and Definitions} We denote by $|\cdot|$ the
usual Euclidean metric on $\mathbb{R}^k$ and by
$\mathbf{S}^{k-1}=\{x\in\mathbb{R}^k\mid |x|=1\}$ the unit sphere.

If $A$ is a set, we write $\sharp A$ for the cardinality of $A$. If
$G$ is a group and $e$ is its unit element, we write
$G^{\,*}=G\setminus\{e\}$. If $(S, \sigma,\mu)$ is a measured space
and $V$ is any finite dimensional linear space, denote
$L^1(S,V,\mu)$  the set of $\mu$-integrable functions from $S$ to
$V$. If $X$ is a topological space and $A$ is a subset of $X$,
denote respectively $\mathrm{Int}(A)$, $\mathrm{Cl}(A)$ and
$\partial A$ the interior, the closure and the boundary of $A$.

If $M$ is a smooth manifold (with boundary or not), we denote by
$\mathrm{Homeo}(M)$ the set of all homeomorphisms of $M$.

%  and
%$\mathrm{Homeo}_*(M)$ (resp. $\mathrm{Diff}_*(M)$) the identity
%component of the space of $\mathrm{Homeo}(M)$ (resp.
%) for the compact-open
%topology. %(resp. $C^1$-topology).(resp. $\mathrm{Diff}(M)$)

\subsection{Identity isotopies}\label{sec:Identity isotopies}An
\emph{identity isotopy} $I=(F_t)_{t\in[0,1]}$ on $M$ is a continuous
path
\begin{eqnarray*}
% \nonumber to remove numbering (before each equation)
  [0,1] &\rightarrow& \mathrm{Homeo}(M) \\
  t&\mapsto& F_t
\end{eqnarray*}
such that $F_0=\mathrm{Id}_{M}$, the last set being endowed with the
compact-open topology. We naturally extend this map to $\mathbb{R}$
by writing $F_{t+1}=F_t\circ F_1$. We can also define the inverse
isotopy of $I$ as $I^{-1}=(F_{-t})_{t\in[0,1]}=(F_{1-t}\circ
F_1^{-1})_{t\in[0,1]}$. We denote by $\mathrm{Homeo}_*(M)$ (resp.
$\mathrm{Diff}_*(M)$, $\mathrm{Diff}_*^1(M)$)  the set of all
homeomorphisms (resp. diffeomorphisms, $C^1$-diffeomorphisms) of $M$
that are isotopic to the identity.\smallskip

A \emph{path} on a manifold $M$ is a continuous map $\gamma:
J\rightarrow M$ defined on a nontrivial interval $J$ (up to an
increasing reparametrization). We can talk of a proper path (i.e.
$\gamma^{-1}(K)$ is compact for any compact set $K$) or a compact
path (i.e. $J$ is compact). When $\gamma$ is a compact path,
$\gamma(\inf J)$ and $\gamma(\sup J)$ are the \emph{ends} of
$\gamma$. We say that a compact path $\gamma$ is a \emph{loop} if
the two ends of $\gamma$ coincide. The inverse of the path $\gamma$
is defined by $\gamma^{-1}:t\mapsto\gamma(-t),\,t\in -J$. If
$\gamma_1: J_1\rightarrow M$ and $\gamma_2: J_2\rightarrow M$ are
two paths such that $$b_1=\sup J_1\in J_1,\quad a_2=\inf J_2\in
J_2\quad\mathrm{and}\quad\gamma_1(b_1)=\gamma_2(a_2),$$ then the
\emph{concatenation $\gamma_1$ and $\gamma_2$} is defined on
$J=J_1\cup(J_2+(b_1-a_2))$  in the classical way, where
$(J_2+(b_1-a_2))$ represents the translation of
 $J_2$ by $(b_1-a_2)$:
\begin{equation*}\gamma_1\gamma_2(t)=
\begin{cases}\gamma_1(t)& \textrm{if} \quad t\in J_1;
\\\gamma_2(t+a_2-b_1)& \textrm{if} \quad
t\in J_2+(b_1-a_2).\end{cases}
\end{equation*}

Let $\mathcal {I}$ be an interval (maybe infinite) of $\mathbb{Z}$.
If $\{\gamma_i:J_i\rightarrow M\}_{i\in\mathcal {I}}$ is a family of
compact paths satisfying that
$\gamma_i(\sup(J_i))=\gamma_{i+1}(\inf(J_{i+1}))$ for every
$i\in\mathcal {I}$, then we can define their concatenation
$\prod_{i\in\mathcal {I}}\gamma_i$.

If $\{\gamma_i\}_{i\in \mathcal {I}}$ is a family of compact paths
where $\mathcal {I}=\bigsqcup_{j\in\mathcal {J}}\mathcal {I}_j$ and
$\mathcal {I}_j$ is an interval of $\mathbb{Z}$ such that
$\prod_{i\in\mathcal {I}_j}\gamma_i$ is well defined (in the
concatenation sense) for every $j\in\mathcal {J}$, we define their
\emph{product} by abusing notations:
$$\prod\limits_{i\in\mathcal
{I}}\gamma_i=\prod_{j\in\mathcal {J}}\prod\limits_{i\in\mathcal
{I}_j}\gamma_i.$$

The \emph{trajectory} of a point $z$ for the isotopy
$I=(F_t)_{t\in[0,1]}$ is the oriented path $I(z): t\mapsto F_t(z)$
defined on $[0,1]$. Suppose that $\{I_k\}_{1\leq k\leq k_0}$ is a
family of identity isotopies on $M$. Write
$I_k=(F_{k,t})_{t\in[0,1]}$. We can define a new identity isotopy
$I_{k_0}\cdots I_{2}I_1=(F_t)_{t\in[0,1]}$ by concatenation as
follows
\begin{equation}\label{the product operate of the symplectic group of isotopy}
    F_t(z)=F_{k,\,k_0t-(k-1)}(F_{k-1,1}\circ F_{k-2,1}\circ\cdots\circ F_{1,1}(z))\quad\mathrm{if}
\quad \frac{k-1}{k_0}\leq t\leq\frac{k}{k_0}.
\end{equation}
In particular, $I^{k_0}(z)=\prod_{k=0}^{k_0-1}I({F^k(z)})$ when
$I_k=I$ for all $1\leq k\leq k_0$.\smallskip

We write $\mathrm{Fix}(F)$ for the set of fixed points of $F$. A
fixed point $z$ of $F=F_1$ is \emph{contractible} if $I(z)$ is
homotopic to zero. We write $\mathrm{Fix}_{\mathrm{Cont},I}(F)$ for
the set of contractible fixed points of $F$, which obviously depends
on $I$.

\subsection{The
algebraic intersection number}\label{sec:the algebraic intersection
number} The choice of an orientation on $M$ permits us to define the
algebraic intersection number $\Gamma\wedge\Gamma'$ between two
loops. We keep the same notation $\Gamma\wedge\gamma$ for the
algebraic intersection number between a loop and a path $\gamma$
when it is defined, for example, when $\gamma$ is proper or when
$\gamma$ is compact path whose extremities are not in $\Gamma$.
Similarly, we write $\gamma\wedge\gamma'$ for the algebraic
intersection number of two path $\gamma$ and $\gamma'$ when it is
defined, for example, when $\gamma$ and $\gamma'$ are compact paths
and the ends of $\gamma$ (resp. $\gamma'$) are not on $\gamma'$
(resp. $\gamma$).

\subsection{Rotation vector}\label{subsec:rotation vector}
\subsubsection{The definition of rotation vector}\label{subsec:rotation vector of Le
Calvez} Let us introduce the classical notion of rotation vector
(defined originally in \cite{S}). If $\Gamma$ is a loop on a smooth
manifold $M$, write $[\Gamma]_M\in H_1(M,\mathbb{Z})$ for the
homology class of $\Gamma$. Suppose that $F\in \mathrm{Homeo}_*(M)$
is the time-one map of an
identity isotopy $I=(F_t)_{t\in[0,1]}$. %and symplectic to $\mu$.
Let $\mathrm{Rec}^+(F)$ be the set of positively recurrent points of
$F$. If $z\in \mathrm{Rec}^+(F)$, fix an open disk $U\subset M$
containing $z$, and write $\{F^{n_k}(z)\}_{k\geq 1}$ for the
subsequence of the positive orbit of $z$ obtained by keeping the
points that are in $U$. For any $k\geq 0$, choose a simple path
$\gamma_{F^{n_k}(z),z}$ in $U$ joining $F^{n_k}(z)$ to $z$. The
homology class $[\Gamma_k]_M\in H_1(M,\mathbb{Z})$ of the loop
$\Gamma_k= I^{n_k}(z)\gamma_{F^{n_k}(z),z}$ does not depend on the
choice of $\gamma_{F^{n_k}(z),z}$. Say that $z$ has a \emph{rotation
vector} $\rho_{M,I}(z)\in H_1(M,\mathbb{R})$ if
\[\lim_{l\rightarrow
+\infty}\frac{1}{n_{k_l}}[\Gamma_{k_l}]_M=\rho_{M,I}(z)\] for any
subsequence $\{F^{n_{k_l}}(z)\}_{l\geq 1}$ which converges to $z$.
Neither the existence nor the value of the rotation vector depends
on the choice of $U$.

\smallskip
\subsubsection{The existence of rotation number in the compact case}
\label{sec:the existion of rotation vector in the copact case}
Suppose that $M$ is compact and that $F$ is the time-one map of an
identity isotopy $I=(F_t)_{t\in[0,1]}$ on $M$. Recall that $\mathcal
{M}(F)$ is the set of Borel finite measures on $M$ whose elements
are invariant by $F$. If $\mu\in\mathcal {M}(F)$, we can define the
rotation vector $\rho_{M,I}(z)$ for $\mu$-almost every positively
recurrent point \cite{P1}. Let us explain why.

Let $U$ be an open disk of $M$ that is the interior of a closed
topological disk. For every couple $(z',z'')\in U^2$, choose a
simple path $\gamma_{z',z''}$ in $U$ joining $z'$ to $z''$. We can
define the first return map $\Phi: \mathrm{Rec}^+(F)\cap
U\rightarrow \mathrm{Rec}^+(F)\cap U$ and write
$\Phi(z)=F^{\tau(z)}(z)$, where $\tau(z)$ is the first return time,
that is, the least number $n\geq1$ such that $F^n(z)\in U$. By
Poincar\'{e} Recurrence Theorem, this map is defined $\mu$-almost
everywhere on $U$. For every $z\in \mathrm{Rec}^+(F)\cap U$ and
$n\geq 1$, define
$$\tau_n(z)=\sum_{i=0}^{n-1}\tau(\Phi^i(z)),\quad \Gamma_z^n=I^{\tau_n(z)}(z)\gamma_{\Phi^n(z),z}.$$
Observe now that
$$[\Gamma_z^n]_M
=\sum_{i=0}^{n-1}[\Gamma^1_{\Phi^i(z)}]_M.$$

By the classical Kac's lemma (see \cite{kac}), we have
$$\int_U\tau\, \mathrm{d}\mu=\mu\bigg(\bigcup_{k\geq
0}F^k(U)\bigg)=\mu\bigg(\bigcup_{k\in \mathbb{Z}}F^k(U)\bigg).$$ %Let
%$U_N=\bigcup_{k=-N}^{N}F^k(U)$. We have $U_N\subset U_{N+1}$. Hence
%\begin{eqnarray*}
%      % \nonumber to remove numbering (before each equation)
%        \mu\left(\bigcup_{k\in\mathbb{Z}}F^k(U)\right)=\lim_{N\rightarrow+\infty}
%        \mu\left(\bigcup_{k=-N}^{N}F^k(U)\right)% lower semi continuity of measure
%        &=& \lim_{N\rightarrow+\infty}
%        \mu\left(F^{-N}\left(\bigcup_{k=0}^{2N}F^k(U)\right)\right) \\
%       =\lim_{N\rightarrow+\infty}
%        \mu\left(\bigcup_{k=0}^{2N}F^k(U)\right) &=&
%        \mu\left(\bigcup_{k\geq0}F^k(U)\right).
%      \end{eqnarray*}
Indeed, we have the following measurable partitions (modulo sets of
measure zero):
$$U=\bigsqcup_{i\geq1}U_i\quad \text{and}\quad \bigcup_{k\geq0}F^k(U)
=\bigsqcup_{i\geq1}\;\bigsqcup_{0\leq j\leq i-1}F^j(U_i),$$ where
$U_i=\tau^{-1}(\{i\})$, therefore $$\mu\bigg(\bigcup_{k\geq
0}F^k(U)\bigg)=\sum_{i\geq1}\;\sum_{0\leq j\leq
i-1}\mu(U_i)=\sum_{i\geq1}i\mu(U_i)=\int_U\tau\, \mathrm{d}\mu.$$

Hence, we get $\tau\in L^1(U,\mathbb{R},\mu)$. In the case where $M$
is compact, let us prove that the function
$z\mapsto[\Gamma^1_z]_M/\tau(z)$ is bounded on
$\mathrm{Rec}^+(F)\cap U$ and hence that the map
$z\mapsto[\Gamma^1_z]_M$ belongs to $L^1(U,H_1(M,\mathbb{R}),\mu)$.

Indeed, it is sufficient to prove that for every cohomology class
$\kappa\in H^1(M,\mathbb{R})$, there exists a constant $K_\kappa$
such that $|\langle\kappa,[\Gamma^1_z]_M\rangle|\leq
K_\kappa\tau(z)$. Let $\lambda$ be a closed form that represents
$\kappa$. The function $g_{\lambda}: z\mapsto \int_{I(z)}\lambda$ is
well defined, since $\lambda$ is closed, and continuous. It is
bounded since $M$ is compact. As $\mathrm{Cl}(U)$ is a closed disk,
we can find an open disk $U'$ containing $\mathrm{Cl}(U)$ and a
primitive $h_\lambda$ of $\lambda$ on $U'$. This primitive is
bounded on $\mathrm{Cl}(U)$. This implies that for every $z\in
\mathrm{Rec}^+(F)\cap U$, we have

\begin{eqnarray}\label{eq:rotation vector}
% \nonumber to remove numbering (before each equation)
  |\langle[\lambda],[\Gamma^1_z]_M\rangle|=\left|\int_{\Gamma^1_z}\lambda\right| &=& \left|\sum_{i=0}^{\tau(z)-1}
\int_{I(F^i(z))}\lambda+\int_{\lambda_{\Phi(z),z}}\lambda\right|\\
&\leq&
 \tau(z)\max_{z\in
M}|g_\lambda(z)|+2\sup_{z\in U}|h_\lambda(z)|\\
   &\leq& \tau(z)(\max_{z\in
M}|g_\lambda(z)|+2\sup_{z\in U}h_\lambda(z)|).
\end{eqnarray}

By Birkhoff Ergodic Theorem, for $\mu$-almost every point on
$\mathrm{Rec}^+(F)\cap U$, the sequence $\{\tau_n(z)/n\}_{n\geq1}$
converges to a real number $\tau^*(z)\geq 1$, and the sequence
$\{[\Gamma_z^n]_M/n\}_{n\geq1}$ converges to $[\Gamma^*_z]_M\in
H_1(M,\mathbb{R})$. The positively recurrent points of $F$ in $U$
are exactly the positively recurrent points of $\Phi$ because $U$ is
open. We deduce that $\mu$-almost every point $z\in
\mathrm{Rec}^+(F)\cap U$ has a rotation vector
$\rho_{M,I}(z)=[\Gamma^*_z]_M/\tau^*(z)$. Since $U$ is arbitrarily
chosen, we deduce that $\mu$-almost every point $z\in
\mathrm{Rec}^+(F)$ has a rotation vector. The function $z\mapsto
[\Gamma^1_z]_M/\tau(z)$ is bounded on $\mathrm{Rec}^+(F)\cap U$, so
is the function
$$\rho_{M,I}: z\mapsto \lim\limits_{n\rightarrow
+\infty}\frac{\sum_{i=0}^{n-1}[\Gamma^1_{\Phi^i(z)}]_M}{\sum_{i=0}^{n-1}\tau(\Phi^i(z))}$$
on $\mathrm{Rec}^+(F)\cap U$. As $M$ can be covered by finitely many
such open disks, we deduce that $\rho_{M,I}$ is uniformly bounded on
$\mathrm{Rec}^+(F)$. Therefore, we can define the \emph{rotation
vector of the measure}
$$\rho_{M,I}(\mu)=\int_M\rho_{M,I}\, \mathrm{d}\mu\in H_1(M,\mathbb{R}).$$

\subsubsection{The rotation number of an open annulus}
\label{subsec:the rotation number of an open annulus}
  Let $\mathbb{A}=\mathbb{R}/\mathbb{Z}\times \mathbb{R}$
be the open annulus. Let us denote the covering map
\begin{eqnarray*}
% \nonumber to remove numbering (before each equation)
\pi\,:\, \mathbb{R}^2&\rightarrow& \mathbb{A} \\
(x,y)&\mapsto&(x+\mathbb{Z},y),
\end{eqnarray*}
and  $T$ the generator of the covering transformation group
\begin{eqnarray*}
% \nonumber to remove numbering (before each equation)
T\,:\, \mathbb{R}^2&\rightarrow& \mathbb{R}^2 \\
(x,y)&\mapsto&(x+1,y).
\end{eqnarray*}

When $F\in \mathrm{Homeo}_*(\mathbb{A})$, we have a simple way to
define the ``rotation vector'' given in \ref{subsec:rotation vector
of Le Calvez} if we observe that
$H_1(\mathbb{A},\mathbb{R})=\mathbb{R}$. We will say that a
positively recurrent point $z$ has a \emph{rotation number}
$\rho_{\mathbb{A},\widetilde{F}}(z)$ for a lift $\widetilde{F}$ of
$F$ to the universal cover $\mathbb{R}^2$ of $\mathbb{A}$, if for
every subsequence $\{F^{n_k}(z)\}_{k\geq 1}$ of $\{F^n(z)\}_{n\geq
1}$ which converges to $z$, we have
\[\lim_{k\rightarrow+\infty}\frac{p_1\circ
\widetilde{F}^{n_k}(\widetilde{z})-p_1(\widetilde{z})}{n_k}=\rho_{\mathbb{A},\widetilde{F}}(z)\]
for every $\widetilde{z}\in \pi^{-1}(z)$, where $p_1:(x,y)\mapsto x$
is the first projection. We denote the set of rotation numbers of
positively recurrent points of $F$ for $\widetilde{F}$ as
$\mathrm{Rot}(\widetilde{F})$. In particular, the rotation number
$\rho_{\mathbb{A},\widetilde{F}}(z)$ always exists when $z$ is a
fixed point of $F$. We denote the set of rotation numbers of fixed
points of $F$ as $\mathrm{Rot}_{\mathrm{Fix}(F)}(\widetilde{F})$.

It is well known that a positively recurrent point of $F$ is also a
positively recurrent point of $F^q$ for all $q\in \mathbb{N}$ (we
give a proof in Appendix, see Lemma \ref{subsec:positively
recurrent}). By the definition of rotation number, we easily get
that $\mathrm{Rot}(\widetilde{F})$ satisfies the following
elementary properties.
\begin{enumerate}\label{prop:ROT}
  \item[1.] $\mathrm{Rot}(T^k\circ \widetilde{F})=\mathrm{Rot}(\widetilde{F})+k$
  for every $k\in \mathbb{Z}$;
  \item[2.] $\mathrm{Rot}(\widetilde{F}^q)=q\mathrm{Rot}(\widetilde{F})$ for every $q\geq 1$.
\end{enumerate}\smallskip
%\begin{rem}\label{rem:a result of the fundamental group of annulus}
%%Let $\mathbb{A}$ be an open annulus.
%If $I=\{h_t\}_{t\in [0,1]}$ with $h_0=h_1=\mathrm{Id}_{\mathbb{A}}$
%is a loop in $\mathrm{Homeo}_*(\mathbb{A})$, write $[I]_1\in
%\pi_1(\mathrm{Homeo}_*(\mathbb{A}))$ for the homotopy class of $I$.
%It is well known that $\pi_1(\mathrm{Homeo}_*(\mathbb{A}))\simeq
%\mathbb{Z}$ (see \cite{H2}), hence we can write
%$\pi_1(\mathrm{Homeo}_*(\mathbb{A}))=\bigcup_{k\in\mathbb{Z}}\mathscr{C}_k$
%where $\mathscr{C}_k$ is the class which satisfies that for
%$[I]_1\in \mathscr{C}_k$, any lift $\widetilde{I}$ of $I$ to the
%universal covering space $\widetilde{\mathbb{A}}$ satisfies
%$\widetilde{h}_1-\widetilde{h}_0=T^k$.
%\end{rem}

\subsection{Linking number of contractible fixed points}

\subsubsection{}We begin by recalling some results about identity isotopies,
which will be often used in the literature.

\begin{rem}\label{rem: contractible fixed point and isotopy}
Suppose that $M$ is an oriented compact surface and that $F$ is the
time-one map of an identity isotopy $I=(F_t)_{t\in[0,1]}$ on $M$.
When $z\in \mathrm{Fix}_{\mathrm{Cont},I}(F)$, there is another
identity isotopy $I'=(F'_t)_{t\in[0,1]}$ homotopic to $I$ with fixed
endpoints such that $I'$ fixes $z$ (see, for example,
\cite[Proposition 4.1]{J}), that is, there is a continuous map $H:
[0,1]\times[0,1]\times M\rightarrow M$ such that
\begin{itemize}
\item $H(0,t,z)=F_t(z)$ and $H(1,t,z)=F'_t(z)$ for all $t\in[0,1]$;
\item $H(s,0,z)=\mathrm{Id}_{M}(z)$ and $H(s,1,z)=F(z)$ for all $s\in[0,1]$;
\item $F'_t(z)=z$ for all $t\in[0,1]$.
\end{itemize}

\end{rem}

\begin{lem}\label{rem:identity isotopies fix three points on sphere}
Let $\mathbf{S}^2$ be the 2-sphere and $I=(F_t)_{t\in[0,1]}$ be an
identity isotopy on $\mathbf{S}^2$. For every three different fixed
points $z_i$ ($i=1,2,3$) of $F_1$, there exists another identity
isotopy $I'=(F'_t)_{t\in[0,1]}$ from $\mathrm{Id}_{\mathbf{S}^2}$ to
$F_1$ such that $I'$ fixes $z_i$ $(i=1,2,3)$.
\end{lem}
\begin{proof}We identify the sphere $\mathbf{S}^2$ to the Riemann sphere
$\mathbb{C}\cup\{\infty\}$. The M\"{o}bius transformation $\mathcal
{M}(z)=\frac{az+b}{cz+d}$ maps the triple $(v_1,v_2,v_3)$ to the
triple $(\omega_1,\omega_2,\omega_3)$ (see Chapter 3 of \cite{N} for
a beautifully illustrated introduction to M\"{o}bius
transformations) where
\begin{equation}\label{eq:mobius trsformation}
a=\det\left(
     \begin{array}{ccc}
       v_1\omega_1 & \omega_1 & 1 \\
       v_2\omega_2 & \omega_2 & 1 \\
       v_3\omega_3 & \omega_3 & 1 \\
     \end{array}
   \right)\quad b=\det\left(
     \begin{array}{ccc}
       v_1\omega_1 & v_1 & \omega_1 \\
       v_2\omega_2 & v_2 & \omega_2\\
       v_3\omega_3 & v_3 & \omega_3 \\
     \end{array}
   \right)$$ $$c=\det\left(
     \begin{array}{ccc}
       v_1 & \omega_1 & 1 \\
       v_2 & \omega_2 & 1 \\
       v_3 & \omega_3 & 1 \\
     \end{array}
   \right)\qquad d=\det\left(
     \begin{array}{ccc}
       v_1\omega_1 & v_1 & 1 \\
       v_2\omega_2 & v_2 & 1 \\
       v_3\omega_3 & v_3 & 1 \\
     \end{array}
   \right).\end{equation}
If one of the points $v_i$ or $w_i$ in \ref{eq:mobius trsformation}
is $\infty$, then we first divide all four determinants by this
variable and then take the limit as the variable approaches
$\infty$. Replacing $v_i$, $w_i$ by $v_i(t)=F_t(z_i)$ and
$w_i(t)=z_i$ ($i=1,2,3$ and $t\in[0,1]$) in the matrices above, we
get the \emph{matrix functions} $a_t$, $b_t$, $c_t$ and $d_t$.

Let $$\mathcal {M}(t,z)=\frac{a_tz+b_t}{c_tz+d_t}$$ and
$$I'(z)(t)=F'_t(z)=\mathcal {M}(t,F_t(z)).$$ By the construction,
$I'$ is an isotopy of $\mathbf{S}^2$ from
$\mathrm{Id}_{\mathbf{S}^2}$ to $F_1$ that fixes $z_i$ ($i=1,2,3$).
\end{proof}
As a consequence, we have the following corollary.

\begin{cor}\label{cor:identity isotopy fixes two points}
Let $I=(F_t)_{t\in[0,1]}$ be an identity isotopy on $\mathbb{C}$.
For any two different fixed points $z_1$ and $z_2$ of $F_1$, there
exists another identity isotopy $I'$ from $\mathrm{Id}_{\mathbb{C}}$
to $F_1$ such that $I'$ fixes $z_1$ and $z_2$.
\end{cor}

\begin{rem}\label{rem:some result of of sphere delete three points}
Let $z_i\in \mathbf{S}^2$ ($i=1,2,3$) and
$\mathrm{Homeo}_*(\mathbf{S}^2,{z_1,z_2,z_3})$ be the identity
component of the space of all homeomorphisms of $\mathbf{S}^2$
leaving $z_i$ ($i=1,2,3$) pointwise fixed (for the compact-open
topology). It is well known that
$\pi_1(\mathrm{Homeo}_*(\mathbf{S}^2,{z_1,z_2,z_3}))=0$ (see
\cite{H2}, \cite{Han}). It implies that any two identity isotopies
$I ,I' \subset \mathrm{Homeo}_*(\mathbf{S}^2,{z_1,z_2,z_3})$ with
fixed endpoints are homotopic. As a consequence, let
$\mathrm{Homeo}_*(\mathbb{C},{z_1,z_2})$ be the identity component
of the space of all homeomorphisms of $\mathbb{C}$ leaving two
different points $z_1$ and $z_2$ of $\mathbb{C}$ pointwise fixed, we
have $\pi_1(\mathrm{Homeo}_*(\mathbb{C},{z_1,z_2}))=0$.
\end{rem}\smallskip

\subsubsection{}\label{sec:linking number in the simple connected
case} Let $M$ be a surface that is homeomorphic to the complex plane
 $\mathbb{C}$ and $I=(F_t)_{t\in[0,1]}$ be an identity isotopy on
$M$. Let us define the \emph{linking number} $i_I(z,z')\in
\mathbb{Z}$ for every two different fixed points $z$ and $z'$ of
$F_1$. It is the degree of the map $\xi: \mathbf{S}^1\rightarrow
\mathbf{S}^1$ defined by
\[\xi(e^{2i\pi t})=\frac{h\circ F_{t}(z')-h\circ F_{t}(z)}
{|h\circ F_{t}(z')-h\circ F_{t}(z)|}\,,\] where
$h:M\rightarrow\mathbb{C}$ is a homeomorphism. The linking number
does not depend on the chosen $h$.

It is well known that $U(1)$ is a strong deformation retract of
$\mathrm{Homeo}_*(\mathbb{C})$ (see \cite{Kn} or \cite[Theorem
2.9]{Ler}).
%$\pi_1(\mathrm{Home}_0(\mathbb{C}))\simeq\mathbb{Z}$ .
Consider the isotopy $R=(r_t)_{t\in[0,1]}$ where
$r_t=\mathrm{e}^{2i\pi t}$. If $I=(F_t)_{t\in[0,1]}$ is an identity
isotopy and $k\in \mathbb{Z}$, we can define the identity isotopy
$R^kI$ by concatenation. If $I'=(F'_t)_{t\in[0,1]}$ is another
identity isotopy with $F'_1=F_1$, then there exists a unique integer
$k$ such that $I'$ is homotopic to $R^kI$.\smallskip

Therefore, if $I=(F_t)_{t\in[0,1]}$ and $I'=(F'_t)_{t\in[0,1]}$ are
two identity isotopies on $M$ with $F'_1=F_1$, then there exist
$k\in\mathbb{Z}$ such that $i_{I'}(z,z')=i_{I}(z,z')+k$ for any
distinct fixed points $z'$ and $z'$ of $F_1$.
\smallskip

%For any fixed point $z$ of $F_1$, we get by restriction a
%homeomorphism $F_z$ of $A_z=M\setminus\{z\}$. Denote $\pi_z:
%\widetilde{A}_{z}\rightarrow A_z$ the universal cover of $A_z$ and
%$T_z: \widetilde{A}_z\rightarrow A_z$ the generator of the covering
%transformation group defined by the orientated boundary of a disk
%whose center is on $z$. We can always choose an isotopy
%$I'=\{F'_t\}_{t\in[0,1]}$ from $\mathrm{Id}_M$ to $F_1$, homotopic
%to $I$ and fixing $z$ (by Remark \ref{rem: contractible fixed point
%and isotopy} or, for example, $F_t'=(h^{-1}\circ\tau_t\circ h)\circ
%F_t$ where $\tau_t$ is the unique translation of $\mathbb{C}$ that
%map $h\circ F_{t}(z)$ to $h(z)$). As $I'$ is homotopic to $I$, we
%have $i_I(z,z')=i_{I'}(z,z')$ for all $z'\in
%\mathrm{Fix}(F_1)\setminus\{z\}$. Consider the lift
%$\widetilde{F}_{I',z}$ of $F_z$ that is the time-one map of the
%lifted identity isotopy of $I'$ on $A_z$. Then $i_I(z,z')$ is
%nothing else but $\rho_{A_z,\widetilde{F}_{I',z}}(z')$. Moreover, we
%can also defined the rotation number
%$\rho_{A_z,\widetilde{F}_{I',z}}(z')$ when $z'\in
%\mathrm{Rec}^+(F_1)\setminus\{z\}$, we will use it later.

\subsubsection{}\label{sec:linking number in the special case}
Let $F$ be the time-one map of an identity isotopy
$I=(F_t)_{t\in[0,1]}$ on a closed oriented  surface $M$ of genus
$g\geq 1$ and $\widetilde{F}$ be the time-one map of the lifted
identity isotopy $\widetilde{I}=(\widetilde{F}_t)_{t\in[0,1]}$ on
the universal cover $\widetilde{M}$ of $M$. When $g>1$, it is well
known that $\pi_1(\mathrm{Homeo}_*(M))\simeq0$ (\cite{H2}). It
implies that any two identity isotopies $I ,I' \subset
\mathrm{Homeo}_*(M)$ with fixed endpoints are homotopic. Hence, $I$
is unique up to homotopy, it implies that $\widetilde{F}$ is
uniquely defined and does not depend on the choice of the isotopy
from $\mathrm{Id}_M$ to $F$. When $g=1$,
$\pi_1(\mathrm{Homeo}_*(M))\simeq\mathbb{Z}^2$ (see \cite{H1}),
$\widetilde{F}$ depends on the isotopy $I$. The universal cover
$\widetilde{M}$ is homeomorphic to $\mathbb{C}$.

Let $\pi: \widetilde{M}\rightarrow M$ be the covering map  and $G$
be the covering transformation group. Denote respectively by
$\Delta$ and $\widetilde{\Delta}$ the diagonal of
$\mathrm{Fix}_{\mathrm{Cont},I}(F)\times
\mathrm{Fix}_{\mathrm{Cont},I}(F)$ and the diagonal of
$\mathrm{Fix}(\widetilde{F})\times \mathrm{Fix}(\widetilde{F})$.
Endow the surface $M$ with a Riemannian metric and denote by $d$ the
distance induced by the metric. Lift the Riemannian metric to
$\widetilde{M}$ and write $\widetilde{d}$ for the distance induced
by the metric.

We define the \emph{linking number}
$i(\widetilde{F};\widetilde{z},\widetilde{z}\,')$ for every pair
$(\widetilde{z},\widetilde{z}\,')\in(\mathrm{Fix}(\widetilde{F})\times
\mathrm{Fix}(\widetilde{F}))\setminus\widetilde{\Delta}$ as
\begin{equation}\label{eq:linking number for two fixed points}
i(\widetilde{F};\widetilde{z},\widetilde{z}\,')=i_{\widetilde{I}}
(\widetilde{z},\widetilde{z}\,').
\end{equation}
This is a special case of the linking number that we have defined in
\ref{sec:linking number in the simple connected case}.

We give some properties of
$i(\widetilde{F};\widetilde{z},\widetilde{z}\,')$ as
follows.\smallskip

\begin{description}
  \item[(P1)] $i(\widetilde{F};\widetilde{z},\widetilde{z}\,')$ is locally constant
on $(\mathrm{Fix}(\widetilde{F})\times
\mathrm{Fix}(\widetilde{F}))\setminus\widetilde{\Delta}$;
  \item[(P2)] $i(\widetilde{F};\widetilde{z},\widetilde{z}\,')$ is invariant by
covering transformation, that is,
$$\qquad i(\widetilde{F};\alpha(\widetilde{z}),\alpha(\widetilde{z}\,'))=
i(\widetilde{F};\widetilde{z},\widetilde{z}\,')\quad\mathrm{for\,\,\,
every}\,\,\,\alpha\in G;$$
  \item[(P3)] $i(\widetilde{F};\widetilde{z},\widetilde{z}\,')=0$
if $\pi(\widetilde{z})=\pi(\widetilde{z}\,')$;
  \item[(P4)] there exists $K$ such that $i(\widetilde{F};\widetilde{z},\widetilde{z}\,')=0$
if $\widetilde{d}(\widetilde{z},\widetilde{z}\,')\geq K$.
\end{description}
\smallskip

Indeed, the property P1 is true by continuity. The property P2 is
true because the linking number does not change when you replace $h$
by $h\circ\alpha$ (see \ref{sec:linking number in the simple
connected case}). By Remark \ref{rem: contractible fixed point and
isotopy}, we can choose an isotopy $I'$ that is homotopic to $I$ and
fixes $\pi(\widetilde{z})$, then the lift $\widetilde{I}\,'$ of $I'$
fixes $\widetilde{z}$ and $\widetilde{z}\,'$. Thus the property P3
holds. Finally, let
$$K=\sup\{\,\widetilde{d}(\widetilde{F}_t(\widetilde{z}),\widetilde{F}_{t'}(\widetilde{z}))\mid(t,t',\widetilde{z})\in[0,1]^2\times\mathrm{Fix}(\widetilde{F})\}
.$$ The value $K$ is well defined because
$\mathrm{Fix}_{\mathrm{Cont},I}(F)=\pi(\mathrm{Fix}(\widetilde{F}))$
is compact and $\widetilde{F}_t\circ\alpha=\alpha\circ
\widetilde{F}_t$ for all $t\in[0,1]$ and $\alpha\in G$. Obviously,
when $\widetilde{d}(\widetilde{z},\widetilde{z}\,')\geq 3K$,
$i(\widetilde{F};\widetilde{z},\widetilde{z}\,')=0$. We get the
property P4.
\subsubsection{}\label{sec:linking number of a and b}Now we define the \emph{linking number}
$I(\widetilde{F};z,z')\in \mathbb{Z}$ for every distinct
contractible fixed points $z$ and $z'$ of $F$ as follows:
$$I(\widetilde{F};z,z')=\sum_{\alpha\in
G}i(\widetilde{F};\widetilde{z},\alpha(\widetilde{z}\,')),$$ where
$\widetilde{z}\in\pi^{-1}(z)$ and $\widetilde{z}\,'\in\pi^{-1}(z')$.
The sum is well defined since there are finite nonzero items in the
sum (by the property P4). Obviously, $I(\widetilde{F};z,z')$ does
not depend on the chosen lifts $\widetilde{z}$ and
$\widetilde{z}\,'$ (by the property P2) and is locally constant on
$(\mathrm{Fix}_{\mathrm{Cont},I}(F)\times
\mathrm{Fix}_{\mathrm{Cont},I}(F))\setminus\Delta$ (by the property
P1 and the fact that there is a finite number of nonzero in the
sum).

\begin{prop}\label{rem:the equivalence of two linking numbers}
The following statements are equivalent
\begin{enumerate}
\item The set of linking numbers
$i(\widetilde{F};\widetilde{z},\widetilde{z}\,')$ where
$(\widetilde{z},\widetilde{z}\,')\in(\mathrm{Fix}(\widetilde{F})\times
\mathrm{Fix}(\widetilde{F}))\setminus\widetilde{\Delta}$ is bounded;
\item The set of linking numbers $I(\widetilde{F};z,z')$ where
$(z,z')\in(\mathrm{Fix}_{\mathrm{Cont},I}(F)\times
\mathrm{Fix}_{\mathrm{Cont},I}(F))\setminus\Delta$ is bounded.
\end{enumerate}
\end{prop}

\begin{proof}(1)$\,\Rightarrow\,$(2).\quad Let $N$ be a bound of
that set. According to the property P4, there exists a positive
integer $K$ such that $\sharp\{\alpha\in G\mid
i(\widetilde{F};\widetilde{z},\alpha(\widetilde{z}\,'))\neq0\}\leq
K$ for all $\widetilde{z},\widetilde{z}\,'\in
\mathrm{Fix}(\widetilde{F})$. Then we have
$|I(\widetilde{F};z,z')|\leq\sum_{\alpha\in
G}|i(\widetilde{F};\widetilde{z},\alpha(\widetilde{z}\,'))|\leq
KN$.\smallskip

(2)$\,\Rightarrow\,$(1).\quad If the statement (1) does not hold,
then there exist two sequences $\{\widetilde{z}_n\}_{n\geq1}$ and
$\{\widetilde{z}_n\!'\}_{n\geq1}$ such that
$\lim\limits_{n\rightarrow+\infty}i(\widetilde{F};\widetilde{z}_n,\widetilde{z}_n\!')=+\infty$
\, or
$\lim\limits_{n\rightarrow+\infty}i(\widetilde{F};\widetilde{z}_n,\widetilde{z}_n\!')=-\infty$.
We suppose
$\lim\limits_{n\rightarrow+\infty}i(\widetilde{F};\widetilde{z}_n,\widetilde{z}_n\!')=+\infty$,
 the other case being similar. As $M$ is compact, there is a
subsequence $\{\widetilde{z}_{n_k}\}_{k\geq1}$ of
$\{\widetilde{z}_n\}_{n\geq1}$ and a subsequence
$\{\widetilde{z}_{n_k}\!'\}_{k\geq1}$ of
$\{\widetilde{z}_n\!'\}_{n\geq1}$ such that
$\pi(\widetilde{z}_{n_k})\rightarrow z$ and
$\pi(\widetilde{z}_{n_k}\!')\rightarrow z'$ when $k\rightarrow
+\infty$. By the continuity of $I$, we have $z,z'\in
\mathrm{Fix}_{\mathrm{Cont},I}(F)$. Fix two points $\widetilde{z}\in
\pi^{-1}(z)$ and $\widetilde{z}\,'\in \pi^{-1}(z')$. We can choose a
sequence $\{\alpha_k\}_{k\geq1}\subset G$
 such that $\alpha_k(\widetilde{z}_{n_k})\rightarrow \widetilde{z}$ as
$k\rightarrow +\infty$. By the property P2, we have
\begin{equation}\label{eq:I implies i}
    \lim_{k\rightarrow+\infty}i(\widetilde{F};\alpha_k(\widetilde{z}_{n_k}),
    \alpha_k(\widetilde{z}_{n_k}\!'))=\lim_{k\rightarrow+\infty}i(\widetilde{F};
\widetilde{z}_{n_k},\widetilde{z}_{n_k}\!')=+\infty.
\end{equation}

The property P4 implies that the sequence
$\{\alpha_k(\widetilde{z}_{n_k}\!')\}_{k\geq1}$ is bounded, then the
property P1 tell us that
$\lim\limits_{k\rightarrow+\infty}\alpha_k(\widetilde{z}_{n_k}\!')=\widetilde{z}$.
 By the properties P1 and P3, we have
$$i(\widetilde{F};\alpha_k(\widetilde{z}_{n_k}),
    \alpha(\alpha_k(\widetilde{z}_{n_k}\!')))=0$$ for all $\alpha\in G^*$ when $k$ is large enough. Thus we have
$\lim\limits_{k\rightarrow+\infty}I(\widetilde{F};
\pi(\widetilde{z}_{n_k}),\pi(\widetilde{z}_{n_k}\!'))=+\infty.$
\end{proof}

\subsubsection{}\label{sec:linking number of a and b,c} In the rest of the paper, when we take two distinct fixed
points $\widetilde{a}$ and $\widetilde{b}$ of $\widetilde{F}$, it
does not mean that $\pi(\widetilde{a})$ and $\pi(\widetilde{b})$ are
distinct.

Fix two distinct fixed points $\widetilde{a}$ and $\widetilde{b}$ of
$\widetilde{F}$. For any
$z\in\mathrm{Fix}_{\mathrm{Cont},I}(F)\setminus\pi(\{\widetilde{a},\widetilde{b}\})$,
we define \emph{the linking number of $z$ for $\widetilde{a}$ and
$\widetilde{b}$} as
$$i(\widetilde{F};\widetilde{a},\widetilde{b},z)=\sum_{\pi(\widetilde{z})=z}
\left(i(\widetilde{F};\widetilde{a},\widetilde{z})-i(\widetilde{F};\widetilde{b},\widetilde{z})\right)
=I(\widetilde{F};\pi(\widetilde{a}),z)-I(\widetilde{F};\pi(\widetilde{b}),z)
.$$

We will extend it to the case where $z\in
\mathrm{Rec}^+(F)\setminus\pi(\{\widetilde{a},\widetilde{b}\})$ in
Section \ref{sec:define a new linking number}. Note here that the
linking number only depends on $\pi(\widetilde{a})$ and
$\pi(\widetilde{b})$ in the case where $z$ is a contractible fixed
point of $F$, but the extension of
$i(\widetilde{F};\widetilde{a},\widetilde{b},z)$ for $z\in
\mathrm{Rec}^+(F)\setminus\mathrm{Fix}_{\mathrm{Cont},I}(F)$ in
Section \ref{sec:define a new linking number} depends on the choices
of $\widetilde{a}$ and $\widetilde{b}$.\bigskip

\subsection{The weak boundedness property and the boundedness
property}\label{subsec:boundedness} We can compactify
$\widetilde{M}$ into a sphere by adding a point $\infty$ at infinity
and the lift $\widetilde{F}$ may be extended by fixing this point.
In all the text, we write
$\mathbf{S}=\widetilde{M}\sqcup\{\infty\}$. If $\widetilde{a}$ and
$\widetilde{b}$ are distinct fixed points of $\widetilde{F}$, the
restriction of $\widetilde{F}$ to the annulus
$A_{{\widetilde{a}},{\widetilde{b}}}=\mathbf{S}\setminus\{\widetilde{a},\widetilde{b}\}$
denoted by $\widetilde{F}_{{\widetilde{a}},{\widetilde{b}}}$, has a
natural lift $\widehat{F}_{\widetilde{a},\widetilde{b}}$ to the
universal cover $\widehat{A}_{\widetilde{a},\widetilde{b}}$ of
$A_{{\widetilde{a}},{\widetilde{b}}}$ that fixes the preimages of
$\infty$ by the covering projection
$\widehat{\pi}_{{\widetilde{a}},{\widetilde{b}}}:
\widehat{A}_{{\widetilde{a}},{\widetilde{b}}}\rightarrow
A_{{\widetilde{a}},{\widetilde{b}}}$. Denote by
$T_{{\widetilde{a}},{\widetilde{b}}}$ the generator of
$H_1(A_{\widetilde{a},\widetilde{b}},\mathbb{R})$ defined by the
oriented boundary of a small disk centered at $\widetilde{a}$.

If $\pi(\widetilde{a})\neq \pi(\widetilde{b})$, by Remark \ref{rem:
contractible fixed point and isotopy}, there exist two identity
isotopies $I'$  and  $I''$ homotopic to $I$ with fixed endpoints
such that $I'$ fixes $\pi(\widetilde{a})$ and $I''$ fixes
$\pi(\widetilde{b})$. However, in general, there does not exist an
identity isotopy $I'''$ homotopic to $I$ with fixed endpoints such
that $I'''$ fixes both $\pi(\widetilde{a})$ and
$\pi(\widetilde{b})$, which is an obstacle that prevents us to
generalize the action function to a more general cases (see Section
\ref{subsec:the second extension of the classical action}). That is
a reason that we introduce the following lemma.

\begin{lem}\label{lem:P} If $\widetilde{z}$ is another fixed point of $\widetilde{F}$ which
is different from $\widetilde{a}$, $\widetilde{b}$ and $\infty$,
then the rotation number of $\widetilde{z}\in
A_{{\widetilde{a}},{\widetilde{b}}}$ for the natural lift
$\widehat{F}_{{\widetilde{a}},{\widetilde{b}}}$ is equal to
$i(\widetilde{F};\widetilde{a},\widetilde{z})-i(\widetilde{F};\widetilde{b},\widetilde{z})$,
that is
$$\rho_{A_{{\widetilde{a}},{\widetilde{b}}},\widehat{F}_{{\widetilde{a}},{\widetilde{b}}}}
(\widetilde{z})=i(\widetilde{F};\widetilde{a},\widetilde{z})-i(\widetilde{F};\widetilde{b},\widetilde{z}).$$
\end{lem}
\begin{proof}
If $J$ and $J'$ are two isotopies of $\widetilde{M}$ from
$\text{Id}_{\widetilde{M}}$ to $\widetilde{F}$, then there exists
$k\in \mathbb{Z}$ such that $i_J=i_{J'}+k$ (see \ref{sec:linking
number in the simple connected case}). Therefore, if
$\widetilde{a},\widetilde{b}$ and $\widetilde{z}$ are distinct fixed
points of $\widetilde{F}$, the quantity
$i_{J}(\widetilde{a},\widetilde{z})-i_{J}(\widetilde{b},\widetilde{z})$
is independent of $J$ and hence equals to
$i(\widetilde{F};\widetilde{a},\widetilde{z})-i(\widetilde{F};\widetilde{b},\widetilde{z})$
if we choose $J=\widetilde{I}$ where $\widetilde{I}$ is the identity
isotopy in \ref{sec:linking number in the special case}. Suppose now
that $J$ is an isotopy that fixes $\widetilde{a}$ and
$\widetilde{b}$. The trajectory $J(\widetilde{z})$ defines a loop in
the sphere $\mathbf{S}$. If $\gamma_{\widetilde{a},\infty}$ and
$\gamma_{\widetilde{b},\infty}$ are two paths in $\mathbf{S}$ that
join respectively $\widetilde{a}$ and $\widetilde{b}$ to $\infty$,
we have
$i_J(\widetilde{a},\widetilde{z})=\gamma_{\widetilde{a},\infty}\wedge
J(\widetilde{z})$ and
$i_J(\widetilde{b},\widetilde{z})=\gamma_{\widetilde{b},\infty}\wedge
J(\widetilde{z})$. The loop $J(\widetilde{z})$ being homologous to
zero in $\mathbf{S}$, we deduce that
$i(\widetilde{F};\widetilde{a},\widetilde{z})-i(\widetilde{F};\widetilde{b},\widetilde{z})
=i_J(\widetilde{a},\widetilde{z})-i_J(\widetilde{b},\widetilde{z})
=\gamma_{\widetilde{a},\widetilde{b}}\wedge J(\widetilde{z})$, where
$\gamma_{\widetilde{a},\widetilde{b}}$ is a path in $\mathbf{S}$
that joins $\widetilde{a}$ to $\widetilde{b}$. Note that this
integer is nothing else but the rotation number of
\,$\widetilde{z}$\, for the lift
$\widehat{F}_{\widetilde{a},\widetilde{b}}$ defined by
$T_{{\widetilde{a}},{\widetilde{b}}}$.
\end{proof}

Remark here that, by the definition
$i(\widetilde{F};\widetilde{a},\widetilde{b},z)$ of Section
\ref{sec:linking number of a and b,c}, we have
\begin{equation*}\label{def:linking number of triple of the sphere}
i(\widetilde{F};\widetilde{a},\widetilde{b},z)=\sum_{\pi(\widetilde{z})=z}
i(\widetilde{F};\widetilde{a},\widetilde{z})-i(\widetilde{F};\widetilde{b},\widetilde{z})
=\sum_{\pi(\widetilde{z})=z}\rho_{A_{{\widetilde{a}},{\widetilde{b}}},\widehat{F}_{{\widetilde{a}},{\widetilde{b}}}}
(\widetilde{z}).
\end{equation*}

\begin{defn}\label{def:wb property and b property} We say that
$I$ satisfies \emph{the weak boundedness property at
$\widetilde{a}\in\mathrm{Fix}(\widetilde{F})$ (WB-property at
$\widetilde{a}$)} if $i(\widetilde{F};\widetilde{a},\widetilde{b})$
is uniformly bounded for all fixed point $\widetilde{b}\in
\mathrm{Fix}(\widetilde{F})\setminus\{\widetilde{a}\}$. We say that
$I$ satisfies \emph{the weak boundedness property (WB-property)} if
it satisfies the weak boundedness property at every
$\widetilde{a}\in\mathrm{Fix}(\widetilde{F})$. We say that $I$
satisfies \emph{the boundedness property (B-property)} if the set of
$i(\widetilde{F};\widetilde{a},\widetilde{b})$ where
$(\widetilde{a},\widetilde{b})\in(\mathrm{Fix}(\widetilde{F})\times
\mathrm{Fix}(\widetilde{F}))\setminus\widetilde{\Delta}$ is bounded.
\end{defn}

%\begin{rem}By the properties P1--P4 and Proposition \ref{rem:the
%equivalence of two linking numbers}, replacing
%$i(\widetilde{F};\widetilde{a},\widetilde{b})$ with
%$I(\widetilde{F};a,b)$ where
%$(a,b)\in(\mathrm{Fix}_{\mathrm{Cont},I}(F)\times
%\mathrm{Fix}_{\mathrm{Cont},I}(F))\setminus\Delta$, we can get a
%definition equivalent to Definition \ref{def:wb property and b
%property}.
%\end{rem}

%In Appendix of this paper, we construct an isotopy $I$ of $M$ such
%that $I$ satisfies the weak bonudedness property at $\widetilde{a}$
%but not $\widetilde{b}$ and
%$\rho_{A_{{\widetilde{a}},{\widetilde{b}}},\widehat{F}_{{\widetilde{a}},{\widetilde{b}}}}
%(\widetilde{c})$ is not uniformly bounded for $\widetilde{c}\in
%\mathrm{Fix}(\widetilde{F})\setminus\{\widetilde{a},\widetilde{b}\}$
%(see Example \ref{exem:not satisfies the property WB and not
%differntial on o}). However, we have the following equivalence.

\begin{lem}\label{lem:rotation number and weak boundess at two point}
Let $\widetilde{a}$ and $\widetilde{b}$ be two distinct fixed points
of $\widetilde{F}$. The following statements are equivalent
\begin{enumerate}
\item $I$ satisfies the WB-property at
$\widetilde{a}$ and $\widetilde{b}$;
  \item there exists $K\geq0$ such that $\left|\rho_{A_{{\widetilde{a}},{\widetilde{b}}},\widehat{F}_{{\widetilde{a}},{\widetilde{b}}}}
(\widetilde{c})\right|\leq K$  for all fixed point $\widetilde{c}\in
\mathrm{Fix}(\widetilde{F})\setminus\{\widetilde{a},\widetilde{b}\}$.
\end{enumerate}
\end{lem}
\begin{proof}
From Lemma \ref{lem:P}, we  have (1)$\,\Rightarrow\,$(2)
immediately. Next we prove (2)$\,\Rightarrow\,$(1) by contradiction.
Without loss of generality, we suppose that there exists a sequence
$\{\widetilde{c}_n\}_{n\geq1}\subset
\mathrm{Fix}(\widetilde{F})\setminus\{\widetilde{a},\widetilde{b}\}$
such that
$\lim\limits_{n\rightarrow+\infty}i(\widetilde{F};\widetilde{a},\widetilde{c}_{n})=+\infty$
(the case
$\lim\limits_{n\rightarrow+\infty}i(\widetilde{F};\widetilde{a},\widetilde{c}_{n})=-\infty$
is similar). Lemma \ref{lem:P} and the hypothesis (2) imply that
$\lim\limits_{n\rightarrow+\infty}i(\widetilde{F};\widetilde{b},\widetilde{c}_{n})=+\infty$.
The property P4 implies that the sequence
$\{\widetilde{c}_n\}_{n\geq1}$ is bounded. The property P1 implies
that
$\lim\limits_{n\rightarrow+\infty}\widetilde{c}_n=\widetilde{a}$ and
$\lim\limits_{n\rightarrow+\infty}\widetilde{c}_n=\widetilde{b}$,
which gives a contradiction.
\end{proof}

\begin{lem}\label{lem:diffeomorphism on a and b}
For any two distinct fixed points $\widetilde{a}$ and
$\widetilde{b}$ of $\widetilde{F}$, if $F$ and $F^{-1}$ are
differentiable at $\pi(\widetilde{a})$ and $\pi(\widetilde{b})$,
then
$\rho_{A_{{\widetilde{a}},{\widetilde{b}}},\widehat{F}_{{\widetilde{a}},{\widetilde{b}}}}
(\widetilde{z})$ is uniformly bounded for any
$\widetilde{z}\in\mathrm{Rec}^+(\widetilde{F})\setminus\{\widetilde{a},\widetilde{b}\}$
if it exists. In particular,
$\rho_{A_{{\widetilde{a}},{\widetilde{b}}},\widehat{F}_{{\widetilde{a}},{\widetilde{b}}}}
(\widetilde{c})$ is uniformly bounded for any fixed point
$\widetilde{c}\in\mathrm{Fix}(\widetilde{F})\setminus\{\widetilde{a},\widetilde{b}\}$.
\end{lem}

\begin{proof}
Let $\bar{A}_{\widetilde{a},\widetilde{b}}=S_{\widetilde{a}}\sqcup
A_{\widetilde{a},\widetilde{b}}\sqcup S_{\widetilde{b}}$ where
$S_{\widetilde{a}}$ and $S_{\widetilde{b}}$ are the tangent unit
circles at $\widetilde{a}$ and $\widetilde{b}$ such that
$\bar{A}_{\widetilde{a},\widetilde{b}}$ is the natural
compactification of $A_{\widetilde{a},\widetilde{b}}$. The maps $F$
and $F^{-1}$ are differentiable at $\pi(\widetilde{a})$ and
$\pi(\widetilde{b})$. Hence the lift $\widetilde{F}$ (resp.
$\widetilde{F}^{-1}$) of $F$ (resp. $F^{-1}$) to $\widetilde{M}$ is
differentiable at $\widetilde{a}$ and $\widetilde{b}$. By the method
of \emph{blowing-up}, it induces a homeomorphism $f:
\bar{A}_{\widetilde{a},\widetilde{b}}\rightarrow
\bar{A}_{\widetilde{a},\widetilde{b}}$,

\begin{equation*}f(u)=
\begin{cases}\widetilde{F}_{\widetilde{a},\widetilde{b}}(u)& \textrm{when} \quad u\in
A_{\widetilde{a},\widetilde{b}}\\[6pt]\frac{D\widetilde{F}(\widetilde{a}).u}{|D\widetilde{F}(\widetilde{a}).u|}&
\textrm{when} \quad u\in
S_{\widetilde{a}}\\[6pt]\frac{D\widetilde{F}(\widetilde{b}).u}{|D\widetilde{F}(\widetilde{b}).u|}&
\textrm{when} \quad u\in S_{\widetilde{b}} .\end{cases}
\end{equation*}

The universal cover of $\bar{A}_{\widetilde{a},\widetilde{b}}$ is
$\mathbb{R}\times [0,1]$. We suppose that $\widehat{f}$ is the lift
of $f$ fixing the preimages of $\infty$ by the covering projection
$\widehat{\pi}_{{\widetilde{a}},{\widetilde{b}}}$. For any $u\in
\bar{A}_{\widetilde{a},\widetilde{b}}$, we have that
$p_1(\widehat{f}(\widehat{u}))-p_1(\widehat{u})$ is uniformly
bounded because $\bar{A}_{\widetilde{a},\widetilde{b}}$ is compact,
where $\widehat{u}$ is any lift of $u$. There exists $N$, depending
on $I$, such that for every
$\widehat{z}\in\widehat{A}_{\widetilde{a},\widetilde{b}}$, one has
$\left|p_1(\widehat{F}_{{\widetilde{a}},{\widetilde{b}}}(\widehat{z}))-
p_1(\widehat{z})\right|\leq N$. Moreover, for every $n\geq1$, we
have
$$\left|\frac{p_1\circ\widehat{F}_{\widetilde{a},\widetilde{b}}^n(\widehat{z})-
p_1(\widehat{z})}{n}\right|\leq\frac{1}{n}\sum_{i=0}^{n-1}\left|p_1\circ\widehat{F}_
{\widetilde{a},\widetilde{b}}^{i+1}(\widehat{z})-p_1\circ\widehat{F}_
{\widetilde{a},\widetilde{b}}^{i}(\widehat{z})\right|\leq N.$$

If $\widetilde{z}\in
\mathrm{Rec}^+(\widetilde{F}_{\widetilde{a},\widetilde{b}})$ and
$\rho_{A_{{\widetilde{a}},{\widetilde{b}}},\widehat{F}_{{\widetilde{a}},{\widetilde{b}}}}
(\widetilde{z})$ exists, by the definition of rotation number (see
\ref{subsec:the rotation number of an open annulus}), we deduce that
$\left|\rho_{A_{{\widetilde{a}},{\widetilde{b}}},\widehat{F}_{{\widetilde{a}},{\widetilde{b}}}}
(\widetilde{z})\right|\leq N$. We have completed the proof.
\end{proof}

 Observe that the proof of Lemma
\ref{lem:diffeomorphism on a and b} gives us an information about
how rotate not only the positively recurrent points of
$\widetilde{F}_{\widetilde{a},\widetilde{b}}$ but in fact every
point in $A_{\widetilde{a},\widetilde{b}}$, we will use this fact in
Section \ref{sec:i(F;a,b,z) is bounded}.
\bigskip

By Lemma \ref{lem:rotation number and weak boundess at two point}
and Lemma \ref{lem:diffeomorphism on a and b}, we have the following
proposition immediately.

\begin{prop}\label{prop:F is differentiable on every point of M, WB}
The WB-property is satisfied if $F\in\mathrm{Diff}(M)$.
\end{prop}
Remark that if $F$ and $F^{-1}$ are differentiable at
$\pi(\widetilde{a})$, similarly to the proof of Lemma
\ref{lem:diffeomorphism on a and b}, we can prove that $I$ satisfies
the WB-property at $\widetilde{a}$, from which Proposition
\ref{prop:F is differentiable on every point of M, WB} can be proven
directly. However, the proof of the current Lemma
\ref{lem:diffeomorphism on a and b} will be necessary for further
proofs of this paper, thus being adopted here.

\bigskip

Obviously, $I$ satisfies the B-property  if
$\sharp\mathrm{Fix}_{\mathrm{Cont},I}(F)<+\infty$. In Appendix, we
construct an isotopy $I=(F_t)_{0\leq t\leq1}$ such that $F=F_1$ is a
diffeomorphism of $M$ but does not satisfy the B-property. In that
example, we show that $F$ is not a $C^1$-diffeomorphism of $M$ (see
Example \ref{ex:the action of non C1-diffeo isnot bounded and
continuous}). If $F$ is a $C^1$-diffeomorphism of $M$, we have the
following result:

\begin{prop}\label{rem:any isotopy of M satisfies condition B}
The B-property  is satisfied if $F\in\mathrm{Diff}^1(M)$.
\end{prop}
Before proving Proposition \ref{rem:any isotopy of M satisfies
condition B}, we need the following lemma (\cite[Lemma 5.6]{BLFM}).

\begin{lem}\label{lem: c1 diffeomorphism on sphere}Let $h$ be a $C^1$-diffeomorphism of $\mathbf{S}^2$ and
$a\in\mathrm{Fix}(h)$. For all point $z\in\mathbf{S}^2$ different
from $a$ and its antipodal point, denote $\gamma_z$ the unique great
circle that passes through them and $a$, and denote $\gamma^-_z$
(resp. $\gamma^+_z$) the small (resp. large) arc of $\gamma_z$
joining $a$ and $z$. Then there exists a neighborhood $W$ of $a$ on
$\mathbf{S}^2$ such that for all $z\in\mathrm{Fix}(h)\cap W$, we
have $h(\gamma^-_z)\cap \gamma^+_z=\{z,a\}$.
\end{lem}

\begin{proof}[Proof of Proposition \ref{rem:any isotopy of M satisfies condition B}]
We only need to consider the case where
$\sharp\mathrm{Fix}_{\mathrm{Cont},I}(F)=+\infty$. To get a proof by
contradiction, according to Definition \ref{def:wb property and b
property}, we suppose that there exist a sequence of pairs
$\{(\widetilde{a}_n,\widetilde{b}_n)\}_{n\geq1}\subset(\mathrm{Fix}(\widetilde{F})\times
\mathrm{Fix}(\widetilde{F}))\setminus\widetilde{\Delta}$ such that
$\lim\limits_{n\rightarrow
+\infty}i(\widetilde{F};\widetilde{a}_n,\widetilde{b}_n)=+\infty$
(the case where $\lim\limits_{n\rightarrow
+\infty}i(\widetilde{F};\widetilde{a}_n,\widetilde{b}_n)=-\infty$ is
similar). By the property P2, we can suppose that the sequence
$\{\widetilde{a}_n\}_{n\geq1}$ is bounded by replacing
$\widetilde{a}_n$ and $\widetilde{b}_n$ with
$\alpha_n(\widetilde{a}_n)$ and $\alpha_n(\widetilde{b}_n)$ where
$\alpha_n\in G$ if necessary. The property P4 implies that the
sequence $\{\widetilde{b}_n\}_{n\geq1}$ is also bounded. Therefore,
by continuity,  we can suppose that
$\lim\limits_{n\rightarrow+\infty}\widetilde{a}_{n}=\widetilde{a}$
and
$\lim\limits_{n\rightarrow+\infty}\widetilde{b}_{n}=\widetilde{b}$
where $\widetilde{a}\in \mathrm{Fix}(\widetilde{F})$ and
$\widetilde{b}\in \mathrm{Fix}(\widetilde{F})$ by extracting
subsequences if necessary.
 According to the
property P1, we deduce that $\widetilde{a}=\widetilde{b}$. Moreover,
as $F$ is a diffeomorphism, so $I$ satisfies the WB-property at
$\widetilde{a}$. That is, there is a number $N_{\widetilde{a}}\geq0$
such that $|i(\widetilde{F};\widetilde{a},\widetilde{z})|\leq
N_{\widetilde{a}}$ for all $\widetilde{z}\in
\mathrm{Fix}(\widetilde{F})\setminus\{\widetilde{a}\}$. Hence, we
can suppose that $\widetilde{a}_n\neq \widetilde{a}$ and
$\widetilde{b}_n\neq \widetilde{a}$ for all $n$ by taking $n$ large
enough.

For every $n\geq1$, let $\widetilde{I}_n$ be an isotopy that fixes
$\widetilde{a}$ and $\widetilde{a}_{n}$ (Corollary \ref{cor:identity
isotopy fixes two points}). Then there exists $k_n$ such that
\begin{equation}\label{eq: B i n}
    i_{\widetilde{I}_n}(\widetilde{z},\widetilde{z}\,')=i(\widetilde{F};\widetilde{z},\widetilde{z}\,')+k_n
\end{equation}
for every two distinct fixed points $\widetilde{z}$ and
$\widetilde{z}\,'$ of $\widetilde{F}$ (see \ref{sec:linking number
in the simple connected case}).  Observing that
$i_{\widetilde{I}_n}(\widetilde{a},\widetilde{a}_{n})=0$ for every
$n$, so Equation \ref{eq: B i n} implies that $|k_n|\leq
N_{\widetilde{a}}$ and
$\lim\limits_{n\rightarrow+\infty}i_{\widetilde{I}_n}(\widetilde{a}_{n},\widetilde{b}_{n})=+\infty$.
Moreover, we have
$i_{\widetilde{I}_n}(\widetilde{a},\widetilde{b}_n)=i(\widetilde{F};\widetilde{a},\widetilde{b}_n)+k_n$,
hence $|i_{\widetilde{I}_n}(\widetilde{a},\widetilde{b}_n)|\leq
2N_{\widetilde{a}}$.

Consider the annulus
$A_{\widetilde{a},\widetilde{a}_n}=\mathbf{S}\setminus\{\widetilde{a},\widetilde{a}_n\}$
and $\widetilde{F}_{\widetilde{a},\widetilde{a}_n}$. By the proof of
Lemma \ref{lem:P}, we know that
$$\rho_{A_{{\widetilde{a}},{\widetilde{a}_n}},\widehat{F}_{{\widetilde{a}},{\widetilde{a}_n}}}
(\widetilde{b}_n)=i_{\widetilde{I}_n}(\widetilde{a},\widetilde{b}_n)-i_{\widetilde{I}_n}(\widetilde{a}_n,\widetilde{b}_n).$$
 Therefore,
 \begin{equation}\label{eq:B property n}
    \lim_{n\rightarrow+\infty}\rho_{A_{{\widetilde{a}},{\widetilde{a}_n}},\widehat{F}_{{\widetilde{a}},{\widetilde{a}_n}}}(\widetilde{b}_n)=-\infty.
 \end{equation}

Fix $q\geq1$, apply Lemma \ref{lem: c1 diffeomorphism on sphere} to
$\widetilde{F}_{\widetilde{a},\widetilde{a}_n}$. When $n$ is large
enough, there are two arcs $\widetilde{\gamma}^-$ and
$\widetilde{\gamma}^+$ in $A_{\widetilde{a},\widetilde{a}_n}$
joining $\widetilde{a}$ and $\widetilde{a}_n$ that are disjoint and
$\widetilde{F}_{\widetilde{a},\widetilde{a}_n}^q(\widetilde{\gamma}^-)\cap\widetilde{\gamma}^+=\emptyset$.
 Recall that $\widehat{\pi}_{\widetilde{a},\widetilde{a}_n}:
\widehat{A}_{{\widetilde{a}},{\widetilde{a}_n}}\rightarrow
A_{{\widetilde{a}},{\widetilde{a}_n}}$ is the universal cover of
$A_{{\widetilde{a}},{\widetilde{a}_n}}$,
$\widehat{F}_{\widetilde{a},\widetilde{a}_n}$ is the lift of
$\widetilde{F}_{\widetilde{a},\widetilde{a}_n}$ that fixes the
preimages of $\infty$ by
$\widehat{\pi}_{\widetilde{a},\widetilde{a}_n}$ and
$T_{{\widetilde{a}},{\widetilde{a}_n}}$ is the generator of
$H_1(A_{\widetilde{a},\widetilde{a}_n},\mathbb{R})$ defined by the
oriented boundary of small disk centered at $\widetilde{a}$. Choose
a connected component $\widehat{\gamma}^{\,-}$  of
$\widehat{\pi}_{\widetilde{a},\widetilde{a}_n}^{-1}(\widetilde{\gamma}^-)$
and endow $\widehat{\gamma}^{\,-}$ with an orientation from the
lower end to the upper end. The arc
$\widehat{F}^q_{\widetilde{a},\widetilde{a}_n}(\widehat{\gamma}^{\,-})$
does not meet any connected component of
$\widehat{\pi}_{\widetilde{a},\widetilde{a}_n}^{-1}(\widetilde{\gamma}^+)$
and thus meets at most a translated
$T_{{\widetilde{a}},{\widetilde{a}_n}}^k(\widehat{\gamma}^{\,-})$.
As $\widehat{F}_{\widetilde{a},\widetilde{a}_n}$ has a fixed point
(the lift $\widehat{\infty}$ of $\infty$), the arc
$\widehat{F}^q_{\widetilde{a},\widetilde{a}_n}(\widehat{\gamma}^{\,-})$
can not be on the right of
$T_{{\widetilde{a}},{\widetilde{a}_n}}(\widehat{\gamma}^{\,-})$
(otherwise, $\widehat{F}^q_{\widetilde{a},\widetilde{a}_n}$ has no
fixed point). Therefore, it is on the left of the arc
$T_{{\widetilde{a}},{\widetilde{a}_n}}^2(\widehat{\gamma}^{\,-})$.
For the same reason, it is on the right of the arc
$T_{{\widetilde{a}},{\widetilde{a}_n}}^{-2}(\widehat{\gamma}^{\,-})$.
As $\widehat{F}_{\widetilde{a},\widetilde{a}_n}$ and
$T_{{\widetilde{a}},{\widetilde{a}_n}}$ commute, it implies that the
arc
$\widehat{F}^q_{\widetilde{a},\widetilde{a}_n}(T(\widehat{\gamma}^{\,-}))$
is on the left of
$T_{{\widetilde{a}},{\widetilde{a}_n}}^3(\widehat{\gamma}^{\,-})$
and on the right of
$T_{{\widetilde{a}},{\widetilde{a}_n}}^{-1}(\widehat{\gamma}^{\,-})$.
Consider a point
$\widetilde{z}\in\mathrm{Rec}^+(\widetilde{F})\setminus\{\widetilde{a},\widetilde{a}_n\}$
such that the rotation number
$\rho_{A_{{\widetilde{a}},{\widetilde{a}_n}},\widehat{F}_{{\widetilde{a}},{\widetilde{a}_n}}}(\widetilde{z})$
is well defined. There exists a unique lift $\widehat{z}$ of $z$
that is in the region between $\widehat{\gamma}^{\,-}$ and
$T_{{\widetilde{a}},{\widetilde{a}_n}}(\widehat{\gamma}^{\,-})$. By
induction, we deduce that the point
$\widehat{F}^{qm}_{\widetilde{a},\widetilde{a}_n}(\widehat{z})$ is
in the region between
$T_{{\widetilde{a}},{\widetilde{a}_n}}^{-2m}(\widehat{\gamma}^{\,-})$
and
$T_{{\widetilde{a}},{\widetilde{a}_n}}^{3m}(\widehat{\gamma}^{\,-})$
for all $m\geq1$. By the definition of the rotation number (see
\ref{subsec:the rotation number of an open annulus}), we have
$|\rho_{A_{{\widetilde{a}},{\widetilde{a}_n}},\widehat{F}_{{\widetilde{a}},{\widetilde{a}_n}}}(\widetilde{z})|\leq
3/q$. As $q$ can be choose arbitrarily large, we have
\begin{equation}\label{eq:c1 diffeomorphism satisfied B property}
    \lim_{n\rightarrow+\infty}\rho_{A_{{\widetilde{a}},{\widetilde{a}_n}},\widehat{F}_{{\widetilde{a}},{\widetilde{a}_n}}}
(\widetilde{z})=0.
\end{equation}
In particular, we have
$$\lim_{n\rightarrow+\infty}\rho_{A_{{\widetilde{a}},{\widetilde{a}_n}},\widehat{F}_{{\widetilde{a}},{\widetilde{a}_n}}}
(\widetilde{b}_n)=0,$$ which conflicts with the limit \ref{eq:B
property n}. We have completed the proof.
\end{proof}\bigskip

\section{Symplectic Action}
The action is a classical object in symplectic geometry and we will
first recall it in this section. Then, we will explain how to
generalize the action to a simple case where the time-one map $F$ of
$I$ is a diffeomorphism, the set ${\rm Fix}_{{\rm Cont},I}(F)$ of
contractible fixed points is finite and unlinked (we will define
what it means), the measure $\mu\in\mathcal {M}(F)$ has no atoms on
${\rm Fix}_{{\rm Cont},I}(F)$ and satisfies $\rho_{M,I}(\mu)=0$.

Suppose that $I$ is an identity isotopy of $M$, $F$ is the time-one
map of $I$, $\mu\in\mathcal{M}(F)$ has no atoms on
$\mathrm{Fix}_{\mathrm{Cont},I}(F)$, and $\rho_{M,I}(\mu)=0$. At the
end of the section, we will state generalizations of the action to
the cases
\begin{itemize}
  \item $F\in\mathrm{Diff}(M)$;
  \item $I$ satisfies the WB-property, the measure $\mu$ has total
  support;
  \item $I$ satisfies the WB-property, the measure $\mu$ is ergodic.
\end{itemize}
We will prove them in Section \ref{sec:proof of the main
theorem}.\smallskip

\subsection{The classical action}\label{sec:action function}
Let us recall what is the action. In this section, we suppose that
$(M,\omega)$ is a symplectic manifold (not necessarily closed).
% Let $M$ is a smooth manifold, which has no
%boundary. A symplectic structure on a smooth manifold $M$ is a
%non-degenerate closed 2-form $\omega$. The non-degeneracy condition
%means that for all $p\in M$ we have the property that there does not
%exist non-zero $X\in T_pM$ such that $\omega(X,Y) = 0$ for all $Y\in
%T_pM$. The manifold $M$ is necessarily of even dimension $2n$ and
%the $n$-fold wedge product $\omega\wedge\ldots\wedge\omega$ never
%vanishes. Thus $M$ is orientable. We assume $H:\mathbf{S}^1\times
%M\rightarrow\mathbb{R}$ to be a regular Hamiltonian, regular means
%that every fixed point $x\in \mathrm{Fix}\phi^1_H$ is
%non-degenerate, where $\phi_H^1: \mathbb{R}\times M\rightarrow M$ is
%the time-one map for the flow of the non-autonomous Hamiltonian
%vector field $X_H$ defined by
%$$\omega(X_H,\cdot)=-dH.$$
%
%Denoting the space of contractible free loops by $\Omega^0(M)\subset
%C^\infty(\mathbf{S}^1,M)$ we define the set of 1-periodic
%contractible solutions of the Hamilton equation
%by$$\mathcal{P}_1(H)=\{x\in \Omega^0(M)\mid\dot{x}(t)=X_H(t,x)\}.$$
%In some conditions, for example, $\omega_{|\pi_2(M)}=0$, we have
%$\mathcal{P}_1(H)=\mathrm{Crit}\mathcal{A}_H$ where
%$\mathrm{Crit}\mathcal{A}_H$ is the critical points set of the
%\emph{action functional}
%$$\mathcal{A}_H(x)=\int_{D^2}\overline{x}^{\,*}\omega-\int_{\mathbf{S}^1}H(t,x)dt$$
%where $\overline{x}:D^2\rightarrow M$ is any extension of $x\in
%\Omega^0(M)$ to the unit disk.
\subsubsection{Symplectic and Hamiltonian}\label{subsec:symplectic
and hamiltonian} A diffeomorphism $F: M \rightarrow M$ is called
\emph{symplectic} if it preserves the form $\omega$. Symplectic
diffeomorphisms form a group denoted by $\mathrm{Symp}(M,\omega)$.
Let $\mathrm{Symp}_*(M,\omega)$ denote the path-connected component
of the $\mathrm{Id}_{M}$ in $\mathrm{Symp}(M,\omega)$.

Consider a smooth isotopy $I=(F_t)_{t\in[0,1]}$ in
$\mathrm{Symp}_*(M,\omega)$ with $F_0 =\mathrm{Id}_M$ and $F_1 = F$.
Let $\xi_t$ be the corresponding time-dependent vector field on $M$:

$$\frac{d}{dt}F_t(x)=\xi_t(F_t(x))\quad \text{for}\quad \text{all}\quad x\in M,\quad t\in[0,1].$$

Since the Lie derivative $L_{\xi_t}\omega$ vanishes, we get that the
$1$-forms $\lambda_t=- i_{\xi_t}\omega$ are closed. Write
$[\lambda_t]$ for the cohomology class of $\lambda_t$. The quantity

$$\mathrm{Flux}(I)=\int_0^1[\lambda_t]\,\mathrm{d}t\in H^1(M,\mathbb{R}),$$
is called the \emph{flux} of the isotopy $I$. It is well known
 that $\mathrm{Flux}(I)$ does not change under
a homotopy of the path $I$ with fixed end points (see \cite{MS}).

An isotopy $I$ is called \emph{Hamiltonian} if the $1$-forms
$\lambda_t$ are exact for all $t$. In this case there exists a
smooth function $H : [0,1]\times M \rightarrow \mathbb{R}$ so that
$\lambda_t = dH_t$, where $H_t(x)$ stands for $H(t,x)$. The function
$H$ is called the Hamiltonian function generating the flow $I$. Note
that $H_t$ is defined uniquely up to an additive time-dependent
constant.

A symplectic diffeomorphism $F: M \rightarrow M$ is called
\emph{Hamiltonian} if there exists a Hamiltonian isotopy
$I=(F_t)_{t\in[0,1]}$ with $F_0 =\mathrm{Id}_M$ and $F_1 = F$.
Hamiltonian diffeomorphisms form a group denoted by
$\mathrm{Ham}(M,\omega)$. The following theorem characterizes the
relation between flux and Hamiltonian diffeomorphisms (see \cite{MS}
for the details).

\begin{thm}\label{thm:smooth hamiltonian theorem}Let $F\in \mathrm{Symp}_*(M,\omega)$. Then $F$ is
Hamiltonian if and only if there exists an isotopy
$I=(F_t)_{t\in[0,1]}$ in $\mathrm{Symp}_*(M,\omega)$ such that
$F_0=\mathrm{Id}_M$, $F_1=F$ and $\mathrm{Flux}(I)=0$. In that case,
$I$ is isotopic with fixed endpoints to a Hamiltonian isotopy.
\end{thm}

Suppose that $(M,\omega)$ is a closed symplectic surface and
$I=(F_t)_{t\in[0,1]}$ is a smooth isotopy in
$\mathrm{Symp}_*(M,\omega)$. Let denote by  $\mu$ the measure
induced by $\omega$. We have the following relation between the
$\mathrm{Flux}(I)$ and $\rho_{M,I}(\mu)$ (see \cite{F2}): for any
smooth loop $\sigma$ on $M$, we have
$$\langle\mathrm{Flux}(I),[\sigma]_M\rangle=\rho_{M,I}(\mu)\wedge[\sigma]_M.$$
Hence, $I$ is Hamiltonian if and only if $\rho_{M,I}(\mu)=0$.

\subsubsection{Action function and action difference}\label{sec:the classical
action} In this section, we suppose that $(M,\omega)$ is a
symplectic manifold with $\pi_2(M)=0$ (for example, a closed
oriented surface with genus $g\geq1$).

Let $I=(F_t)_{t\in[0,1]}$ be a Hamiltonian isotopy on $M$ with
$F_0=\mathrm{Id}_M$ and $F_1=F$. Suppose that the function $H$ is
the Hamiltonian function generating the flow $I$.

Let $x$ be a contractible fixed point of $F$. Take any immersed disk
$D_x\subset M$ with $\partial D_x=I(x)$, and define the \emph{
action function}
\begin{equation}\label{eq:action functional}
\mathcal{A}_{H}(x)=\int_{D_x}\omega-\int_0^1
H_t(F_t(x))\,\mathrm{d}t.
\end{equation}

The definition is well defined, that is $\mathcal{A}_{H}(x)$ does
not depend on the choice of $D_x$. It is sufficient to prove the
integral $\int_{D_x}\omega$ does not depend on the choice of  $D_x$.
Indeed, let $D_x'$ be another choice, the 2-chain $\Pi=D_x-D_x'$
represents an immersed 2-sphere in $M$, and hence $\int_\Pi\omega=0$
since $\pi_2(M)=0$. Hence the claim follows.\smallskip

 Given two contractible fixed  points $x$ and $y$ of
$F$, take a path $\gamma:[0,1]\rightarrow M$ with $\gamma(0)=x$ and
$\gamma(1)=y$. Choose two immersed disks $D_x$ and $D_y$ so that
$\partial D_x=I(x)$ and $\partial D_y=I(y)$. Let us define
$\Delta:[0,1]\times[0,1]\rightarrow M$ by
$\Delta(t,s)=F_t(\gamma(s))$ where we assume that the boundary of
the square $[0,1]\times[0,1]$ is oriented counter-clockwise and
observer that $\partial \Delta=-\gamma+F\gamma-I(y)+I(x)$. So
$F\gamma-\gamma=\partial\Delta+\partial D_y-\partial D_x$ is a
1-cycle and is the boundary of $\Sigma$ where $\Sigma$ is a 2-chain.

Define the \emph{action difference} for $x$ and $y$:
\begin{equation}\label{eq:action difference}
    \delta(F;x,y)=\int_\Sigma \omega.
\end{equation}
Since $\pi_2(M)=0$, it does not depend on the choice of $\Sigma$,
and hence not on $D_x$ and $D_y$. Let us prove that it does not
depend on the choice of $\gamma$.

Denote by $\xi_t$ the vector field of the flow $F_t$ (see
\ref{subsec:symplectic and hamiltonian}). Then
\begin{eqnarray*}
 % \nonumber to remove numbering (before each equation)
 \Delta^*\omega&=&\omega\left(\xi_t(F_t\gamma(s)),\frac{\partial}{\partial s}F_t\gamma(s)\right)\,\mathrm{d}t\wedge \mathrm{d}s \\
  &=&-\mathrm{d}H_t\left(\frac{\partial}{\partial s}F_t\gamma(s)\right)\,\mathrm{d}t\wedge \mathrm{d}s
  .
 \end{eqnarray*} Hence,
 \begin{eqnarray*}
 % \nonumber to remove numbering (before each equation)
 \int_{\Delta}\omega &=&\int_{[0,1]^2}\Delta^*\omega=-\int_{0}^1\,\mathrm{d}t\int_{0}^1\,
  \mathrm{d}H_t\left(\frac{\partial}{\partial s}F_t\gamma(s)\right)\mathrm{d}s \\
 &=& \int_{0}^1 H_t(F_t(x))\,\mathrm{d}t-\int_{0}^1
 H_t(F_t(y))\,\mathrm{d}t.
 \end{eqnarray*}
Finally, we have
\begin{equation}\label{eq:action and action difference}
    \delta(F;x,y)=\int_{\Sigma}\omega=\int_{\Delta}\omega+\int_{D_y}\omega-\int_{D_x}\omega=\mathcal{A}_{H}(y)-\mathcal{A}_{H}(x).
\end{equation}

Equation \ref{eq:action and action difference} shows that the action
difference does not depend on the choice of $\gamma$, we have
completed our claim. Moreover, we also give a relation between the
action difference and the action function.

\subsubsection{The action function and action difference on the universal covering
space}\label{subsec:action and action difference on universal cover}
When $I=(F_t)_{t\in[0,1]}\subset \mathrm{Symp}_*(M,\omega)\setminus
\mathrm{Ham}(M,\omega)$, the action function (see Definition
\ref{eq:action functional}) is not meaningful. However, observing
that the universal cover $\widetilde{M}$ of $M$ is simply connected,
the lifted identity isotopy
$\widetilde{I}=(\widetilde{F}_t)_{t\in[0,1]}\subset
\mathrm{Symp}_*(\widetilde{M},\widetilde{\omega})$ of $I$ to
$\widetilde{M}$ where $\widetilde{\omega}$ is the lift of the
symplectic structure $\omega$ to $\widetilde{M}$ is automatically
Hamiltonian since $H^1(\widetilde{M},\mathbb{R})=0$ (see Theorem
\ref{thm:smooth hamiltonian theorem}). Let $\widetilde{H}$ be the
Hamiltonian function generating the flow $\widetilde{I}$. As before,
 we can define the action function
$\mathcal{A}_{\widetilde{H}}(\widetilde{x})$ for any fixed point
$\widetilde{x}$ of $\widetilde{F}=\widetilde{F}_1$ and the action
difference $\delta(\widetilde{F};\widetilde{x},\widetilde{y})$ for
any two distinct fixed points $\widetilde{x}$ and $\widetilde{y}$ of
$\widetilde{F}$, and we have the relation
$\delta(\widetilde{F};\widetilde{x},\widetilde{y})=\mathcal{A}_
{\widetilde{H}}(\widetilde{y})-\mathcal{A}_{\widetilde{H}}(\widetilde{x})$.\smallskip

Let us see what happens in the the particular case where $I$ is
Hamiltonian. Suppose that $H$ is the Hamiltonian function generating
the flow $I$ and $\widetilde{H}$ is its lift to $\widetilde{M}$. For
any contractible fixed point $x$ of $F$ and its lift
$\widetilde{x}$, we have
$\mathcal{A}_{\widetilde{H}}(\widetilde{x})=\mathcal{A}_{H}(x)$ (see
\cite[Theorem 2.1.C]{P} and \cite[Remark 2.7]{F2}). Hence, for any
two distinct contractible fixed points $x$ and $y$ of $F$, and their
lifts $\widetilde{x}$ and $\widetilde{y}$, we have
\begin{equation}\label{eq:action difference of cohomology}
    \delta(\widetilde{F};\widetilde{x},\widetilde{y})=\mathcal{A}_{\widetilde{H}}
(\widetilde{y})-\mathcal{A}_{\widetilde{H}}(\widetilde{x})=\mathcal{A}_{H}(y)-\mathcal{A}_{H}(x).
\end{equation}\smallskip

\subsection{A generalization of the action
function}\label{subsec:the first extension of the classical action}
The action difference of two contractible fixed points $x,y$ of $F$
equals to the algebraic area of any path $\gamma$ connecting $x$ and
$y$ along the isotopy $I$, that is, the area of the path  $\gamma$
along $I$ swept out. By this observation, we would like to
generalize such an object to the case where $\omega$ is replaced by
a finite Borel measure $\mu$ and the Hamitonian isotopy by an
identity isotopy $I$ with $\rho_{M,I}(\mu)=0$.
\smallskip

There is a case where this can be done easily (see \cite{P1}).
Suppose that $I=(F_t)_{t\in [0,1]}$ is an identity isotopy of $M$,
the time-one map $F$ of $I$ is a diffeomorphism, the set ${\rm
Fix}_{{\rm Cont},I}(F)$ of contractible fixed points is finite and
unlinked, that means that there exists an isotopy $I'=(F'_t)_{t\in
[0,1]}$ homotopic to $I$ that fixes every point of
$\mathrm{Fix}_{\mathrm{Cont},I}(F)$, the measure $\mu\in\mathcal
{M}(F)$ has no atoms on ${\rm Fix}_{{\rm Cont},I}(F)$ and satisfies
$\rho_{M,I}(\mu)=0$.

Let $N=M\setminus \mathrm{Fix}_{\mathrm{Cont},I}(F)$, by the method
of blowing-up, we can naturally get  a compactification
$\overline{N}$ of $N$ if we replace each point $x\in {\rm Fix}_{{\rm
Cont},I}(F)$ by $S_x$, the tangent unit circle at $x$. The
diffeomorphism $F|_N$ can be extended to a homeomorphism
$\overline{F}$ on $\overline{N}$ which is isotopic to identity and
induces the natural action  by the linear map $DF(x)$ on $S_x$. As
$\mu$ does not charge any point of
$\mathrm{Fix}_{\mathrm{Cont},I}(F)$, we can define a measure on
$\overline{N}$ which is invariant by $\overline{F}$, denoted also
$\mu$. Therefore, we can define the rotation vector in
$H_1(\overline{N},\mathbb{R})$. The inclusion $\iota:
N\hookrightarrow\overline{N}$ induces an isomorphism $\iota_*:
H_1(N,\mathbb{R})\rightarrow H_1(\overline{N},\mathbb{R})$. We
denote by $\rho_{N,I}(\mu)\in H_1(N,\mathbb{R})$ the rotation vector
transported by this isomorphism. Let $\gamma$ be a simple path in
$N$ joining $a\in \mathrm{Fix}_{\mathrm{Cont},I}(F)$ and $b\in
\mathrm{Fix}_{\mathrm{Cont},I}(F)$. We can define the algebraic
intersection number $\gamma\wedge \rho_{N,I}(\mu)$. Remark here that
$\gamma\wedge \rho_{N,I}(\mu)$ is independent on the chosen $\gamma$
because the rotation vector $\rho_{M,I}(\mu)\in H_1(M,\mathbb{R})$
is zero. Moreover, we can write
$$\gamma\wedge \rho_{N,I}(\mu)=L(b)-L(a),$$
where $L: \mathrm{Fix}_{\mathrm{Cont},I}(F)\rightarrow \mathbb{R}$
is a function, defined up to an additive constant. We call that $L$
is the \emph{action function}.\smallskip

\subsection{Our main theorem}\label{subsec:the second extension of the classical
action} It is natural to ask if we can generalize the action to a
more general case. Let us first analyze what has been done above.
The key points of his generalization are that $F$ is a
diffeomorphism of $M$ and that there is another identity isotopy
$I'$ homotopic to $I$ that fixes all contractible fixed points of
$F$. The differentiability hypothesis prevents the dynamics to be
too wild in a neighborhood of a contractible fixed point so that it
provides some boundedness condition, which means one can compactify
the sub-manifold $N=M\setminus \mathrm{Fix}_{\mathrm{Cont},I}(F)$ by
blowing-up. It seems to us that keeping the boundedness condition is
necessary and that is why we define the boundedness properties in
\ref{subsec:boundedness}. However, in general case, there may not
exist such an isotopy $I'$ that fixes all contractible fixed points
of $F$. %even if it fixes every two points of
%$\mathrm{Fix}_{\mathrm{Cont},I}(F)$.
How to deal with this obstacle?
The section \ref{subsec:action and action difference on universal
cover} reminds us that it will be a good idea if we consider the
universal covering space $\widetilde{M}$. A key point is that we can
always find an isotopy $\widetilde{I}\,'$ from
$\mathrm{Id}_{\widetilde{M}}$ to $\widetilde{F}$ that fixes any two
fixed points of $\widetilde{F}$, where $\widetilde{F}$ is the
time-one map of the lifted identity isotopy $\widetilde{I}$ of $I$
to $\widetilde{M}$ (Corollary \ref{cor:identity isotopy fixes two
points}). It makes us able to define the action difference for every
two fixed points of $\widetilde{F}$ and generalize the classical
action. Our main result is following.\bigskip

\newenvironment{thm01}{\noindent\textbf{Theorem \ref{thm:PW}}~\itshape}{\par}
\begin{thm01}
Let $M$ be a closed oriented surface with genus $g\geq1$ and $F$ be
the time-one map of an identity isotopy $I$ on $M$. Suppose that
$\mu$ is a Borel finite measures on $M$ that is invariant by $F$,
has no atoms on $\mathrm{Fix}_{\mathrm{Cont},I}(F)$ and
$\rho_{M,I}(\mu)=0$. In all of the following cases
\begin{itemize}
  \item $F\in\mathrm{Diff}(M)$;
  \item $I$ satisfies the WB-property, the measure $\mu$ has total
  support;
  \item $I$ satisfies the WB-property, the measure $\mu$ is ergodic,
\end{itemize}
an action function can be defined which generalizes the classical
case.
\end{thm01}\smallskip

We will prove it in Section \ref{sec:proof of the main
theorem}.\bigskip

\section{Disk Chains}\label{sec:preliminaries}
In this section, we will recall some classical results of the plane
and the open annulus, and extend some results of Franks so that we
can use them in Section \ref{sec:i(F;a,b,z) is bounded}.

Let $M$ be a surface and let $h$ be a homeomorphism of $M$. A
\emph{disk chain $C$} of $h$ in $M$ is given by a family
$\{D_i\}_{1\leq i\leq n}$ of embedded open disks of $M$ and a family
$\{m_i\}_{1\leq i<n}$ of positive integers satisfying
\begin{enumerate}
\item if $i\neq j$, then either $D_i=D_j$
or $D_i\cap D_j=\emptyset$;
\item for $1\leq i<n$, $h^{m_i}(D_i)\cap D_{i+1}\neq\emptyset$.
\end{enumerate}
We will write $C=\{D_i\}_{1\leq i\leq n}$ or $C=(\{D_i\}_{1\leq
i\leq n}, \{m_i\}_{1\leq i\leq n})$ in a more detailed way. If
$D_1=D_n$ we will say that $\{D_i\}_{1\leq i\leq n}$ is a
\emph{periodic disk chain}. We define the \emph{length} of the chain
$C$ to be the integer $l(C)=\sum_{i=1}^{n-1}m_i$.

A \emph{free} disk of $h$ is a disk in $M$ which does not meet its
image by $h$. A \emph{free disk chain} of $h$ is a disk chain
$C=\{D_i\}_{1\leq i\leq n}$
 such that every $D_i$ is a free
disk of $h$. \vspace{2mm}

%In 1912, Brouwer \cite{Br} got the following theorem:
%\begin{thm}\label{thm:br}If an orientation preserving homeomorphism $H$ of
%$\mathbb{R}^{2}$ has a periodic point of period $q\geq 2$, then one
%can construct a loop $\Gamma$ such that $i_{\Gamma} H=1$, where the
%index $i_{\Gamma} H$ is the degree of the map
%$$s\mapsto\frac{H(\Gamma(s))-\Gamma(s)}{|H(\Gamma(s))-\Gamma(s)|},$$
%and $s\mapsto\Gamma(s)$ is a parametrization defined on the unit
%circle $\mathbf{S}^1$. The same conclusion occurs if there is a non
%wandering point which is not fixed.\end{thm}
%
%By a perturbation, Franks got the following useful lemma:

Recall the following fundamental result:
\begin{prop}[Franks' Lemma \cite{F}]\label{prop:Franks' Lemma}
Let $H: \mathbb{R}^2\rightarrow \mathbb{R}^2$ be an orientation
preserving homeomorphism. If $H$  possesses a  periodic free disk
chain, then %there is a loop $\Gamma$ such that $i_{\Gamma} H=1$. In
%particular,
$H$ has at least one fixed point.
\end{prop}

Recall that $\mathbb{A}=\mathbb{R}/\mathbb{Z}\times\mathbb{R}$ is
the open annulus and $T: (x,y)\mapsto(x+1,y)$ is the generator of
the covering transformation group. Let $h\in
\mathrm{Homeo}_*(\mathbb{A})$ and $H$ be a lift of $h$ to
$\mathbb{R}^2$. We say that $\widetilde{D}\subset \mathbb{R}^2$ is a
\emph{positively returning disk} if all the following conditions
hold:
\begin{itemize}
\item
$T^k(\widetilde{D})\cap\widetilde{D}=\emptyset$ for all
$k\in\mathbb{Z}\setminus\{0\}$;
\item $H(\widetilde{D})\cap\widetilde{D}=\emptyset$;
\item there exist $n>0$ and $k>0$ such that $H^n(\widetilde{D})\cap T^k(\widetilde{D})\neq\emptyset$.
\end{itemize}
A \emph{negatively returning disk} is defined similarly but with
$k<0$.

If there exists an open disk that is both positively and negatively
returning, then it is easy to construct a periodic free disk chain
of $H$. Hence, by Franks' Lemma, we have the following result (see
\cite{F} for the detail):

\begin{cor}\label{lem:Franks'lemma2}
If $H$ has an open disk $\widetilde{D}\subset\mathbb{R}^2$ which is
both positively and negatively returning, then there is a fixed
point of $H$.
\end{cor}

Suppose that $D\subset \mathbb{A}$ is a free disk of $h$, we define
the following set:
\begin{equation}\label{the rotation number set for D}
    \mathrm{Rot}_{D}(H)=\mathrm{Conv}\{\,p/q\,\,\big\vert\,
p\in\mathbb{Z}\;\;\mathrm{and}\;\;q\in \mathbb{N}\setminus\{0\},\,
 H^q(\widetilde{D})\cap
T^p(\widetilde{D})\neq\emptyset\}
\end{equation}
where $\mathrm{Conv}(A)$ represents the convex hull of the set $A$
and $\widetilde{D}$ is an arbitrary connected component of
$\pi^{-1}(D)$. Observe here that $\mathrm{Rot}_{D}(H)$ does not
depend on the choice of $\widetilde{D}$. By Corollary
\ref{lem:Franks'lemma2}, we have the following result:

\begin{cor}\label{lem:Franks'lemma1}
For every $k\in \mathrm{Rot}_{D}(H)\cap \mathbb{Z}$, there exists a
point $\widetilde{z}_0$ such that
$H(\widetilde{z}_0)=T^k(\widetilde{z}_0)$.
\end{cor}
\begin{proof}
Choose any connected component $\widetilde{D}$ of $\pi^{-1}(D)$. We
first suppose that there is an integer $k$ such that
$H^q(\widetilde{D})\cap T^{kq}(\widetilde{D})\neq\emptyset$. Note
here that this case covers the case where $k$ is a boundary point of
$\mathrm{Rot}_{D}(H)$. Denote by $H'$ the lift $H'=T^{-k}\circ H$ of
$h$. We have $H'^{q}(\widetilde{D})\cap \widetilde{D}\neq \emptyset$
and $H'(\widetilde{D})\cap\widetilde{D}=\emptyset$ since $D$ is
free. According to Proposition \ref{prop:Franks' Lemma}, $H'$ has a
fixed point $\widetilde{z}_0$, that is,
$H(\widetilde{z}_0)=T^k(\widetilde{z}_0)$.

We now suppose that there are two rational numbers $p_i/q_i$
($i=1,2$) and an integer $k$ such that
\begin{itemize}
\item $p_1/q_1<k<p_2/q_2$\,;
\item $H^{q_1}(\widetilde{D})\cap
T^{p_1}(\widetilde{D})\neq\emptyset$;
\item $H^{q_2}(\widetilde{D})\cap
T^{p_2}(\widetilde{D})\neq\emptyset$.
\end{itemize}
Considering the lift $H'=T^{-k}\circ H$, we have
$$H'^{q_1}(\widetilde{D})\cap
T^{p_1-q_1k}(\widetilde{D})\neq\emptyset$$ and
$$H'^{q_2}(\widetilde{D})\cap
T^{p_2-q_2k}(\widetilde{D})\neq\emptyset.$$ Therefore,
$\widetilde{D}$ is a both positively and negatively returning disk
of $H'$. By Corollary \ref{lem:Franks'lemma2}, $H'$ has a fixed
point. We have completed the proof.
\end{proof}\smallskip

Let $C=(\{D_i\}_{1\leq i\leq n},\{m_i\}_{1\leq i<n})$ be a periodic
disk chain of $h$ in $\mathbb{A}$. A \emph{lift of $C$} for $H$ in
$\mathbb{R}^2$ is a disk chain
$\widetilde{C}=(\{\widetilde{D}_i\}_{1\leq i\leq n},\{m_i\}_{1\leq
i<n})$ in $\mathbb{R}^2$ such that $\pi(\widetilde{D}_i)=D_i$ for
every $i$.

We define the \emph{width} of the lift $\widetilde{C}$ of $C$ to be
the integer $w(H;\widetilde{C})=k$ such that
$\widetilde{D}_n=T^k(\widetilde{D}_1)$. For every $p\in\mathbb{Z}$,
the disk chain $T^p(\widetilde C)=(\{T^p(\widetilde D_i)\}_{1\leq
i\leq n}, \{m_i\}_{1\leq i< n})$ is also a lift of $C$ for $H$ since
$H$ commutes with $T$. The disk chain
$$T^p\cdot\widetilde{C}=\{\widetilde{D}_1,T^{pm_1}(\widetilde{D}_2)
,T^{p(m_1+m_2)}(\widetilde{D}_3),\cdots,T^{p\,\,l(C)}(\widetilde{D}_n)\}$$
is a lift of $C$ for $T^p\circ H$. Therefore, the width of
$\widetilde{C}$ satisfies
$$w(H;\widetilde{C})=w(H;T^p(\widetilde{C}))$$ and $$w(T^p\circ
H;T^p\cdot\widetilde{C})=p\,\,l(C)+w(H;\widetilde{C})$$ for every
$p\in\mathbb{Z}$.\smallskip

%We say a periodic disk chain $C=\{D_i\}_{1\leq i\leq n}$ with
%$\{m_i\}_{1\leq i<n}$ of $h$ in $\mathbb{A}$ is \emph{C-free
%periodic disk chain} of $h$ if it satisfies
%\begin{itemize}
%\item $D_1=D_n$;
%\item if $D_i\cap D_1=\emptyset$, then $D_i$ is a free disk of $h$ ($D_1$ is not necessary a free disk);
%\item for $1\leq i<
%n$, $h^{m_i}(D_i)\cap D_{i+1}\neq\emptyset$.
%\end{itemize}
\bigskip

Using Corollary \ref{lem:Franks'lemma2} and Corollary
\ref{lem:Franks'lemma1}, we have the following lemma.

\begin{lem}\label{lem:disk chain of Franks lemma}
Let $h\in \mathrm{Homeo}_*(\mathbb{A})$ and $H$ be a lift of $h$ to
$\mathbb{R}^2$. Suppose that
$\mathrm{Rot}_{\mathrm{Fix}(h)}(H)\subset[-N,N]$ for some
$N\in\mathbb{N}$, and that there is a disk $D$ in $\mathbb{A}$
satisfying $H(\widetilde{D})\cap
 T^k(\widetilde{D})\neq\emptyset$ if and only if $k=0$, where $\widetilde{D}$ is any connected component of
 $\pi^{-1}(D)$, and that a periodic disk chain  $C=(\{D_i\}_{1\leq i\leq n},
\{m_i\}_{1\leq i< n})$ of $h$ such that
 \begin{enumerate}
 \item $D_1=D$;
 \item if $D_i\not=D$ then $D_i$ is a free disk of $h$.
\end{enumerate}
Then, we have
\begin{itemize}
  \item $|w(H;\widetilde{C})|< (N+1)l(C)$ for all lift $\widetilde{C}$ of $C$;
  \item $\mathrm{Rot}_{D_i}(H)\subset]-(N+1),N+1[$
if $D_i\neq D$.
\end{itemize}
\end{lem}

\begin{proof}
Obviously, $C\,'=(\{D,D\},\{1\})$ is a periodic disk chain of $h$.

Fix a connected component $\widetilde{D}$ of
 $\pi^{-1}(D)$ and a lift $\widetilde{C}=\{\widetilde{D}_i\}_{1\leq i\leq
n}$ of $C$ for $H$ that satisfies $\widetilde{D}_1=\widetilde{D}$.
Define $\mathcal {D}$ as the family of all connected components of
$\pi^{-1}(D_i),\,\, 1\leq i\leq n$.

Suppose first that $w(H;\widetilde{C})\geq 0$, consider the lift
$H'=H\circ T^{-(N+1)}$, we have the following facts

\begin{itemize}
\item $\mathrm{Fix}(H')=\emptyset$;
  \item $H'(\widetilde{D})\cap\widetilde{D}=\emptyset$;
  \item there is a free disk chain $\widetilde{C}\,'$ in $\mathcal {D}$ of length $1$ from $\widetilde{D}$ to $T^{-(N+1)}(\widetilde{D})$ for $H'$
 (indeed, this disk chain is a lift of $C\,'$ for $H'$);
 \item there is a free disk chain $\widetilde{C}$ in $\mathcal {D}$ of length $l(C)$ from $\widetilde{D}$ to $T^{-(N+1)l(C)+w(H;\,\widetilde{C})}(\widetilde{D})$ for
   $H'$ (indeed, this disk chain is a lift of $C$ for $H'$).
\end{itemize}
The first item follows from
$\mathrm{Rot}_{\mathrm{Fix}(h)}(H)\subset[-N,N]$. The second and
third items hold by the hypothesis of $D$. The last one follows from
the hypothesis (1) and the property of $w(H;\widetilde{C})$.

If $-(N+1)l(C)+w(H;\widetilde{C})=0$, then $\widetilde{C}$ is a
periodic free disk chain for $H'$. By Proposition \ref{prop:Franks'
Lemma}, $H'$ has a fixed point, which conflicts with the first item.
If $r=-(N+1)l(C)+w(H;\widetilde{C})>0$, then the disk chain
$$\widetilde{C}\cup T^r(\widetilde{C})\cup\cdots\cup
T^{Nr}(\widetilde{C})\cup T^{(N+1)r}(\widetilde{C}\,')\cup\cdots\cup
T^{N+1}(\widetilde{C}\,')$$ is a periodic free disk chain for $H'$.
By Proposition \ref{prop:Franks' Lemma} again, $H'$ has a fixed
point, which still conflicts with the first item. Hence
$w(H;\widetilde{C})<(N+1)l(C)$.

In the case where $w(H;\widetilde{C})<0$, replacing $H'=H\circ
T^{-(N+1)}$ by $H'=H\circ T^{N+1}$, similarly to the case
$w(H;\widetilde{C})\geq0$, we get $w(H;\widetilde{C})> -(N+1)l(C)$.
The first conclusion is proven.

Fix a disk $D_i\neq D$ and $p/q\in\mathrm{Rot}_{D_i}(H)$. For every
$s\geq1$, consider the following  periodic disk chain of $h$
$$C_s=\{D_1,\cdots,\underbrace{D_i,\cdots,D_i}_{s+1},\cdots,
D_n\}$$ with
$$\{m_1,\cdots,m_{i-1},\underbrace{q,\cdots,q}_{s},m_i,\cdots,m_{n-1}\}$$ and its lift for $H$
$$\widetilde{C}_s=\{\widetilde{D}_1,\cdots,\widetilde{D}_i,T^p(\widetilde{D}_i),\cdots,T^{sp}(\widetilde{D}_i)
,T^{sp}(\widetilde{D}_{i+1}),\cdots,T^{sp}(\widetilde{D}_{n})\}.$$
Then we have $l(C_s)=l(C)+sq$ and
$w(H;\widetilde{C}_s)=w(H;\widetilde{C})+sp$. By the first
conclusion, we get $|w(H;\widetilde{C}_s)|<(N+1)l(C_s)$. Letting $s$
tend to $+\infty$, we get $|p/q|\leq N+1$. Moreover, if $p/q=N+1$
(resp. $p/q=-(N+1)$), according to Corollary
\ref{lem:Franks'lemma1}, then there exists a fixed point of $h$ with
rotation number $N+1$ (resp. $-(N+1)$) for $H$, which conflicts with
the hypothesis $\mathrm{Rot}_{\mathrm{Fix}(h)}(H)\subset[-N,N]$.
Therefore $|p/q|< N+1$. We have completed the proof.
\end{proof}

%\begin{lem}\label{lem:two disks of Franks lemma 2}
%Let $h\in \mathrm{Homeo}_*(\mathbb{A})$ and $H$ be a lift of $h$ to
%$\mathbb{R}^2$. Suppose that
%$\mathrm{Rot}_{\mathrm{Fix}(h)}(H)\subset[-N,N]$ for some
%$N\in\mathbb{N}$ and that there are two open disjoint disks $D$ and
%$D'$ of $\mathbb{A}$ such that
% \begin{enumerate}
% \item $h(D')\cap D'=\emptyset$ and $\bigcap_{n\geq1}\bigcup_{k\geq n}h^k(D')\cap D'\neq\emptyset$;
% \item $H(\widetilde{D})\cap
% T^k(\widetilde{D})\neq\emptyset$ if and only if $k=0$
% where $\widetilde{D}$ is any connected component of $\pi^{-1}(D)$.
% \end{enumerate}
%Then, we have
%\begin{itemize}
%  \item if\, there is $k\geq1$ such that $h^k(D')\cap
%  D\neq\emptyset$,\, then\, $\mathrm{Rot}_{D'}(H)\subset]-(N+1),N+1[$;
%  \item if\, $\mathrm{Rot}_{D'}(H)\nsubseteq\,]-(N+1),N+1[$,\, then
%  \,$\bigcup_{k\geq1}h^k(D')\cap
% D=\emptyset$ and $\mathrm{Rot}_{D'}(H)\subset]l,l+1[$ for some integer $l$
%with $l\geq N+1$ or $l\leq -(N+2)$.
%\end{itemize}
%\end{lem}
%\begin{proof}
%Let us prove the first conclusion. By the hypothesis (1), there
%exist $k'> k$ such that $h^{k'}(D')\cap D'\neq \emptyset$. Then
%$\{D,D',D\}$ is a pseudo-periodic disk chain of $h$. By Lemma
%\ref{lem:disk chain of Franks lemma}, we have
%$\mathrm{Rot}_{D'}(H)\subset[-(N+1),N+1]$.
%
% It remains to prove the case where $\bigcup_{k\geq1}h^k(D')\cap
% D=\emptyset$ and $\mathrm{Rot}_{D'}(H)\nsubseteq]-(N+1),N+1[$. It
% follows from Corollary \ref{lem:Franks'lemma1} and the hypothesis $\mathrm{Rot}_{\mathrm{Fix}(h)}(H)\subset[-N,N]$
%  immediately.
%\end{proof}

The following Theorem is due to Franks \cite{F} when $\mathbb{A}$ is
a closed annulus and $h$ has no wandering point, and it was improved
by Le Calvez \cite{P1} to the case where $\mathbb{A}$ is an open
annulus and $h$ satisfies the intersection property:

\begin{thm}\label{thm:FP}Let $h\in \mathrm{Homeo}_*(\mathbb{A})$ and
 $H$ be a lift of $h$ to
$\mathbb{R}^2$. Suppose that there exist two positively recurrent
points of rotation numbers $\nu^-$ and $\nu^+$(eventually equal to
$\pm\infty$) with $\nu^-<\nu^+$, and suppose that $h$ satisfies the
following intersection property: any simple closed curve of
$\mathbb{A}$ which is not null-homotopic meets its image by $h$.
Then for any rational number $p/q\in ]\nu^-,\nu^+[$ written in an
irreducible way, there exists a periodic point of period $q$ whose
rotation number is $p/q$.
\end{thm}\bigskip
%Let us state a recent result due to Jaulent \cite{J}:
%\begin{thm}\label{thm:O.Jaulent}
%Let $M$ be a oriented surface and $F$ be the time-one map of an
%identity isotopy $I$ on $M$. There exists a closed subset $X\subset
%\mathrm{Fix}(F)$ and an isotopy $I'$ joining
%$\mathrm{Id}_{M\setminus X}$ to $F|_{M\setminus X}$ in
%$\mathrm{Homeo(M\setminus X)}$ such that
%\begin{enumerate}
%  \item For all $z\in X$, the loop $I(z)$ is homotopic to zero in
%  $M$.
%  \item For all $z\in\mathrm{Fix(F)}\setminus X$, the loop $I'(z)$
%  is not homotopic to zero in $M\setminus X$.
%  \item For all $z\in M\setminus X$, the trajectories $I(z)$ and
%  $I'(z)$ are homotopic with fixed endpoints in $M$.
%  \item There exists an oriented topological foliation $\mathcal{F}$
%  on $M\setminus X$ such that, for all $z\in M\setminus X$, the trajectory
%  $I'(z)$ is homotopic to an arc $\gamma$ joining $z$ and $F(z)$ in
%  $M\setminus X$
%  which is positively to $\mathcal{F}$.
%\end{enumerate}
%Moreover, the isotopy $I'$ satisfies the following property:

%\begin{enumerate}
%\item[(5)] For all finite $Y\subset X$, there exists an isotopy $I_Y'$
%joining $\mathrm{Id}_M$ and $F$ in $\mathrm{Homeo}(M)$ which fixes
%$Y$ such that, if $z\in M\setminus X$, the arc $I'(z)$ and $I_Y'(z)$
%are homotopic in $M\setminus Y$.
%\end{enumerate}
%\end{thm}

\section{Extension of the Linking Number}\label{sec:define a new linking number}
In this section, we will first extend the notion of linking number
defined in \ref{sec:linking number of a and b,c}, and then state
some properties about it.

\subsection{Extension of the linking number for a positively recurrent point}
\label{subsec:the definition of a new linking
number}\quad\par

 Recall that $F$ is the time-one map of an identity
isotopy $I=(F_t)_{t\in[0,1]}$ on a closed oriented surface $M$ of
genus $g\geq 1$ and $\widetilde{F}$ is the time-one map of the
lifted identity isotopy
$\widetilde{I}=(\widetilde{F}_t)_{t\in[0,1]}$ on the universal cover
$\widetilde{M}$ of $M$. For every distinct fixed points
$\widetilde{a}$ and $\widetilde{b}$ of $\widetilde{F}$, by Lemma
\ref{rem:identity isotopies fix three points on sphere}, we can
choose an isotopy $\widetilde{I}_1$ from
$\mathrm{Id}_{\widetilde{M}}$ to $\widetilde{F}$ that fixes
$\widetilde{a}$ and $\widetilde{b}$.

Let us fix $z\in
\mathrm{Rec}^+(F)\setminus\pi(\{\widetilde{a},\widetilde{b}\})$ and
consider an open disk $U\subset
M\setminus\pi(\{\widetilde{a},\widetilde{b}\})$ containing $z$. For
every pair $(z',z'')\in U^2$, choose an oriented simple path
$\gamma_{z',z''}$ in $U$ from $z'$ to $z''$. Denote by
$\widetilde{\Phi}$ the lift
 of the first return map $\Phi$:
\begin{eqnarray*}
% \nonumber to remove numbering (before each equation)
  \widetilde{\Phi}: \pi^{-1}(\mathrm{Rec}^+(F))\cap\pi^{-1}(U)&\rightarrow&
 \pi^{-1}(\mathrm{Rec}^+(F))\cap\pi^{-1}(U) \\
\widetilde{z} &\mapsto&
  \widetilde{F}^{\tau(z)}(\widetilde{z}),
\end{eqnarray*}
where $z=\pi(\widetilde{z})$ and $\tau(z)$ is the first return time
in $U$.

For any $\widetilde{z}\in\pi^{-1}(U)$, write $U_{\widetilde{z}}$ the
connected component of $\pi^{-1}(U)$ that contains $\widetilde{z}$.
For every $j\geq0$, recall that
$\tau_j(z)=\sum\limits_{i=0}^{j-1}\tau(\Phi^i(z))$. For every
$n\geq1$, consider the following curves in $\widetilde{M}$:
$$%\widetilde{\Gamma}_{\widetilde{I}_1,\widetilde{z}}=\widetilde{I}_1^{\,\tau(z)}(\widetilde{z})
%\widetilde{\gamma}_{\widetilde{\Phi}(\widetilde{z}),\widetilde{x}_1},\quad
\widetilde{\Gamma}_{\widetilde{I}_1,\widetilde{z}}^n=\widetilde{I}_1^{\,\tau_n(z)}(\widetilde{z})
\widetilde{\gamma}_{\widetilde{\Phi}^n(\widetilde{z}),\widetilde{z}_n}\,,$$
where $\widetilde{z}_n\in \pi^{-1}(\{z\})\cap
\widetilde{U}_{\widetilde{\Phi}^n(\widetilde{z})}$, and
$\widetilde{\gamma}_{\widetilde{\Phi}^n(\widetilde{z}),\widetilde{z}_n}$
is the lift of $\gamma_{\Phi^n(z),z}$ in that is contained
$\widetilde{U}_{\widetilde{\Phi}^n(\widetilde{z})}$. We can define
the following infinite product (see \ref{sec:Identity isotopies}):
$$\widetilde{\Gamma}_{\widetilde{I}_1,z}^n=\prod_{\pi(\widetilde{z})=z}\widetilde
{\Gamma}_{\widetilde{I}_1,\widetilde{z}}^n\,.$$In particular, when
$z\in\mathrm{Fix}(F)$,
$\widetilde{\Gamma}_{\widetilde{I}_1,z}^1=\prod\limits_{\pi(\widetilde{z})=z}\widetilde{I}_1(\widetilde
z)$.

\smallskip
%Write $(F^{n_{k}}(z))_{k\geq 0}$ for the subsequence of the positive
%orbit of $z$ obtained by keeping the points that are in $U$. For any
%$k\geq 0$, choose a simple path $\gamma_{F^{n_k}(z),z}$ in $U$
%joining $F^{n_k}(z)$ to $z$. Write
%$$\widetilde{\Gamma}_{\widetilde{I}_1,z}^{n_k}=\prod_{\pi(\widetilde{z})=z}
%\widetilde{I}_1^{n_{k}}(\widetilde{z})\widetilde{\gamma}_{\widetilde
%{F}^{n_{k}}(\widetilde{z}),\widetilde{z}_{n_{k}}},$$ where
%$\widetilde{z}_{n_{k}}\in \pi^{-1}(z)\cap \widetilde{U}_{\widetilde
%{F}^{n_{k}}(\widetilde{z})}$ and $\widetilde{\gamma}_{\widetilde
%{F}^{n_{k}}(\widetilde{z}),\widetilde{z}_{n_{k}}}$ is the lift of
%$\gamma_{F^{n_k}(z),z}$ in $\widetilde{U}_{\widetilde
%{F}^{n_{k}}(\widetilde{z})}$. It is a collection of proper paths of
%$\widetilde{M}$ which means that all of two ends tend to $\infty$
%(say it \emph{multi-path}) or a collection of proper loops  (say it
%\emph{multi-loop}).

When
$\widetilde{U}_{\widetilde{\Phi}^n(\widetilde{z})}=\widetilde{U}_{\widetilde{z}}$,
the curve $\widetilde{\Gamma}_{\widetilde{I}_1,\widetilde{z}}^n$ is
a loop and hence $\widetilde{\Gamma}_{\widetilde{I}_1,z}^n$ is an
infinite family of loops, that will be called a \emph{multi-loop}.
When
$\widetilde{U}_{\widetilde{\Phi}^n(\widetilde{z})}\neq\widetilde{U}_{\widetilde{z}}$,
the curve $\widetilde{\Gamma}_{\widetilde{I}_1,\widetilde{z}}^n$ is
a compact path and hence $\widetilde{\Gamma}_{\widetilde{I}_1,z}^n$
is an infinite family of
 paths (it can be seen as a family of proper paths, that means all of two ends of these paths going to
$\infty$), that will be called a \emph{multi-path}.

In the both cases, for every neighborhood $\widetilde{V}$ of
$\infty$, there are finitely many loops or paths
$\widetilde{\Gamma}_{\widetilde{I}_1,\widetilde{z}}^n$ that are not
included in $\widetilde{V}$. By adding the point $\infty$ at
infinity, we get a multi-loop on the sphere
$\mathbf{S}=\widetilde{M}\sqcup\{\infty\}$.

In fact, $\widetilde{\Gamma}_{\widetilde{I}_1,z}^n$ can be seen as a
multi-loop in the annulus $A_{\widetilde{a},\widetilde{b}}$ with a
finite homology. As a consequence, if $\widetilde{\gamma}$ is a path
from $\widetilde{a}$ to $\widetilde{b}$, the intersection number
$\widetilde{\gamma}\wedge\widetilde{\Gamma}_{\widetilde{I}_1,z}^{n}$
is well defined and does not depend on $\widetilde{\gamma}$. By
Remark \ref{rem:some result of of sphere delete three points} and
the properties of intersection number, the intersection number is
also independent of the choice of the identity isotopy
$\widetilde{I}_1$ but depends on $U$.
 Moreover, observe that the path $(\prod_{i=0}^{n-1}\gamma_{\Phi^{n-i}(z)
 ,\Phi^{n-i-1}(z)})(\gamma_{\Phi^n(z),z})^{-1}$ is a loop in $U$, we have
\begin{equation}\label{eq:Birkhoff sum}
\widetilde{\gamma}\wedge\widetilde{\Gamma}_{\widetilde{I}_1,z}^{n}=
\widetilde{\gamma}\wedge\prod_{j=0}^{n-1}\widetilde{\Gamma}_{\widetilde{I}_1,\Phi^j(z)}^{1}
=\sum_{j=0}^{n-1}\widetilde{\gamma}\wedge\widetilde{\Gamma}_{\widetilde{I}_1,\Phi^j(z)}^{1}.
\end{equation}
\smallskip

For $n\geq1$, we can define the functions $$L_{n}:
((\mathrm{Fix}(\widetilde{F})\times
\mathrm{Fix}(\widetilde{F}))\setminus\widetilde{\Delta})\times
(\mathrm{Rec}^+(F)\cap U)\rightarrow \mathbb{Z},$$
\begin{equation}\label{eq: Ln Birkhoff sum}
    L_{n}(\widetilde{F};\widetilde{a},\widetilde{b},z)=\widetilde{\gamma}\wedge
\widetilde{\Gamma}^n_{\widetilde{I}_1,z}=\sum\limits_{j=0}^{n-1}
L_1(\widetilde{F};\widetilde{a},\widetilde{b},\Phi^j(z)).
\end{equation}
where $U\subset M\setminus\pi(\{\widetilde{a},\widetilde{b}\})$. The
last equation follows from Equation \ref{eq:Birkhoff sum}. The
function $L_n$ depends on $U$ but not on the choice of
$\gamma_{\Phi^n(z),z}$.

\begin{defn}\label{def:Intersection number density} Fix $z
\in \mathrm{Rec}^+(F)\setminus\pi(\{\widetilde{a},\widetilde{b}\})$.
Let us say that the linking number
$i(\widetilde{F};\widetilde{a},\widetilde{b},z)\in \mathbb{R}$ is
defined, if
$$\lim_{k\rightarrow +\infty} \frac{L_{n_k}(\widetilde{F};\widetilde{a},\widetilde{b},z)}{\tau_{n_{k}}(z)}
=i(\widetilde{F};\widetilde{a},\widetilde{b},z)$$ for any
subsequence $\{\Phi^{n_{k}}(z)\}_{k\geq 1}$ of
$\{\Phi^{n}(z)\}_{n\geq 1}$ which converges to $z$.
\end{defn}
Note here that the linking number
$i(\widetilde{F};\widetilde{a},\widetilde{b},z)$ does not depend on
$U$ since if $U$ and $U'$ are open disks containing $z$, there
exists a disk containing $z$ that is contained in $U\cap U'$. In
particular, when $z\in\mathrm{Fix}(F)\setminus\pi(\{\widetilde{a}
,\widetilde{b}\})$, the linking number
$i(\widetilde{F};\widetilde{a},\widetilde{b},z)$ always exists and
is equal to $L_1(\widetilde{F};\widetilde{a},\widetilde{b},z)$.

\begin{rem}\label{rem:linking number depends on the choice of the lift of a and
b}When
$z\in\mathrm{Rec}^+(F)\setminus\mathrm{Fix}_{\mathrm{Cont},I}(F)$,
the linking number $i(\widetilde{F};\widetilde{a},\widetilde{b},z)$
depends on the choice of $\widetilde{a}$ and $\widetilde{b}$ if it
exists. Indeed, consider the following smooth identity isotopy on
$\mathbb{R}^2$: $\widetilde{I}=(\widetilde{F}_t)_{t\in[0,1]}:
(x,y)\mapsto(x,y+t\sin(2\pi x))$. It induces an identity smooth
isotopy $I=(F_t)_{t\in[0,1]}$ on
$\mathbb{T}^2=\mathbb{R}/\mathbb{Z}\times\mathbb{R}/\mathbb{Z}$.
Obviously $\mathrm{Fix}(\widetilde{F})=\{(x,y)\mid x=k,\,\,
x=k+1/2,\,\, k\in\mathbb{Z}\}$ and $z=(1/4,0)\in \mathbb{T}^2$ is a
fixed point of $F$ but not a contractible fixed point of $F$. Let
$\widetilde{a}_k=(k,1/2)\in \mathbb{R}^2$ where $k\in \mathbb{Z}$.
It is easy to see that
$i(\widetilde{F};\widetilde{a}_0,\widetilde{a}_k,z)=k$ and
$\pi(\widetilde{a}_k)=\pi(\widetilde{a}_{k'})$ where
$k,k'\in\mathbb{Z}$.
\end{rem}

\subsection{Some properties of the linking number}
Now we give some properties of the linking number we have
defined.\smallskip

For any $q\geq1$,  $F^q$ is the time-one map of the identity isotopy
$I^q=(F_t)_{t\in[0,q]}$ on $M$. By Lemma \ref{subsec:positively
recurrent} in Appendix, a positively recurrent point of $F$ is also
a positively recurrent point of $F^q$, so we can define the linking
number $i(\widetilde{F}^q,\widetilde{a},\widetilde{b},z)$.

\begin{prop}\label{prop:i(Fnabx)=ni(abx)}
If $i(\widetilde{F};\widetilde{a},\widetilde{b},z)$ exists, then
$i(\widetilde{F}^{q};\widetilde{a},\widetilde{b},z)$ exists for
every $q\geq1$ and
$i(\widetilde{F}^{q};\widetilde{a},\widetilde{b},z)=qi(\widetilde{F};\widetilde{a},\widetilde{b},z)$.
\end{prop}
\begin{proof}Let $\widetilde{\gamma}$ be any simple path from
$\widetilde{a}$ to $\widetilde{b}$ and $\widetilde{I}_1$ be an
isotopy that fixes $\widetilde{a}$ and $\widetilde{b}$.
%When $z\in
%\mathrm{Fix}(F)\setminus\{\pi(\widetilde{a}),\pi(\widetilde{b})\}$,
%we have
%\begin{eqnarray*}
%i(\widetilde{F}^{q};\widetilde{a},\widetilde{b},z)&=&L_1(\widetilde{F}^{q};\widetilde{a},\widetilde{b},z)\\&=&
%\widetilde{\gamma}\wedge\widetilde{\Gamma}^1_{\widetilde{I}_1^{\,q},z}\\
%&=&\widetilde{\gamma}\wedge\left(\prod_{\pi(\widetilde{z})=z}\widetilde{I}_1^{\,q}(\widetilde
%{z})\right)\\&=&q\widetilde{\gamma}\wedge\left(\prod_{\pi(\widetilde{z})=z}\widetilde{I}_1(\widetilde{z})\right)\\
%&=&q i(\widetilde{F};\widetilde{a},\widetilde{b},z).
%\end{eqnarray*}
%For every $z\in
%\mathrm{Rec}^+(F)\setminus\{\pi(\widetilde{a}),\pi(\widetilde{b})\}$
We suppose that $i(\widetilde{F};\widetilde{a},\widetilde{b},z)$
exists. Let $U$ be an open disk containing $z$. For every $q\geq1$,
write respectively $\tau'(z)$ and $\Phi'(z)$ for the first return
time and the first return map of $F^q$ in this proof. Recall that
$$\tau'_n(z)=\sum_{i=0}^{n-1}\tau'(\Phi'^i(z))$$ and $$\widetilde{\Gamma}_
{\widetilde{I}_1^{\,q},\widetilde{z}}^n=\widetilde{I}_1^{\,q\tau'_n(z)}(\widetilde{z})
\widetilde{\gamma}_{\widetilde{\Phi}'^n(\widetilde{z}),\widetilde{z}_n}\,,\quad
\widetilde{\Gamma}_{\widetilde{I}_1^{\,q},z}^n=\prod_{\pi(\widetilde{z})=z}\widetilde
{\Gamma}_{\widetilde{I}_1^{\,q},\widetilde{z}}^n$$ where
$\widetilde{\Phi}'$ is the lift of $\Phi'$ to $\pi^{-1}(U)$,
$\widetilde{z}_n\in \pi^{-1}(\{z\})\cap
\widetilde{U}_{\widetilde{\Phi}'^n(\widetilde{z})}$ and
$\widetilde{\gamma}_{\widetilde{\Phi}'^n(\widetilde{z}),\widetilde{z}_n}$
is the lift of $\gamma_{\Phi'^n(z),z}$ that is in
$\widetilde{U}_{\widetilde{\Phi}'^n(\widetilde{z})}$.

We suppose that the subsequence $\{\Phi'^{n_{k}}(z)\}_{k\geq 1}$
converges to $z$. For every $k$, there is $n_k'\in\mathbb{N}$ such
that $\tau_{n'_k}(z)=q\tau'_{n_k}(z)$. By Definition
\ref{def:Intersection number density}, for any subsequence
$\{\Phi'^{n_{k}}(z)\}_{k\geq 1}$ which converges to $z$, we have
\begin{eqnarray*}\lim_{k\rightarrow+\infty}
\frac{L_{n_k}(\widetilde{F}^q;\widetilde{a},\widetilde{b},z)}{\tau'_{n_k}(z)}
&=&\lim_{k\rightarrow+\infty}
\frac{\widetilde{\gamma}\wedge\widetilde{\Gamma}^{n_k}_{\widetilde{I}^q,z}}{\tau'_{n_k}(z)}\\
&=&q\cdot\lim_{k\rightarrow+\infty}
\frac{\widetilde{\gamma}\wedge\prod_{\pi(\widetilde{z})=z}\widetilde{I}_1^{\,q\tau'_{n_k}(z)}(\widetilde{z})
\widetilde{\gamma}_{\widetilde{\Phi}'^{n_k}(\widetilde{z}),\widetilde{z}_{n_k}}}{q\tau'_{n_k}(z)}\\
&=&q\cdot\lim_{k\rightarrow+\infty}
\frac{L_{n_k'}(\widetilde{F};\widetilde{a},\widetilde{b},z)}{\tau_{n_k'}(z)}\\&=&
q i(\widetilde{F}; \widetilde{a},\widetilde{b},z).
\end{eqnarray*}
\end{proof}

\begin{prop}\label{prop:alpha- one exists then
 other exists}
For every $\alpha\in G$, every distinct fixed points $\widetilde{a}$
and $\widetilde{b}$ of $\widetilde{F}$, and every $z\in
\mathrm{Rec}^+(F)\setminus\pi(\{\widetilde{a},\widetilde{b}\})$, if
$i(\widetilde{F};\widetilde{a},\widetilde{b},z)$ exists, then
$i(\widetilde{F};\alpha(\widetilde{a}),\alpha(\widetilde{b}),z)$
also exists and
$$i(\widetilde{F};\alpha(\widetilde{a}),\alpha(\widetilde{b}),z)=
i(\widetilde{F};\widetilde{a},\widetilde{b},z).$$
\end{prop}
\begin{proof}
Let $\widetilde{\gamma}$ be any simple path from $\widetilde{a}$ to
$\widetilde{b}$. Observe that the isotopy
$\widetilde{I}\,'_1=\alpha\circ\widetilde{I}_1\circ\alpha^{-1}$
fixes $\alpha(\widetilde{a})$ and $\alpha(\widetilde{b})$,
$\widetilde{\gamma}\wedge\widetilde{\Gamma}_{\widetilde{I}_1,\widetilde{z}}^n
=\alpha(\widetilde{\gamma})\wedge\widetilde{\Gamma}_{\widetilde{I}\,'_1,\alpha(\widetilde{z})}^n$
for every $n$. The proposition follows from Definition
\ref{def:Intersection number density}.
\end{proof}

Let $H$ be an orientation preserving homeomorphism of $M$ and
$\widetilde{H}$ be a lift of $H$ to $\widetilde{M}$. Consider the
time-one map $H\circ F\circ H^{-1}$ of the isotopy $I'=H\circ I\circ
H^{-1}$ and write the time-one map of the identity isotopy
$\widetilde{I}^{~'}$ as $\widetilde{H}\circ \widetilde{F}\circ
\widetilde{H}^{\,-1}$, where $\widetilde{I}^{~'}$ is the lift of
$I'$ to $\widetilde{M}$. Similarly to the Proposition
\ref{prop:alpha- one exists then
 other exists}, we have the following result:

\begin{prop}\label{prop:linking number and conjugation}
For every distinct fixed points $\widetilde{a}$, $\widetilde{b}$ of
$\widetilde{F}$ and every $z\in
\mathrm{Rec}^+(F)\setminus\pi(\{\widetilde{a},\widetilde{b}\})$, if
$i(\widetilde{F};\widetilde{a},\widetilde{b},z)$ exists, then
$i(\widetilde{H}\circ \widetilde{F}\circ
\widetilde{H}^{\,-1};\widetilde{H}(\widetilde{a}),\widetilde{H}(\widetilde{b}),H(z))$
also exists and $$i(\widetilde{H}\circ \widetilde{F}\circ
\widetilde{H}^{\,-1};\widetilde{H}(\widetilde{a}),\widetilde{H}(\widetilde{b}),H(z))
=i(\widetilde{F};\widetilde{a},\widetilde{b},z).$$
\end{prop}

\begin{prop}\label{lem:i is 3coboundary for point}
For every distinct fixed points $\widetilde{a}$, $\widetilde{b}$ and
$\widetilde{c}$ of $\widetilde{F}$, and every $z\in
\mathrm{Rec}^+(F)\setminus\pi(\{\widetilde{a},\widetilde{b},\widetilde{c}\})$,
if two among the three linking numbers
$i(\widetilde{F};\widetilde{a}, \widetilde{b},z)$,
$i(\widetilde{F};\widetilde{b},\widetilde{c},z)$ and
$i(\widetilde{F};\widetilde{c},\widetilde{a},z)$ exist, then the
last one also exists and we have
$$i(\widetilde{F};\widetilde{a},\widetilde{b},z)+i(\widetilde{F};\widetilde{b},
\widetilde{c},z)+i(\widetilde{F};\widetilde{c},\widetilde{a},z)=0.$$
%That is, $i$ is coboundary on $\mathrm{Fix}(\widetilde{F})$ for
%$\mu$ almost everywhere point of
%$M\setminus\{\pi(\widetilde{a}),\pi(\widetilde{b}),\pi(\widetilde{c})\}$.
\end{prop}\smallskip

Before proving Proposition \ref{lem:i is 3coboundary for point}, we
introduce some notations and recall some results of the annulus.

If $\{\gamma_i\}_{1\leq i\leq k}$ and $\{\gamma'_j\}_{1\leq j\leq
k'}$ are two finite families of loops or compact paths in
$\mathbf{S}=\widetilde{M}\cup\{\infty\}$ such that
$\prod_{i=1}^k\gamma_i$ and $\prod_{j=1}^{k'}\gamma'_j$ are well
defined (in the concatenation sense) (see \ref{sec:Identity
isotopies}) and the algebraic intersection number
$\left(\prod_{i=1}^k\gamma_i\right)\wedge\left(\prod_{j=1}^{k'}\gamma'_j\right)$
is well defined (see \ref{sec:the algebraic intersection number}),
then we formally write
$$\left(\prod_{i=1}^k\gamma_i\right)\wedge\left(\prod_{j=1}^{k'}\gamma'_j\right)=\sum_{i,j}\gamma_i\wedge\gamma'_j.$$

Recall that $\mathbb{A}=\mathbb{R}/\mathbb{Z}\times\mathbb{R}$ is
the open annulus and $T: (x,y)\mapsto(x+1,y)$ is the generator of
the covering transformation group. If $I=(h_t)_{t\in [0,1]}$ with
$h_0=h_1=\mathrm{Id}_{\mathbb{A}}$ is a loop in
$\mathrm{Homeo}_*(\mathbb{A})$, write $[I]_1\in
\pi_1(\mathrm{Homeo}_*(\mathbb{A}))$ for the homotopy class of $I$.
Recall that $\pi_1(\mathrm{Homeo}_*(\mathbb{A}))\simeq \mathbb{Z}$.
Therefore, we may write
$\pi_1(\mathrm{Homeo}_*(\mathbb{A}))=\bigcup_{k\in\,\mathbb{Z}}\mathscr{C}_k$
where $\mathscr{C}_k$ is the class which satisfies that for every
$[I]_1\in \mathscr{C}_k$, any lift $\widetilde{I}$ of $I$ to the
universal covering space $\widetilde{\mathbb{A}}$ satisfies
$\widetilde{h}_1(\widetilde{z})-\widetilde{h}_0(\widetilde{z})=T^k(\widetilde{z})$
for every $\widetilde{z}\in\widetilde{\mathbb{A}}$.\bigskip

\begin{proof}[Proof of Proposition \ref{lem:i is 3coboundary for point}]
Suppose that $\widetilde{\gamma}_1$, $\widetilde{\gamma}_2$ and
$\widetilde{\gamma}_3$ are oriented simple paths from
$\widetilde{a}$ to $\widetilde{b}$,  $\widetilde{b}$ to
$\widetilde{c}$ and  $\widetilde{c}$ to $\widetilde{a}$,
respectively. We choose isotopies $\widetilde{I}_j$ ($j=0,1,2,3$)
such that $\widetilde{I}_1$ fixes $\widetilde{a},\widetilde{b}$ and
$\infty$, $\widetilde{I}_2$ fixes $\widetilde{b},\widetilde{c}$ and
$\infty$, $\widetilde{I}_3$ fixes $\widetilde{c},\widetilde{a}$ and
$\infty$, and $\widetilde{I}_0$ fixes $\widetilde{a},\widetilde{b}$
and $\widetilde{c}$.\bigskip

% Let $\mathbb{A}_j$
%($j=1,2,3)$ be respectively $\mathbf{S}\setminus
%\{\widetilde{a},\widetilde{b}\}$, $\mathbf{S}\setminus
%\{\widetilde{b},\widetilde{c}\}$ and $\mathbf{S}\setminus
%\{\widetilde{c},\widetilde{a}\}$. For every $n\in \mathbb{N}$,
%consider the loops $\widetilde{I}_0^n\widetilde{I}_j^{-n}\subset
%\mathrm{Homeo}_*(\mathbb{A}_j)$, where $\widetilde{I}_j^{-1}$ stands
%for the reverse of $\widetilde{I}_j$. By Remark \ref{rem:a result of
%the fundamental group of annulus}, we have
%$[\widetilde{I}_0^n\widetilde{I}_j^{-n}]_1\in
%\mathscr{C}^j_{n\cdot(\widetilde{\gamma}_j
%\wedge\widetilde{I}_0(\infty))}$ $(j=1,2,3)$, where
%$\mathscr{C}^j_k$ is a class in
%$\pi_1(\mathrm{Homeo}_*(\mathbb{A}_j))$.

For every $z\in M\setminus\pi(\{\widetilde{a}
,\widetilde{b},\widetilde{c}\})$, every lift $\widetilde{z}$ of $z$,
every $j\in\{1,2,3\}$ and every $n\geq1$, the path
$\widetilde{I}^n_0(\widetilde{z})(\widetilde{I}^{n}_j(\widetilde{z}))^{-1}$
is a loop where $(\widetilde{I}^{n}_j(\widetilde{z}))^{-1}$ is the
inverse of the path $\widetilde{I}^{n}_j(\widetilde{z})$. We claim
that
\begin{equation}\label{eq:n homotopy}
    \widetilde{\gamma}_j\wedge\left(\widetilde{I}^n_0(\widetilde{z})(\widetilde{I}^{n}_j(\widetilde{z}))^{-1}\right)
=\widetilde{\gamma}_j\wedge\widetilde{I}^n_0(\widetilde{z})-
\widetilde{\gamma}_j\wedge\widetilde{I}^n_j(\widetilde{z})=n\cdot(\widetilde{\gamma}_j\wedge\widetilde{I}_0(\infty)).
\end{equation}

Indeed, let $\mathbb{A}_j$ ($j=1,2,3)$ be respectively
$\mathbf{S}\setminus \{\widetilde{a},\widetilde{b}\}$,
$\mathbf{S}\setminus \{\widetilde{b},\widetilde{c}\}$ and
$\mathbf{S}\setminus \{\widetilde{c},\widetilde{a}\}$. For every
$n\in \mathbb{N}$, considering the loops
$\widetilde{I}_j^{-n}\widetilde{I}_0^n\subset
\mathrm{Homeo}_*(\mathbb{A}_j)$ (see \ref{the product operate of the
symplectic group of isotopy}) where $\widetilde{I}_j^{-1}$ is the
inverse of $\widetilde{I}_j$, we have
$[\widetilde{I}_j^{-n}\widetilde{I}_0^n]_1\in
\mathscr{C}^j_{n\cdot(\widetilde{\gamma}_j
\wedge\widetilde{I}_0(\infty))}$ $(j=1,2,3)$ where $\mathscr{C}^j_k$
is a class in $\pi_1(\mathrm{Homeo}_*(\mathbb{A}_j))$. Observing
that
$(\widetilde{I}_j^{-n}\widetilde{I}_0^n)(\widetilde{z})=\widetilde{I}^n_0
(\widetilde{z})(\widetilde{I}^{n}_j(\widetilde{z}))^{-1}$, the claim
\ref{eq:n homotopy} follows. \bigskip

In the case where $z\in \mathrm{Fix}(F)\setminus\pi(\{\widetilde{a}
,\widetilde{b},\widetilde{c}\})$, for every lift $\widetilde{z}$ of
$z$, we have
$$\widetilde{\gamma}_j\wedge\widetilde{I}_0(\widetilde{z})-
\widetilde{\gamma}_j\wedge\widetilde{I}_j(\widetilde{z})=\widetilde{\gamma}_j
\wedge\widetilde{I}_0(\infty)\quad(j=1,2,3).$$

Write $C_z$ for the set of points $\widetilde{z}\in\pi^{-1}(\{z\})$
such that
$\widetilde{I}_j(\widetilde{z})\cap\bigcup\limits_{j'=1}^3\widetilde{\gamma}_{j'}
\neq\emptyset$ for every $j$. As all $\widetilde{I}_j$ fix $\infty$,
we know that $C_z$ is finite.\smallskip

Recall that
$$i(\widetilde{F};\widetilde{a},\widetilde{b},z)=\widetilde{\gamma}_1\wedge\widetilde{\Gamma}_{\widetilde{I}_1,z}^1,\quad
i(\widetilde{F};\widetilde{b},\widetilde{c},z)=\widetilde{\gamma}_2\wedge\widetilde{\Gamma}_{\widetilde{I}_2,z}^1\quad
\mathrm{and}\quad
i(\widetilde{F};\widetilde{c},\widetilde{a},z)=\widetilde{\gamma}_3\wedge\widetilde{\Gamma}_{\widetilde{I}_3,z}^1$$
where
$$\widetilde{\Gamma}_{\widetilde{I}_j,z}^1=\prod_{\pi(\widetilde{z})=z}\widetilde{I}_j(\widetilde
z)\quad (j=1,2,3).$$

Observe that$$\sum_{j=1}^3\sum_{\widetilde{z}\in
C_z}\widetilde{\gamma}_j\wedge\widetilde{I}_0(\widetilde{z})=\sum_{\widetilde{z}\in
C_z}\sum_{j=1}^3\widetilde{\gamma}_j\wedge\widetilde{I}_0(\widetilde{z})=0
\mathrm{,}\quad
\sum_{j=1}^3\widetilde{\gamma}_j\wedge\widetilde{I}_0(\infty)=0$$
and
$$\widetilde{\gamma}_j\wedge\widetilde{\Gamma}_{\widetilde{I}_j,z}^1=\widetilde{\gamma}_j\wedge\prod_{\pi(\widetilde{z})
=z}\widetilde{I}_j(\widetilde{z})=\sum_{\widetilde{z}\in
C_z}\widetilde{\gamma}_j\wedge\widetilde{I}_j(\widetilde{z})\quad
(j=1,2,3).$$ We get
\begin{eqnarray*}
% \nonumber to remove numbering (before each equation)
  &&i(\widetilde{F};\widetilde{a},\widetilde{b},z)+i(\widetilde{F};\widetilde{b},
\widetilde{c},z)+i(\widetilde{F};\widetilde{c},\widetilde{a},z)\\&=&
\sum_{j=1}^3\left(\widetilde{\gamma}_j\wedge\widetilde{\Gamma}_{\widetilde{I}_j,z}^1\right)\\
%&=&
%\sum_{j=1}^3\left(\widetilde{\gamma}_j\wedge\left(\prod_{\pi(\widetilde{z})=z}\widetilde{I}_j(\widetilde{z})\right)\right)\\
%&=&
%\sum_{j=1}^3\left(\widetilde{\gamma}_j\wedge\left(\prod_{\widetilde{z}\in
%C_z}\widetilde{I}_j(\widetilde{z})\right)\right)\\
&=&
-\sum_{j=1}^3\,\sum_{\widetilde{z}\in
C_z}\left(\widetilde{\gamma}_j\wedge\widetilde{I}_0(\widetilde{z})-
\widetilde{\gamma}_j\wedge\widetilde{I}_j(\widetilde{z})\right)\\
   &=&  -\sum_{\widetilde{z}\in
C_z}\,\sum_{j=1}^3\widetilde{\gamma}_j \wedge\widetilde{I}_0(\infty)\\
   &=& 0.
\end{eqnarray*}
Hence we have proved the proposition in this case.\smallskip

In the case where $z\in \mathrm{Rec}^+(F)\setminus\mathrm{Fix}(F)$,
recall that
$$\widetilde{\Gamma}_{\widetilde{I}_j,\widetilde{z}}^n=\widetilde{I}_j^{\,\tau_n(z)}(\widetilde{z})
\widetilde{\gamma}_{\widetilde{\Phi}^n(\widetilde{z}),\widetilde{z}_n}\,\,
(0\leq j\leq3),$$ where $\widetilde{z}_n\in \pi^{-1}(\{z\})\cap
\widetilde{U}_{\widetilde{\Phi}^n(\widetilde{z})}$ and
$\widetilde{\gamma}_{\widetilde{\Phi}^n(\widetilde{z}),\widetilde{z}_n}$
is the lift of $\gamma_{\Phi^n(z),z}$ in
$\widetilde{U}_{\widetilde{\Phi}^n(\widetilde{z})}$. For every
$1\leq j\leq3$, we have
$\widetilde{\Gamma}_{\widetilde{I}_0,\widetilde{z}}^{n}(\widetilde{\Gamma}^n_{\widetilde{I}_j,\widetilde{z}})^{-1}$
is a loop where
$(\widetilde{\Gamma}^n_{\widetilde{I}_j,\widetilde{z}})^{-1}$ is the
inverse of the path
$\widetilde{\Gamma}^n_{\widetilde{I}_j,\widetilde{z}}$\,. Therefore,
for every lift $\widetilde{z}$ of $z$ and $n\geq1$, we have
$$\widetilde{\gamma}_j\wedge\left(\widetilde{\Gamma}_{\widetilde{I}_0,\widetilde{z}}^{n}(\widetilde{\Gamma}^n_{
\widetilde{I}_j,\widetilde{z}})^{-1}\right)=\widetilde{\gamma}_j\wedge\widetilde{\Gamma}_{\widetilde{I}_0,\widetilde{z}}^n-
\widetilde{\gamma}_j\wedge\widetilde{\Gamma}_{\widetilde{I}_j,\widetilde{z}}^n=\tau_n(z)\cdot(\widetilde{\gamma}_j
\wedge\widetilde{I}_0(\infty))\quad(j=1,2,3).$$

For every $n$, write $C_z^{\,n}$ for the set of points
$z\in\pi^{-1}(\{z\})$ such that
$\widetilde{\Gamma}_{\widetilde{I}_j,\widetilde{z}}^n\cap\bigcup\limits_{j=1}^3\widetilde{\gamma}_j
\neq\emptyset$. Here again, we know that $C_z^{\,n}$ is
finite.\smallskip

Recall that
$$L_n(\widetilde{F};\widetilde{a},\widetilde{b},z)=\widetilde{\gamma}_1\wedge\widetilde{\Gamma}_{\widetilde{I}_1,z}^n,\quad
L_n(\widetilde{F};\widetilde{b},\widetilde{c},z)=\widetilde{\gamma}_2\wedge\widetilde{\Gamma}_{\widetilde{I}_1,z}^n\quad
\mathrm{and}\quad
L_n(\widetilde{F};\widetilde{c},\widetilde{a},z)=\widetilde{\gamma}_3\wedge\widetilde{\Gamma}_{\widetilde{I}_1,z}^n$$
where
$$\widetilde{\Gamma}_{\widetilde{I}_1,z}^n=\prod_{\pi(\widetilde{z})=z}\widetilde
{\Gamma}_{\widetilde{I}_1,\widetilde{z}}^n\,.$$

Then for any subsequence $\{\Phi^{n_{k}}(z)\}_{k\geq 1}$ which
converges to $z$, similarly to the fixed point case, we get
\begin{eqnarray}\label{eq:L(abz)+L(bcz)+L(caz)=0}
% \nonumber to remove numbering (before each equation)
 &&\frac{L_{n_k}(\widetilde{F};\widetilde{a},\widetilde{b},z)+L_{n_k}(\widetilde{F};\widetilde{b},
\widetilde{c},z)+L_{n_k}(\widetilde{F};\widetilde{c},\widetilde{a},z)}{\tau_{n_k}(z)}\\
&=&\frac{1}{\tau_{n_k}(z)}
\sum\limits_{j=1}^3\left(\widetilde{\gamma}_j\wedge\widetilde{\Gamma}_{\widetilde{I}_j,z}^{n_k}\right)\nonumber\\
&=&-\frac{1}{\tau_{n_k}(z)}\sum\limits_{j=1}^3\,\sum\limits_{\widetilde{z}\in
C^{n_k}_z}\left(\widetilde{\gamma}_j\wedge\widetilde{\Gamma}_{\widetilde{I}_0,\widetilde{z}}^{n_k}
-\widetilde{\gamma}_j\wedge\widetilde{\Gamma}_{\widetilde{I}_j,\widetilde{z}}^{n_k}\right)\nonumber\\
&=&-\sum_{\widetilde{z}\in
C^{n_k}_z}\,\sum_{j=1}^3\widetilde{\gamma}_j \wedge\widetilde{I}_0(\infty)\nonumber\\
&=&0\,.\nonumber
\end{eqnarray}
Letting $k\rightarrow+\infty$ in Equation
\ref{eq:L(abz)+L(bcz)+L(caz)=0}, we have completed the proposition.
\end{proof}\bigskip

\section{Boundedness and Existence of
the Linking Number}\label{sec:i(F;a,b,z) is bounded}

This section is divided into two parts. In the first part, we study
the boundedness of the linking number when it exists. In the second
part, we study the existence and boundedness of the linking number
if the map $F$ preserves a measure on $M$. The tools we will use are
Franks' Lemma and Birkhoff Ergodic Theorem.

\subsection{Boundedness}\label{sec:boundedness}In this section,
let $\widetilde{a}$ and $\widetilde{b}$ be two distinct fixed points
of $\widetilde{F}$. We suppose that $I$ satisfies WB-property at
$\widetilde{a}$ and $\widetilde{b}$. By Lemma \ref{lem:rotation
number and weak boundess at two point}, there is a positive number
$N_{\widetilde{a},\widetilde{b}}$ such that
$\mathrm{Rot}_{\mathrm{Fix}(\widetilde{F}_{\widetilde{a},\widetilde{b}})}
(\widehat{F}_{\widetilde{a},\widetilde{b}})\subset
[-N_{\widetilde{a},\widetilde{b}},N_{\widetilde{a},\widetilde{b}}]$.\smallskip

Fix an isotopy $\widetilde{I}_1$ from $\mathrm{Id}_{\widetilde{M}}$
to $\widetilde{F}$ which fixes $\widetilde{a}$ and $\widetilde{b}$.
Let $\widetilde{\gamma}$ be any oriented path in $\widetilde{M}$
from $\widetilde{a}$ to $\widetilde{b}$. Fix an open disk
$\widetilde{W}$ that contains $\infty$ and is disjoint from
$\widetilde{\gamma}$. We choose an open disk $\widetilde{V}\subset
\widetilde{W}$ that contains $\infty$ such that for every
$\widetilde{z}\in \widetilde{V}$, we have
$\widetilde{I}_1(\widetilde{z})\subset \widetilde{W}$. Observe that
if $\widehat{\infty}$ is a given lift of $\infty$ in
$\widehat{A}_{\widetilde{a},\widetilde{b}}$, if $\widehat{W}$ (resp.
$\widehat{V}$) is the connected component of
$\pi^{-1}(\widetilde{W})$ (resp. $\pi^{-1}(\widetilde{V})$) that
contains $\widehat{\infty}$, then we have
$\widehat{F}_{\widetilde{a},\widetilde{b}}(\widehat{V})\subset
\widehat{W}$, which implies that $\widehat{V}$ is free for every
other lift $\widehat{F}_{\widetilde{a},\widetilde{b}}\circ
T_{\widetilde{a},\widetilde{b}}^k$, where $k\in
\mathbb{Z}\setminus\{0\}$. Let $A^c$ denote the complement of a set
$A$. For every $z\in
M\setminus\pi(\{\widetilde{a},\widetilde{b}\})$, write
$X_z=\pi^{-1}(\{z\})\cap (\widetilde{V}\cap
\widetilde{F}_{\widetilde{a},\widetilde{b}}^{-1}(\widetilde{V}))^c$.
Observe that there exists $K_{\widetilde{a},\widetilde{b}}\in
\mathbb{N}$ such that $\sharp X_z\leq
K_{\widetilde{a},\widetilde{b}}$ for every $z\in
M\setminus\pi(\{\widetilde{a},\widetilde{b}\})$. %Note here that $K$
%depends on $\widetilde{V}$ and of course depends on $\widetilde{a}$
%and $\widetilde{b}$.
\smallskip

In the case where $z\in \mathrm{Rec}^+(F)\setminus\mathrm{Fix}(F)$,
we choose an open disk $U$ that contains $z$ and is free for $F$. As
the value $i(\widetilde{F};\widetilde{a},\widetilde{b},z)$ depends
neither on $\widetilde{\gamma}$ nor on $U$, we can always suppose
that $\widetilde{\gamma}\cap\pi^{-1}(U) =\emptyset$ by perturbing
$\widetilde{\gamma}$ a little and shrinking $U$ if necessary. For
every $n\geq1$, write
$$X_{z}^n=\pi^{-1}(\{z,F(z),\cdots,F^{\tau_n(z)-1}(z)\})\cap(\widetilde{V}\cap
\widetilde{F}^{-1}_{\widetilde{a},\widetilde{b}}(\widetilde{V}))^c.$$
Observe that $\sharp X_{z}^n\leq
\tau_n(z)K_{\widetilde{a},\widetilde{b}}$.\smallskip

The following result is the main proposition of this section.
\begin{prop}\label{prop:l(a,b,x)divid tau(x)is boundedonU}The following two statements hold:
\begin{itemize}\item If $z\in
\mathrm{Fix}(F)\setminus\pi(\{\widetilde{a},\widetilde{b}\})$, we
have $|i(\widetilde{F};\widetilde{a},\widetilde{b},z)|<
K_{\widetilde{a},\widetilde{b}}
(N_{\widetilde{a},\widetilde{b}}+1)$.
\item If $z\in
\mathrm{Rec}^+(F)\setminus\mathrm{Fix}(F)$ and
$i(\widetilde{F};\widetilde{a},\widetilde{b},z)$ is defined, then
$|i(\widetilde{F};\widetilde{a},\widetilde{b},z)|\leq
K_{\widetilde{a},\widetilde{b}}K_U$, where
 $K_U\in \mathbb{N}$ depends only on $U$.\end{itemize}
\end{prop}

In order to prove Proposition \ref{prop:l(a,b,x)divid tau(x)is
boundedonU}, we consider two cases: the fixed point case and the
non-fixed point case. The first case is more easy to deal with and
the second case is a little more complicated, but the ideas are
similar.

%In order to study the boundedness of the linking number, we consider
%two cases: the fixed point case (Section \ref{sec:the fixed point
%case}) and the non-fixed point case (Section \ref{sec:the non fixed
%point case}). The first case is more easy to deal with and the
%second case is a little more complicated, but the ideas are similar.
%In the end of this section (Section \ref{sec:boundedness in the case
%where F is C1 diffeomorphism}), we study the boundedness in a
%special case where the time-one map is a $C^1$-diffeomorphism of
%$M$.
\smallskip

\begin{flushleft}
\emph{The fixed point case.}\end{flushleft}
 \label{sec:the fixed point case}

When $z\in\mathrm{Fix}(F)\setminus\pi(\{\widetilde{a}
,\widetilde{b}\})$, then $\tau(z)=1$ and
$i(\widetilde{F};\widetilde{a},\widetilde{b},z)=L_1(\widetilde{F};\widetilde{a},\widetilde{b},z)$,
we have the following results.

\begin{lem}\label{lem:contractible fixed point case}If $z\in
\mathrm{Fix}_{\mathrm{Cont},I}(F)\setminus\pi(\{\widetilde{a},\widetilde{b}\})$,
then $|i(\widetilde{F};\widetilde{a},\widetilde{b},z)|\leq
K_{\widetilde{a},\widetilde{b}} N_{\widetilde{a},\widetilde{b}}$.
\end{lem}

\begin{proof}By Definition \ref{def:Intersection number density} and Lemma \ref{lem:P}, we
have
$$i(\widetilde{F};\widetilde{a},\widetilde{b},z)=\sum_{\pi(\widetilde{z})=z}
\rho_{A_{{\widetilde{a}},{\widetilde{b}}},\widehat{F}_{{\widetilde{a}},{\widetilde{b}}}}
(\widetilde{z})=\sum_{\widetilde{z}\in
X_z}\rho_{A_{{\widetilde{a}},{\widetilde{b}}},\widehat{F}_{{\widetilde{a}},{\widetilde{b}}}}
(\widetilde{z}).$$ The lemma follows from the fact that $\sharp
X_z\leq K_{\widetilde{a},\widetilde{b}}$ and that
$\mathrm{Rot}_{\mathrm{Fix}(\widetilde{F}_{\widetilde{a},\widetilde{b}})}
(\widehat{F}_{\widetilde{a},\widetilde{b}})\subset
[-N_{\widetilde{a},\widetilde{b}},N_{\widetilde{a},\widetilde{b}}]$.
\end{proof}

\begin{lem}\label{lem:non contractible case of fixed point}If $z\in
\mathrm{Fix}(F)\setminus\mathrm{Fix}_{\mathrm{Cont},I}(F)$, then
$|i(\widetilde{F};\widetilde{a},\widetilde{b},z)|<K_{\widetilde{a},\widetilde{b}}
(N_{\widetilde{a},\widetilde{b}}+1)$.
\end{lem}
\begin{proof}
We have
$$i(\widetilde{F};\widetilde{a},\widetilde{b},z)=\widetilde{\gamma}\wedge\widetilde{\Gamma}_{\widetilde{I}_1,z}^1=
\sum_{\widetilde{z}\in
X_z}\widetilde{\gamma}\wedge\widetilde{I}_1(\widetilde{z}).$$

Observe that if $\widetilde{z}\in X_z$, then the trajectory of
$\widetilde{I}_1(\widetilde{z})$ is not included in $\widetilde{V}$.
Therefore we can write the multi-path $\prod_{\widetilde{z}\in
X_z}\widetilde{I}_1(\widetilde{z})$ as finitely many sub-paths:
$$\prod_{\widetilde{z}\in
X_z}\widetilde{I}_1(\widetilde{z})=\prod_{1\leq i\leq
P(z)}\widetilde{\Gamma}_i(z),$$ where
%$0<P(z)<K_{\widetilde{a},\widetilde{b}}$ and
$$\widetilde{\Gamma}_i(z)=\prod_{0\leq
j<m^i(z)}\widetilde{I}_1(\widetilde{F}^{j}_{\widetilde{a},\widetilde{b}}(\widetilde{z}_i))$$
is a path with $\widetilde{z}_i\in X_z\cap\widetilde{V}$,
$\widetilde{F}_{\widetilde{a},\widetilde{b}}^{j}(\widetilde{z}_i)\in
X_z\cap\widetilde{V}^c$ for $1\leq j<m^i$ and
$\widetilde{F}_{\widetilde{a},\widetilde{b}}^{m^i}(\widetilde{z}_i)\in
\widetilde{V}$. For every $i$, we get a periodic disk chain
$C_i=(\{\widetilde{V},\widetilde{V}\},\{m^i\})$ whose length
$l(C_i)$ is equal to $m^i$ (see Section \ref{sec:preliminaries}).

Obviously, $\sum_i m^i\leq K_{\widetilde{a},\widetilde{b}}$. Let
$k^i(z)=\widetilde{\gamma}\wedge\widetilde{\Gamma}_i$. We have
$i(\widetilde{F};\widetilde{a},\widetilde{b},z)=\widetilde{\gamma}\wedge\widetilde{\Gamma}_{\widetilde{I}_1,z}^1=\sum_i
k^i$. Therefore, to get the lemma, it is sufficient to prove that
$|k^i|< m^i(N_{\widetilde{a},\widetilde{b}}+1)$.

For every $i$, the path $\widetilde{\Gamma}_i$ is lifted to a path
from a point $\widehat{z}_i\in \widehat{V}$ to
$\widehat{F}^{m^i}_{\widetilde{a},\widetilde{b}}(\widehat{z}_i)\in
T_{\widetilde{a},\widetilde{b}}^{k^i}(\widehat{V})$ and hence we get
a lift
$\widetilde{C}_i=(\{\widehat{V},T_{\widetilde{a},\widetilde{b}}^{k^i}(\widehat{V})\},\{m^i\})$
of $C_i$ for $\widehat{F}_{\widetilde{a},\widetilde{b}}$ with width
$w(\widehat{F}_{\widetilde{a},\widetilde{b}}\,;\widetilde{C}_i)=k^i$.
By the construction of $\widetilde{V}$, replacing $\mathbb{A}$ by
$\widetilde{A}_{\widetilde{a},\widetilde{b}}$, $h$ by
$\widetilde{F}_{\widetilde{a},\widetilde{b}}$, $H$ by
$\widehat{F}_{\widetilde{a},\widetilde{b}}$, $D$ by $\widetilde{V}$
and $C$ by $C_i$ in Lemma \ref{lem:disk chain of Franks lemma}, we
get $|k^i|< m^i(N_{\widetilde{a},\widetilde{b}}+1)$. We have
completed the proof.
\end{proof}

\begin{flushleft}
    \emph{The non-fixed point case.}
\end{flushleft}
\label{sec:the non fixed point case}

Let $z\in \mathrm{Rec}^+(F)\setminus\mathrm{Fix}(F)$ and $U$ be an
open free disk for $F$ that contains $z$. Recall that, for every
lift $\widetilde{z}$ of $z$ and every $n\geq0$, there is a unique
connected component
$\widetilde{U}_{\widetilde{\Phi}^n(\widetilde{z})}$ of $\pi^{-1}(U)$
such that
$\widetilde{\Phi}^n(\widetilde{z})\in\widetilde{U}_{\widetilde{\Phi}^n(\widetilde{z})}$
and a unique
 $\alpha_{z,n}\in G$ such that
 $\widetilde{U}_{\widetilde{\Phi}^n(\widetilde{z})}=
\alpha_{z,n}(\widetilde{U}_{\widetilde{z}})$.
 For
convenience, we define

\begin{equation*}\widetilde{F}_{\widetilde{a},\widetilde{b}}^{*}(\widetilde{z}\,')=
\begin{cases}\widetilde{F}_{\widetilde{a},\widetilde{b}}(\widetilde{z}\,')& \textrm{if}
 \quad \pi(\widetilde{z}\,')\in\{z,\cdots,F^{\tau_n(z)-2}(z)\};
\\\alpha_{z,n}(\widetilde{z})& \textrm{if} \quad
\pi(\widetilde{z}\,')=F^{\tau_n(z)-1}(z)\,\,\mathrm{and}\,\,
\widetilde{F}_{\widetilde{a},\widetilde{b}}(\widetilde{z}\,')\in\widetilde{U}_{\alpha_{z,n}(\widetilde{z})}.\end{cases}
\end{equation*}
and
\begin{equation*}\widetilde{I}_{1}^{*}(\widetilde{z}\,')=
\begin{cases}\widetilde{I}_{1}(\widetilde{z}\,')& \textrm{if}
 \quad \pi(\widetilde{z}\,')\in\{z,\cdots,F^{\tau_n(z)-2}(z)\};
\\\widetilde{I}_1
(\widetilde{z}\,')\widetilde{\gamma}_{\widetilde{F}_{\widetilde{a},\widetilde{b}}(\widetilde{z}\,'),\,\alpha_{z,n}(\widetilde{z})}&
\textrm{if} \quad
\pi(\widetilde{z}\,')=F^{\tau_n(z)-1}(z)\,\,\mathrm{and}\,\,
\widetilde{F}_{\widetilde{a},\widetilde{b}}(\widetilde{z}\,')\in\widetilde{U}_{\alpha_{z,n}(\widetilde{z})},\end{cases}
\end{equation*}
where
$\widetilde{\gamma}_{\widetilde{F}_{\widetilde{a},\widetilde{b}}(\widetilde{z}\,'),\,\alpha_{z,n}(\widetilde{z})}$
is the lift of $\gamma_{\Phi^n(z),z}$ that is in
$\widetilde{U}_{\alpha_{z,n}(\widetilde{z})}$.
\smallskip

%Fixing a lift $\widetilde{z}$ of $z$,
We have to consider two cases: $\alpha_{z,n}=e$ and
$\alpha_{z,n}\neq e$. First, we consider the case where
$\alpha_{z,n}\neq e$. We have the following lemma.

\begin{lem}\label{lem:the boundedness of the case multi-path}
If $\alpha_{z,n}\neq e$, then
$|L_n(\widetilde{F};\widetilde{a},\widetilde{b},z)|<
\tau_n(z)K_{\widetilde{a},\widetilde{b}}
(N_{\widetilde{a},\widetilde{b}}+1)$.
\end{lem}

\begin{proof}
In this case, the curve $\widetilde{\Gamma}^n_{\widetilde{I}_1,z}$
is a multi-path in $\widetilde{M}$. %There are finitely many paths,
%written
%$\widetilde{\Gamma}^n_{\beta_l}(z)=\prod_{k\in\mathbb{Z}}\widetilde{
%\Gamma}^n_{\widetilde{I}_1,\beta_{l}(\alpha_{z,n}^k(\widetilde{z}))}$
%with $\beta_l\in G$ and $1\leq l<\tau_n(z)K$, such that
%$$\{\widetilde{F}^{j}(\beta_l(\alpha_{z,n}^k(\widetilde{z})))\mid
%0\leq j<\tau_n(z),k\in\mathbb{Z}\}\cap X_z^n\neq\emptyset.$$
By the definition of
$L_n(\widetilde{F};\widetilde{a},\widetilde{b},z)$, we have
$$L_n(\widetilde{F};\widetilde{a},\widetilde{b},z)=\widetilde{\gamma}
\wedge\widetilde{\Gamma}_{\widetilde{I}_1,z}^n=\sum_{\widetilde{z}\,'\in
X_z^n}\widetilde{\gamma}\wedge
\widetilde{I}^*_1(\widetilde{z}\,').$$

%Observe that if $\widetilde{z}\,'\in X_z^n$, then the trajectory of
%$\widetilde{I}^*_1(\widetilde{z}\,')$ does not belong to
%$\widetilde{V}$.
We can write the multi-path

\begin{equation}\label{eq:multi-path1}
    \prod_{\widetilde{z}\,'\in
X_z^n}\widetilde{I}^*_1(\widetilde{z}\,')=\prod_{1\leq i\leq
P_n(z)}\widetilde{\Gamma}^n_i(z),
\end{equation}
where
\begin{equation}\label{eq:multi-path2}
    \widetilde{\Gamma}^n_i(z)=\prod_{0\leq
j<m_n^i(z)}\widetilde{I}^*_1(\widetilde{F}^{*j}_{\widetilde{a},\widetilde{b}}(\widetilde{z}_i))
\end{equation}
is a path with $\widetilde{z}_i\in X^n_z\cap\widetilde{V}$,
$\widetilde{F}_{\widetilde{a},\widetilde{b}}^{*j}(\widetilde{z}_i)\in
X^n_z\cap\widetilde{V}^c$ for $1\leq j<m_n^i$ and
$\widetilde{F}_{\widetilde{a},\widetilde{b}}^{*m_n^i}(\widetilde{z}_i)\in
\widetilde{V}$. Hence, for every $i$, we get a periodic disk chain
$C_i$ that satisfies the hypothesis of Lemma \ref{lem:disk chain of
Franks lemma} with length $m_n^i$. When we lift the path
$\widetilde{\Gamma}^n_i$, we can get a lift of $C_i$ for
$\widehat{F}_{\widetilde{a},\widetilde{b}}$ with width $k_n^i$.

Obviously, we have $\sum_i m_n^i<
\tau_nK_{\widetilde{a},\widetilde{b}}$. Let
$k_n^i(z)=\widetilde{\gamma}\wedge\widetilde{\Gamma}^n_i$. Hence
$L_n(\widetilde{F};\widetilde{a},\widetilde{b},z)=\sum_i k_n^i$. It
is sufficient to prove that $|k_n^i|<
m_n^i(N_{\widetilde{a},\widetilde{b}}+1)$.

Similarly to the proof of Lemma \ref{lem:non contractible case of
fixed point}, replacing $\mathbb{A}$ by
$\widetilde{A}_{\widetilde{a},\widetilde{b}}$, $h$ by
$\widetilde{F}_{\widetilde{a},\widetilde{b}}$, $H$ by
$\widehat{F}_{\widetilde{a},\widetilde{b}}$, $D$ by $\widetilde{V}$
and $C$ by $C_i$ in Lemma \ref{lem:disk chain of Franks lemma}, we
get $|k_n^i|< m_n^i(N_{\widetilde{a},\widetilde{b}}+1)$. This proves
the first case.
\end{proof}

As a consequence, we have the following proposition.
\begin{prop}\label{prop:roM(z)not0 linking number uniformly bounded}
We suppose that $i(\widetilde{F};\widetilde{a},\widetilde{b},z)$ and
$\rho_{M,I}(z)$ exist, then
$$|i(\widetilde{F};\widetilde{a},\widetilde{b},z)|\leq K_{\widetilde{a},\widetilde{b}}
(N_{\widetilde{a},\widetilde{b}}+1)\quad\text{if}\quad \rho_{M,I}(z)\neq0.$$% and locally bounded if $\rho_M(z)=0$
\end{prop}
\begin{proof}If $z\in \mathrm{Fix}(F)$ and $\rho_{M,I}(z)\neq0$,
then $z$ is not a contractible fixed point and the conclusion
follows from Lemma \ref{lem:non contractible case of fixed point}.
Suppose now that $z\in \mathrm{Rec}^+(F)\setminus\mathrm{Fix}(F)$
and $U\subset M\setminus\mathrm{Fix}(F)$ is a free open disk
containing $z$. If $\rho_{M,I}(z)\neq0$, then there exists a
positive number $N$ such that $\alpha_{z,n}\neq e$ when $n\geq N$
(see \ref{sec:the existion of rotation vector in the copact case}).
In that case, the conclusion follows from Lemma \ref{lem:the
boundedness of the case multi-path}.
\end{proof}
\bigskip

Let us study the case where $\alpha_{z,n}=e$.
%For every $n\geq1$, let
%$p_n(\widetilde{z})=\widetilde{\gamma}\wedge\widetilde{\Gamma}_{\widetilde{I}_1,\widetilde{z}}^n$.
%By Definition \ref{the rotation number set for D}, we have
%$p_n(\widetilde{z})/\tau_n(z)\in
%\mathrm{Rot}_{\widetilde{U}_{\widetilde{z}}}(\widehat{F}_{\widetilde{a},\widetilde{b}})$
%and hence
%\begin{equation}\label{lem:boundedness of multi-loop 1}
%\left\{\frac{p_n(\widetilde{z})}{\tau_n(z)}~\Big\vert
%~\widetilde{z}\in\widetilde{U},\,z=\pi(\widetilde{z})\in
%\mathrm{Rec}^+(F)\cap
%U\,\,\mathrm{and}\,\,\alpha_{z,n}=e\right\}\subset
%\mathrm{Rot}_{\widetilde{U}}(\widehat{F}_{\widetilde{a},\widetilde{b}}).
%\end{equation}\smallskip

\begin{lem}\label{lem:alphazn is e, Ln divide taun is less
then kU} There exists a positive integer $K_U$ which depends on $U$
such that $$|L_n(\widetilde{F};\widetilde{a},\widetilde{b},z)|\leq
\tau_n(z)K_{\widetilde{a},\widetilde{b}}K_U\quad \mathrm{if}\quad
\alpha_{z,n}=e.$$
\end{lem}

Before proving Lemma \ref{lem:alphazn is e, Ln divide taun is less
then kU}, we require the following lemma.

\begin{lem}\label{lem:boundedness of the rotation number disk in some
condition}Let $\widetilde{U}$ be any connected component of
$\pi^{-1}(U)$ in $\widetilde{V}^c$. If
$$\mathrm{Rot}_{\widetilde{U}}(\widehat{F}_{\widetilde{a},\widetilde{b}})\nsubseteq\,
]-(N_{\widetilde{a},\widetilde{b}}+1),N_{\widetilde{a},\widetilde{b}}+1[,$$
then we have \begin{enumerate}
               \item $\alpha_{z',n}=e$ for all $z'\in \mathrm{Rec}^+(F)\cap
               U$ and all $n\geq1$;
               \item $\bigcup_{k\geq1}\widetilde{F}^k(\pi^{-1}(\mathrm{Rec}^+(F))\cap \widetilde{U})\subset
               \widetilde{V}^c$;
               \item $\mathrm{Rot}_{\widetilde{U}}(\widehat{F}_{\widetilde{a},
 \widetilde{b}})\subset]l,l+1[$
for some integer $l$ with $l\geq N_{\widetilde{a},\widetilde{b}}+1$
or $l\leq -(N_{\widetilde{a},\widetilde{b}}+2)$ where $l$ depends on
$\widetilde{U}$.
             \end{enumerate}
\end{lem}

Let us prove now Lemma \ref{lem:alphazn is e, Ln divide taun is less
then kU} supposing Lemma \ref{lem:boundedness of the rotation number
disk in some condition} whose proof will be given later.

\begin{proof}[Proof of Lemma \ref{lem:alphazn is e, Ln divide taun is less then kU}]
As $\alpha_{z,n}=e$, the curve
$\widetilde{\Gamma}^n_{\widetilde{I}_1,z}$ is a multi-loop in
$\widetilde{M}$. Let
$p_n(\widetilde{z})=\widetilde{\gamma}\wedge\widetilde{\Gamma}^n_{\widetilde{I}_1,\widetilde{z}}$
where $\widetilde{z}\in\pi^{-1}(z)$. Obviously,
$p_n(\widetilde{z})/\tau_n(z)\in\mathrm{Rot}_{\widetilde{U}_{\widetilde{z}}}(\widehat{F}_{\widetilde{a},\widetilde{b}})$.

Let us first analyze the possible cases that need to be considered
in the proof. The set $X_z^n$ maybe contain a ``whole orbit'' of
some lift $\widetilde{z}$ of $z$, that means
$\widetilde{F}^{j}(\widetilde{z})\in X^n_z$ for all $0\leq
j<\tau_n(z)$, or a ``partial orbit'' of $\widetilde{z}$. In the case
where a ``partial orbit'' of $\widetilde{z}$ is contained in
$X_z^n$, similarly to the proof of Lemma \ref{lem:non contractible
case of fixed point}, we can get a  periodic disk chain of
$\widetilde{F}_{\widetilde{a},\widetilde{b}}$ that satisfies the
hypothesis of Lemma \ref{lem:disk chain of Franks lemma} and hence
we can estimate the intersection number of $\widetilde{\gamma}$ and
the path on which the ``partial orbit'' of $\widetilde{z}$ lies. In
the case where the ``whole orbit'' of $\widetilde{z}$ is contained
in $X_z^n$, we can use Lemma \ref{lem:boundedness of the rotation
number disk in some condition} to get either
$|p_n(\widetilde{z})/\tau_n(z)|< N_{\widetilde{a},\widetilde{b}}+1$,
or $l<p_n(\widetilde{z})/\tau_n(z)<l+1$ where $l\in\mathbb{Z}$
depends on $\widetilde{U}$ and satisfies $l\geq
N_{\widetilde{a},\widetilde{b}}+1$ or $l\leq
-(N_{\widetilde{a},\widetilde{b}}+2)$. Finally, we only need to sum
the intersection numbers of all the cases above.
\smallskip

Let us begin the rigorous proof. Write
$$S^n_z=\{\widetilde{z}\in\pi^{-1}(z)\mid
\widetilde{F}^{j}(\widetilde{z})\in\widetilde{V}^c\,\,\,\mathrm{for\,\,all}\,\,\,0\leq
j<\tau_n(z)\}$$ and
$$Y^n_z=\{\widetilde{F}^{j}(\widetilde{z})\mid \widetilde{z}\in
S_z^n,0\leq j<\tau_n(z)\}.$$\smallskip
%Note that the set $Y^n_z$ may be empty. and $0\leq\sharp
%Y^n_z<\tau_n(z)K$.

As before, we write
$$L_n(\widetilde{F};\widetilde{a},\widetilde{b},z)=\widetilde{\gamma}
\wedge\widetilde{\Gamma}_{\widetilde{I}_1,z}^n=\sum_{\widetilde{z}\,'\in
X_z^n}\widetilde{\gamma}\wedge
\widetilde{I}^*_1(\widetilde{z}\,').$$ We can write the multi-path
as follows
\begin{equation}\label{eq:multi-path3}
\prod_{\widetilde{z}\,'\in
X_z^n}\widetilde{I}^*_1(\widetilde{z}\,')=\prod_{\widetilde{z}\,'\in
Y_z^n}\widetilde{I}^*_1(\widetilde{z}\,')\cdot\prod_{\widetilde{z}\,'\in
X_z^n\setminus
Y_z^n}\widetilde{I}^*_1(\widetilde{z}\,')=\prod_{1\leq i\leq
P'_n(z)}\widetilde{\Gamma}^n_i(z)\cdot\prod_{P'_n(z)< i\leq
P_n(z)}\widetilde{\Gamma}^n_i(z),
\end{equation}
where
\begin{equation}\label{eq:multi-path4}
\widetilde{\Gamma}^n_i(z)=\widetilde{\Gamma}^n_{\widetilde{I}_1,\widetilde{z}_i}=
\prod_{0\leq
j<m_n^i(z)}\widetilde{I}^*_1(\widetilde{F}^{*j}_{\widetilde{a},\widetilde{b}}(\widetilde{z}_i))
\end{equation}
for $1\leq i\leq P'_n$ %(note that $0\leq P'_n<K$)
with $\widetilde{z}_i\in S_z^n$ and $m_n^i=\tau_n$; and
\begin{equation}\label{eq:multi-path5}
\widetilde{\Gamma}^n_i(z)=\prod_{0\leq
j<m_n^i(z)}\widetilde{I}^*_1(\widetilde{F}^{*j}_{\widetilde{a},\widetilde{b}}(\widetilde{z}_i))
\end{equation}
for $P'_n< i\leq P_n$ with $\widetilde{z}_i\in
X^n_z\cap\widetilde{V}$,
$\widetilde{F}_{\widetilde{a},\widetilde{b}}^{*j}(\widetilde{z}_i)\in
X^n_z\cap\widetilde{V}^c$ for $1\leq j<m_n^i$ and
$\widetilde{F}_{\widetilde{a},\widetilde{b}}^{*m_n^i}(\widetilde{z}_i)\in
\widetilde{V}$.\smallskip

Obviously, $\sum_i m_n^i\leq
\tau_n(z)K_{\widetilde{a},\widetilde{b}}$. Let
$k_n^i(z)=\widetilde{\gamma}\wedge\widetilde{\Gamma}^n_i$. Hence
$L_n(\widetilde{F};\widetilde{a},\widetilde{b},z)=\sum_i k_n^i$. To
prove Lemma \ref{lem:alphazn is e, Ln divide taun is less then kU},
it is sufficient to prove that there exists a positive integer $K_U$
which depends only on $U$ such that $|k_n^i|\leq
m_n^iK_U$.\smallskip

When $1\leq i\leq P'_n$, by Lemma \ref{lem:boundedness of the
rotation number disk in some condition} and the fact that $P'_n\leq
K_{\widetilde{a},\widetilde{b}}$, there exists a positive integer
$r$ that depends on $U$ such that
$\mathrm{Rot}_{\widetilde{U}_{\widetilde{z}_i}}(\widehat{F}_{\widetilde{a},\widetilde{b}})
\subset[-r,r]$. Observing that
$k_n^i=p_n(\widetilde{z}_i)=\widetilde{\gamma}\wedge\widetilde{\Gamma}^n_{\widetilde{I}_1,\widetilde{z}_i}$,
$m_n^i=\tau_n$, and $k_n^i/m_n^i=p_n(\widetilde{z}_i)/\tau_n(z)\in
\mathrm{Rot}_{\widetilde{U}_{\widetilde{z}_i}}(\widehat{F}_{\widetilde{a},\widetilde{b}})$,
we have $|k_n^i|\leq m_n^ir$.

When $P'_n< i\leq P_n$, similarly to the proof of Lemma \ref{lem:non
contractible case of fixed point}, we can get $|k_n^i|<
m_n^i(N_{\widetilde{a},\widetilde{b}}+1)$.

Let $K_U =\max\{N_{\widetilde{a},\widetilde{b}}+1,r\}$. We have
$|k_n^i|\leq m_n^iK_U$ for every $1\leq i\leq P_n$ and hence
$$|L_n(\widetilde{F};\widetilde{a},\widetilde{b},z)|=\left|\sum_i
k_n^i\right|\leq \tau_n(z)K_{\widetilde{a},\widetilde{b}}K_U.$$ We
have completed the proof.
\end{proof}
\smallskip

\begin{proof}[Proof of Lemma \ref{lem:boundedness of the rotation number disk in some
condition}] (1).\quad Suppose that there is a point $z'\in
\mathrm{Rec}^+(F)\cap U$ and some $n_0\geq1$ such that
$\alpha_{z',n_0}\neq e$. Let $\widetilde{z}\,'$ be the lift of $z'$
that is in $\widetilde{U}$. Similarly to the proof of Lemma
\ref{lem:the boundedness of the case multi-path}, we can find a path
$$\widetilde{\Gamma}^{n_0}_i(z')=\prod_{0\leq
j<m_{n_0}^i(z')}\widetilde{I}^*_1(\widetilde{F}^{*j}_{\widetilde{a},\widetilde{b}}(\widetilde{z}_i))$$
which satisfies $\widetilde{z}_i\in X^{n_0}_{z'}\cap\widetilde{V}$,
$\widetilde{F}_{\widetilde{a},\widetilde{b}}^{*j}(\widetilde{z}_i)\in
X^{n_0}_{z'}\cap\widetilde{V}^c$ for all $1\leq j<m_{n_0}^i$,
$\widetilde{z}\,'=\widetilde{F}_{\widetilde{a},\widetilde{b}}^{*j_0}(\widetilde{z}_i)$
for some $1\leq j_0<m_{n_0}^i$, and
$\widetilde{F}_{\widetilde{a},\widetilde{b}}^{*m_{n_0}^i}(\widetilde{z}_i)\in
\widetilde{V}$. Hence, we get a periodic disk chain $C\,'$ that
contains $\widetilde{U}$ as an element and satisfies the hypothesis
of Lemma \ref{lem:disk chain of Franks lemma}. Replacing
$\mathbb{A}$ by $\widetilde{A}_{\widetilde{a},\widetilde{b}}$, $h$
by $\widetilde{F}_{\widetilde{a},\widetilde{b}}$, $H$ by
$\widehat{F}_{\widetilde{a},\widetilde{b}}$, $D$ by $\widetilde{V}$
and $C$ by $C\,'$ in Lemma \ref{lem:disk chain of Franks lemma} (the
second conclusion), we get
$\mathrm{Rot}_{\widetilde{U}}(\widehat{F}_{\widetilde{a},\widetilde{b}})\subset\,
]-(N_{\widetilde{a},\widetilde{b}}+1),N_{\widetilde{a},\widetilde{b}}+1[$.
We have a contradiction.\smallskip

(2).\quad Suppose that there is a point $\widetilde{z}\,'\in
\pi^{-1}(z')\cap \widetilde{U}$ where $z'\in \mathrm{Rec}^+(F)$ and
an integer $n_0\geq 1$ such that
$\widetilde{F}^{n_0}(\widetilde{z}\,')\in \widetilde{V}$. By (1), it
is sufficient to consider the case where $\alpha_{z',n}=e$ for all
$n\geq1$, that means,
$\widetilde{F}^{\tau_{n}(z')}(\widetilde{z}\,')\in \widetilde{U}$
for all $n\geq1$. We choose a positive integer $n_1$ large enough
such that $\tau_{n_1}(z')>n_0$. We have
$\widetilde{F}^{\tau_{n_1}(z')-n_0}(\widetilde{F}^{n_0}(\widetilde{z}\,'))\in\widetilde{U}$.
Then we get $\widetilde{F}^{\tau_{n_1}(z')-n_0}(\widetilde{V})\cap
\widetilde{U}\neq\emptyset$ and
$\widetilde{F}^{n_0}(\widetilde{U})\cap \widetilde{V}\neq\emptyset$.
Therefore, the disk chain
$(\{\widetilde{V},\widetilde{U},\widetilde{V}\},\{\tau_{n_1}(z')-n_0,n_0\})$
is a  periodic disk chain that satisfies the hypothesis of Lemma
\ref{lem:disk chain of Franks lemma}. Applying Lemma \ref{lem:disk
chain of Franks lemma} again, we get
$\mathrm{Rot}_{\widetilde{U}}(\widehat{F}_{\widetilde{a},\widetilde{b}})\subset\,
]-(N_{\widetilde{a},\widetilde{b}}+1),N_{\widetilde{a},\widetilde{b}}+1[$.
It is still a contradiction.\smallskip

(3).\quad This follows from Lemma  \ref{lem:disk chain of Franks
lemma} immediately.
\end{proof}

\begin{proof}[Proof of Proposition \ref{prop:l(a,b,x)divid tau(x)is boundedonU}]
This follows from Lemma \ref{lem:contractible fixed point case},
Lemma \ref{lem:non contractible case of fixed point}, Lemma
\ref{lem:the boundedness of the case multi-path} and Lemma
\ref{lem:alphazn is e, Ln divide taun is less then kU}.
\end{proof}

\bigskip

In the end of this section, we study the boundedness in the case
where the time-one map $F$ of $I$ satisfies some differential
conditions. \label{sec:boundedness in the case where F is C1
diffeomorphism}

\begin{prop}\label{prop:diffeomorphism has a bounded
linking number for every two distinct fixed points} For any two
distinct fixed points $\widetilde{a}$ and $\widetilde{b}$ of
$\widetilde{F}$, if $F$ and $F^{-1}$ are differentiable at
$\pi(\widetilde{a})$ and $\pi(\widetilde{b})$, then there exists
$N\in\mathbb{R}$ such that
$|i(\widetilde{F};\widetilde{a},\widetilde{b},z)|\leq N$ if
$i(\widetilde{F};\widetilde{a},\widetilde{b},z)$ exists.
\end{prop}
\begin{proof}
We make a proof by contradiction. If it is not true, without loss of
generality, we suppose that there is a sequence
$\{z_k\}_{k\geq1}\subset \mathrm{Rec}^+(F)$ such that
$\lim\limits_{k\rightarrow+\infty}i(\widetilde{F};\widetilde{a},\widetilde{b},z_k)=+\infty$.
By the proof of Lemma \ref{lem:alphazn is e, Ln divide taun is less
then kU} and the conclusion (1) of Lemma \ref{lem:boundedness of the
rotation number disk in some condition}, we have $\alpha_{z_k,n}=e$
for every $n\geq1$ when $k$ is large enough. Hence
$\widetilde{z}_k\in
\mathrm{Rec}^+(\widetilde{F})\setminus\mathrm{Fix}(\widetilde{F})$
when $k$ is large enough where $\widetilde{z}_k\in\pi^{-1}(z_k)$. By
the proof of Lemma \ref{lem:alphazn is e, Ln divide taun is less
then kU} and the conclusion (2) of Lemma \ref{lem:boundedness of the
rotation number disk in some condition}, we only need consider the
lifts $\widetilde{z}_k$ of $z_k$ whose whole orbit is in
$\widetilde{V}^c$
 when $k$ is large enough. However, such lifts are
finite (at most $K_{\widetilde{a},\widetilde{b}}$). This implies
that there exists a sequence $\{\widetilde{z}_k\}_{k\geq1}$ with
$\widetilde{z}_k\in\pi^{-1}(z_k)$ such that
$\lim\limits_{k\rightarrow+\infty}\rho_{A_{\widetilde{a},\widetilde{b}},\widehat{F}
_{\widetilde{a},\widetilde{b}}}(\widetilde{z}_k)=+\infty$, which
conflicts with Lemma \ref{lem:diffeomorphism on a and b}.
\end{proof}

In Example \ref{exem:proposition 21 and mu integrable} of Appendix,
we will construct an identity isotopy $I$ of a closed surface such
that $I$ satisfies the B-property but its time-one map is not a
diffeomorphism and there are two different fixed points
$\widetilde{z}_0$ and $\widetilde{z}_1$ of $\widetilde{F}$ such that
the linking number
$i(\widetilde{F};\widetilde{z}_0,\widetilde{z}_1,z)$ is not
uniformly bounded for
$z\in\mathrm{Rec}^+(F)\setminus\pi(\{\widetilde{z}_0,\widetilde{z}_1\})$.\smallskip

In order to study the boundedness and continuity of the generalize
action in the next section, we need the following proposition:

\begin{prop}\label{prop:C1 diffeomorphism has a bounded
linkingnumber} Let $F\in\mathrm{Diff}^1(M)$ be the time-one map of
$I$ and $\widetilde{P}\subset \widetilde{M}$ be a connected compact
set. There exists $N_{\widetilde{P}}\geq0$ such that, for every two
distinct fixed points $\widetilde{a}$ and $\widetilde{b}$ of
$\widetilde{F}$ in $\widetilde{P}$, and $z\in
\mathrm{Rec}^+(F)\setminus\pi(\{\widetilde{a},\widetilde{b}\})$, we
have $|i(\widetilde{F};\widetilde{a},\widetilde{b},z)|\leq
N_{\widetilde{P}}$ when
$i(\widetilde{F};\widetilde{a},\widetilde{b},z)$  exists.
\end{prop}
\begin{proof}
We make a proof by contradiction. If it is not true, without loss of
generality, we suppose that there is a sequence
$\{(\widetilde{a}_k,\widetilde{b}_k)\}_{k\geq1}\subset\mathrm{Fix}
(\widetilde{F})\times\mathrm{Fix}(\widetilde{F})\setminus\widetilde{\Delta}$
and a sequence $\{z_k\}_{k\geq1}$ where $z_k\in
\mathrm{Rec}^+(F)\setminus\pi(\{\widetilde{a}_k,\widetilde{b}_k\})$
such that
$\lim\limits_{k\rightarrow+\infty}i(\widetilde{F};\widetilde{a}_k,\widetilde{b}_k,z_k)=+\infty$.
As $\widetilde{P}$ is compact, by extracting subsequences, we can
suppose that there exist two fixed points $\widetilde{a}$ and
$\widetilde{b}$ of $\widetilde{F}$ in $\widetilde{P}$ such that
$\lim\limits_{k\rightarrow+\infty}\widetilde{a}_k=\widetilde{a}$ and
$\lim\limits_{k\rightarrow+\infty}\widetilde{b}_k=\widetilde{b}$.

We identify the sphere $\mathbf{S}$ as the Riemann sphere
$\mathbb{C}\cup\{\infty\}$. Recall that
$\widetilde{I}=(\widetilde{F}_t)_{t\in[0,1]}$ is the lifted identity
isotopy of $I$ on the universal cover $\widetilde{M}$ of $M$.
Replacing $v_i$, $w_i$ ($i=1,2,3$) by
$v_1(k,t)=\widetilde{F}_t(\widetilde{a}_k)$,
$v_2(k,t)=\widetilde{F}_t(\widetilde{b}_k)$,
$w_1(k,t)=\widetilde{a}_k$, $w_2(k,t)=\widetilde{b}_k$,
$v_3(k,t)=w_3(k,t)=\infty$ ($t\in[0,1]$) in the matrices in the
proof of Lemma \ref{rem:identity isotopies fix three points on
sphere}, we get the matrix functions $a_t(k)$, $b_t(k)$, $c_t(k)$
and $d_t(k)$.

Let $$\mathcal {M}_{k}(t,z)=\frac{a_t(k)z+b_t(k)}{c_t(k)z+d_t(k)}$$
and
$$\widetilde{I}_{k}(z)(t)=\mathcal {M}_{k}(t,F_t(z)).$$
By the construction, $\widetilde{I}_k$ is an isotopy on
$\widetilde{M}$ from $\mathrm{Id}_{\widetilde{M}}$ to
$\widetilde{F}$ that fixes $\widetilde{a}_k$ and $\widetilde{b}_k$.

Similarly, we can construct $$\mathcal
{M}'_{k}(t,z)=\frac{a'_t(k)z+b'_t(k)}{c'_t(k)z+d'_t(k)},\quad\mathcal
{M}''_{k}(t,z)=\frac{a''_t(k)z+b''_t(k)}{c''_t(k)z+d''_t(k)}$$ and
$$\widetilde{I}\,'_{k}(z)(t)=\mathcal {M}'_{k}(t,F_t(z)),\quad \widetilde{I}\,''_{k}(z)(t)=\mathcal {M}''_{k}(t,F_t(z))$$
 such that $\widetilde{I}\,'_k$ (resp. $\widetilde{I}\,''_k$) is an isotopy on $\widetilde{M}$ from
$\mathrm{Id}_{\widetilde{M}}$ to $\widetilde{F}$ that fixes
$\widetilde{a}$ (resp. $\widetilde{b}$) and $\widetilde{a}_k$ (resp.
$\widetilde{b}_k$). As $\widetilde{I}_k$ (resp.
$\widetilde{I}\,'_k$, $\widetilde{I}\,''_k$) fixes $\infty$, we have
$c_t(k)=0$ (resp. $c'_t(k)=0$, $c''_t(k)=0$) and
$a_t(k)\,d_t(k)\neq0$ (resp. $a'_t(k)\,d'_t(k)\neq0$,
$a''_t(k)\,d''_t(k)\neq0$) for all $t\in[0,1]$ and $k\geq1$.
\smallskip

Fix an open disk $\widetilde{W}=\{z\in\mathbb{C}\mid |z|>p\}$ that
contains $\infty$ and is disjoint from $\widetilde{P}$.  As
$\lim\limits_{k\rightarrow+\infty}\widetilde{a}_k=\widetilde{a}$\,
$\lim\limits_{k\rightarrow+\infty}\widetilde{b}_k=\widetilde{b}$ and
all of the matrix functions above are continuous on
$\widetilde{P}\times[0,1]$, it is easy to see that the norms of
these functions have a maximal value $p_{\mathrm{max}}>0$ and a
minimal value $p_{\mathrm{min}}>0$. Define the open disk
$$\widetilde{V}=\left\{z\in\mathbb{C}\,\Big\vert
|z|>\frac{(p+1)p_{\mathrm{max}}}{p_{\mathrm{min}}}\right\}.$$
Obviously, $\widetilde{V}\subset \widetilde{W}$ containing $\infty$
and for every $\widetilde{z}\in \widetilde{V}$, we have
$\widetilde{I}_k(\widetilde{z})\subset \widetilde{W}$,
$\widetilde{I}\,'_k(\widetilde{z})\subset \widetilde{W}$ and
$\widetilde{I}\,''_k(\widetilde{z})\subset \widetilde{W}$ for all
$k\geq1$. Let $\widetilde{a}\,'$ and $\widetilde{b}\,'$ be two
distinct fixed points of $\widetilde{F}$ in $\widetilde{P}$. As the
linking number
$i(\widetilde{F};\widetilde{a}\,',\widetilde{b}\,',z)$ does not
depend on the choice of $\widetilde{\gamma}$ that joins
$\widetilde{a}\,'$ and $\widetilde{b}\,'$ (see \ref{subsec:the
definition of a new linking number}), we can suppose
$\widetilde{\gamma}\subset\widetilde{P}$ in this proof when we talk
of the linking number
$i(\widetilde{F};\widetilde{a}\,',\widetilde{b}\,',z)$. For
$\widetilde{W}$ and $\widetilde{V}$ here, Lemma \ref{lem:alphazn is
e, Ln divide taun is less then kU} and Lemma \ref{lem:boundedness of
the rotation number disk in some condition} are still valid.
\smallskip

As $F\in\mathrm{Diff}^1(M)$, by Proposition \ref{rem:any isotopy of
M satisfies condition B}, $I$ satisfies the B-property. Consider the
annulus $A_{\widetilde{a}_k,\widetilde{b}_k}$. Similarly to the
proof of Proposition \ref{prop:diffeomorphism has a bounded linking
number for every two distinct fixed points}, we have
%$\alpha_{z_k,n}=e$ for every $n\geq1$, and
$\widetilde{z}_k\in
\mathrm{Rec}^+(\widetilde{F})\setminus\mathrm{Fix}(\widetilde{F})$
 where $\widetilde{z}_k\in\pi^{-1}(z_k)$ when
$k$ is large enough.

For every $k$, we choose an open disk $U_k$ containing $z_k$. Let
$\Phi_k(z)$ be the first return map of $z\in U_k$ and $\tau(k,z)$ be
the first return time of $z$ in this proof. Recall that
$\tau_n(k,z)=\sum_{i=0}^{n-1}\tau(k,\Phi_k(z))$.

In the proof of Proposition \ref{lem:i is 3coboundary for point}, we
have proved that, for every $k$ and any subsequence
$\{\Phi_k^{n_{l}}(z_k)\}_{l\geq 1}$ which converges to $z_k$, we
have
\begin{eqnarray}\label{eq: C1 diffeomorphism the linking number of a
compact}
% \nonumber to remove numbering (before each equation)
  \frac{L_{n_l}(\widetilde{F};\widetilde{a},\widetilde{a}_k,z_k)+L_{n_l}(\widetilde{F};\widetilde{a}_k,
\widetilde{b}_k,z_k)+L_{n_l}(\widetilde{F};\widetilde{b}_k,\widetilde{b},z_k)}{\tau_{n_l}(k,z_k)}=
0.
\end{eqnarray}

On one hand, by the definition of linking number, we have
$$i(\widetilde{F};\widetilde{a}_k,\widetilde{b}_k,z_k)=
\lim_{l\rightarrow+\infty}\frac{L_{n_l}(\widetilde{F};
\widetilde{a}_k,\widetilde{b}_k,z_k)}{\tau_{n_l}(k,z_k)}.$$ As
$\lim\limits_{k\rightarrow+\infty}i(\widetilde{F};\widetilde{a}_k,\widetilde{b}_k,z_k)=+\infty$,
we have that, for any $N>0$, there is $K_N\in \mathbb{N}$ such that
when $l,k\geq K_N$,
\begin{equation}\label{eq: C1 diffeomorphism the linking number of a
compact 1}
    \frac{L_{n_l}(\widetilde{F};
\widetilde{a}_k,\widetilde{b}_k,z_k)}{\tau_{n_l}(k,z_k)}>N.
\end{equation}

On the other hand, let us study
$\frac{L_{n_l}(\widetilde{F};\widetilde{a},\widetilde{a}_k,z_k)}{\tau_{n_l}(k,z_k)}$
and
$\frac{L_{n_l}(\widetilde{F};\widetilde{b}_k,\widetilde{b},z_k)}{\tau_{n_l}(k,z_k)}$
when $l$ and $k$ are large enough. By the proof of Lemma
\ref{lem:alphazn is e, Ln divide taun is less then kU} and the
conclusion (2) of Lemma \ref{lem:boundedness of the rotation number
disk in some condition}, we only need consider the lift
$\widetilde{z}_k$ of $z_k$ whose whole orbit is in $\widetilde{V}^c$
when $k$ is large enough. Note that such lifts are finite. Observing
 the proof of Proposition \ref{rem:any isotopy
of M satisfies condition B}, there exists $N'\geq0$ such that
\begin{equation*}\label{eq: C1 diffeomorphism the linking number of a
compact 2}
    \left|\frac{L_{n_l}(\widetilde{F};\widetilde{a},\widetilde{a}_k,z_k)}{\tau_{n_l}(k,z_k)}\right|\leq
N'\quad
\mathrm{and}\quad\left|\frac{L_{n_l}(\widetilde{F};\widetilde{b}_k,\widetilde{b},z_k)}{\tau_{n_l}(k,z_k)}\right|\leq
N'
\end{equation*}
when $k$ and $l$ are large enough, which conflicts with Equation
\ref{eq: C1 diffeomorphism the linking number of a compact} and
Inequation \ref{eq: C1 diffeomorphism the linking number of a
compact 1}. We have completed the proof.
\end{proof}

In Example \ref{ex:the action of non C1-diffeo isnot bounded and
continuous} of Appendix, we will construct an identity isotopy $I$
whose time-one map is a diffeomorphism but not a
$C^1$-diffeomorphism, that does not satisfy the conclusion of
Proposition \ref{prop:C1 diffeomorphism has a bounded
linkingnumber}.\bigskip

\subsection{Existence and Boundedness in the conservative case}\label{subsec:the proof of theorem 1}

\begin{prop}\label{prop:measure and linking number exsiting}
Suppose that $I$ satisfies the WB-property at $\widetilde{a}$ and
$\widetilde{b}$. If $\mu\in \mathcal{M}(F)$ satisfies
$\mu(\pi(\widetilde{a}))=\mu(\pi(\widetilde{b}))=0$, then
$\mu$-almost every point $z\in \mathrm{Rec}^+(F)$ has a rotation
vector $\rho_{M,I}(z)\in H_1(M,\mathbb{R})$ and has a linking number
$i(\widetilde{F};\widetilde{a},\widetilde{b},z)\in \mathbb{R}$.
There exists $C > 0$ such that, for every point $z$ such that
$\rho_{M,I}(z)$ exists and is not equal to zero, one has
$|i(\widetilde{F};\widetilde{a},\widetilde{b},z)|\leq C$ if this
linking number exists.
\end{prop}
\begin{proof}According to Poincar\'{e} Recurrence Theorem, we have $\mu(\mathrm{Rec}^+(F))=\mu(M)$.

When $z\in
\mathrm{Fix}(F)\setminus\pi(\{\widetilde{a},\widetilde{b}\})$, by
Section \ref{sec:the existion of rotation vector in the copact case}
and \ref{sec:the fixed point case}, $\rho_{M,I}(z)$ and
$i(\widetilde{F};\widetilde{a},\widetilde{b},z)$ exist and are
bounded. Thus we only need to consider the non-fixed point case.

Fix a free open disk $U\subset
M\setminus\pi(\{\widetilde{a},\widetilde{b}\})$ with $\mu(U)>0$. For
any $z\in \mathrm{Rec}^+(F)\cap U$, by Lemma \ref{lem:the
boundedness of the case multi-path} and Lemma \ref{lem:alphazn is e,
Ln divide taun is less then kU}, we have
$|L_1(\widetilde{F};\widetilde{a},\widetilde{b},z)|\leq
\tau(z)K_{\widetilde{a},\widetilde{b}}
(N_{\widetilde{a},\widetilde{b}}+1)$ if $\alpha_{z,1}\neq e$ and
$|L_{1}(\widetilde{F};\widetilde{a},\widetilde{b},z)|\leq
\tau(z)K_{\widetilde{a},\widetilde{b}}K_U$ if $\alpha_{z,1}=e$. This
implies that $L_{1}(\widetilde{F};\widetilde{a},\widetilde{b},z)\in
L^1(U,\mathbb{R},\mu)$. By %Equation \ref{eq: Ln Birkhoff sum} and
Birkhoff Ergodic Theorem, we deduce that the sequence
$\{L_{n}(\widetilde{F};\widetilde{a},\widetilde{b},z)/n\}_{n=1}^{+\infty}$
converges to a real number
$L^*(\widetilde{F};\widetilde{a},\widetilde{b},z)$ for $\mu$-almost
every point on $\mathrm{Rec}^+(F)\cap U$. Recall that, for
$\mu$-almost every point on $\mathrm{Rec}^+(F)\cap U$, the sequence
$\{\tau_n(z)/n\}_{n=1}^{+\infty}$ converges to a real number
$\tau^*(z)$ (see \ref{sec:the existion of rotation vector in the
copact case}).

We can define the linking number on $U$ as follows (modulo sets of
measure zero):

\begin{equation}\label{eq: linking number with a measure}
    i(\widetilde{F};\widetilde{a},\widetilde{b},z)=\lim_{n\rightarrow+\infty}
\frac{L_{n}(\widetilde{F};\widetilde{a},\widetilde{b},z)}{\tau_n(z)}=
\frac{L^*(\widetilde{F};\widetilde{a},\widetilde{b},z)}{\tau^*(z)}.
\end{equation}
By Proposition \ref{prop:l(a,b,x)divid tau(x)is boundedonU}, the
linking number $i(\widetilde{F};\widetilde{a},\widetilde{b},z)$ has
a bound $K_U$ for $\mu$-almost every point $z\in
\mathrm{Rec}^+(F)\cap U$. As $U$ is arbitrarily chosen, this implies
that we can define the function
$i(\widetilde{F};\widetilde{a},\widetilde{b},z)$ for $\mu$-almost
every point $z\in M\setminus\pi(\{\widetilde{a},\widetilde{b}\})$.

Finally, by Proposition \ref{prop:roM(z)not0 linking number
uniformly bounded}, we can uniformly bound
$i(\widetilde{F};\widetilde{a},\widetilde{b},z)$ if
$\rho_{M,I}(z)\neq 0$.
\end{proof}
Remark here that, under the hypothesis of Proposition
\ref{prop:measure and linking number exsiting},
$i(\widetilde{F};\widetilde{a},\widetilde{b},z)$ is bounded on $U$,
but does not necessarily possess a uniform bound on
$M\setminus\pi(\{\widetilde{a},\widetilde{b}\})$ (see Example
\ref{exem:proposition 21 and mu integrable} in Appendix). However,
when $F$ is a diffeomorphism of $M$ (see Proposition
\ref{prop:diffeomorphism has a bounded linking number for every two
distinct fixed points}), we can get a uniform bound. Moreover, we
can get a uniform bound in the case where the support of the measure
is the whole space, as stated in the following proposition.
\begin{prop}\label{prop:the support of mu is M}
With the same hypotheses as Proposition \ref{prop:measure and
linking number exsiting} and if furthermore $\mu\in\mathcal {M}(F)$
has total support, we have
$|i(\widetilde{F};\widetilde{a},\widetilde{b},z)|\leq
K_{\widetilde{a},\widetilde{b}}(N_{\widetilde{a},\widetilde{b}}+1)$
if it exists.
\end{prop}
\begin{proof}
The measure $\mu$ may naturally be lifted to a (non finite) measure
$\widetilde{\mu}$ on $\widetilde{M}$. Since $\mu$ does not charge
$\pi(\widetilde{a})$ and $\pi(\widetilde{b})$, $\widetilde{\mu}$ can
be seen as a measure on $A_{\widetilde{a},\widetilde{b}}$ invariant
by $\widetilde{F}_{\widetilde{a},\widetilde{b}}$ satisfying
$\widetilde{\mu}(A_{\widetilde{a},\widetilde{b}})=+\infty$. As the
support of $\widetilde{\mu}$ is $\widetilde{M}$ and
$\widetilde{F}_{{\widetilde{a}},{\widetilde{b}}}$ preserves the
measure $\widetilde{\mu}$, the homeomorphism
$\widetilde{F}_{{\widetilde{a}},{\widetilde{b}}}$ satisfies the
intersection property, that is, any simple closed curve of
$A_{\widetilde{a},\widetilde{b}}$ which is not null-homotopic meets
its image by $\widetilde{F}_{{\widetilde{a}},{\widetilde{b}}}$.
Indeed, any closed curve which goes through $\infty$ will meet its
image by $\widetilde{F}_{{\widetilde{a}},{\widetilde{b}}}$ since
$\widetilde{F}_{{\widetilde{a}},{\widetilde{b}}}$ fixes the point
$\infty$. If the closed curve does not pass through $\infty$, we may
go back to $\widetilde{M}$ and consider a component enclosed by the
closed curve which contains $\widetilde{a}$ or $\widetilde{b}$ and
which has finite measure, then it will meet its image since
$\widetilde{F}$ preserves the measure $\widetilde{\mu}$.

In the case where $z\in \mathrm{Fix}(F)$, it is obvious that
$i(\widetilde{F};\widetilde{a},\widetilde{b},z)$ is uniformly
bounded.

Choose any free open disk $U\subset M\setminus\mathrm{Fix}(F)$,
according to Lemma \ref{lem:the boundedness of the case multi-path},
we only need to consider the points $z\in \mathrm{Rec}^+(F)\cap U$
such that $\alpha_{z,n}=e$ for $n$ large enough. We suppose that $z$
is a such point and $i(\widetilde{F};\widetilde{a},\widetilde{b},z)$
exists. We go to the annulus $A_{\widetilde{a},\widetilde{b}}$, for
any lift $\widetilde{z}$ of $z$, then we have
$\rho_{A_{{\widetilde{a}},{\widetilde{b}}},\widehat{F}_{{\widetilde{a}},{\widetilde{b}}}}
(\widetilde{z})=
\lim\limits_{n\rightarrow+\infty}\frac{\widetilde{\gamma}\wedge\widetilde
{\Gamma}^n_{\widetilde{I}_1,\widetilde{z}}}{\tau_n(z)}$. %Since $I$
%satisfies the WB-property  for pair
%$(\widetilde{a},\widetilde{b})\in(\mathrm{Fix}(\widetilde{F})\times
%\mathrm{Fix}(\widetilde{F}))\setminus\widetilde{\Delta}$, we have
%$\mathrm{Rot}_{\mathrm{Fix}(\widetilde{F}_{\widetilde{a},\widetilde{b}})}
%(\widehat{F}_{\widetilde{a},\widetilde{b}})\subset[-N_{\widetilde{a},\widetilde{b}},
%N_{\widetilde{a},\widetilde{b}}]$.

We claim that, for any $\epsilon>0$,
$|i(\widetilde{F};\widetilde{a},\widetilde{b},z)|\leq
(N_{\widetilde{a},\widetilde{b}}+1+\epsilon)K_{\widetilde{a},\widetilde{b}}$.
Otherwise, without loss of generality, we can suppose that
$i(\widetilde{F};\widetilde{a},\widetilde{b},z)>
(N_{\widetilde{a},\widetilde{b}}+1+\epsilon)K_{\widetilde{a},\widetilde{b}}$.
Then there exists a number $N$ large enough such that for every
$n\geq N$, there is a lift $\widetilde{z}_n$ of $z$ in
$\widetilde{V}^c$ satisfying
$\frac{\widetilde{\gamma}\wedge\widetilde
{\Gamma}^n_{\widetilde{I}_1,\widetilde{z}_n}}{\tau_n(z)}>N_{\widetilde{a},\widetilde{b}}+1+\epsilon$.
This implies that there exists a lift $\widetilde{z}$ of $z$ in
$\widetilde{V}^c$ such that
$\rho_{A_{{\widetilde{a}},{\widetilde{b}}},\widehat{F}_{{\widetilde{a}},{\widetilde{b}}}}
(\widetilde{z})\geq
N_{\widetilde{a},\widetilde{b}}+1+\epsilon>N_{\widetilde{a},\widetilde{b}}+1$.
By the fact
$\rho_{A_{{\widetilde{a}},{\widetilde{b}}},\widehat{F}_{{\widetilde{a}},{\widetilde{b}}}}
(\infty)=0$ and according to Theorem \ref{thm:FP},
$\widetilde{F}_{\widetilde{a},\widetilde{b}}$ has a fixed point
whose rotation number is $N_{\widetilde{a},\widetilde{b}}+1$, which
is a contradiction. This proves the claim.

As $\epsilon$ is arbitrarily chosen, we get
$|i(\widetilde{F};\widetilde{a},\widetilde{b},z)|\leq
K_{\widetilde{a},\widetilde{b}}(N_{\widetilde{a},\widetilde{b}}+1)$.
\end{proof}

The function $i(\widetilde{F};\widetilde{a},\widetilde{b},z)$ is not
necessarily $\mu$-integrable (see Example \ref{exem:proposition 21
and mu integrable} in Appendix). But in some cases, as we have
stated above, where the time-one map $F$ is a diffeomorphism of $M$,
or $I$ satisfies the WB-property at $\widetilde{a}$ and
$\widetilde{b}$, and $\mu$ is ergodic (because it is constant
$\mu$-a.e.) or the support of $\mu$ is the whole space, the function
$i_{\mu}(\widetilde{F};\widetilde{a},\widetilde{b},z)$ is
$\mu$-integrable.

Suppose now the function
$i(\widetilde{F};\widetilde{a},\widetilde{b},z)$ is
$\mu$-integrable. We can define a function as follows
\begin{equation}\label{eq:imu}
    i_{\mu}(\widetilde{F};\widetilde{a},\widetilde{b})= \int_{M\setminus\pi(\{\widetilde{a},\widetilde{b}\})}
i(\widetilde{F};\widetilde{a},\widetilde{b},z)\,\mathrm{d}\mu.
\end{equation}

From Propositions \ref{prop:i(Fnabx)=ni(abx)} and \ref{prop:alpha-
one exists then
 other exists}, we get the following corollaries immediately:
\begin{cor}\label{cor:imu(Fq,a,b)=qimu(F,a,b)}
We have
$i_{\mu}(\widetilde{F}^q;\widetilde{a},\widetilde{b})=q\,i_{\mu}(\widetilde{F};\widetilde{a},\widetilde{b})$
for all $q\geq1$.
\end{cor}

\begin{cor}\label{cor:imu(F,alphaa,alphab)=imu(F,a,b)}
We have
$i_{\mu}(\widetilde{F};\alpha(\widetilde{a}),\alpha(\widetilde{b}))=i_{\mu}(\widetilde{F};\widetilde{a},\widetilde{b})$
for any $\alpha\in G$.
\end{cor}

Let $H$ be an orientation preserving homeomorphism of $M$ and
$\widetilde{H}$ be a lift of $H$
to $\widetilde{M}$. %If $H$ preserves the measure $\mu$,
From Proposition \ref{prop:linking number and conjugation}, we get
the following corollary.

\begin{cor}\label{cor:imu(HFH-1,Ha,Hb)=imu(F,a,b)}
We have $i_{H_*(\mu)}(\widetilde{H}\circ\widetilde{F}\circ
\widetilde{H}^{\,-1};\widetilde{H}(\widetilde{a}),\widetilde{H}(\widetilde{b}))
=i_\mu(\widetilde{F};\widetilde{a},\widetilde{b})$.
\end{cor}\bigskip

In the end of this section, we will give the integral \ref{eq:imu} a
geometric description when $F$ and $F^{-1}$ are differentiable at
$\pi(\widetilde{a})$ and $\pi(\widetilde{b})$. Before that, let us
introduce a definition.\bigskip

Let $\mathbb{A}=\mathbf{T}^1\times[0,1]$ be a closed annulus and let
$T$ be the generator of the covering transformation group $\pi:
\widetilde{\mathbb{A}}\rightarrow \mathbb{A}$ where
$\widetilde{\mathbb{A}}=\mathbb{R}\times[0,1]$. Suppose that
$J=(h_t)_{t\in[0,1]}$ is an isotopy of $\mathbb{A}$ from
$\mathrm{Id}_{\mathbb{A}}$ to $h$, $\nu$ is a Borel measure ($\nu$
is admitted to be an infinite measure here) invariant by $h$ on
$\mathbb{A}$. Let $\gamma:[0,1]\rightarrow \mathbb{A}$ be a simple
oriented path which satisfies $\gamma(0)\in
\mathbf{T}^1\times\{0\}$, $\gamma(1)\in \mathbf{T}^1\times\{1\}$ and
$\mathrm{Int}(\gamma)\subset \mathrm{Int}(\mathbb{A})$. Denote by
$\Sigma': [0,1]\times[0,1]\rightarrow \mathbb{A}$ the 2-chain
$\Sigma'(s,t)=h_s^{-1}(\gamma(t))$ and by
$|\Sigma'|=\{z\in\mathbb{A}\mid z=h_s^{-1}(\gamma(t)), (s,t)\in
[0,1]\times[0,1]\}$ the support of $\Sigma'$. When $\nu(\gamma)=0$,
the intersection number $\gamma\wedge J(z)$ is well defined for
$\nu$-almost every $z$ on $\mathbb{A}$. Define \emph{the algebraic
area of the 2-chain $\Sigma'$ in $\mathbb{A}$}, that is, the
algebraic area (for $\nu$) ``swept out'' by $\bigcup_{s\in
[0,1]}h_s^{-1}(\gamma)$, as follows
\begin{equation}\label{eq:algebric area of 2 chain}
\int_{\Sigma'}\, \mathrm{d}\nu=\int_{\mathbb{A}}\gamma\wedge
J(z)\,\mathrm{d}\nu\nonumber.
\end{equation}
When $\nu(|\Sigma'|)<+\infty$, the integral is well defined. Indeed,
there exist a number $N\geq0$ such that $|\gamma\wedge J(z)|\leq N$
since $\mathbb{A}$ is compact. Obviously, $\gamma\cap
J(z)=\emptyset$ if $z\notin \bigcup_{s\in[0,1]}
h_s^{-1}(\gamma(t))$. Therefore,
$$\left|\int_{\Sigma'}\, \mathrm{d}\nu\right|\leq\int_{\mathbb{A}}|\gamma\wedge
J(z)|\,\mathrm{d}\nu\leq \nu(|\Sigma'|)N<+\infty.$$\smallskip

Let $H$ be the lift of $h$ that is the time-one map of the lifted
identity isotopy $\widetilde{J}$ of $J$, $\widetilde{\gamma}$ be a
connected component of $\gamma$ in $\widetilde{\mathbb{A}}$ and
$\widetilde{\nu}$ be the lift of $\nu$ to $\widetilde{\mathbb{A}}$.
Let $\widetilde{D}'$ be the closed region between
$H^{-1}(\widetilde{\gamma})$ and $T(H^{-1}(\widetilde{\gamma}))$
which is a fundamental domain of $T$. We have
\begin{eqnarray}\label{eq:stocks' theorem}
% \nonumber to remove numbering (before each equation)
  \int_{\Sigma'}\, \mathrm{d}\nu=\int_{\mathbb{A}}\gamma\wedge
J(z)\,\mathrm{d}\nu = \int_{\widetilde{D}'}\widetilde{\gamma}\wedge
\widetilde{J}(\widetilde{z})\,\mathrm{d}\widetilde{\nu}
\end{eqnarray}
which does not depend on the choice of $\widetilde{\gamma}$.

Denote by $\Sigma=h\ast\Sigma': [0,1]\times[0,1]\rightarrow
\mathbb{A}$ the 2-chain $\Sigma(s,t)=h_{s}^{-1}(h(\gamma(t)))$ and
suppose that $\nu(|\Sigma|)<+\infty$. Let
$\widetilde{D}=H(\widetilde{D}')$ be the closed region between
$\widetilde{\gamma}$ and $T(\widetilde{\gamma})$ which is also a
fundamental domain of $T$.  By Equation \ref{eq:stocks' theorem}, we
have
\begin{eqnarray}\label{eq: stocks' theorem 1}
% \nonumber to remove numbering (before each equation)
  \int_{\Sigma}\, \mathrm{d}\nu=\int_{\mathbb{A}}h(\gamma)\wedge J(z)\,\mathrm{d}\nu =
\int_{\widetilde{D}}H(\widetilde{\gamma})\wedge
\widetilde{J}(\widetilde{z})\,\mathrm{d}\widetilde{\nu}.\end{eqnarray}
%The algebraic area of $\Sigma$ is swept out by $h(\gamma)$ along the
%isotopy $J^{-1}$.
The equation \ref{eq: stocks' theorem 1} tell us that the value
$\int_{\Sigma}\, \mathrm{d}\nu$ is equal to the algebraic area (for
$\widetilde{\nu}$) of the region of $\widetilde{\mathbb{A}}$
situated between $\widetilde{\gamma}$ and its image
$H(\widetilde{\gamma})$. Furthermore, if we suppose that $J$ fixes a
point $\infty$ in $\mathbb{A}$, we have
\begin{eqnarray}\label{eq: stocks' theorem 2}
% \nonumber to remove numbering (before each equation)
  \int_{\Sigma}\, \mathrm{d}\nu
   &=& \int_{\mathbb{A}}h(\gamma)\wedge
J(z)\,\mathrm{d}\nu\\
   &=& \int_{\mathbb{A}}\gamma\wedge
\left(h^{-1}\circ J\right)(z)\,\mathrm{d}\nu\nonumber\\
   &=& \int_{\mathbb{A}}\gamma\wedge
\left(h^{-1}\circ J\circ h\right)(z)\,\mathrm{d}\nu\nonumber\\
&=&\int_{\mathbb{A}}\gamma\wedge J(z)\,\mathrm{d}\nu.
\nonumber\end{eqnarray} Indeed, write the isotopy $J'=h^{-1}\circ
J\circ h=\left(h^{-1}\circ h_{t}\circ h\right)_{0\leq t\leq 1}$. The
third equation holds because $h$ is a homeomorphism of $\mathbb{A}$
and preserves the measure $\nu$. Observing that the isotopy
$J^{-1}J'$ is a loop (whose base point is
$\mathrm{Id_{\mathbb{A}}}$) in $\mathrm{Homeo}_*(\mathbb{A})$ and
fixes the point $\infty$, recall that
$\pi_1(\mathrm{Homeo}_*(\mathbb{A}))=\bigcup_{k\in\,\mathbb{Z}}\mathscr{C}_k$
(see the proof of Proposition \ref{lem:i is 3coboundary for point}),
we get $[J^{-1}J']_1\in\mathscr{C}_0$ . Hence, we get the last
equation. It is easy to prove that, by induction and Equation
\ref{eq: stocks' theorem 2}, $\int_{\Sigma}\, \mathrm{d}\nu$ is
equal to $\int_{h^k\ast\Sigma'}\, \mathrm{d}\nu$ for every
$k\in\mathbb{Z}$.

\smallskip

Remark that we can also define the algebraic area of the 2-chain
$\Sigma$ when $\gamma$ is not simple if we consider the oriented
domain enclosed by $\widetilde{\gamma}$, $H(\widetilde{\gamma})$ and
$\partial\widetilde{\mathbb{A}}$ in $\widetilde{\mathbb{A}}$.
However, to prove Theorem \ref{thm:PW} in the next section, it is
enough to merely consider the case of a simple oriented path.

Suppose now the measure $\nu$ is defined by a symplectic form
$\omega$, that is, $\nu(A)=\int_{A}\omega$ for all measurable sets
$A\subset \mathbb{A}$. Observe that $\widetilde{\omega}$ is exact in
$\widetilde{\mathbb{A}}$ where $\widetilde{\omega}$ is the lift of
$\omega$ to $\widetilde{\mathbb{A}}$. The equation \ref{eq: stocks'
theorem 1} and Stokes' theorem imply that $\int_{\Sigma}\omega$
(defined by the integrals of differential 2-form on 2-chain) is
nothing else but the algebraic area of the 2-chain $\Sigma$ in
$\mathbb{A}$, $\int_{\Sigma}\, \mathrm{d}\nu$ (defined by Equation
\ref{eq: stocks' theorem 1}).
%It is an ingredient to prove Theorem \ref{thm:PW}.
\bigskip

We now suppose that the time-one map $F$ of $I$ and its inverse
$F^{-1}$ are differentiable at $\pi(\widetilde{a})$ and
$\pi(\widetilde{b})$. Let
$\widetilde{I}_1=(\widetilde{F}'_t)_{t\in[0,1]}$ be an isotopy from
$\mathrm{Id}_{\widetilde{M}}$ to $\widetilde{F}$ that fixes
$\widetilde{a}$ and $\widetilde{b}$, and $\widetilde{\mu}$ be the
lift of $\mu$ to $\widetilde{M}$. Let
$\widetilde{\gamma}:[0,1]\rightarrow \widetilde{M}$ be a simple
oriented path from $\widetilde{a}$ to $\widetilde{b}$ with
$\widetilde{\gamma}(0)=\widetilde{a}$ and
$\widetilde{\gamma}(1)=\widetilde{b}$. Consider the annulus
$A_{\widetilde{a},\widetilde{b}}$ and the annulus map
$\widetilde{F}_{\widetilde{a},\widetilde{b}}$. Recall that, in the
proof of Lemma \ref{lem:diffeomorphism on a and b},
$\bar{A}_{\widetilde{a},\widetilde{b}}=S_{\widetilde{a}}\sqcup
A_{\widetilde{a},\widetilde{b}}\sqcup S_{\widetilde{b}}$ is the
natural compactification of $A_{\widetilde{a},\widetilde{b}}$ where
$S_{\widetilde{a}}$ and $S_{\widetilde{b}}$ are the tangent unit
circles at $\widetilde{a}$ and $\widetilde{b}$. We can identify
$\widetilde{\gamma}$ as an oriented path in
$\bar{A}_{\widetilde{a},\widetilde{b}}$ and $\widetilde{I}_1$ as an
identity isotopy of $\bar{A}_{\widetilde{a},\widetilde{b}}$. As the
measure $\widetilde{\mu}$ is invariant by $\widetilde{F}$ and
$\widetilde{\mu}(\widetilde{a})=\widetilde{\mu}(\widetilde{b})=0$,
it naturally induces a measure on
$\bar{A}_{\widetilde{a},\widetilde{b}}$, denoted still by
$\widetilde{\mu}$.

Suppose that $\widetilde{\Sigma}$ is the 2-chain
$\widetilde{\Sigma}:[0,1]\times[0,1]\rightarrow\widetilde{M}$
defined by
$\widetilde{\Sigma}(s,t)=\widetilde{F}'^{-1}_s(\widetilde{F}(\widetilde{\gamma}(t)))$
whose boundary is
$\widetilde{F}(\widetilde{\gamma})\widetilde{\gamma}^{-1}$ with the
boundary of the square $[0,1]\times[0,1]$ oriented
counter-clockwise.  As $\widetilde{I}_1$ fixes $\infty$, the
intersection number $\widetilde{\gamma}
\wedge\widetilde{I}_1(\widetilde{z})$ is zero when $\widetilde z$
belongs to a neighborhood of $\infty$. Therefore, if
$\widetilde{\mu}(\widetilde{\gamma})=0$, we can define \emph{the
algebraic area of the 2-chain $\widetilde{\Sigma}$ in
$\widetilde{M}\setminus\{\widetilde{a},\widetilde{b}\}$} as follows
$$\int_{\widetilde{\Sigma}}\,
\mathrm{d}\widetilde{\mu}=\int_{\widetilde{M}\setminus\{\widetilde{a},\widetilde{b}\}}\widetilde{\gamma}
\wedge\widetilde{I}_1(\widetilde{z})
 \,\mathrm{d}\widetilde{\mu}=\int_{\bar{A}_{\widetilde{a},\widetilde{b}}}\widetilde{\gamma}
\wedge\widetilde{I}_1(\widetilde{z})
 \,\mathrm{d}\widetilde{\mu}.$$

Remark here that if the measure $\mu$ is defined by a symplectic
form $\omega$, then $\int_{\widetilde{\Sigma}}\widetilde{\omega}$
(see Equation \ref{eq:action difference} and Equation \ref{eq:action
difference of cohomology}) is nothing else but
$\int_{\widetilde{\Sigma}}\,
\mathrm{d}\widetilde{\mu}$ %the algebraic area of the 2-chain
%$\widetilde{\Sigma}$ in
%$\widetilde{M}\setminus\{\widetilde{a},\widetilde{b}\}$
where $\widetilde{\omega}$ is the lift of $\omega$ to
$\widetilde{M}$. Moreover, we have the following result which is a
key step to prove that our generalized action function defined in
next section is a generalization of the classical one.

\begin{lem}\label{lem:mu gamma is zero}
If $\widetilde{\mu}(\widetilde{\gamma})=0$, then we have
\begin{equation*}\label{equ: lemma of main theorem}
i_\mu(\widetilde{F},\widetilde{a},\widetilde{b})=\int_{\widetilde{\Sigma}}\,
\mathrm{d}\widetilde{\mu}.
\end{equation*}
\end{lem}

\begin{proof}From Proposition \ref{prop:diffeomorphism has a bounded
linking number for every two distinct fixed points}, we know that
$i_\mu(\widetilde{F},\widetilde{a},\widetilde{b})$ is well defined.
Let
$$Z=\bigcup\limits_{k=0}^{+\infty}(F^{-k}(\pi(\widetilde{\gamma}))).$$
Observe that $\mu(\mathrm{Rec}^+(F)\setminus Z)=\mu(M)$. For every
$z\in \mathrm{Rec}^+(F)\setminus Z$ and every $n\geq1$, consider the
following infinite family of paths in $\widetilde{M}$:
$$\widetilde{\Gamma}'^{\,n}_{\widetilde{I}_1,z}=\prod_{\pi(\widetilde{z})=z}\widetilde{I}_1^{\,n}(\widetilde{z})\,.$$
Define the function
$$G_{n}(\widetilde{F};\widetilde{a},\widetilde{b},z)=\widetilde{\gamma}\wedge
\widetilde{\Gamma}'^{\,n}_{\widetilde{I}_1,z}.
$$

Let us verify that this is well defined. Consider the annulus
$A_{\widetilde{a},\widetilde{b}}$ and the annulus map
$\widetilde{F}_{\widetilde{a},\widetilde{b}}$. For any $z\in
\mathrm{Rec}^+(F)\setminus Z$, let
 $\widetilde{z}$ be any lift of $z$ to $\widetilde{M}$ (we also write
 $\widetilde{z}$ in $A_{\widetilde{a},\widetilde{b}}$), and
$\widehat{z}$ be any lift of $\widetilde{z}$ to
$\widehat{A}_{\widetilde{a},\widetilde{b}}$. In the proof of Lemma
\ref{lem:diffeomorphism on a and b}, we have proved that
$|p_1(\widehat{F}_{{\widetilde{a}},{\widetilde{b}}}(\widehat{z}))-
p_1(\widehat{z})|$ is uniformly bounded for any
$\widehat{z}\in\widehat{A}_{\widetilde{a},\widetilde{b}}$, say $N$
as a bound, and depends on the isotopy $I$ but not on the choice of
$\widehat{z}$. Fix an open disk $\widetilde{W}$ containing $\infty$
and disjoint from $\widetilde{\gamma}$. As
$\widetilde{I}_1(\infty)=\infty$, for every $n\geq1$, we can choose
an open disk $\widetilde{V}_n\subset \widetilde{W}$ containing
$\infty$ such that for every $\widetilde{z}\in \widetilde{V}_n$, we
have $\widetilde{I}^n_1(\widetilde{z})\in \widetilde{W}$. Write
$X'^n_z=\pi^{-1}(\{z\})\cap \widetilde{V}_n^c$. We deduce that there
is a positive integer $K'_n$ such that $\sharp X'^n_z\leq K'_n$ and
$$\left|G_n(\widetilde{F};\widetilde{a},\widetilde{b},z)\right|=\left|\widetilde{\gamma}\wedge
\widetilde{\Gamma}'^{\,n}_{\widetilde{I}_1,z}\right|
=\left|\sum_{\widetilde{z}\in X'^n_z}\widetilde{\gamma}\wedge
\widetilde{I}^{\,n}_1(\widetilde{z})\right|\leq K'_nN.$$ Hence we
complete the claim. As a consequence,
$G_1(\widetilde{F};\widetilde{a},\widetilde{b},z)\in
L^1(M\setminus\pi(\{\widetilde{a},\widetilde{b}\}),\mathbb{R},\mu)$.

Moreover, we can write
$G_{n}(\widetilde{F};\widetilde{a},\widetilde{b},z)$ as a Birkhoff
sum:
$$G_{n}(\widetilde{F};\widetilde{a},\widetilde{b},z)=\widetilde{\gamma}\wedge
\widetilde{\Gamma}'^{\,n}_{\widetilde{I}_1,z}=\widetilde{\gamma}\wedge
\prod_{i=0}^{n-1}\widetilde{\Gamma}'^{\,1}_{\widetilde{I}_1,F^i(z)}=\sum_{j=0}^{n-1}
G_1(\widetilde{F};\widetilde{a},\widetilde{b},F^j(z)).
$$According to Birkhoff Ergodic theorem, the limit
$$\lim_{n\rightarrow+\infty}\frac{G_{n}(\widetilde{F};\widetilde{a},\widetilde{b},z)}{n}
=\lim\limits_{n\rightarrow+\infty}\frac{1}{n}\sum_{j=0}^{n-1}G_1(\widetilde{F};\widetilde{a},\widetilde{b},F^j(z))$$
exists for $\mu$-almost everywhere on
$M\setminus\pi(\{\widetilde{a},\widetilde{b}\})$. We know that
$$i(\widetilde{F};\widetilde{a},\widetilde{b},z)=\lim\limits_{n\rightarrow+\infty}
\frac{L_{n}(\widetilde{F};\widetilde{a},\widetilde{b},z)}{\tau_n(z)}=
\frac{L^*(\widetilde{F};\widetilde{a},\widetilde{b},z)}{\tau^*(z)}$$
for $\mu$-almost every point $z\in
M\setminus\pi(\{\widetilde{a},\widetilde{b}\})$ exists \,(see
Proposition \ref{prop:measure and linking number exsiting}). As
$i(\widetilde{F};\widetilde{a},\widetilde{b},z)$ does not depend on
the choice of $U$ (see Definition \ref{def:Intersection number
density}), when $z\notin\pi(\widetilde{\gamma})$, we can suppose
that the disk $U$ is small enough such that
$U\cap\pi(\widetilde{\gamma})=\emptyset$. Therefore,
$\{L_{n}(\widetilde{F};\widetilde{a},\widetilde{b},z)/\tau_n(z)\}_{n\geq1}$
 is a subsequence of
$\{G_{n}(\widetilde{F};\widetilde{a},\widetilde{b},z)/n\}_{n\geq1}$.
We get
$$i(\widetilde{F};\widetilde{a},\widetilde{b},z)=\lim\limits_{n\rightarrow+\infty}
\frac{G_{n}(\widetilde{F};\widetilde{a},\widetilde{b},z)}{n}$$ for
$\mu$-almost everywhere on
$M\setminus\pi(\{\widetilde{a},\widetilde{b}\})$.
\smallskip

 By Birkhoff Ergodic theorem, we have
\begin{eqnarray}i_\mu(\widetilde{F};\widetilde{a},\widetilde{b})&=&
\int_{M\setminus\pi(\{\widetilde{a},\widetilde{b}\})}i(\widetilde{F};\widetilde{a},\widetilde{b},z)\,\mathrm{d}\mu\nonumber\\
&=&
\int_{M\setminus\pi(\{\widetilde{a},\widetilde{b}\})}G_1(\widetilde{F};\widetilde{a},\widetilde{b},z)\,\mathrm{d}\mu\nonumber\\
&=&\int_{\widetilde{M}\setminus\pi^{-1}(\pi(\{\widetilde{a},\widetilde{b}\}))}\widetilde{\gamma}
\wedge\widetilde{I}_1(\widetilde{z})
 \,\mathrm{d}\widetilde{\mu}\nonumber\\
&=&\int_{\widetilde{\Sigma}}\, \mathrm{d}\widetilde{\mu},\nonumber
\label{equ:scan}
\end{eqnarray}
We have completed the proof.\end{proof}\bigskip
%\begin{rem}Recall that $\widehat{\pi}:
%\widehat{A}_{{\widetilde{a}},{\widetilde{b}}}\rightarrow
%A_{{\widetilde{a}},{\widetilde{b}}}$ is the covering map and  $T$ is
%the generator of the covering transformation group $\widehat{G}$
%defined by the oriented boundary of a small disk centered at
%$\widetilde{a}$. Let $\widehat{\gamma}$ be a connected component of
%$\widetilde{\gamma}$ to $\widehat{A}_{\widetilde{a},\widetilde{b}}$
%and $\widehat{\mu}$ be the lift of $\widetilde{\mu}$ to
%$\widehat{A}_{{\widetilde{a}},{\widetilde{b}}}$. Let $\widehat{D}$
%be the closed region between $\widehat{\gamma}$ and
%$T(\widehat{\gamma})$ which is the fundamental domain of the action
%group generated by $T$. We can give the integral
%$\int_{\widetilde{\Sigma}}\, \mathrm{d}\widetilde{\mu}$ a more
%precise explanation as following:
%$$\int_{\widetilde{\Sigma}}\, \mathrm{d}\widetilde{\mu}=\int_{\widehat{D}}
%\left(p_1(\widehat{F}_{\widetilde{a},\widetilde{b}}(\widehat{z}))-p_1(\widehat{z})
%\right)\,\mathrm{d}\widehat{\mu}.$$
%\end{rem}\smallskip

\section{Action Function}\label{sec:proof of the main
theorem}

This section will be divided into three parts. In the first part, we
will define the action function and prove Theorem \ref{thm:PW}. In
the second part, we will study some properties of the action. In the
third part, we will define the action spectrum and prove that the
action is not constant in the case where the contractible fixed
points set is finite.\smallskip

Firstly, we state some results we can already get immediately. Let
$F\in \mathrm{Homeo}_*(M)$ be the time-one map of an identity
isotopy $I=(F_t)_{t\in[0,1]}$ of $M$. As we have proved in the last
section, we know that the function
$i(\widetilde{F};\widetilde{a},\widetilde{b},z)$ is $\mu$-integrable
for every pair
$(\widetilde{a},\widetilde{b})\in(\mathrm{Fix}(\widetilde{F})\times
\mathrm{Fix}(\widetilde{F}))\setminus\widetilde{\Delta}$ in each of
the following cases:
\begin{itemize}
\item $F\in\mathrm{Diff}(M)$, and $\mu\in\mathcal{M}(F)$ has no atoms on $\mathrm{Fix}_{\mathrm{Cont},I}(F)$
 (Proposition \ref{prop:diffeomorphism has a bounded
linking number for every two distinct fixed points});
\item $I$ satisfies the
WB-property, and $\mu\in \mathcal{M}(F)$ has total support but no
atoms on $\mathrm{Fix}_{\mathrm{Cont},I}(F)$ (Proposition
\ref{prop:the support of mu is M});
\item $I$ satisfies the
WB-property, $\mu\in\mathcal {M}(F)$ has no atoms on
$\mathrm{Fix}_{\mathrm{Cont},I}(F)$ and $\mu$ is ergodic
(Proposition \ref{prop:measure and linking number exsiting} and the
Birkhoff Ergodic theorem).
\end{itemize}

\subsection{Definition of the action function}\label{sec:definition of action function}
In this subsection, we
suppose that the function
$i(\widetilde{F};\widetilde{a},\widetilde{b},z)$ is $\mu$-integrable
for every two distinct fixed points $\widetilde{a}$ and
$\widetilde{b}$ of $\widetilde{F}$.
\smallskip

We define the \emph{action difference} as follows:
\begin{eqnarray*}
% \nonumber to remove numbering (before each equation)
  i_{\mu}: (\mathrm{Fix}(\widetilde{F})\times
\mathrm{Fix}(\widetilde{F}))\setminus\widetilde{\Delta}
&\rightarrow&\mathbb{R} \\
  (\widetilde{a},\widetilde{b}) &\mapsto&
  i_{\mu}(\widetilde{F};\widetilde{a},\widetilde{b}).
\end{eqnarray*}

From Proposition \ref{lem:i is 3coboundary for point}, we have the
following corollary immediately:
\begin{cor}\label{clm:I is coboundary}
For any distinct fixed points $\widetilde{a}$,
 $\widetilde{b}$ and $\widetilde{c}$ of $\widetilde{F}$, we have
$$i_{\mu}(\widetilde{F};\widetilde{a},\widetilde{b})+i_{\mu}(\widetilde{F};\widetilde{b},\widetilde{c})+i_{\mu}(\widetilde{F};\widetilde{c},\widetilde{a})=0.$$
That is, $i_{\mu}$ is a coboundary on $\mathrm{Fix}(\widetilde{F})$.
So there is a function $l_{\mu}:
\mathrm{Fix}(\widetilde{F})\rightarrow\mathbb{R}$, defined up to an
additive constant, such that
\begin{equation}\label{eq:imu and lmu}
    i_{\mu}(\widetilde{F};\widetilde{a},\widetilde{b})=l_{\mu}(\widetilde{F};\widetilde{b})-l_{\mu}(\widetilde{F};\widetilde{a}).
\end{equation}
\end{cor}
We call the function $l_\mu$ the \emph{action} on
$\mathrm{Fix}(\widetilde{F})$ defined by the measure $\mu$.
\smallskip

As a consequence, if $F$ is a diffeomorphism of $M$ and the measure
$\mu\in\mathcal {M}(F)$ has no atoms on
$\mathrm{Fix}_{\mathrm{Cont},I}(F)$, or the isotopy $I$ satisfies
the WB-property and the measure $\mu\in\mathcal {M}(F)$ has total
support but no atoms on $\mathrm{Fix}_{\mathrm{Cont},I}(F)$, then
the action function is well defined on
$\mathrm{Fix}(\widetilde{F})$, but the action can be unbounded. In
Example \ref{ex:the action of non C1-diffeo isnot bounded and
continuous} of Appendix, we will construct an isotopy $I$ and a
measure $\mu\in\mathcal {M}(F)$ such that the time-one map $F$ is a
diffeomorphism (hence $I$ satisfies the WB-property), and the
measure $\mu$ has total support but no atoms on
$\mathrm{Fix}_{\mathrm{Cont},I}(F)$, while the action is unbounded.

%Certainly, the action is bounded when $F$ is a $C^1$-diffeomorphism
%of $M$ and $\mu\in\mathcal{M}_{c}(F)$.
\smallskip

\begin{prop}\label{clm:L is well defined}
If $\rho_{M,I}(\mu)=0$, then
$i_{\mu}(\widetilde{F};\widetilde{a},\alpha(\widetilde{a})) = 0$ for
every $\widetilde{a}\in \mathrm{Fix}(\widetilde{F})$ and every
$\alpha\in G^{\,*}$. As a consequence, there exists a function
$L_{\mu}$ defined on $\mathrm{Fix}_{\mathrm{Cont},I}(F)$ such that
for every two distinct fixed points $\widetilde{a}$ and
$\widetilde{b}$ of $\widetilde{F}$, we have
$$i_{\mu}(\widetilde{F};\widetilde{a},\widetilde{b})=
L_{\mu}(\widetilde{F};\pi(\widetilde{b}))-L_{\mu}(\widetilde{F};\pi(\widetilde{a}))
.$$
\end{prop}
\begin{proof}
There exists an isotopy $I'$ homotopic to $I$ that fixes
$\pi(\widetilde{a})$. It is lifted to an isotopy $\widetilde{I}\,'$
that fixes $\widetilde{a}$ and $\alpha(\widetilde{a})$. Observe that
if $\widetilde{\gamma}$ is an oriented path from $\widetilde{a}$ to
$\alpha(\widetilde{a})$, then the intersection number
$\widetilde{\gamma}\wedge\widetilde{\Gamma}^n_{\widetilde{I}\,',z}$
(see \ref{subsec:the definition of a new linking number}) is equal
to the intersection between the loop $\pi(\widetilde{\gamma})$ and
the loop $I'\,^{\tau_n(z)}(z)\gamma_{\Phi^n(z), z}$ (see
\ref{sec:the existion of rotation vector in the copact case}). As
$\rho_{M,I}(\mu)=\rho_{M,I'}(\mu)=0$ and $\pi(\widetilde{a})\in
\mathrm{Fix}_{\mathrm{Cont},I}(F)$ (or $\mu(\pi(\widetilde{a}))=0$),
we have
\begin{eqnarray*}i_\mu(\widetilde{F};\widetilde{a},\alpha(\widetilde{a}))&=&
\int_{M\setminus\{\pi(\widetilde{a})\}}i(\widetilde{F};\widetilde{a},\alpha(\widetilde{a}),z)\,\mathrm{d}\mu\\
&=&
\int_{M\setminus\{\pi(\widetilde{a})\}}\lim_{n\rightarrow+\infty}\frac{L_n(\widetilde{F};
\widetilde{a},\alpha(\widetilde{a}),z)}{\tau_n(z)}\,\mathrm{d}\mu\\
&=&\int_{M\setminus\{\pi(\widetilde{a})\}}\lim_{n\rightarrow+\infty}\frac{\widetilde{\gamma}
\wedge\widetilde{\Gamma}^n_{\widetilde{I}\,',z}}{\tau_n(z)}\,\mathrm{d}\mu\\
&=&\pi(\widetilde{\gamma})\wedge\rho_{M,I'}(\mu)\\
&=&0
\end{eqnarray*}

The second conclusion follows from Corollary \ref{clm:I is
coboundary}. We have completed the proof.
\end{proof}

We call the function $L_\mu$ the \emph{action function} or
\emph{action} on $\mathrm{Fix}_{\mathrm{Cont},I}(F)$ defined by the
measure $\mu$.
\smallskip

\begin{proof}[Proof of Theorem \ref{thm:PW}]From Corollary \ref{clm:I is coboundary} and Proposition
\ref{clm:L is well defined}, we define the \emph{action difference}
$I_{\mu}:
(\mathrm{Fix}_{\mathrm{Cont},I}(F)\times\mathrm{Fix}_{\mathrm{Cont},I}(F))\setminus\Delta\rightarrow
\mathbb{R}$ and the action
$L_\mu:\mathrm{Fix}_{\mathrm{Cont},I}(F)\rightarrow \mathbb{R}$ as
follows
\begin{equation}\label{eq:Imu and Lmu}
I_{\mu}(\widetilde{F};a,b)=i_{\mu}(\widetilde{F};\widetilde{a},\widetilde{b})
=L_{\mu}(\widetilde{F};b)-L_{\mu}(\widetilde{F};a),
\end{equation}
where $\widetilde{a}$ and $\widetilde{b}$ are any lifts of $a$ and
$b$. We only need to prove that the function $L_{\mu}$ defined in
this section is a generalization of the action difference in
\ref{sec:the classical action}.

Observe that, in the classical case,
$I=(F_t)_{t\in[0,1]}\subset\mathrm{Diff}_*(M)$ where
$\mathrm{Diff}_*(M)$ is the set of diffeomorphisms that are isotopic
to the identity. The measure $\mu$ is defined by a symplectic form
$\omega$. Therefore, $\mu$ is non-atomic. Comparing the Equation
\ref{eq:action difference of cohomology} with Equation \ref{eq:Imu
and Lmu}, it sufficient to prove that
$I_{\mu}(\widetilde{F};a,b)=i_{\mu}(\widetilde{F};\widetilde{a},\widetilde{b})
=\delta(\widetilde{F},\widetilde{a},\widetilde{b})$.

Let $\widetilde{\gamma}$ be any oriented path from $\widetilde{a}$
to $\widetilde{b}$. By Lemma \ref{lem:mu gamma is zero}, we have
\begin{equation*}
i_\mu(\widetilde{F},\widetilde{a},\widetilde{b})=\int_{\widetilde{\Sigma}}\,
\mathrm{d}\widetilde{\mu}
\end{equation*}
where $\widetilde{\Sigma}$ is the 2-chain whose boundary is
$\widetilde{F}(\widetilde{\gamma})-\widetilde{\gamma}$ (that is,
identify $\widetilde{F}(\widetilde{\gamma})\widetilde{\gamma}^{-1}$
as a 1-chain) as defined in Lemma \ref{lem:mu gamma is zero}. As
$\delta(\widetilde{F},\widetilde{a},\widetilde{b})$ does not depend
on the choices of $\widetilde{\gamma}$ and $\widetilde{\Sigma}$ (see
\ref{sec:the classical action}), we have
$$i_{\mu}(\widetilde{F};\widetilde{a},\widetilde{b})
=\delta(\widetilde{F},\widetilde{a},\widetilde{b}).$$ We have
completed the proof.
\end{proof}\smallskip

\subsection{The properties of the action function} In this section, we will give some properties of the
action function that we have defined in \ref{sec:definition of
action function}.\smallskip

From Theorem \ref{thm:PW} and Corollary
\ref{cor:imu(Fq,a,b)=qimu(F,a,b)}, we get the following corollary
immediately:

\begin{cor}\label{cor:Imu(Fq,a,b)=qImu(F,a,b)}
Under the same hypotheses as Theorem \ref{thm:PW}, for every two
distinct contractible fixed points $a$ and $b$ of $F$, we have
$I_{\mu}(\widetilde{F}^q;a,b)=qI_{\mu}(\widetilde{F};a,b)$ for all
$q\geq1$.
\end{cor}
\smallskip
Let us study the continuity and boundedness of the actions $l_\mu$
and $L_\mu$.

\begin{lem}\label{lem:In Vn nW}Let $I=(F_t)_{t\in[0,1]}$ be an identity isotopy of
$\mathbf{S}^2$ and $\bar{z},\bar{\bar{z}}\in\mathbf{S}^2$ be two
fixed points of $F_1$. If $\{z_n\}_{n\geq1}\subset
\mathrm{Fix}(F_1)\setminus\{\bar{z},\bar{\bar{z}}\}$ satisfies
$z_n\rightarrow \bar{z}$ as $n\rightarrow+\infty$, then for any open
neighborhood $W$ of $\bar{z}$, there exists a positive integer $n_W$
such that for every $n\geq n_W$, there exists an isotopy $I_n$ from
$\mathrm{Id}_{\mathbf{S}^2}$ to $F_1$ that fixes
$\bar{z},\bar{\bar{z}}$ and $z_n$, and there exists an open
neighborhood $V_n$ of $\bar{z}$ containing $z_n$, such that
$I_n(z)\cap V_n=\emptyset$ when $z\notin W$.
\end{lem}
\begin{proof}We identify the sphere $\mathbf{S}^2$ to the Riemann sphere
$\mathbb{C}\cup\{\infty\}$. %Recall that the M\"{o}bius
%transformation $\mathcal {M}(z)=\frac{az+b}{cz+d}$ maps the triple
%$(v_1,v_2,v_3)$ to the triple $(\omega_1,\omega_2,\omega_3)$ where
%$$a=\det\left(
%     \begin{array}{ccc}
%       v_1\omega_1 & \omega_1 & 1 \\
%       v_2\omega_2 & \omega_2 & 1 \\
%       v_3\omega_3 & \omega_3 & 1 \\
%     \end{array}
%   \right)\quad b=\det\left(
%     \begin{array}{ccc}
%       v_1\omega_1 & v_1 & \omega_1 \\
%       v_2\omega_2 & v_2 & \omega_2\\
%       v_3\omega_3 & v_3 & \omega_3 \\
%     \end{array}
%   \right)$$ $$c=\det\left(
%     \begin{array}{ccc}
%       v_1 & \omega_1 & 1 \\
%       v_2 & \omega_2 & 1 \\
%       v_3 & \omega_3 & 1 \\
%     \end{array}
%   \right)\qquad d=\det\left(
%     \begin{array}{ccc}
%       v_1\omega_1 & v_1 & 1 \\
%       v_2\omega_2 & v_2 & 1 \\
%       v_3\omega_3 & v_3 & 1 \\
%     \end{array}
%   \right).$$

For simplicity, up to conjugacy by a M\"{o}bius transformation (see
the proof of Lemma \ref{rem:identity isotopies fix three points on
sphere}) that maps the triple $(\bar{z},\bar{\bar{z}},z_1)$ to the
triple $(0,1,z_1)$, we can suppose that $\bar{z}=0$ and
$\bar{\bar{z}}=1$. We choose an isotopy $I_1=(F'_t)_{0\leq t\leq1}$
fixing the triple $(0,1,z_1)$ (using Lemma \ref{rem:identity
isotopies fix three points on sphere}).

Let $$\mathcal{M}_n(t,z)=\frac{z_n(F'_t(z_n)-1)z
}{(F'_t(z_n)-z_n)z+F'_t(z_n)(z_n-1)}$$ and
$$I_n(z)(t)=\mathcal{M}_n(t,F'_t(z)).$$

By construction,  $I_n$ is an isotopy from
$\mathrm{Id}_{\mathbf{S}^2}$ to $F_1$ that fixes the triple
$(0,1,z_n)$.

Let $W$ be any open neighborhood of $0$ and $V_n$ be the ball whose
center is at $0$ and radius is $2|z_n|$. Write
$$m(W)=\inf_{(t,z)\in [0,1]\times W^c}|F'_t(z)| \quad\mathrm{and}\quad M(z)=\sup_{t\in[0,1]}|F'_t(z)|.$$ As $I_1$ fixes
$0$, we have $m(W)>0$ and $M(z_n)\rightarrow0$ as
$n\rightarrow+\infty$. Therefore, there exists a positive number
$n_W$ such that when $n\geq n_W$,
$$M(z_n)<\min\{\frac{1}{2},\frac{m(W)}{8m(W)+4}\}.$$

For any $z\notin W$, every $n\geq n_W$ and $t\in[0,1]$, we have
\begin{eqnarray*}
% \nonumber to remove numbering (before each equation)
|I_n(z)(t)|&=&\frac{|z_n(F'_t(z_n)-1)F_t'(z)|}{|(F'_t(z_n)-z_n)F_t'(z)+F'_t(z_n)(z_n-1)|}\\
&=&\frac{|z_n(F'_t(z_n)-1)|}{|(F'_t(z_n)-z_n)+\frac{F'_t(z_n)(z_n-1)}{F_t'(z)}|}\\
&\geq&\frac{|(F'_t(z_n)-1)|}{|(F'_t(z_n)-z_n)|+\left|\frac{F'_t(z_n)(z_n-1)}{m(W)}\right|}|z_n|\\
&>&\frac{1/2}{2M(z_n)+M(z_n)/m(W)}|z_n|\\
&>&2|z_n|.
\end{eqnarray*}
Hence $I_n(z)\cap V_n=\emptyset$. We have completed the proof.
\end{proof}\bigskip

\begin{lem}\label{lem:L1anaz}We suppose that
$\widetilde{a}\in\mathrm{Fix}(\widetilde{F})\setminus\{\infty\}$ and
$\{\widetilde{a}_n\}_{n\geq1}\subset
\mathrm{Fix}(\widetilde{F})\setminus\{\widetilde{a},\infty\}$
satisfying $\widetilde{a}_n\rightarrow\widetilde{a}$ as
$n\rightarrow+\infty$. Then
$$\lim_{n\rightarrow+\infty}i(\widetilde{F};\widetilde{a}_n,\widetilde{a},z)
=0$$ when $z\in \mathrm{Fix}(F)\setminus\{\pi(\widetilde{a})\}$,
while
$$\lim_{n\rightarrow+\infty}L_1(\widetilde{F};\widetilde{a}_n,\widetilde{a},z)=0$$
when $z\in \mathrm{Rec}^+(F)\cap U$ where $U$
 is a disk of $M\setminus\{\pi(\widetilde{a})\}$.
\end{lem}
\begin{proof}When $z\in U$, recall that the
first return map is $\tau(z)$. For convenience, we write $\tau(z)=1$
if $z\in\mathrm{Fix}(F)$. For any given $z\in
\mathrm{Rec}^+(F)\setminus\pi(\{\widetilde{a}\})$, let
$\widetilde{W}$ be any open neighborhood of $\widetilde{a}$
satisfying
$\widetilde{W}\cap\pi^{-1}(\{z,F(z),\cdots,F^{\tau(z)-1}(z)\})=\emptyset$.
By Lemma \ref{lem:In Vn nW}, there exist a number
$n_{\widetilde{W}}$, a family of isotopies
$\{\widetilde{I}_n\}_{n\geq n_{\widetilde{W}}}$ with
$\widetilde{I}_n$ fixing the points $\widetilde{a}$ ,$\infty$ and
$\widetilde{a}_n$, and a family of neighborhoods
$\{\widetilde{V}_n\}_{n\geq n_{\widetilde{W}}}$ of $\widetilde{a}$
with $\widetilde{V}_n$ containing $\widetilde{a}_n$, such that
$\widetilde{I}_n(\widetilde{z})\cap \widetilde{V}_n=\emptyset$ for
any
$\widetilde{z}\in\pi^{-1}(\{z,F(z),\cdots,F^{\tau(z)-1}(z)\})$.\smallskip

The functions $i(\widetilde{F};\widetilde{a}_n,\widetilde{a},z)$
when $z$ is a fixed point of $F$ that is disjoint from
$\pi(\widetilde{a})$ and $\pi(\widetilde{a}_n)$, and
$L_1(\widetilde{F};\widetilde{a}_n,\widetilde{a},z)$ when $z\in
\mathrm{Rec}^+(F)\cap U$ and $n$ is large enough, depend neither on
the choice of the isotopy $\widetilde{I}$ that fixes the points
$\widetilde{a}$, $\widetilde{a}_n$ and $\infty$, nor on the path
from $\widetilde{a}_n$ to $\widetilde{a}$ (see \ref{subsec:the
definition of a new linking number}). Therefore, for every $n\geq
n_{\widetilde{W}}$, we can choose the isotopy $\widetilde{I}_n$ as
above and a path in $\widetilde{V}_n$ from $\widetilde{a}_n$ to
$\widetilde{a}$. As a consequence, we have
$$\lim_{n\rightarrow+\infty}i(\widetilde{F};\widetilde{a}_n,\widetilde{a},z)
=0$$ in the case where $z\in
\mathrm{Fix}(F)\setminus\{\pi(\widetilde{a})\}$, and
$$\lim_{n\rightarrow+\infty}L_1(\widetilde{F};\widetilde{a}_n,\widetilde{a},z)=0$$
in the case where $z\in \mathrm{Rec}^+(F)\cap U$.
\end{proof}

\begin{lem}\label{lem:B property with a ergodic measure}
Suppose that an isotopy $I$ satisfies the B-property,
$\mu\in\mathcal{M}(F)$ is ergodic and no atoms on
$\mathrm{Fix}_{\mathrm{Cont},I}(F)$ where $F$ is the time-one map of
$I$. Let $\widetilde{P}\subset \widetilde{M}$ be a connected compact
set. There exists $N_{\widetilde{P}}\geq0$ such that
$|i(\widetilde{F};\widetilde{a},\widetilde{b}, z)|\leq
N_{\widetilde{P}}$ for all two distinct fixed points $\widetilde{a}$
and $\widetilde{b}$ of $\widetilde{F}$ in $\widetilde{P}$ and
$\mu$-a.e. $z$.
\end{lem}
\begin{proof}Take a disk $U$ such that $\mu(U)>0$. Recall that
$\Phi$ and $\tau$ are respectively the first return map and the
first return time. Denote by $\mu_{U}$ the measure on $U$ induced by
$\mu$ and $\mathrm{Orb}^+(z)$ the positive orbit of $z$, that is,
$\mathrm{Orb}^+(z)=\{F^n(z)|n\geq0\}$. Then $\mu_{U}\in\mathcal
{M}(\Phi)$ and $\mu_{U}$ is ergodic with regard to $\Phi$. Indeed,
if $h$ is a measurable function on $U$ and satisfies $h\circ\Phi=h$.
Observe that $\mu(\cup_{n\geq0}F^n(U))=\mu(M)$ since $\mu$ is
ergodic. We may extend $h$ on $M$ in the following way:
\begin{equation*}h'(z')=
\begin{cases}h(z)& \textrm{if} \quad z\in \mathrm{Rec}^+(F)\cap U \,\,\mathrm{and}\,\, z'\in\mathrm{Orb}^+(z);
\\0& \textrm{if} \quad
z\notin \mathrm{Rec}^+(F)\cap U\,\,\mathrm{and}\,\,
z'\in\mathrm{Orb}^+(z).\end{cases}
\end{equation*}
By the construction of $h'$, we have $h'\circ F=h'$. Hence, $h'$ is
constant $\mu$-a.e. on $M$. This implies that $h$ is constant
$\mu_U$-a.e. on $U$ and therefore $\mu_U$ is ergodic with regard to
$\Phi$. Observe that
$\int_U\tau(z)d\mu=\mu(\cup_{k\geq0}F^k(U))=\mu(M)$. By Birkhoff
Ergodic theorem and Equation \ref{eq: linking number with a
measure}, for all two distinct fixed points $\widetilde{a}$ and
$\widetilde{b}$ of $\widetilde{F}$ in $\widetilde{P}$ and $\mu$-a.e.
$z$, we have

\begin{eqnarray*}
% \nonumber to remove numbering (before each equation)
i(\widetilde{F};\widetilde{a},\widetilde{b},z)&=&\lim_{n\rightarrow+\infty}
\frac{L_{n}(\widetilde{F};\widetilde{a},\widetilde{b},z)}{\tau_n(z)}
=\lim_{n\rightarrow+\infty}
\frac{\frac{L_{n}(\widetilde{F};\widetilde{a},\widetilde{b},z)}{n}}{\frac{\tau_n(z)}{n}}\\
&=&\frac{\int_UL_1(\widetilde{F};\widetilde{a},\widetilde{b},z)d\mu}{\int_U\tau(z)d\mu}\\
&=&\frac{1}{\mu(M)}\int_UL_1(\widetilde{F};\widetilde{a},\widetilde{b},z)d\mu.\\
\end{eqnarray*}

If the lemma is not true, then we can find
$\{(\widetilde{a}_n,\widetilde{b}_n)\}_{n\geq1}\subset
\mathrm{Fix}(\widetilde{F})\times\mathrm{Fix}(\widetilde{F})\setminus\widetilde{\Delta}\cap
\widetilde{P}$ and $z\in U$ such that
$|i(\widetilde{F};\widetilde{a}_n,\widetilde{b}_n,z)|\geq n$. That
is
\begin{equation}\label{eq:L1anbnz}
    \lim_{n\rightarrow+\infty}\left|\int_UL_1(\widetilde{F};\widetilde{a}_n,\widetilde{b}_n,z)d\mu\right|=
+\infty.
\end{equation}

We can suppose that there are two fixed points $\widetilde{a}$ and
$\widetilde{b}$ of $\widetilde{F}$ such that
$\lim\limits_{n\rightarrow+\infty}\widetilde{a}_n=\widetilde{a}$ and
$\lim\limits_{n\rightarrow+\infty}\widetilde{b}_n=\widetilde{b}$ by
taking subsequences if necessary.

When $\widetilde{a}=\widetilde{b}$, by Proposition \ref{lem:i is
3coboundary for point}, we have
$$L_1(\widetilde{F};\widetilde{a}_n,\widetilde{b}_n,z)+L_1(\widetilde{F};
\widetilde{b}_n,\widetilde{a},z)+L_1(\widetilde{F};\widetilde{a},\widetilde{a}_n,z)=0.$$
By Lemma \ref{lem:L1anaz}, we have
$L_1(\widetilde{F};\widetilde{a}_n,\widetilde{b}_n,z)\rightarrow 0$
as $n\rightarrow+\infty.$

When $\widetilde{a}\neq\widetilde{b}$, we have
$$L_1(\widetilde{F};\widetilde{a}_n,\widetilde{b}_n,z)+L_1(\widetilde{F};
\widetilde{b}_n,\widetilde{b},z)+L_1(\widetilde{F};\widetilde{b},\widetilde{a},z)+
L_1(\widetilde{F};\widetilde{a},\widetilde{a}_n,z)=0.$$ By Lemma
\ref{lem:L1anaz} again, we have
$L_1(\widetilde{F};\widetilde{a}_n,\widetilde{b}_n,z)\rightarrow
L_1(\widetilde{F};\widetilde{a},\widetilde{b},z)$ as
$n\rightarrow+\infty.$ By the definition B-property, Section
\ref{sec:boundedness} and Lebesgue's dominating convergence theorem,
we have
$$\lim_{n\rightarrow+\infty}\left|\int_UL_1(\widetilde{F};\widetilde{a}_n,\widetilde{b}_n,z)d\mu\right|=
0$$ when $\widetilde{a}=\widetilde{b}$, and
$$\lim_{n\rightarrow+\infty}\left|\int_UL_1(\widetilde{F};\widetilde{a}_n,\widetilde{b}_n,z)d\mu\right|=
\left|\int_UL_1(\widetilde{F};\widetilde{a},\widetilde{b},z)d\mu\right|<+\infty$$
when $\widetilde{a}\neq\widetilde{b}$, which conflicts with the
limit \ref{eq:L1anbnz}.
\end{proof}

From the definition of the B-property, Proposition \ref{prop:C1
diffeomorphism has a bounded linkingnumber}, Proposition
\ref{prop:the support of mu is M}, Proposition \ref{clm:L is well
defined} and Lemma \ref{lem:B property with a ergodic measure}, we
get the following proposition.

\begin{prop}\label{prop:boundedness}Suppose that $F$ is the time-one map of an identity isotopy
$I$ on $M$, the measure $\mu\in\mathcal{M}(F)$ has no atoms on
$\mathrm{Fix}_{\mathrm{Cont},I}(F)$ and $\rho_{M,I}(\mu)=0$. If one
of the following three cases is satisfied\begin{itemize}
           \item $I$ satisfies the
           B-property and $F\in\mathrm{Diff}(M)$ (in particular $F\in\mathrm{Diff}^1(M)$);
           \item $I$ satisfies the B-property, the measure $\mu\in\mathcal{M}(F)$ has total
           support;
           \item $I$ satisfies the B-property, the measure $\mu\in\mathcal{M}(F)$
           is ergodic,
         \end{itemize}
then the action $L_\mu$ is uniformly bounded on
$\mathrm{Fix}_{\mathrm{Cont},I}(F)$.
\end{prop}
\begin{proof}
By Proposition \ref{clm:L is well defined}, we only need consider a
compact set $\widetilde{P}$ of $\widetilde{M}$ such that
$\widetilde{P}$ contains a fundamental domain of the covering
transformation group $G$ (see \ref{sec:linking number in the special
case}).
\end{proof}
We now study the continuity of the actions $l_\mu$ and $L_\mu$. In
Example \ref{ex:the action of non C1-diffeo isnot bounded and
continuous} of Appendix, we will construct an isotopy $I$ and a
measure $\mu\in\mathcal{M}(F)$ such that the time-one map $F$ is a
diffeomorphism (hence satisfies the WB-property) but not a
$C^1$-diffeomorphism and the measure $\mu$ has total support and no
atoms on $\mathrm{Fix}_{\mathrm{Cont},I}(F)$, while the action is
not continuous. However, we have the following results.

\begin{prop}\label{prop:the continuity of lmu}
Suppose that $F$ is the time-one map of an isotopy $I$ on $M$ and
the measure $\mu\in\mathcal{M}(F)$ has no atoms on
$\mathrm{Fix}_{\mathrm{Cont},I}(F)$. If one of the following three
cases is satisfied\begin{itemize}
           \item $I$ satisfies the
           B-property and $F\in\mathrm{Diff}(M)$ (in particular $F\in\mathrm{Diff}^1(M)$);
           \item $I$ satisfies the B-property, the measure $\mu\in\mathcal{M}(F)$ has total
           support;
           \item $I$ satisfies the B-property, the measure $\mu\in\mathcal{M}(F)$
           is ergodic,
         \end{itemize} then the action $l_\mu$ is
continuous on $\mathrm{Fix}(\widetilde{F})$. As a consequence, if
$\rho_{M,I}(\mu)=0$, the action $L_\mu$ is continuous on
$\mathrm{Fix}_{\mathrm{Cont},I}(F)$.
\end{prop}

\begin{proof}
We suppose that
$\widetilde{a}\in\mathrm{Fix}(\widetilde{F})\setminus\{\infty\}$ and
$\{\widetilde{a}_n\}_{n\geq1}\subset
\mathrm{Fix}(\widetilde{F})\setminus\{\widetilde{a},\infty\}$
satisfying $\widetilde{a}_n\rightarrow\widetilde{a}$ as
$n\rightarrow+\infty$. We consider the value
$i_\mu(\widetilde{F};\widetilde{a}_n,\widetilde{a})$. There exists a
triangulation $\{\mathrm{Cl}(U_i)\}_{i=1}^{+\infty}$ of
$M\setminus\mathrm{Fix}(F)$ such that, for every $i$, the interior
$U_i$ of $\mathrm{Cl}(U_i)$ is an open free disk for $F$ and
satisfies $\mu(\partial U_i)=0$. %When $z\in U_i$, recall that the
%first return map is $\tau(z)$. For convenience, we write $\tau(z)=1$
%when $z\in\mathrm{Fix}(F)$.
\smallskip

%For any given $z\in
%\mathrm{Rec}^+(F)\setminus\pi(\{\widetilde{a}\})$, let
%$\widetilde{W}$ be any open neighborhood of $\widetilde{a}$
%satisfying
%$\widetilde{W}\cap\pi^{-1}(\{z,F(z),\cdots,F^{\tau(z)-1}(z)\})=\emptyset$.
%By Lemma \ref{lem:In Vn nW}, there exist a number
%$n_{\widetilde{W}}$, a family of isotopies
%$\{\widetilde{I}_n\}_{n\geq n_{\widetilde{W}}}$ with
%$\widetilde{I}_n$ fixing the points $\widetilde{a}$ ,$\infty$ and
%$\widetilde{a}_n$, and a family of neighborhoods
%$\{\widetilde{V}_n\}_{n\geq n_{\widetilde{W}}}$ of $\widetilde{a}$
%with $\widetilde{V}_n$ containing $\widetilde{a}_n$, such that
%$\widetilde{I}_n(\widetilde{z})\cap \widetilde{V}_n=\emptyset$ for
%any
%$\widetilde{z}\in\pi^{-1}(\{z,F(z),\cdots,F^{\tau(z)-1}(z)\})$.\smallskip
%
%The function $L_1(\widetilde{F};\widetilde{a}_n,\widetilde{a},z)$
%depends neither on the choice of the isotopy $\widetilde{I}$ that
%fixes the points $\widetilde{a}$, $\widetilde{a}_n$ and $\infty$,
%nor on the path from $\widetilde{a}_n$ to $\widetilde{a}$ (see
%\ref{subsec:the definition of a new linking number}). Therefore, for
%every $n\geq n_{\widetilde{W}}$, we can choose the isotopy
%$\widetilde{I}_n$ as above and a path in $\widetilde{V}_n$ from
%$\widetilde{a}_n$ to $\widetilde{a}$.
By Lemma \ref{lem:L1anaz}, we have
$\lim\limits_{n\rightarrow+\infty}i(\widetilde{F};\widetilde{a}_n,\widetilde{a},z)=0$
in the case where $z\in
\mathrm{Fix}(F)\setminus\{\pi(\widetilde{a})\}$, and
$\lim\limits_{n\rightarrow+\infty}L_1(\widetilde{F};\widetilde{a}_n,\widetilde{a},z)=0$
in the case where $z\in \mathrm{Rec}^+(F)\cap U_i$, for every $i$.
\smallskip

Choose a compact set $\widetilde{P}\subset\widetilde{M}$ such that
$\widetilde{a}\in\widetilde{P}$ and
$\{\widetilde{a}_n\}_{n\geq1}\subset\widetilde{P}$. As before, when
$\widetilde{a}\,'$ and $\widetilde{b}\,'$ are two distinct fixed
points of $\widetilde{F}$ in $\widetilde{P}$, we can always suppose
that the path $\widetilde{\gamma}$ that joins $\widetilde{a}\,'$ and
$\widetilde{b}\,'$ is in $\widetilde{P}$ in this proof when we talk
of the linking number
$i(\widetilde{F};\widetilde{a}\,',\widetilde{b}\,',z)$. By the
definition of B-property, Proposition \ref{prop:C1 diffeomorphism
has a bounded linkingnumber}, \ref{prop:the support of mu is M} and
Lemma \ref{lem:B property with a ergodic measure} , we can suppose
that there exists a number $N\geq0$ such that
$$N=\sup_{n\geq1}\left\{\,|i(\widetilde{F};\widetilde{a}_n,\widetilde{a},z)|\,\big\vert z\in
\mathrm{Rec}^+(F)\setminus\pi(\{\widetilde{a}_n,\widetilde{a}\})\right\}.$$

By Lebesgue's dominating convergence theorem (the dominated function
is $N$), we get
\begin{equation*}\label{eq:the estimation of the fixed point case}
    \lim_{n\rightarrow+\infty}\int_{\mathrm{Fix}(F)}
  \left|(i(\widetilde{F};\widetilde{a}_n,\widetilde{a},z)\right|\,\mathrm{d}\mu=0.
\end{equation*}
\smallskip

It is then sufficient to prove that
\begin{equation*}\label{eq:the estimation of non fixed point case}
    \lim_{n\rightarrow+\infty}\int_{M\setminus\mathrm{Fix}(F)}
  \left|(i(\widetilde{F};\widetilde{a}_n,\widetilde{a},z)\right|\,\mathrm{d}\mu=0.
\end{equation*}

Fix any $\epsilon>0$. Since
$\mu(\bigcup_{i=1}^{+\infty}U_i)=\mu(M\setminus\mathrm{Fix}(F))<+\infty$,
there exists a positive integer $N'$ such that
$$\mu(\bigcup_{N'+1}^{+\infty}U_i)<\frac{\epsilon}{2N}.$$

For every pair $(\widetilde{a},\widetilde{b})\in
(\mathrm{Fix}(\widetilde{F})\times\mathrm{Fix}(\widetilde{F}))
\setminus\widetilde{\Delta}$ and each $i$, by Birkhoff Ergodic
theorem, we have $\tau^*(\Phi(z))=\tau^*(z)$ where $\tau^*(z)$ is
the limit of the sequence $\{\tau_n(z)/n\}_{n\geq1}$ and $\Phi$ is
the first return map (see \ref{sec:the existion of rotation vector
in the copact case}), and
$L^*(\widetilde{F};\widetilde{a},\widetilde{b},\Phi(z))=L^*(\widetilde{F};\widetilde{a},\widetilde{b},z)$
. Hence,
$i(\widetilde{F};\widetilde{a},\widetilde{b},\Phi(z))=i(\widetilde{F};\widetilde{a},\widetilde{b},z)$
for $\mu$-almost every point $z\in U_i$. Obviously,
$|i(\widetilde{F};\widetilde{a}_n,\widetilde{a},z)|\tau(z)$ $\in
L^1(U_i,\mathbb{R},\mu)$. Therefore, for $\mu$-almost every point
$z\in U_i$, we have
\begin{eqnarray*}
% \nonumber to remove numbering (before each equation)
&&\lim_{m\rightarrow+\infty}\frac{1}{m}\sum_{j=0}^{m-1}\left(\tau(\Phi^j(z))
\left|i(\widetilde{F};\widetilde{a}_n,\widetilde{a},\Phi^j(z))\right|\right)\\
 &=&\lim_{m\rightarrow+\infty}
\left(\frac{1}{m}\sum_{j=0}^{m-1}\tau(\Phi^j(z))\right)\cdot\left|i(\widetilde{F};\widetilde{a}_n,\widetilde{a},z)\right|\\
 &=&\tau^*(z)\left|i(\widetilde{F};\widetilde{a}_n,\widetilde{a},z)\right|.
\end{eqnarray*}
This implies that
\begin{equation}\label{eq:an ergodic result of pruduction}
    \int_{U_i}\tau(z)\left|i(\widetilde{F};\widetilde{a}_n,\widetilde{a},z)
\right|\,\mathrm{d}\mu=\int_{U_i}\tau^*(z)\left|i(\widetilde{F};\widetilde{a}_n,\widetilde{a},z)\right|\,\mathrm{d}\mu
\end{equation}
for every $i$ and every $n$.

\smallskip

Recall that, for every pair $(\widetilde{a},\widetilde{b})\in
(\mathrm{Fix}(\widetilde{F})\times\mathrm{Fix}(\widetilde{F}))
\setminus\widetilde{\Delta}$ and every $i$,
$$L^{*}(\widetilde{F};\widetilde{a},\widetilde{b},z)=\lim\limits_{m\rightarrow+\infty}
\frac{1}{m}\sum_{j=1}^{m-1}L_1(\widetilde{F};\widetilde{a},\widetilde{b},\Phi^j(z))$$
exists for $\mu$-almost every point $z\in U_i$. From Proposition
\ref{prop:C1 diffeomorphism has a bounded linkingnumber} and
Proposition \ref{prop:the support of mu is M}, we have
$\left|L_1(\widetilde{F};\widetilde{a}_n,\widetilde{a},z)\right|\leq
  N\tau(z)$, which implies that $L_1(\widetilde{F};\widetilde{a}_n,\widetilde{a},z)\in
  L^1(U_i,\mathbb{R},\mu)$ for every $i$.
Therefore, we have the following inequality (modulo subsets of
measure zero of $U_i$)
\begin{eqnarray}\label{ineq:two birkhoff ineqs}
% \nonumber to remove numbering (before each equation)
\left|L^{*}(\widetilde{F};\widetilde{a}_n,\widetilde{a},z)\right|
&=& \lim_{m\rightarrow +\infty} \frac{1}{m}\left|\sum_{j=0}^{m-1}
   (L_1(\widetilde{F};\widetilde{a}_n,\widetilde{a},\Phi^j(z))\right|\\
  &\leq& \lim_{m\rightarrow +\infty} \frac{1}{m}\sum_{j=0}^{m-1}\left
  |L_1(\widetilde{F};\widetilde{a}_n,\widetilde{a},\Phi^j(z))\right|\nonumber \\
   &\stackrel{\triangle}{=}& \left|L_1(\widetilde{F};\widetilde{a}_n,\widetilde{a},z)\right|^{*}.\nonumber
\end{eqnarray}
The last definition and equation hold due to Birkhoff Ergodic
theorem.
\smallskip

Moreover, we have
\begin{eqnarray*}
% \nonumber to remove numbering (before each equation)
\int_{\bigcup\limits_{i=1}^{N'}U_i}
  \left|i(\widetilde{F};\widetilde{a}_n,\widetilde{a},z)\right|\,\mathrm{d}\mu&\leq&
\sum_{i=1}^{N'}\int_{\bigcup_{k\geq0}F^k(U_i)}\left|i(\widetilde{F};\widetilde{a}_n,\widetilde{a},z)\right|\,\mathrm{d}\mu \\
  &=& \sum_{i=1}^{N'}\int_{U_i}\tau(z)\left|i(\widetilde{F};\widetilde{a}_n,\widetilde{a},z)\right|\,\mathrm{d}\mu \\
   &=& \sum_{i=1}^{N'}\int_{U_i}\tau^{*}(z)\left|i(\widetilde{F};\widetilde{a}_n,\widetilde{a},z)\right|\,\mathrm{d}\mu \\
   &=& \sum_{i=1}^{N'}\int_{U_i}\left|L^{*}(\widetilde{F};\widetilde{a}_n,\widetilde{a},z)\right|\,\mathrm{d}\mu  \\
   &\leq& \sum_{i=1}^{N'}\int_{U_i}\left|L_1(\widetilde{F};\widetilde{a}_n,\widetilde{a},z)\right|^{*}\,\mathrm{d}\mu \\
   &=& \sum_{i=1}^{N'}\int_{U_i}\left|L_1(\widetilde{F};\widetilde{a}_n,\widetilde{a},z)\right|\,\mathrm{d}\mu \\
   &\rightarrow& 0 \qquad(n \rightarrow +\infty).
\end{eqnarray*}
The second equation holds since $F$ preserve the measure $\mu$ and
$i(\widetilde{F};\widetilde{a},\widetilde{b},z)$ is the action of
$F$. The third equation holds by Equation \ref{eq:an ergodic result
of pruduction}. The forth equation is true because
$i(\widetilde{F};\widetilde{a},\widetilde{b},z)=L^{*}(\widetilde{F};\widetilde{a},\widetilde{b},z)/\tau^{*}(z)$.
The fifth inequality holds by the Inequality \ref{ineq:two birkhoff
ineqs}. The sixth equation holds due to Birkhoff Ergodic theorem.
The last limit holds due to Lebesgue's dominating convergence
theorem (the dominated function is $N\tau(z)$) and Lemma
\ref{lem:L1anaz}.
\smallskip

Therefore, there exists a positive number $N''$ such that when
$n\geq N''$, $$\int_{\bigcup\limits_{i=1}^{\,N'}U_i}
  \left|i(\widetilde{F};\widetilde{a}_n,\widetilde{a},z)\right|\,\mathrm{d}\mu\\
< \frac{\epsilon}{2}.$$ Finally, when $n\geq N''$, we have
\begin{eqnarray*}
% \nonumber to remove numbering (before each equation)
 \int_{M\setminus\mathrm{Fix}(F)}
  \left|(i(\widetilde{F};\widetilde{a}_n,\widetilde{a},z)\right|\,\mathrm{d}\mu
   &=& \int_{\bigcup\limits_{i=1}^{N'}U_i}\left|(i(\widetilde{F};\widetilde{a}_n,\widetilde{a},z)\right|\,\mathrm{d}\mu+\int_{\bigcup\limits_{N'+1}^{+\infty}U_i}
\left|(i(\widetilde{F};\widetilde{a}_n,\widetilde{a},z)\right|\,\mathrm{d}\mu\\
   &<& \frac{\epsilon}{2}+\frac{\epsilon}{2N}\cdot
   N\\
   &=&\epsilon.
\end{eqnarray*}
We have completed the first statement.\bigskip

Now we turn to prove the second statement. Let
$a\in\mathrm{Fix}_{\mathrm{Cont},I}(F)$ and $\{a_n\}_{n\geq1}\subset
\mathrm{Fix}_{\mathrm{Cont},I}(F)\setminus\{a\}$ satisfying
$a_n\rightarrow a$ as $n\rightarrow+\infty$. By Proposition
\ref{clm:L is well defined}, we only need to consider a lift
$\widetilde{a}$ of $a$ and a lift sequence
$\{\widetilde{a}_n\}_{n\geq1}\subset \mathrm{Fix}(\widetilde{F})$ of
$\{a_n\}_{n\geq1}$ satisfying $\widetilde{a}_n\rightarrow
\widetilde{a}$ as $n\rightarrow+\infty$. Then it follows from the
first statement.
\end{proof}\bigskip

%As a consequence, if $F$ is $C^1$-diffeomorphism of $M$, by
%Proposition \ref{rem:any isotopy of M satisfies condition B}, the
%isotopy $I$ satisfies the B-property  and hence the action is
%continuous.

\subsection{Action spectrum}\label{sec:action spectrum}
In this section, we suppose that the action $l_\mu$ is well defined.
Write $\widetilde{F}$ as the lift of $F$ obtained by lifting $I$ to
an isotopy $\widetilde{I}$ to $\widetilde{M}$ starting
$\mathrm{Id}_{\widetilde{M}}$.\smallskip

Define the \emph{action spectrum of $I$} as follows (up to an
additive constant):
\begin{equation*}
\sigma(\widetilde{F})=\{l_\mu(\widetilde{F};\widetilde{z})\mid z\in
\mathrm{Fix}(\widetilde{F})\}\subset\mathbb{R}\cup\{\pm\infty\}.
\end{equation*}

By Corollary \ref{cor:imu(HFH-1,Ha,Hb)=imu(F,a,b)}, the action
spectrum of $I$ is invariant under conjugation by an orientation
preserving and measure preserving homeomorphism of $M$.\smallskip

Define the \emph{action width of $I$} as follows:
\begin{equation*}
    \mathrm{width}(\widetilde{F})=\sup_{x,y\in\sigma(\widetilde{F})}|x-y|=
    \sup_{\widetilde{z}_1,\widetilde{z}_2\in\mathrm{Fix}(\widetilde{F})}i_\mu
    (\widetilde{F};\widetilde{z}_1,\widetilde{z}_2).
\end{equation*}\smallskip

Moreover, if $\rho_{M,I}(\mu)=0$, we can write the action spectrum
of $I$ as (up to an additive constant):
\begin{equation*}
\sigma(\widetilde{F})=\{L_\mu(\widetilde{F};z)\mid z\in
\mathrm{Fix}_{\mathrm{Cont},I}(F)\}\subset\mathbb{R}\cup\{\pm\infty\}.
\end{equation*}

Furthermore, if $L_\mu$ is continuous (see Proposition
\ref{prop:boundedness} and Proposition \ref{prop:the continuity of
lmu}), $\sigma(\widetilde{F})$ is a compact set of $\mathbb{R}$ and
we can write the action width of $I$ as:
$$\mathrm{width}(\widetilde{F})=\max\limits_{z_1,z_2\in\mathrm{Fix}
_{\mathrm{Cont},I}(F)}I_{\mu}(\widetilde{F},z_1,z_2).$$\bigskip

The following Theorem is the Arnold conjecture for surface
homeomorphisms that is due to Matsumoto \cite{Ma} (see also
\cite{P1}):
\begin{thm}\label{thm:hamiltonian map has at least 3 contricible fixed
points} Let $M$ be a compact surface with genus $g\geq1$ and $F$ be
the time-one map of an identity isotopy $I$ on $M$. We suppose that
$\mu\in\mathcal{M}(F)$ has total support and $\rho_{M,I}(\mu)=0$.
Then there exist at least three contractible fixed points of $F$.
\end{thm}

On a closed surface, based on Theorem \ref{thm:hamiltonian map has
at least 3 contricible fixed points}, we can get the following
result that is a generalization of Lemma 2.8 that is proved in
\cite{SM} by using Floer homology.\bigskip

\newenvironment{prop02}{\noindent\textbf{Proposition \ref{prop:F is not
constant if the contractible fixed points is
finite}}~\itshape}{\par}
\begin{prop02}Let $F$ be the
time-one map of an identity isotopy $I$ on a closed oriented surface
$M$ with $g\geq 1$. If $I$ satisfies the WB-property and $F\in
\mathrm{Homeo}_*(M)\setminus\{\,\mathrm{Id}_{M}\}$, $\mu\in\mathcal
{M}(F)$ has total support, no atoms on
$\mathrm{Fix}_{\mathrm{Cont},I}(F)$ and $\rho_{M,I}(\mu)=0$, then
$\sharp\sigma(\widetilde{F})\geq2$, that is, the action function
$L_\mu$ is not constant.
\end{prop02}\bigskip

The proof of Proposition \ref{prop:F is not constant if the
contractible fixed points is finite} will be divided two cases: the
set $\mathrm{Fix}_{\mathrm{Cont},I}(F)$ is finite and the set
$\mathrm{Fix}_{\mathrm{Cont},I}(F)$ is infinite. Firstly, let us
prove the case where the set $\mathrm{Fix}_{\mathrm{Cont},I}(F)$ is
finite which is an easer case.

\begin{proof}[Proof of the case of Proposition \ref{prop:F is not
constant if the contractible fixed points is finite} where the set
$\mathrm{Fix}_{\mathrm{Cont},I}(F)$ is finite]\qquad\par\smallskip

We say that $X\subseteq \mathrm{Fix}_{\mathrm{Cont},I}(F)$ is
\emph{unlinked}
 if there exists an isotopy $I'=(F'_t)_{t\in [0,1]}$ homotopic
to $I$ which fixes every point of $X$. Moreover, we say that $X$ is
a \emph{maximal unlinked set}, if any set $X'\subseteq
\mathrm{Fix}_{\mathrm{Cont},I}(F)$ which strictly contains $X$ is
not unlinked.\smallskip

In the proof of Theorem \ref{thm:hamiltonian map has at least 3
contricible fixed points} (\cite[Theorem 10.1]{P1}), Le Calvez
proved that there exists a maximal unlinked set $X\subseteq
\mathrm{Fix}_{\mathrm{Cont},I}(F)$ with $\sharp X\geq 3$ if
$\sharp\mathrm{Fix}_{\mathrm{Cont},I}(F)<+\infty$.\smallskip

There exists an oriented topological foliation $\mathcal{F}$ on
$M\setminus X$ (or, equivalently, a singular oriented foliation
$\mathcal{F}$ on $M$ with $X$ equal to the singular set) such that,
for all $z\in M\setminus X$, the trajectory $I(z)$ is homotopic to
an arc $\gamma$ joining $z$ and $F(z)$ in $M\setminus X$ which is
positively transverse to $\mathcal{F}$. That means that for every
$t_0\in[0,1]$ there exists an open neighborhood $V\subset M\setminus
X$ of $\gamma(t_0)$ and an orientation preserving homeomorphism
$h:V\rightarrow(-1,1)^2$ which sends the foliation $\mathcal{F}$ on
the horizontal foliation (oriented with $x_1$ increasing) such that
the map $t\mapsto p_2(h(\gamma(t)))$ defined in a neighborhood of
$t_0$ is strictly increasing where $p_2(x_1,x_2)=x_2$.\smallskip

We can choose a point $z\in
\mathrm{Rec}^+(F)\setminus\mathrm{Fix}(F)$ and a leaf $\lambda$
containing $z$. Proposition 10.4 in \cite{P1} states that the
$\omega$-limit set $\omega(\lambda)\in X$, the $\alpha$-limit set
$\alpha(\lambda)\in X$ and $\omega(\lambda)\neq\alpha(\lambda)$. Fix
an isotopy $I'$ homotopic to $I$ that fixes $\omega(\lambda)$ and
$\alpha(\lambda)$ and a lift $\widetilde{\lambda}$ of $\lambda$ that
joins $\widetilde{\omega(\lambda)}$ and
$\widetilde{\alpha(\lambda)}$. Let us now study the linking number
$i(\widetilde{F};\widetilde{\omega(\lambda)},\widetilde{\alpha(\lambda)},z')$
for $z'\in \mathrm{Rec}^+(F)\setminus X$ if it exists. Observe that
for all $z'\in M\setminus X$, the trajectory $I'(z')$ is still
homotopic to an arc that is positively transverse to $\mathcal{F}$.
Hence, for all $z'\in \mathrm{Rec}^+(F)\setminus X$ and disk $U$
containing $z'$ (here, we suppose that $U\cap \lambda=\emptyset$ by
shrinking $U$ and perturbing $\lambda$ if necessary), we have
\begin{equation*}\label{eq: Ln is not negtive}
    L_n(\widetilde{F};\widetilde{\omega(\lambda)},\widetilde{\alpha(\lambda)},z')
=\widetilde{\lambda}\wedge\widetilde{\Gamma}^n_{\widetilde{I}\,',z'}=
\lambda\wedge\Gamma^n_{I\,',z'}\geq0
\end{equation*}
for every $n\geq1$, where $\widetilde{I}\,'$ is the lift of $I'$ to
$\widetilde{M}$ (refer to Section \ref{subsec:the definition of a
new linking number}). Finally, we have
$$i(\widetilde{F};\widetilde{\omega(\lambda)},\widetilde{\alpha(\lambda)},z')\geq0$$
for $\mu$-almost every point
$z'\in\mathrm{Rec}^+(F)\setminus\{\omega(\lambda),\alpha(\lambda)\}$
(refer to Definition \ref{def:Intersection number density}).

By the continuity of $I'$ and the hypothesis on $\mu$, there exists
an open free disk $U$ containing $z$ such that $\mu(U)>0$ and
$L_1(\widetilde{F};\widetilde{\omega(\lambda)},\widetilde{\alpha(\lambda)},z')>0$
when $z'\in U\cap\mathrm{Rec}^+(F)$.

Similarly to the proof of Proposition \ref{prop:the continuity of
lmu}, we have
\begin{eqnarray*}
% \nonumber to remove numbering (before each equation)
I_\mu(\widetilde{F};\omega(\lambda),\alpha(\lambda))&\geq&
\int_{\bigcup_{k\geq0}F^k(U)}i(\widetilde{F};\widetilde{\omega(\lambda)},\widetilde{\alpha(\lambda)},z)\,\mathrm{d}\mu \\
  &=& \int_{U}\tau(z)i(\widetilde{F};\widetilde{\omega(\lambda)},\widetilde{\alpha(\lambda)},z)\,\mathrm{d}\mu \\
   &=&\int_{U}\tau^{*}(z)i(\widetilde{F};\widetilde{\omega(\lambda)},\widetilde{\alpha(\lambda)},z)\,\mathrm{d}\mu \\
   &=&\int_{U}L^{*}(\widetilde{F};\widetilde{\omega(\lambda)},\widetilde{\alpha(\lambda)},z)\,\mathrm{d}\mu  \\
   &=&\int_{U}L_1(\widetilde{F};\widetilde{\omega(\lambda)},\widetilde{\alpha(\lambda)},z)\,\mathrm{d}\mu \\
   &>& 0.
\end{eqnarray*}
\end{proof}

Before proving the case where the set
$\mathrm{Fix}_{\mathrm{Cont},I}(F)$ is infinite., let us state a
recent result due to Jaulent \cite{J}:
\smallskip

\begin{thm}[Jaulent]\label{thm:O.Jaulent}
Let $M$ be an oriented surface and $F$ be the time-one map of an
identity isotopy $I$ on $M$. There exists a closed subset $X\subset
\mathrm{Fix}(F)$ and an isotopy $I'$ joining
$\mathrm{Id}_{M\setminus X}$ to $F|_{M\setminus X}$ in
$\mathrm{Homeo(M\setminus X)}$ such that
\begin{enumerate}
  \item For all $z\in X$, the loop $I(z)$ is homotopic to zero in
  $M$.
  \item For all $z\in\mathrm{Fix(F)}\setminus X$, the loop $I'(z)$
  is not homotopic to zero in $M\setminus X$.
  \item For all $z\in M\setminus X$, the trajectories $I(z)$ and
  $I'(z)$ are homotopic with fixed endpoints in $M$.
  \item There exists an oriented topological foliation $\mathcal{F}$
  on $M\setminus X$ such that, for all $z\in M\setminus X$, the trajectory
  $I'(z)$ is homotopic to an arc $\gamma$ joining $z$ and $F(z)$ in
  $M\setminus X$
  which is positively to $\mathcal{F}$.
\end{enumerate}
Moreover, the isotopy $I'$ satisfies the following property:

\begin{enumerate}
\item[(5)] For all finite $Y\subset X$, there exists an isotopy $I_Y'$
joining $\mathrm{Id}_M$ and $F$ in $\mathrm{Homeo}(M)$ which fixes
$Y$ such that, if $z\in M\setminus X$, the arc $I'(z)$ and $I_Y'(z)$
are homotopic in $M\setminus Y$. And if $z\in X\setminus Y$, the
loop $I_Y'(z)$ is contractible in $M\setminus Y$.
\end{enumerate}
\end{thm}

\begin{proof}[Proof of the case of Proposition \ref{prop:F is not
constant if the contractible fixed points is finite} where the set
$\mathrm{Fix}_{\mathrm{Cont},I}(F)$ is infinite]\qquad\par\smallskip

Suppose that $X$, $I'$ and $\mathcal {F}$ are respectively the
closed contractible
 fixed points set, the isotopy and foliation as stated in Theorem
 \ref{thm:O.Jaulent}. Obviously, $X\neq\emptyset$ and $\mu(M\setminus X)>0$. Assume that $X'$ is the union of the connected
 components of $X$ that separates $M$. Write $M\setminus X'=\cup_{i=1}^nS_i$ where $n\geq1$ and $S_i$ are the $F$-invariant
 subsurfaces of $M$. By the definitions of $S_i$ and
 $I'$, we have the following properties
\begin{description}
   \item[(A1)] if $S_i$ is a disk, then we have $(X\setminus \partial S_i)\cap
   S_i\neq\emptyset$ (by Proposition \ref{prop:Franks' Lemma} and item (2) of Theorem \ref{thm:O.Jaulent});
   \item[(A2)] $\rho_{S_i,I'}(\mu)=0\in H_1(S_i,\mathbb{R})$ for every $i$ (by the item (1) and (3) of Theorem \ref{thm:O.Jaulent}).
 \end{description}
It implies that the sum of the number of the connected component of
$\partial S_i$ and the number of the connected component of $X\cap
S_i$ is greater than 2. Indeed, it is sufficient to prove the case
when $S_i$ is not a subsurface of sphere by A1. Identifying every
connected component of $\partial S_i$ as one point, we get a closed
surface $S'_i$ and an identity isotopy induced by $I'$, written
still $I'$, which satisfy $\rho_{S'_i,I'}(\mu)=0\in
H_1(S'_i,\mathbb{R})$ by A2. Using Theorem \ref{thm:hamiltonian map
has at least 3 contricible fixed points}, we prove the claim.

Fix one subsurface $S_i$. Similarly to Proposition \ref{prop:F is
not constant if the contractible fixed points is finite}, we choose
a point $z\in (\mathrm{Rec}^+(F)\setminus\mathrm{Fix}(F))\cap S_i$
and a leaf $\lambda\in\mathcal {F}$ containing $z$. The proofs of
Proposition 4.1 (page 150, when $S_i$ is a disk or an annulus) and
Proposition 6.1 (page 166, when $S_i$ is not a subsurface of the
sphere) in \cite{P3} say that  $\omega(\lambda)$ (resp.,
$\alpha(\lambda)$) is connected and is contained in a connected
component of $\partial S_i$ or a connected component of $X\cap S_i$,
written $X_+(\lambda)$ (resp. $X_-(\lambda)$). Moreover,
$X_-(\lambda)\neq X_+(\lambda)$. Choose a lift $\widetilde{\lambda}$
of $\lambda$. We have to consider the following four cases: the set
$\omega(\widetilde{\lambda})$ or $\alpha(\widetilde{\lambda})$
contains $\infty$ or not.

 Take two points $a\in\alpha(\lambda)$ and
$b\in\omega(\lambda)$. Let $Y=\{a,b\}$ and $I'_Y$ be the isotopy as
stated in Theorem \ref{thm:O.Jaulent}. Suppose that
$\widetilde{I}\,'_Y$ is the identity lift of $I'_Y$ to
$\widetilde{M}$. Notice that
\begin{description}
  \item[(B1)] if $z\in M\setminus X$, the arcs $I'(z)$
and $I_Y'(z)$ are homotopic in $M\setminus Y$ by item (5) of Theorem
\ref{thm:O.Jaulent}, and by item (4) of Theorem \ref{thm:O.Jaulent},
$I_Y'(z)$ is homotopic to an arc $\gamma$ joining $z$ and $F(z)$ in
$M\setminus Y$ which is positively transverse to $\mathcal{F}$;
  \item[(B2)] if $z\in X\setminus Y$, then $\gamma\wedge I_Y'(z)=0$ by
the item (5) of Theorem \ref{thm:O.Jaulent} where $\gamma$ is any
path from $a$ to $b$.
\end{description}

If both $\alpha(\widetilde{\lambda})$ and
$\omega(\widetilde{\lambda})$ do not contain $\infty$, replacing $a$
by $\alpha(\lambda)$, $b$ by $\omega(\lambda)$ and $I'$ by $I'_Y$ in
the proof of Proposition \ref{prop:F is not constant if the
contractible fixed points is finite}, then we can get
$I_\mu(\widetilde{F};a,b)>0$.

We suppose now that at least one of $\alpha(\widetilde{\lambda})$
and $\omega(\widetilde{\lambda})$ contains $\infty$. Recall that
$\widetilde{d}$ is the distance of $\widetilde{M}$ induced by
 a distance $d$ of $M$ which is induced by a Riemannian metric on $M$. Define
$\mathrm{dist}(\widetilde{z},\widetilde{C})=\inf\limits_{\widetilde{c}\in\widetilde{C}}
       \widetilde{d}(\widetilde{z},\widetilde{C})$ if
       $\widetilde{z}\in \widetilde{M}$ and $\widetilde{C}\subset
       \widetilde{M}$. Take a sequence
$\{(\widetilde{a}_m,\widetilde{b}_m)\}_{m\geq1}$ such that
\begin{itemize}
       \item $\pi(\widetilde{a}_m)=a$ and $\pi(\widetilde{b}_m)=b$\,\,, if $\alpha(\widetilde{\lambda})$
        (resp. $\omega(\widetilde{\lambda})$) does not contain $\infty$, we set $\widetilde{a}_m=\widetilde{a}$
         (resp. $\widetilde{b}_m=\widetilde{b})$ for every $m$ where $\widetilde{a}\in\pi^{-1}(a)\cap\alpha(\widetilde{\lambda})$
         (resp. $\widetilde{b}\in\pi^{-1}(b)\cap\omega(\widetilde{\lambda})$);
       \item
       $\lim\limits_{n\rightarrow+\infty}\mathrm{dist}(\widetilde{a}_m,\widetilde{\lambda})=0$
       and
       $\lim\limits_{n\rightarrow+\infty}\mathrm{dist}(\widetilde{b}_m,\widetilde{\lambda})=0$.
     \end{itemize}
For every $m$, suppose that $\widetilde{c}_m$ (resp.
$\widetilde{c}\,'_m$) is a point of $\widetilde{\lambda}$ such that
$\widetilde{d}(\widetilde{a}_m,\widetilde{c}_m)=\mathrm{dist}(\widetilde{a}_m,\widetilde{\lambda})$
(resp.
$\widetilde{d}(\widetilde{b}_m,\widetilde{c}\,'_m)=\mathrm{dist}(\widetilde{b}_m,\widetilde{\lambda})$).
Obviously, if $\alpha(\widetilde{\lambda})$ (resp.
$\omega(\widetilde{\lambda})$) does not contain $\infty$, then
$\widetilde{c}_m=\widetilde{a}_m=\widetilde{a}$ (resp.
$\widetilde{c}\,'_m=\widetilde{b}_m=\widetilde{b}$) and
$\mathrm{dist}(\widetilde{a}_m,\widetilde{\lambda})=0$ (resp.
$\mathrm{dist}(\widetilde{b}_m,\widetilde{\lambda})=0$). Choose a
simple path $\widetilde{l}_m$ (resp. $\widetilde{l}\,'_m$) from
$\widetilde{a}_m$ (resp. $\widetilde{c}\,'_m$) to $\widetilde{c}_m$
(resp. $\widetilde{b}_m$) such that the length of $\widetilde{l}_m$
(resp. $\widetilde{l}\,'_m$) is
$\mathrm{dist}(\widetilde{a}_m,\widetilde{\lambda})$ (resp.
$\mathrm{dist}(\widetilde{b}_m,\widetilde{\lambda})$). Here, we set
the simple path is empty set if its length is $0$. Let
$\widetilde{\gamma}_m=\widetilde{l}_m\widetilde{\lambda}_m\widetilde{l}\,'_m$
where $\widetilde{\lambda}_m$ is the sub-path of
$\widetilde{\lambda}$ from $\widetilde{c}_m$ to
$\widetilde{c}\,'_m$. Then $\widetilde{\gamma}_m$ is a path from
$\widetilde{a}_m$ to
$\widetilde{b}_m$. %By the construction of $\widetilde{\gamma}_n$,
%similarly to the proof of Proposition \ref{prop:F is not constant if
%the contractible fixed points is finite},

We know that, for every $m\geq1$, the linking number
$i(\widetilde{F};\widetilde{a}_m,\widetilde{b}_m,z')$ exists for
$\mu$-almost every $z'\in M\setminus\{a,b\}$. Hence, the linking
number $i(\widetilde{F};\widetilde{a}_m,\widetilde{b}_m,z')$ exists
on a full measure subset of $M\setminus\{a,b\}$ for all $m$. By the
property B2 above, we have
$i(\widetilde{F};\widetilde{a}_m,\widetilde{b}_m,z')=0$ if $z'\in
X\setminus\{a,b\}$. We now claim that
$\liminf\limits_{m\rightarrow+\infty}
i(\widetilde{F};\widetilde{a}_m,\widetilde{b}_m,z')\geq0$ for
$\mu$-almost every $z'\in \mathrm{Rec}^+(F)\setminus X$.

Fix one point $z'\in \mathrm{Rec}^+(F)\setminus X$ and choose a disk
$U$ containing $z'$ (here again, we suppose that $U\cap
\lambda=\emptyset$). By the property B1 and the construction of
$\widetilde{\gamma}_m$, for every $n\geq1$, there exists
$m(z',n)\in\mathbb{N}$ such that when $m\geq m(z',n)$, the value
$$L_n(\widetilde{F};\widetilde{a}_m,\widetilde{b}_m,z')=
\widetilde{\gamma}_m\wedge\widetilde{\Gamma}^n_{\widetilde{I}\,'_Y,z'}
=\pi(\widetilde{\gamma}_m)\wedge\Gamma^n_{I\,'_Y,z'}$$ is constant
with regard to $m$ and

\begin{equation}\label{Ln geq 0}
    L_n(\widetilde{F};\widetilde{a}_m,\widetilde{b}_m,z')\geq0.
\end{equation}

We now suppose that $$\mu\{z'\in \mathrm{Rec}^+(F)\setminus X\mid
\liminf\limits_{m\rightarrow+\infty}
i(\widetilde{F};\widetilde{a}_m,\widetilde{b}_m,z')<0\}>0.$$

There exists a small number $c>0$ such that
\begin{equation}\label{mu(E)>c}
    \mu\{z'\in
\mathrm{Rec}^+(F)\setminus X\mid
\liminf\limits_{m\rightarrow+\infty}
i(\widetilde{F};\widetilde{a}_m,\widetilde{b}_m,z')<-c\}>c.
\end{equation}\smallskip

Write $E=\{z'\in \mathrm{Rec}^+(F)\setminus X\mid
\liminf\limits_{m\rightarrow+\infty}
i(\widetilde{F};\widetilde{a}_m,\widetilde{b}_m,z')<-c\}$. Fix a
point $z'\in E$ and a disk $U$ containing $z'$ as before. By taking
subsequence if necessary, we may suppose that
$$-\infty\leq\lim\limits_{m\rightarrow+\infty}
i(\widetilde{F};\widetilde{a}_m,\widetilde{b}_m,z')<-c.$$ Then there
exists $N(z')\in \mathbb{N}$ such that when $m\geq N(z')$, we have
$$
i(\widetilde{F};\widetilde{a}_m,\widetilde{b}_m,z')=\lim_{n\rightarrow+\infty}
\frac{L_n(\widetilde{F};\widetilde{a}_m,\widetilde{b}_m,z')}{\tau_n(z')}<-c.$$
Fix $m_0\geq N(z')$. There exists $n(z',m_0)\in\mathbb{N}$ such that
when $n\geq n(z',m_0)$, we have
$$\frac{L_n(\widetilde{F};\widetilde{a}_{m_0},\widetilde{b}_{m_0},z')}{\tau_n(z')}<-c.$$
Fix $n_0\geq n(z',m_0)$. Then, we have
$$L_{n_0}(\widetilde{F};\widetilde{a}_{m_0},\widetilde{b}_{m_0},z')<-c\tau_{n_0}(z').$$

By the inequality \ref{Ln geq 0}, there exists $m(z',n_0)>m_0$ such
that, when $m\geq m(z',n_0)$, we have
$$L_{n_0}(\widetilde{F};\widetilde{a}_m,\widetilde{b}_m,z')\geq0.$$
Fix $m_1\geq m(z',n_0)$. There exists $n(z',m_1)>n_1$ such that when
$n\geq n(z',m_1)$, we have
$$\frac{L_n(\widetilde{F};\widetilde{a}_{m_1},\widetilde{b}_{m_1},z')}{\tau_n(z')}<-c.$$
Fix $n_1\geq n(z',m_1)$. Then, we have
$$L_{n_1}(\widetilde{F};\widetilde{a}_{m_1},\widetilde{b}_{m_1},z')<-c\tau_{n_1}(z').$$

By induction, we can construct a sequence
$\{(m_i,n_i)\}_{i\geq0}\subset \mathbb{N}\times\mathbb{N}$ such that
\begin{description}
  \item[(C1)]\, $\{m_i\}_{i\geq0}$ and $\{n_i\}_{i\geq0}$ are strictly increasing
sequences;
  \item[(C2)]\, for every $i\geq0$,
  $$L_{n_i}(\widetilde{F};\widetilde{a}_{m_i},\widetilde{b}_{m_i},z')<-c\tau_{n_i}(z')\quad
\mathrm{and} \quad
L_{n_i}(\widetilde{F};\widetilde{a}_{m_{i+1}},\widetilde{b}_{m_{i+1}},z')\geq0.$$
\end{description}

As the positively transverse property of $\mathcal {F}$, it is easy
to see that the negative part of
$L_{n_i}(\widetilde{F};\widetilde{a}_{m_i},\widetilde{b}_{m_i},z')$
only comes from the intersection $l_{m_i}$ or $l'_{m_i}$, or both of
them with the curve $\widetilde{\Gamma}^n_{\widetilde{I}\,'_Y,z'}$
in the case where $\alpha(\widetilde{\lambda})$ or
$\omega(\widetilde{\lambda})$ contains $\infty$, or both of them
contain $\infty$.

We deal with the case where both $\alpha(\widetilde{\lambda})$ and
$\omega(\widetilde{\lambda})$ contain $\infty$, and other cases
follow similarly. By the item (5) of Theorem \ref{thm:O.Jaulent}, it
is easy to see that there is a positive integer $N$ such that the
number of times that $I_Y'(x)$ rotates around $a$ (resp. $b$) is
less than $N$ when $x$ is close to $a$ (resp. $b$). Since $I_Y'$
fixes $a$ and $b$, the measure $\mu$ has no atoms on
$\mathrm{Fix}_{\mathrm{Cont},I}(F)$, the construction of
$\widetilde{\lambda}_m$, and the property C2, there must be an open
disks sequence $\{U^a_i\}_{i\geq0}$ that contains the set
$(I_Y')^{-1}(\pi(\widetilde{l}_{m_i}))=\cup_{y\in\pi(\widetilde{l}_{m_i})}(I_Y')^{-1}(y)$
and an open disks sequence $\{U^b_i\}_{i\geq0}$ that contains the
set
$(I_Y')^{-1}(\pi(\widetilde{l}\,'_{m_i}))=\cup_{y\in\pi(\widetilde{l}\,'_{m_i})}(I_Y')^{-1}(y)$
satisfying
\begin{description}
  \item[(D1)] $U^a_{i+1}\subset U^a_{i}$ (resp. $U^b_{i+1}\subset U^b_{i}$)
 and $\mu(U^a_i)\rightarrow0$ (resp. $\mu(U^b_i)\rightarrow0$) as
 $i\rightarrow+\infty$;
  \item[(D2)] for every $i\geq0$,
 $$\frac{1}{\tau_{n_i}(z')}\sum_{j=0}^{\tau_{n_i}(z')-1}\chi_{U^a_i}\circ
 F^j(z')>\frac{c}{2N}\quad \mathrm{or} \quad\frac{1}{\tau_{n_i}(z')}\sum_{j=0}^{\tau_{n_i}(z')-1}\chi_{U^b_i}\circ
 F^j(z')>\frac{c}{2N},$$ where $\chi_U$ is the characteristic function
 of $U\subset M$.
\end{description}

Denote by $\chi_{U}^*(x)$ the limit function of
$\frac{1}{n}\sum\limits_{j=0}^{n-1}\chi_{U}\circ
 F^j(x)$ as $n\rightarrow+\infty$ for $\mu$-almost every $x\in M$ (by Birkhoff Ergodic theorem). By the property D2 and the inequality \ref{mu(E)>c}, we have
 $$\mu(\{x\in \mathrm{Rec}^+(F)\setminus X\mid\chi_{U^a_i}^*(x)\geq\frac{c}{2N}\quad\mathrm{or}
 \quad\chi_{U^b_i}^*(x)\geq\frac{c}{2N}\})>c$$ for every $i$. This implies that
 $\int_M(\chi_{U^a_i}^*(x)+\chi_{U^b_i}^*(x))d\mu\geq\frac{c^2}{2N}>0$ for every $i$. On the
 other hand, by Birkhoff Ergodic theorem and the property D1, we have $$\int_M(\chi_{U^a_i}^*(x)+\chi_{U^b_i}^*(x))d\mu
 =\int_M(\chi_{U^a_i}(x)+\chi_{U^b_i}(x))d\mu=\mu(U^a_i)+\mu(U^b_i)\rightarrow0$$
 as $i\rightarrow+\infty$, which is impossible.\smallskip

Finally, we get
\begin{equation}\label{liminf i(F,am,bm,z)geq 0}
    \liminf\limits_{m\rightarrow+\infty}i(\widetilde{F};\widetilde{a}_m,\widetilde{b}_m,z')\geq0
\end{equation}
for $\mu$-almost every point
$z'\in\mathrm{Rec}^+(F)\setminus\{a,b\}$.\smallskip

By the continuity of $I_Y'$ and the hypothesis on $\mu$, there
exists an open free disk $U$ containing $z$ such that $\mu(U)>0$ and
\begin{equation}\label{L1 >0}
    \lim\limits_{m\rightarrow+\infty}L_1(\widetilde{F};\widetilde{a}_m,\widetilde{b}_m,z')>0
\end{equation}
when $z'\in U\cap\mathrm{Rec}^+(F)$.

 As the rotation vector of $\mu$
vanishes, by Proposition \ref{clm:L is well defined}, the
inequalities \ref{liminf i(F,am,bm,z)geq 0}, \ref{L1 >0} and Fatou
Lemma,
%and Lebesgue's dominated convergence Theorem,
 we have
\begin{eqnarray*}
% \nonumber to remove numbering (before each equation)
I_\mu(\widetilde{F};a,b)&=&\lim_{m\rightarrow+\infty}i_\mu(\widetilde{F};\widetilde{a}_m,\widetilde{b}_m)\\
&=&\lim_{m\rightarrow+\infty}\int_{M\setminus\{a,b\}}i(\widetilde{F};\widetilde{a}_m,\widetilde{b}_m,z)\,\mathrm{d}\mu\\&\geq&
\int_{M\setminus\{a,b\}}\liminf_{m\rightarrow+\infty}i(\widetilde{F};\widetilde{a}_m,\widetilde{b}_m,z)\,\mathrm{d}\mu\\&\geq&
\int_{\bigcup_{k\geq0}F^k(U)}\liminf_{m\rightarrow+\infty}i(\widetilde{F};\widetilde{a}_m,\widetilde{b}_m,z)\,\mathrm{d}\mu \\
  &=& \int_{U}\liminf_{m\rightarrow+\infty}\tau(z)i(\widetilde{F};\widetilde{a}_m,\widetilde{b}_m,z)\,\mathrm{d}\mu \\
   &=&\int_{U}\liminf_{m\rightarrow+\infty}\tau^{*}(z)i(\widetilde{F};\widetilde{a}_m,\widetilde{b}_m,z)\,\mathrm{d}\mu \\
   &=&\int_{U}\liminf_{m\rightarrow+\infty}L^{*}(\widetilde{F};\widetilde{a}_m,\widetilde{b}_m,z)\,\mathrm{d}\mu  \\
   &=&\int_{U}\liminf_{m\rightarrow+\infty}L_1(\widetilde{F};\widetilde{a}_m,\widetilde{b}_m,z)\,\mathrm{d}\mu \\
   &>& 0.
\end{eqnarray*}
\end{proof}

As a immediate consequence of Proposition \ref{prop:F is not
constant if the contractible fixed points is finite}, we have the
following corollary.\smallskip

\begin{cor}\label{prop:F is not constant if F is not Id}Suppose that $F$
is the time-one map of an identity isotopy $I$ on a closed oriented
surface $M$ with $g\geq 1$, $\mu\in\mathcal {M}(F)$ has total
support, no atoms on $\mathrm{Fix}_{\mathrm{Cont},I}(F)$ and
$\rho_{M,I}(\mu)=0$. If %$I$ satisfies the WB-property
$F\in \mathrm{Diff}_*(M)\setminus\{\,\mathrm{Id}_{M}\}$, then
$\sharp\sigma(\widetilde{F})\geq2$.
\end{cor}
\bigskip

From Proposition \ref{clm:L is well defined} and Proposition
\ref{prop:F is not constant if the contractible fixed points is
finite}, we can get the following generalization of Theorem 2.1.\,C
in \cite{P} on closed surface.\smallskip

\newenvironment{cor04}{\noindent\textbf{Corollary \ref{cor:the symplectic action when M with genus bigger
1}}~\itshape}{\par}
\begin{cor04}\label{cor:g is great than 1}Let $F$ be the time-one map of an identity isotopy $I$ on a closed oriented
surface $M$ with $g>1$. If $I$ satisfies the WB-property and $F\in
\mathrm{Homeo}_*(M)\setminus\{\,\mathrm{Id}_{M}\}$, $\mu\in\mathcal
{M}(F)$ has total support and no atoms on
$\mathrm{Fix}_{\mathrm{Cont},I}(F)$,
 then $\sharp\sigma(\widetilde{F})\geq2$.
 %there exist two distinct fixed points $\widetilde{a}$ and
%$\widetilde{b}$ of $\widetilde{F}$ such that
%$i_\mu(\widetilde{F};\widetilde{a},\widetilde{b})\neq0$.
\end{cor04}
\begin{proof}
If $\rho_{M,I}(\mu)=0$, by Proposition \ref{prop:F is not constant
if the contractible fixed points is finite} and Proposition
\ref{prop:F is not constant if F is not Id}, there exist two
distinct contractible fixed points $a$ and $b$ of $F$ such that
$I_{\mu}(\widetilde{F};a,b)\neq0$, thus for any their lifts
$\widetilde{a}$ and $\widetilde{b}$ we have
$i_{\mu}(\widetilde{F};\widetilde{a},\widetilde{b})=I_{\mu}(\widetilde{F};a,b)\neq
0$.\smallskip

If $\rho_{M,I}(\mu)\neq0$, by the proof of Proposition \ref{clm:L is
well defined}, there exists $\alpha\in G^{\,*}$ such that
$\varphi(\alpha)\wedge\rho_{M,I}(\mu) \neq0$ where $\varphi$ is the
Hurewitz homomorphism from $G$ to $H_1(M,\mathbb{Z})$. By
Lefschetz-Nielsen's formula, we know that
$\mathrm{Fix}_{\mathrm{Cont},I}(F)\neq \emptyset$. Choose $a\in
\mathrm{Fix}_{\mathrm{Cont},I}(F)$, and an isotopy $I'$ homotopic to
$I$ that fixes $a$. For any lifts $\widetilde{a}$ and
$\alpha(\widetilde{a})$ of $a$, we get that
$i_\mu(\widetilde{F};\widetilde{a},\alpha(\widetilde{a}))=\varphi(\alpha)\wedge\rho_{M,I}(\mu)
\neq0$. We have completed the proof.
\end{proof}\smallskip

Let us now give two examples to see what will happen when
$\mathrm{Supp}(\mu)\neq M$.

\begin{exem}\label{exem:T2 and supp(u)notM}
Consider the following smooth identity isotopy on $\mathbb{R}^2$:
$\widetilde{I}=(\widetilde{F}_t)_{t\in[0,1]}:
(x,y)\mapsto(x+\frac{t}{2\pi}\cos(2\pi y),y+\frac{t}{2\pi}\sin(2\pi
y))$. It induces an identity smooth isotopy $I=(F_t)_{t\in[0,1]}$ on
$\mathbb{T}^2$. Let $\mu$ have constant density on
$\{(x,y)\in\mathbb{T}^2\mid y=0\,\,\mathrm{or}\,\, y=\frac{1}{2}\}$
and vanish on elsewhere. Obviously, $\rho_{\mathbb{T}^2,I}(\mu)=0$
but $\mathrm{Fix}_{\mathrm{Cont},I}(F_1)=\emptyset$.
\end{exem}

The example \ref{exem:T2 and supp(u)notM} tells us that there is no
sense to talk about the action function when $g=1$ and
$\mathrm{Supp}(\mu)\neq M$. The following example belongs to Le
Calvez \cite[page 73]{P1} who mentioned me that this example implies
that Proposition \ref{prop:F is not constant if the contractible
fixed points is finite} is not true anymore in the case where $g>1$
and $\mathrm{Supp}(\mu)\neq M$. For convenience of readers, we
provide the example.

\begin{exem}\label{exem:M and supp(u)notM}
Let $M$ be the closed orientated  surface with $g=2$ and $L$ be the
morse function on $M$ which has six critical points $z_1,\ldots,z_6$
such that the points $z_i$ ($2\leq i\leq 5$) are saddle and the six
critical values $L(z_i)=a_i$ ($1\leq i\leq 6$) satisfy
$a_1<a_2<\cdots<a_6$ (see Figure \ref{fig:counterexample-measure}).
Fix $b_2\in]a_2,a_3[$ and $b_4\in]a_4,a_5[$. Denote by $C_1$,$C_2$
the two connected components of $L^{-1}(\{b_2\})$ and by $C_3$,$C_4$
the two connected components of $L^{-1}(\{b_4\})$. Fix
$b_2'\in]a_2,b_2[$ and $b_4'\in]b_4,a_5[$ and modify the Hamiltonian
vector field with regard to $L$ on the closed annulus
$L^{-1}([a_1,b_2'])$ and $L^{-1}([b_4',a_6])$ to construct two
components of Reeb and obtain a vector field $\xi$ on $M$ such that
the critical points $z_3$ and $z_4$ of $L$ are the only two singular
points of $\xi$.

\begin{figure}[ht]
\begin{center}\scalebox{1.0}[1.0]{\includegraphics{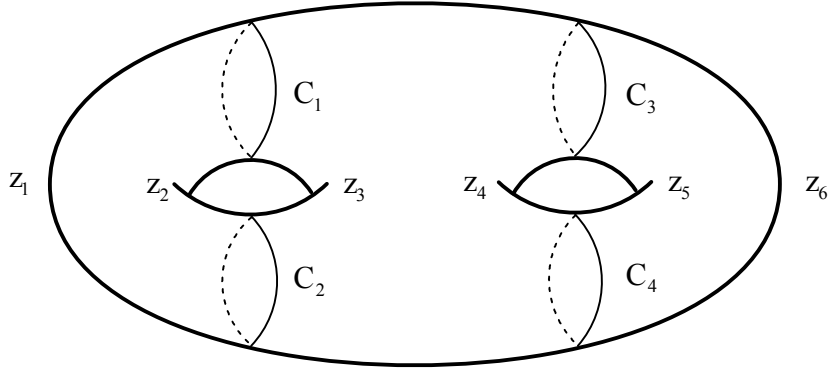}}\end{center}
\caption{The counterexample when $\mathrm{Supp}(\mu)\neq M$ and
$g=2$\label{fig:counterexample-measure}}
\end{figure}

The vector field $\xi$ induces a natural orientation on the circles
$C_i$. For two distinct points $z$ and $z'$ of $C_i$, denote by
$[z,z']_i$ the segment which joins $z$ and $z'$ on $C_i$ with the
orientation. Choose the points as follows (see Figure
\ref{fig:counterexample-measure-2})\begin{itemize}
\item $z_{1,3}$,$z_{3,1}$,$z_{1,4}$,$z_{4,1}$
on $C_1$ whose cyclic order is induced by the orientation of $C_1$;
\item $z_{2,4}$,$z_{4,2}$ on $C_2$;
\item $z_{1,3}'$,$z_{3,1}'$ on $C_3$;
\item $z_{1,4}'$,$z_{4,1}'$,$z_{4,2}'$,$z_{2,4}'$
on $C_4$ whose cyclic order is induced by the orientation of $C_4$.
\end{itemize}
We can construct the following disjoint segments in
$M\setminus\{z_3,z_4\}$
\begin{itemize}
  \item an oriented segment $\gamma_{1,3}$ form $z_{1,3}$ to
  $z_{1,3}'$;
  \item an oriented segment $\gamma_{3,1}$ form $z_{3,1}'$ to
  $z_{3,1}$;
  \item an oriented segment $\gamma_{1,4}$ form $z_{1,4}$ to
  $z_{1,4}'$;
  \item an oriented segment $\gamma_{4,1}$ form $z_{4,1}'$ to
  $z_{4,1}$;
  \item an oriented segment $\gamma_{2,4}$ form $z_{2,4}$ to
  $z_{2,4}'$;
  \item an oriented segment $\gamma_{4,2}$ form $z_{4,2}'$ to
  $z_{4,2}$.
\end{itemize}

\begin{figure}[ht]
\begin{center}\scalebox{1.0}[1.0]{\includegraphics{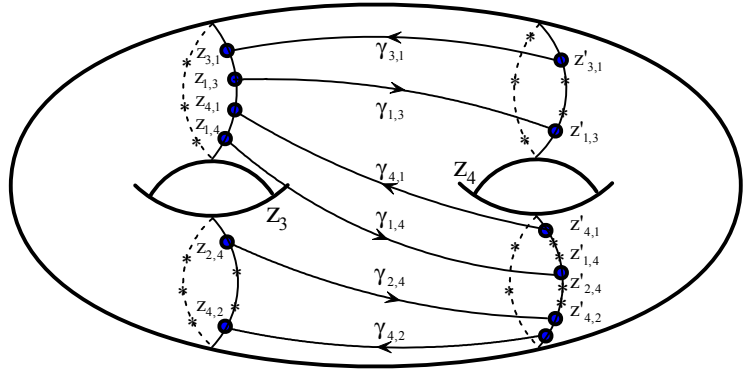}}\end{center}
\caption{The counterexample when $\mathrm{Supp}(\mu)\neq M$ and
$g=2$\label{fig:counterexample-measure-2}}
\end{figure}

The following three closed curves
$$\gamma_{1,3}\cup\gamma_{3,1}\cup[z_{1,3},z_{3,1}]_1\cup[z_{1,3}',z_{3,1}']_3,$$
$$\gamma_{1,4}\cup\gamma_{4,1}\cup[z_{1,4},z_{4,1}]_1\cup[z_{1,4}',z_{4,1}']_4,$$
$$\gamma_{2,4}\cup\gamma_{4,2}\cup[z_{4,2},z_{2,4}]_2\cup[z_{4,2}',z_{2,4}']_4$$
bound three disjoint closed disks in $L^{-1}([b_2,b_4])$. Up to
multiplying the vector field by a strict positive function, denote
by $(\theta_t)_{t\in \mathbb{R}}$ the flow with regard to $\xi$, we
can suppose that
\begin{itemize}
  \item the segments $\theta_t(\gamma_{i,j}),t\in[-1,1]$ are
  pairwise disjoint for all $(i,j)$;
  \item the segments $\theta_t(\gamma_{i,j})$ and
  $\theta_{t'}(\gamma_{i',j'})$ are
  pairwise disjoint for all $(t,t')\in[-1,1]^2$ if $(i,j)\neq(i',j')$;
  \item $z_{3,1}'=\theta_3(z_{1,3}')$,
  $z_{1,3}'=\theta_4(z_{3,1}')$;
  \item $z_{1,4}=\theta_4(z_{3,1})$,
  $z_{1,3}=\theta_4(z_{4,1})$;
  \item $z_{4,1}'=\theta_3(z_{1,4}')$,
  $z_{4,2}'=\theta_3(z_{4,1}')$, $z_{1,4}'=\theta_3(z_{2,4}')$;
  \item $z_{2,4}=\theta_3(z_{4,2})$,
  $z_{4,2}=\theta_4(z_{2,4})$.
\end{itemize}

We now fix neighborhoods $U_{i,j}$ of $\gamma_{i,j}$ such that the
six neighborhoods are pairwise disjoint and $$U_{i,j}\cap
L^{-1}([b_2,b_4])\subset \bigcup_{-1\leq
t\leq1}\theta_t(\gamma_{i,j}).$$

We consider an isotopy $(G_t)_{t\in[0,1]}$ whose support is in the
union of the six neighborhoods $U_{i,j}$ such that
\begin{itemize}
  \item the arc $t\mapsto G_t(z_{i,j})$ is the segment
  $\gamma_{i,j}$, if $i\in\{1,2\}$;
  \item the arc $t\mapsto G_t(z_{i,j}')$ is the segment
  $\gamma_{i,j}$, if $i\in\{3,4\}$.
\end{itemize}

After that, we define an isotopy $I=(F_t)_{t\in[0,1]}$ as follows
\begin{equation*}F_t=
\begin{cases}\theta_{4t}& \textrm{if} \quad t\in [0,1/2];
\\G_{2t-1}\circ\theta_{2}& \textrm{if} \quad
t\in [1/2,1].\end{cases}
\end{equation*}

By construction above, we get the point $z_{1,3}'$ is a periodic
point of $F_1$ with periodic 20 and the arc $\prod\limits_{0\leq
i\leq 19}I(F_1^i(z_{1,3}'))$ is homologic to the sum of circles
$C_i$ in $M\setminus\{z_3,z_4\}$ and hence homologic to zero. Let
the measure
$\mu=\frac{1}{20}\sum\limits_{i=0}^{19}\delta_{F_1^i(z_{1,3}')}$,
where $\delta_z$ is the Dirac measure. The points $z_3$ and $z_4$
are the only two contractible fixed points of $F_1$. Obviously,
$\rho_{M,I}(\mu)=0$ and hence the action function is constant.
\end{exem}

\smallskip

\section{Application to the group of conservative
diffeomorphisms}\label{sec:apply to the group} Fix a Borel finite
measure $\mu$ on $M$, and denote by $\mathrm{Homeo}_*(M,\mu)$ the
subgroup of $\mathrm{Homeo}_*(M)$ whose element preserves the
measure $\mu$. %Obviously, the set $\mathrm{Homeo}_*(M,\mu)$ form a
%group whose product operation is defined as in the formula \ref{the
%product operate of the symplectic group of isotopy}.
Denote by $\mathrm{Hameo}(M,\mu)$ the subset of
$\mathrm{Homeo}_*(M,\mu)$ whose elements satisfy furthermore that
$\rho_{M,I}(\mu)=0$. When $M$ is a compact surface, Franks (see
\cite{F1} for the details) has proved that $\mathrm{Hameo}(M,\mu)$
forms a group. In this section, we suppose that the support of $\mu$
is $M$.

Denote by $\mathrm{Ham}^1(M,\mu)$ the group
$\mathrm{Hameo}(M,\mu)\cap \mathrm{Diff}^1(M)$ and by
$\mathrm{Diff}^1_*(M,\mu)$ the group $\mathrm{Homeo}_*(M,\mu)\cap
\mathrm{Diff}^1(M)$. For convenience, we write $M_g$ the oriented
closed surface with the genus $g\geq1$.

%Let us now recall the definition of distortion (see \cite{P}). If
%$\mathscr{G}$ is a finitely generated group with generators
%$\{g_1,\ldots,g_s\}$, then
%$f\in\mathscr{G}$ is said to be a \emph{distortion element} %(resp.
%%\emph{U-element})
%of $\mathscr{G}$ provided that $f$ has infinite
%order and
%$$\liminf_{n\rightarrow+\infty}\frac{\|f^n\|_{\mathscr{G}}}{n}=0,$$
%%(\mathrm{resp.
%%\quad\liminf_{n\rightarrow+\infty}\frac{\log\|f^n\|_{\mathscr{G}}}{\log
%%n}=0})
%where $\|f^n\|_{\mathscr{G}}$ is the word length of $f^n$ in the
%generators $\{g_1,\ldots,g_s\}$. If $\mathscr{G}$ is not finitely
%generated, then we say that $f\in\mathscr{G}$ is distorted in
%$\mathscr{G}$ if it is distorted in some finitely generated subgroup
%of $\mathscr{G}$. %Obviously, a U-element of $\mathscr{G}$ is a
%%distortion element of $\mathscr{G}$.
%\smallskip
%
%Given two positive sequences $\{a_n\}$ and $\{b_n\}$, we write
%$a_n\succeq b_n$ if there is $c>0$ such that $a_n\geq cb_n$ for all
%$n\in\mathbb{N}$, and $a_n\sim b_n$ if $a_n\succeq b_n$ and
%$a_n\preceq b_n$.\smallskip

In the first part of this section, we will discuss the torsion in
the group $\mathrm{Hameo}(\mathbb{T}^2,\mu)$ and
$\mathrm{Homeo}_*(M_g,\mu)$ with $g>1$. Moreover, by using a result
of Fathi \cite{Fathi}, we can get an indirect proof about periodic
homeomorphisms of surfaces, that is, the group
$\mathrm{Homeo}_*(M_g)$ with $g>1$ is torsion free. In the second
part, we will study the distortion in the group
$\mathrm{Ham}^1(\mathbb{T}^2,\mu)$ and the group
$\mathrm{Diff}^1_*(M_g,\mu)$ with $g>1$. The second part links to
Zimmer conjecture on closed oriented surfaces.\bigskip

\subsection{The absence of torsion in  $\mathrm{Hameo}(\mathbb{T}^2,\mu)$
and $\mathrm{Homeo}_*(M_g,\mu)$ with $g>1$}\label{subsec:application
to conservative diffeomorphism} From Corollary
\ref{cor:Imu(Fq,a,b)=qImu(F,a,b)} and Proposition \ref{prop:F is not
constant if the contractible fixed points is finite}, we have the
following proposition, which is a generalization of Proposition
2.6.\,A in \cite{P}.\smallskip

\newenvironment{prop04}{\noindent\textbf{Proposition \ref{prop:ggeq1 no torsion}}~\itshape}{\par}

\begin{prop04}Under the same hypotheses
as Proposition \ref{prop:F is not constant if the contractible fixed
points is finite}, there exists a constant $C>0$ such that
$\mathrm{width}(\widetilde{F}^n)\geq C\cdot n$ for every $n\geq1$.
\end{prop04}\smallskip

The proposition \ref{prop:ggeq1 no torsion} implies that the groups
$\mathrm{Hameo}(M_g,\mu)$ ($g\geq1$) is torsion free.
\bigskip

From Corollary \ref{cor:imu(Fq,a,b)=qimu(F,a,b)} and Corollary
\ref{cor:the symplectic action when M with genus bigger 1}, we then
get the following conclusion:\smallskip

\newenvironment{prop05}{\noindent\textbf{Proposition \ref{prop:g>1 no torsion}}~\itshape}{\par}

\begin{prop05}
Under the same hypotheses as Corollary \ref{cor:the symplectic
action when M with genus bigger 1}, there exists a constant $C>0$
such that $\mathrm{width}(\widetilde{F}^n)\geq C\cdot n$ for every
$n\geq1$.
\end{prop05}\smallskip

Denote by $\mathscr{I}\mathscr{S}_*(M)$ the group of all identity
isotopies $I=(F_t)_{t\in[0,1]}$ on $M$, where the composition is
given by Equation \ref{the product operate of the symplectic group
of isotopy}. Denote by $\mathscr{I}\mathscr{S}_*(M,\mu)$ the
subgroup of $\mathscr{I}\mathscr{S}_*(M)$ whose element
$(F_t)_{t\in[0,1]}\in\mathscr{I}\mathscr{S}_*(M,\mu)$ satisfies
$(F_t)_*\mu=\mu$ for all $t$. We say that two identity isotppies
$(F_t)_{t\in[0,1]}$
 and $(G_t)_{t\in[0,1]}$ are homotopic with fixed extremities if $F_1=G_1$
  and there exists a continuous map $[0,1]^2\rightarrow\mathrm{Homeo}(M)$, $(t,s)\mapsto
H_{t,s}$ such that $H_{0,s}=\mathrm{Id}_{M}$, $H_{1,s}=F_1=G_1$,
$H_{t,0}=F_t$ and $H_{t,1}=G_t$. The homotopic relation is an
equivalence relation on $\mathscr{I}\mathscr{S}_*(M)$ (resp.
$\mathscr{I}\mathscr{S}_*(M,\mu)$). We note the set of equivalence
classes by $\mathscr{H}_*(M)$ (resp. $\mathscr{H}_*(M,\mu)$). It is
not difficult to see that $\mathscr{H}_*(M)$ and
$\mathscr{H}_*(M,\mu)$ are groups. Indeed, $\mathscr{H}_*(M)$ is the
universal covering space of $\mathrm{Homeo}_*(M)$ (see
\cite{Fathi}).

We can divide the group $\mathscr{H}_*(M,\mu)$ into two sets by
whether an element $I\in\mathscr{H}_*(M,\mu)$ satisfying the
WB-property. In the subset of $\mathscr{H}_*(M,\mu)$ whose elements
satisfy the WB-property, denoted by $\mathscr{W}$, we can continue
to divide this set into two sets by whether an element
$I\in\mathscr{W}$ satisfying $\rho_{M,I}(\mu)=0$.

Given $I\in\mathscr{H}_*(M,\mu)$. If $I\notin\mathscr{W}$, by
Definition \ref{def:wb property and b property}, there exist two
fixed points $\widetilde{a}$ and $\widetilde{b}$ of $\widetilde{F}$
such that $i(\widetilde{F};\widetilde{a},\widetilde{b})\neq 0$. By
Equation \ref{eq:linking number for two fixed points}, we have
$i(\widetilde{F}^n;\widetilde{a},\widetilde{b})=n\cdot
i(\widetilde{F};\widetilde{a},\widetilde{b})\neq 0$ for every
$n\in\mathbb{N}$, that is,
$|i(\widetilde{F}^n;\widetilde{a},\widetilde{b})|\succeq n$. If
$\rho_{M,I}(\mu)\neq0$, by the morphism property of
$\rho_{M,\cdot}(\mu): \mathscr{H}_*(M,\mu)\rightarrow
H_1(M,\mathbb{R})$:
$\rho_{M,II'}(\mu)=\rho_{M,I}(\mu)+\rho_{M,I'}(\mu)$, we have
$\|\rho_{M,I^n}(\mu)\|_{H_1(M,\mathbb{R})}\succeq n$. From
Proposition \ref{prop:ggeq1 no torsion} and Proposition
\ref{prop:g>1 no torsion}, we get

\begin{cor}\label{cor:no torsion cor}
The groups $\mathscr{H}_*(M_g,\mu)$ ($g\geq1$) are torsion
free.\end{cor}

In \cite{Fathi}, Fathi proved the following result: if $M^n$ is a
compact connected manifold without boundary of dimension $n\geq1$,
and $\mu$ is a finite measure on $M^n$ without atoms and with total
support, then the inclusion
$$\mathrm{Homeo}(M^n,\mu)\hookrightarrow\mathrm{Homeo}(M^n)$$is a weak homotopy
equivalence, that is, it induces isomorphisms on all homotopy
groups.

We suppose now that the measure in Corollary \ref{cor:no torsion
cor} has no atoms on $M_g$. Observing that
$\pi_1(\mathrm{Homeo}_*(\mathbb{T}^2))\simeq\mathbb{Z}^2$
(see\cite{H1}) and $\pi_1(\mathrm{Homeo}_*(M_g))\simeq0$ ($g>1$)
(see \cite{H2}), we have\smallskip

\newenvironment{cor01}{\noindent\textbf{Corollary \ref{cor:no
torsion}}~\itshape}{\par}\smallskip

\begin{cor01}
The groups $\mathrm{Hameo}(\mathbb{T}^2,\mu)$,
$\mathrm{Homeo}_*(M_g)$ ($g>1$) are torsion free.
\end{cor01}\smallskip

%Consider the Baumslag-Solitar group $BS(q,p)=\langle a,b\mid
%a^q=ba^pb^{-1}\rangle$, where $p,q\in \mathbb{Z}\setminus\{0\}$ and
%$|p|<|q|$. Using a similar argument of Proposition 1.6 H in
%\cite{P}, we get the following result:
%
%\begin{prop}Assume that $p$ divides $q$. Then $\phi$ (resp. $\phi'$)is trivial for
%every homomorphism $\phi: BS(q,p)\rightarrow \mathrm{Ham}(M_g,\mu)$
%with $g\geq1$ (resp. $\phi': BS(q,p)\rightarrow
%\mathrm{Diff}_*(M_g,\mu)$ with $g>1$).
%\end{prop}\bigskip

\subsection{The absence of distortion in $\mathrm{Ham}^1(\mathbb{T}^2,\mu)$
 and $\mathrm{Diff}^1_*(M_g,\mu)$ with $g>1$.}\label{sec:distortion of group} In
2002, Polterovich \cite{P} showed us a Hamiltonian version of the
Zimmer program (see \cite{Zim,P2} for the detail) dealing with
actions of lattices. It is achieved by using the classical action
defined in symplectic geometry, the symplectic filling function (see
Section 1.2 in \cite{P}), and a result of Schwarz \cite{SM} about
the action function being non-constant which has been proved by
using Floer homology. In 2003, Franks and Handel \cite{F2} developed
the Thurston theory of normal forms for surface homeomorphisms with
finite fixed sets. In 2006, they \cite{F4} used the generalized
normal form to give a more general version (the map is a
$C^1$-diffeomorphism and the measure is a Borel finite measure)
 of the Zimmer program on the closed oriented
surfaces. We recommend the reader a survey by Fisher \cite{Fish} for
the recent progress of Zimmer program. We will give an alternative
proof of the $C^1$-version of the Zimmer conjecture on surfaces when
the measure is a Borel finite measure with full surport.\smallskip

Suppose that $F$ is a $C^1$-diffeomorphism of $M_g$ ($g\geq 1$)
which is the time-one map of an identity isotopy
$I=(F_t)_{t\in[0,1]}$ on $M_g$ and $\widetilde{F}$ is the time-one
map of the lifted identity isotopy
$\widetilde{I}=(\widetilde{F}_t)_{t\in[0,1]}$ on the universal cover
$\widetilde{M}$ of $M_g$.

\begin{lem}\label{lem:zimmer program}If there exist two distinct fixed points
$\widetilde{a}$ and $\widetilde{b}$ of $\widetilde{F}$, and a point
$z_*\in\mathrm{Rec}^+(F)\setminus\pi(\{\widetilde{a},\widetilde{b}\})$
such that $i(\widetilde{F};\widetilde{a},\widetilde{b},z_*)$ exists
and is not zero, then $$\|F^n\|_{\mathscr{G}}\succeq \sqrt{n}.$$
\end{lem}
\begin{proof}
If
$z_*\in\mathrm{Rec}^+(F)\setminus\pi(\{\widetilde{a},\widetilde{b}\})$
and $i(\widetilde{F};\widetilde{a},\widetilde{b},z_*)$ exist, by
Lemma \ref{subsec:positively recurrent} and Proposition
\ref{prop:i(Fnabx)=ni(abx)}, we know that
$z_*\in\mathrm{Rec}^+(F^n)\setminus\pi(\{\widetilde{a},\widetilde{b}\})$
, that $i(\widetilde{F}^n;\widetilde{a},\widetilde{b},z_*)$ exists
and that
$i(\widetilde{F}^n;\widetilde{a},\widetilde{b},z_*)=ni(\widetilde{F};\widetilde{a},\widetilde{b},z_*)$
for all $n\geq1$.\smallskip

Let $N(n)=\|F^n\|_{\mathscr{G}}$. Assume that there exist identity
isotopies $I_i=(F_{i,t})_{t\in[0,1]}\subset \mathrm{Diff^1(M)}$
($1\leq i\leq s$) such that, for every $n\geq1$, we have the isotopy
$I^n$ is homotopic to the isotopy
$\prod\limits_{j=1}^{N(n)}I^{~\epsilon_j}_{i_j}=\left(F_t^{(n)}\right)_{0\leq
t\leq 1}\stackrel{\triangle}{=}I^{(n)}$, where
$i_j\in\{1,2,\cdots,s\}$, $\epsilon_j\in\{-1,1\}$
($j=1,2,\cdots,N(n)$) and
$$F_t^{(n)}(z)=F^{\epsilon_{i_k}}_{i_k,N(n)t-(k-1)}(F^{\epsilon_{i_{k-1}}}_{i_{k-1},1},\circ
\cdots\circ F^{\epsilon_{i_1}}_{i_1,1}(z)),\quad
\mathrm{if} \quad \frac{k-1}{N(n)}\leq t\leq \frac{k}{N(n)}.$$

Let $\widetilde{I}_i=(\widetilde{F}_{i,t})_{t\in[0,1]}$ ($1\leq
i\leq s$) and $\widetilde{I}^{(n)}=(\widetilde{F}_t^{(n)})_{0\leq
t\leq 1}$ be the lifts of $I_i$ ($1\leq i\leq s$) and
$(F_t^{(n)})_{0\leq t\leq 1}$ to $\widetilde{M}$ respectively.
Identify the sphere $\widetilde{M}\cup\{\infty\}$ as the Riemann
sphere $\mathbb{C}\cup\{\infty\}$. Again, for simplicity,
%up to conjugacy by a M\"{o}bius transformation that maps the triple
%$(\widetilde{a},\widetilde{b},\infty)$ to the triple $(0,1,\infty)$,
we can suppose that $\widetilde{a}=0$ and
$\widetilde{b}=1$.\smallskip

Fix $n\geq1$. Using the method of Lemma \ref{rem:identity isotopies
fix three points on sphere}, we can get the isotopy
$\widetilde{I}'^{(n)}=(\widetilde{F}\,'^{(n)}_t)_{0\leq t\leq 1}$
which fixes $0$, $1$ and is an isotopy on $\widetilde{M}$ from
$\mathrm{Id}_{\widetilde{M}}$ to $\widetilde{F}^n$, where

\begin{equation}\label{eq:new isotopy}
\widetilde{F}\,'^{(n)}_t(\widetilde{z})=\frac{\widetilde{F}_t^{(n)}(\widetilde{z})
-\widetilde{F}_t^{(n)}(0)}{\widetilde{F}_t^{(n)}(1)-\widetilde{F}_t^{(n)}(0)}\,.
\end{equation}

Let $\widetilde{\gamma}=\{0\leq r\leq 1\}$ be the straight line from
$0$ to $1$. If
$\widetilde{I}'^{(n)}(\widetilde{z})\cap\widetilde{\gamma}\neq\emptyset$
for some point $\widetilde{z}\in \widetilde{M}\setminus\{0,1\}$,
then there exist $t_0\in[0,1]$ and $r_0\in]0,1[$ such that
$\widetilde{F}\,'^{(n)}_{t_0}(\widetilde{z})=r_0$, that is
\begin{equation}\label{eq:zimmer}
    \widetilde{F}_{t_0}^{(n)}(\widetilde{z}) -\widetilde{F}_{t_0}^{(n)}(0)=r_0
(\widetilde{F}_{t_0}^{(n)}(1) -\widetilde{F}_{t_0}^{(n)}(0)).
\end{equation}

Let
$$C=\max_{i\in\{1,\cdots,s\}}\sup_{t\in[0,1]\,,\,\widetilde{z}\in\widetilde{M}}
\widetilde{d}(\widetilde{F}_{i,t}(\widetilde{z}),\widetilde{z}).$$
We have
\begin{equation}\label{ineq:Ak}
\left|\widetilde{F}_t^{(n)}(1)-\widetilde{F}_t^{(n)}(0)\right|\leq
2CN(n)+1
\end{equation}
and
$$\left|\widetilde{F}_t^{(n)}(\widetilde{z})-\widetilde{F}_t^{(n)}(0)\right|\geq
|\widetilde{z}|-2CN(n)$$ for all $t\in[0,1]$. Hence when
$|\widetilde{z}|\geq 5CN(n)$, we get
$\left|\widetilde{F}\,'^{(n)}_t(\widetilde{z})\right|>1$, i.e.,
$\widetilde{I}'^{(n)}(\widetilde{z})\cap\widetilde{\gamma}=\emptyset$.
Recall that the open disks $\widetilde{V}$ and $\widetilde{W}$ that
contain $\infty$ in  Section \ref{sec:boundedness}. Here, we set
$\widetilde{V}=\{\widetilde{z}\in\widetilde{M}\mid\widetilde{z}|>
5CN(n)\}$ and choose an open disk $\widetilde{W}$ containing
$\infty$ such that $\widetilde{\gamma}\cap \widetilde{W}=\emptyset$,
and for every $\widetilde{z}\in \widetilde{V}$, we have
$\widetilde{I}'^{(n)}(\widetilde{z})\subset \widetilde{W}$. Without
loss of generality, we can suppose that $z_*\notin
\pi(\widetilde{\gamma})$. Choose an open disk $U$ containing $z_*$
such that $U\cap \pi(\widetilde{\gamma})=\emptyset$. Write
respectively $\tau(n,z)$ and $\Phi_n(z)$ for the first return time
and the first return map of $F^n$ throughout this proof. For every
$m\geq1$, recall that
$\tau_m(n,z)=\sum\limits_{i=0}^{m-1}\tau(n,\Phi^i_n(z))$. Let us
consider the following value
$$L_m(\widetilde{F}^n;0,1,z_*)=\widetilde{\gamma}\wedge\widetilde{\Gamma}^m_{\widetilde{I}'^{(n)},z_*}.$$

By the same arguments with Lemma \ref{lem:the boundedness of the
case multi-path} and Lemma \ref{lem:alphazn is e, Ln divide taun is
less then kU}, we can find multi-paths $\widetilde{\Gamma}^m_i(z_*)$
($1\leq i \leq P_m(z_*)$) from $\widetilde{V}$ to $\widetilde{V}$
(see the equations \ref{eq:multi-path1}$-$\ref{eq:multi-path5} for
the details) such that
$$L_m(\widetilde{F}^n;0,1,z_*)=\widetilde{\gamma}\wedge\prod_{1\leq i\leq
P_m(z_*)}\widetilde{\Gamma}^m_i(z_*).$$
\smallskip

For every $j\in\{1,\cdots,s\}$ and
$(\widetilde{z},\widetilde{z}\,')\in\widetilde{M}\times\widetilde{M}\setminus\widetilde{\Delta}$,
there is a unique function $\theta_j: [0,1]\rightarrow \mathbb{R}$
such that $\theta_j(0)=0$ and
$$e^{2i\pi\theta_j(t)}=\frac{\widetilde{F}_{j,t}(\widetilde{z})-\widetilde{F}_{j,t}(\widetilde{z}\,')}
{\left|\widetilde{F}_{j,t}(\widetilde{z})-\widetilde{F}_{j,t}(\widetilde{z}\,')\right|}.$$
Let $\lambda_j(\widetilde{z},\widetilde{z}\,')=\theta_j(1)$. As
$\widetilde{I}_j\subset \mathrm{Diff}^1(\widetilde{M})$, there is a
natural compactification of
$\widetilde{M}\times\widetilde{M}\setminus\widetilde{\Delta}$
obtained by replacing the diagonal $\widetilde{\Delta}$ with the
unit tangent bundle such that the map $\lambda_j$ extends
continuously (see, for example, \cite[page 81]{CFGL}).

Let
$$C_1=\max_{i\in\{1,\cdots,s\}}\sup_{(\widetilde{z},\widetilde{z}\,')
\in\widetilde{M}\times\widetilde{M}\setminus\widetilde{\Delta}}
|\lambda_j(\widetilde{z},\widetilde{z}\,')|.$$

Suppose that $\widetilde{M}_0\subset\widetilde{M}$ is a closed
fundamental domain with regard to the transformation group $G$.
Denote by $\widetilde{I}^{\pm}_i(\widetilde{M}_0)$ the set
$\{\widetilde{F}_{i,\pm t}(\widetilde{z})\mid (\pm
t,\widetilde{z})\in [0,1]\times\widetilde{M}_0\}$ where
$\widetilde{F}_{i,-t}=\widetilde{F}_{i,1-t}\circ\widetilde{F}^{-1}_{i,1}$.
As $\widetilde{F}_{i,\pm
t}\circ\alpha=\alpha\circ\widetilde{F}_{i,\pm t}$ for all $\alpha\in
G$ and $t\in[0,1]$, we can suppose that
$$C_2=\max_{i\in\{1,\cdots,s\},\,z\in M}\sharp\{\widetilde{z}\in
\pi^{-1}(z)\mid \widetilde{I}_i(\widetilde{z})\cap
\widetilde{I}^{+}_i(\widetilde{M}_0)\cup
\widetilde{I}^{-1}_i(\widetilde{z})\cap
\widetilde{I}^{-}_i(\widetilde{M}_0)\neq\emptyset\},$$ which is
independent of $n$. \bigskip

For every $0\leq j\leq m-1$, $\tau_j(n,z_*)\leq l<\tau_{j+1}(n,z_*)$
and $1\leq k\leq N(n)$, let
$F^{\epsilon_{i_0}}_{i_{0},1}=\mathrm{Id}_{M}$,
$\widetilde{F}^{~\epsilon_{i_{0}}}_{i_{0},1}=\mathrm{Id}_{\widetilde{M}}$,

$$z_{j,l,k}=F^{~\epsilon_{i_{k-1}}}_{i_{k-1},1}(F^{~\epsilon_{i_{k-2}}}_{i_{k-2},1}\circ\cdots\circ
F^{~\epsilon_{i_{1}}}_{i_1,1}(F^{n(l-\tau_j(n,z_*))}(\Phi_n^j(z_*))))$$
and
$$\widetilde{z}^{\,0}_{k}=\widetilde{F}^{~\epsilon_{i_{k-1}}}_{i_{k-1},1}(\widetilde{F}^{~\epsilon_{i_{k-2}}}_{i_{k-2},1}\circ\cdots\circ \widetilde{F}^{~\epsilon_{i_{1}}}_{i_1,1}(0))
,\quad\widetilde{z}^1_{k}=\widetilde{F}^{~\epsilon_{i_{k-1}}}_{i_{k-1},1}(\widetilde{F}^{~\epsilon_{i_{k-2}}}_{i_{k-2},1}\circ\cdots\circ
\widetilde{F}^{~\epsilon_{i_{1}}}_{i_1,1}(1)).$$

When $\frac{k-1}{N(n)}\leq t\leq \frac{k}{N(n)}$, recall that
\begin{equation*}\widetilde{F}\,'^{(n)}_t(\widetilde{z})=\frac{\widetilde{F}^
{~\epsilon_{i_{k}}}_{i_k,N(n)t-(k-1)}(\widetilde{z})-\widetilde{F}^{~\epsilon_
{i_{k}}}_{i_k,N(n)t-(k-1)}(\widetilde{z}^{\,0}_{k})}
{\widetilde{F}^{~\epsilon_{i_{k}}}_{i_k,N(n)t-(k-1)}(\widetilde{z}^{\,1}_{k})-
\widetilde{F}^{~\epsilon_{i_{k}}}_{i_k,N(n)t-(k-1)}(\widetilde{z}^{\,0}_{k})}.\end{equation*}

For every $\widetilde{z}\in\widetilde{M}$, denote by
$\widetilde{J}_k(\widetilde{z})$ by the curve
$$\widetilde{J}_k(\widetilde{z})=\left(\widetilde{F}\,'^{(n)}_t(\widetilde{z})\right)_{\frac{k-1}{N(n)}\leq
t\leq \frac{k}{N(n)}}.$$

For every $k$, define the immersed square
\begin{eqnarray*}
% \nonumber to remove numbering (before each equation)
 A_k: [0,1]^2 &\rightarrow& \widetilde{M}\nonumber\\
  (t,r)&\mapsto& \widetilde{F}^{~\epsilon_{i_{k}}}_{i_k,t}(~\widetilde{z}_k^{\,0})+r(\widetilde{F}^{~\epsilon_{i_{k}}}_{i_k,t}(\widetilde{z}_k^1)-\widetilde{F}^{~\epsilon_{i_{k}}}_{i_k,t}
(\widetilde{z}_k^{\,0})).
\end{eqnarray*}
%Obviously,
%$\widetilde{A}=\bigcup\limits_{k=1}^{N(n)}A_k$ is an
%immersed annulus of $\widetilde{M}$.

 By Equality \ref{eq:zimmer}, we know that
$\widetilde{J}_k(\widetilde{z})\cap\widetilde{\gamma}\neq\emptyset$
implies $\widetilde{I}_{i_k}^{~\epsilon_{i_{k}}}(\widetilde{z})\cap
A_k\neq\emptyset$ (see Figure \ref{fig:Zimmer-program-1}). Remark
here that there are two universal covers in Figure
\ref{fig:Zimmer-program-1}, that the curves in $\widetilde{M}$ (the
big one) is generated by the isotopy $\widetilde{I}^{(n)}$,
 and that the curve in $\widetilde{M}$ (the small one) is generated by $\widetilde{J}_k$ (and hence by the
isotopy $\widetilde{I}'^{(n)}$ defined by Formular \ref{eq:new
isotopy}).

Let
$$C_{j,l,k}=\{\widetilde{z}_{j,l,k}\in\pi^{-1}(z_{j,l,k})\mid
\widetilde{I}_{i_k}^{~\epsilon_{i_{k}}}(\widetilde{z}_{j,l,k}) \cap
A_k\neq\emptyset\}.$$\smallskip

\begin{figure}[ht]
\begin{center}\scalebox{0.5}[0.5]{\includegraphics{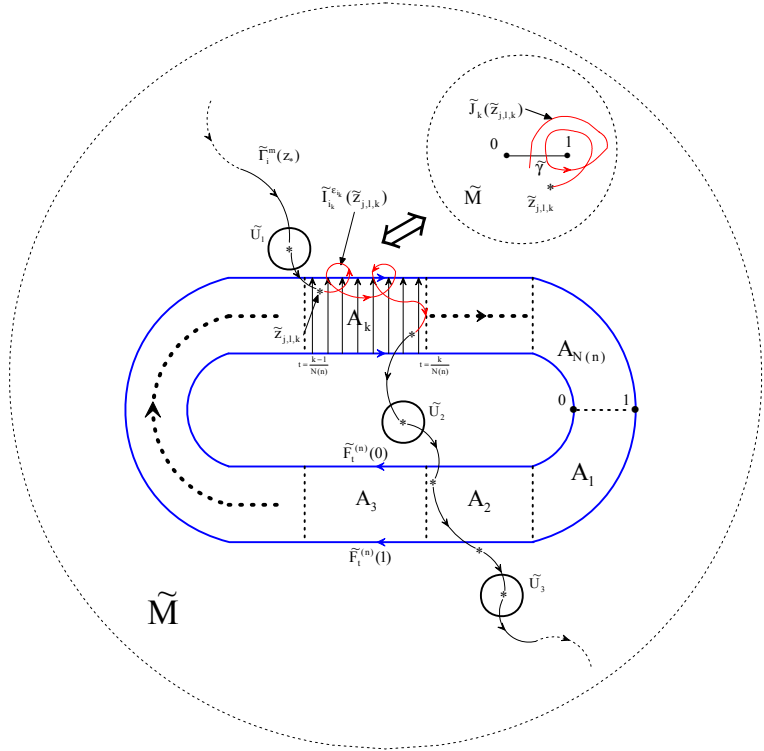}}\end{center}
\caption{The proof of Lemma \ref{lem:zimmer
program}\label{fig:Zimmer-program-1}}
\end{figure}

For every $k$, observing that
$\widetilde{d}(\widetilde{z}^{\,0}_{k},\widetilde{z}^{\,1}_{k})\leq
2C(k-1)+1$, there exists $C_3>0$ (depending only on $\widetilde{a}$
and $\widetilde{b}$) such that
\begin{equation}\label{ineq:Cjlk1}
\sharp\{\alpha\in G\mid A_k\cap
\alpha(\widetilde{M}_0)\neq\emptyset\}\leq C_3N(n).
\end{equation}
Therefore,
\begin{equation}\label{ineq:Cjlk2}
\sum_{j,l,k}\sharp C_{j,l,k}\leq C_2C_3N^2(n)\tau_m(n,z_*).
\end{equation}

%\begin{figure}[ht]
%\begin{center}\scalebox{0.7}[0.7]{\includegraphics{Zimmer-Program-2.eps}}\end{center}
%\caption{The proof of Lemma \ref{lem:zimmer
%program}\label{fig:Zimmer-Program-2}}
%\end{figure}

We have
\begin{equation}\label{ineq:Cjlk3}
L_m(\widetilde{F}^n;0,1,z_*)=\widetilde{\gamma}\wedge\prod_{1\leq
i\leq
P_m(z_*)}\widetilde{\Gamma}^m_i(z_*)=\widetilde{\gamma}\wedge\prod_{j,l,k}\left(\prod_{\widetilde{z}
\in C_{j,l,k}}\widetilde{J}_{k}(\widetilde{z})\right).\end{equation}
We get
\begin{equation}\label{eq: key eq}
\left|L_m(\widetilde{F}^n;0,1,z_*)\right|\leq
c_0N^2(n)\tau_m(n,z_*),
\end{equation}
where $c_0=C_1C_2C_3$. Therefore,
\begin{equation*}
\left|ni(\widetilde{F};0,1,z_*)\right|=\left|i(\widetilde{F}^n;0,1,z_*)\right|\leq
c_0N^2(n).
\end{equation*}
This implies that, for every $n\geq1$,
\begin{equation*}\label{ineq:subline}
0<\left|i(\widetilde{F};0,1,z_*)\right|\leq
c_0\frac{N^2(n)}{n},\end{equation*} That is
$$\|F^n\|_{\mathscr{G}}\succeq \sqrt{n}.$$
We have completed this proof.
\end{proof}\bigskip

By Proposition \ref{lem:zimmer program} and Proposition \ref{prop:F
is not constant if the contractible fixed points is finite}, we can
get the following result which is a generalization of Theorem 1.6 B
in \cite{P} on the closed surface.

\begin{cor}\label{cor:sqrt n of growth}
If $F\in\mathrm{Diff}^1_*(M_g,\mu)\setminus\{\mathrm{Id}_{M_g}\}$
($g>1$) or
$F\in\mathrm{Ham}^1(\mathbb{T}^2,\mu)\setminus\{\mathrm{Id}_{\mathbb{T}^2}\}$,
then for any finitely generated subgroup
$F\in\mathscr{G}\subset\mathrm{Diff}^1_*(M_g,\mu)$ ($g\geq 1$),
$$\|F^n\|_{\mathscr{G}}\succeq \sqrt{n}.$$
\end{cor}

Moreover, we can improve Corollary \ref{cor:sqrt n of growth}. The
following result is our main theorem in this subsection. \smallskip

\newenvironment{thm02}{\noindent\textbf{Theorem \ref{thm:Fn thicksim
n}}~\itshape}{\par}\smallskip

\begin{thm02}
Let $F\in\mathrm{Diff}^1_*(M_g,\mu)\setminus\{\mathrm{Id}_{M_g}\}$
($g>1$) (resp.
$F\in\mathrm{Ham}^1(\mathbb{T}^2,\mu)\setminus\{\mathrm{Id}_{\mathbb{T}^2}\}$),
and $\mathscr{G}\subset\mathrm{Diff}^1_*(M_g,\mu)$ ($g>1$) (resp.
$\mathscr{G}\subset\mathrm{Ham}^1(\mathbb{T}^2,\mu)$) be a finitely
generated subgroup containing $F$, then
$$\|F^n\|_{\mathscr{G}}\thicksim n.$$
As a consequence, the groups $\mathrm{Diff}^1_*(M_g,\mu)$ ($g>1$)
and $\mathrm{Ham}^1(\mathbb{T}^2,\mu)$ have no
distortions.\end{thm02}

The theorem \ref{thm:Fn thicksim n} can be obtained immediately from
the following two lemmas.

\begin{lem}\label{lem:zimmer1}
If $F\in\mathrm{Homeo}_*(M_g,\mu)\setminus\mathrm{Hameo}(M_g,\mu)$
($g>1$), then for any finitely generated subgroup
$F\in\mathscr{G}\subset\mathrm{Homeo}_*(M_g,\mu)$, we have
$$\|F^n\|_{\mathscr{G}}\thicksim n.$$
\end{lem}

\begin{lem}\label{lem:zimmer2}
If $F\in\mathrm{Ham}^1(M_g,\mu)\setminus\{\mathrm{Id}_{M_g}\}$
($g\geq1$), then for any finitely generated subgroup
$F\in\mathscr{G}\subset\mathrm{Diff}^1_*(M_g,\mu)$, we have
$$\|F^n\|_{\mathscr{G}}\thicksim n.$$
\end{lem}\smallskip

\begin{proof}[Proof of Lemma \ref{lem:zimmer1}]
By the definition of $\mathrm{Hameo}(M_g,\mu)$, we know that
$\rho_{M_g,I}(\mu)\neq0$. Assume that $F\in\mathscr{G}=\langle
F_{1,1},\cdots,F_{s,1}\rangle\subset\mathrm{Homeo}_0(M_g,\mu)$ and
$I_i$ ($1\leq i\leq s$) are the identity isotopies corresponding to
$F_{i,1}$. Denote by $\|\cdot\|_{H_1(M_g,\mathbb{R})}$ the norm in
the space $H_1(M_g,\mathbb{R})$. Write
$$\kappa=\max_{i\in\{1,\cdots,s\}}
\left\{\left\|\rho_{M_g,I_i}(\mu)\right\|_{H_1(M_g,\mathbb{R})}\right\}.$$
As $\rho_{M_g,I}(\mu)\neq0$ and $F\in\mathscr{G}$, we have
$\kappa>0$.

For every $n\in\mathbb{N}$, if $I^n$ is homotopic to
$\prod\limits_{s=1}^{N(n)}I^{\;\epsilon_{i_{s}}}_{i_s}$, then we
have
$$n\cdot\|\rho_{M_g,I}(\mu)\|_{H_1(M_g,\mathbb{R})}=
\|\rho_{M_g,I^n}(\mu)\|_{H_1(M_g,\mathbb{R})}=
\left\|\sum_{s=1}^{N(n)}\rho_{M_g,I^{\;\epsilon_{i_{s}}}_{i_s}}(\mu)\right\|_{H_1(M_g,\mathbb{R})}\leq
\kappa\cdot N(n).$$ Hence $\|F^n\|_{\mathscr{G}}\succeq n$. On the
other hand, we have $\|F^n\|_{\mathscr{G}}\preceq n$, which
completes the proof.
\end{proof}\smallskip

\begin{proof}[Proof of Lemma \ref{lem:zimmer2}]
For simplicity, we write $M_g$ as $M$. It is sufficient to prove
that $\|F^n\|_{\mathscr{G}}\succeq n$. We use the notations in the
proofs of Proposition \ref{prop:F is not constant if the
contractible fixed points is finite} and Lemma \ref{lem:zimmer
program}.

If $\sharp\mathrm{Fix}_{\mathrm{Cont},I}(F)=+\infty$, assume that
$X\subset\mathrm{Fix}_{\mathrm{Cont},I}(F)$, $I'$, $Y=\{a,b\}\subset
X$,$I_Y'$ are the notations defined in the proof for the case
$\sharp\mathrm{Fix}_{\mathrm{Cont},I}(F)=+\infty$ of Proposition
\ref{prop:F is not constant if the contractible fixed points is
finite}. If $\sharp\mathrm{Fix}_{\mathrm{Cont},I}(F)<+\infty$, for
convenience, we require $a=\alpha(\lambda)$, $b=\omega(\lambda)$,
and $I_Y'=I'$ where $\alpha(\lambda)$, $\omega(\lambda)$ and $I'$
are the notions defined in the proof for the case
$\sharp\mathrm{Fix}_{\mathrm{Cont},I}(F)<+\infty$ of Proposition
\ref{prop:F is not constant if the contractible fixed points is
finite}. Suppose that $\widetilde{I}'$, $\widetilde{I}_Y'$ are
respectively the lifts of $I'$ and $I_Y'$ to $\widetilde{M}$. Choose
a lift $\widetilde{a}$ of $a$ and a lift $\widetilde{b}$ of $b$. We
know that $I_{\mu}(\widetilde{F};a,b)\neq0$. As
$F\neq\mathrm{Id}_{M}$ and $\mu$ has total support, by the property
(B2) in the proof of Proposition \ref{prop:F is not constant if the
contractible fixed points is finite}, we can choose $z_*\in
\mathrm{Rec}^+(F)\setminus X$, such that $\rho_{M,I}(z_*)$ and
$i(\widetilde{F};\widetilde{a},\widetilde{b},z_*)$ exist, and
$i(\widetilde{F};\widetilde{a},\widetilde{b},z_*)$ is not zero.

Suppose now that $z\in M\setminus X$. By the items (3) and (5) of
Theorem \ref{thm:O.Jaulent}, we know that $I(z)$ and $I_Y'(z)$ are
homotopic in $M$. Hence, for every $n\in\mathbb{N}$,
$I^{(n)}(z)=\prod\limits_{j=1}^{N(n)}I^{~\epsilon_{i_j}}_{i_j}(z)$
is homotopic to $(I_{Y}')^{n}(z)$ in $M$. If $\gamma_{F^n(z),z}$ is
a geodesic path from $F^n(z)$ and $z$ on $M$, similarly to the proof
of Formula \ref{eq:rotation vector}, then there exists $C'>0$ such
that
\begin{equation}\label{homotopy finiteness}
\left\|[I^{(n)}(z)\gamma_{F^n(z),z}]_M\right\|_{H_1(M,\mathbb{R})}=
\left\|[(I_{Y}')^{n}(z)\gamma_{F^n(z),z}]_M\right\|_{H_1(M,\mathbb{R})}\leq
C'N(n).
\end{equation}\smallskip

Assume that $C_{j,l,k}$, $\widetilde{\gamma}$,
$\widetilde{I}'^{(n)}$, $C_1$ and $C_2$ are the notations in the
proof of Lemma~\ref{lem:zimmer program}.

Write
$$C'_{j,l,k}=\{\widetilde{z}_{j,l,k}\in\pi^{-1}(z_{j,l,k})\mid\widetilde{I}_{i_k}^{\,\epsilon_{i_k}}(\widetilde{z}_{j,l,k})
\cap(A_{k}(\{r=0\}\cup \{r=1\})\neq\emptyset\}.$$ Obviously,
$$\sharp C'_{j,l,k}\leq 2 C_2,\quad \sum_{j,l,k}\sharp
C'_{j,l,k}\leq 2C_2N(n)\tau_m(n,z_*).$$

Observing that $(\widetilde{I}_{Y}')^{n}$ and $\widetilde{I}'^{(n)}$
are two isotopies from $\mathrm{Id}_{\widetilde{M}}$ to
$\widetilde{F}^n$ which fix $\widetilde{a}$ and $\widetilde{b}$, by
Remark \ref{rem:some result of of sphere delete three points},
$(\widetilde{I}_{Y}')^{n}$ is homotopic to $\widetilde{I}'^{(n)}$ in
$\widetilde{M}\setminus\{\widetilde{a},\widetilde{b}\}$.

Observe that $N(n)$ has the following simple properties:
\begin{itemize}
  \item For every two numbers
$n_1,n_2\geq 1$, we have $N(n_1+n_2)\leq N(n_1)+N(n_2)$;
  \item For every $1\leq k< n$, we have $N(k)-N(1)\leq
N(k+1)\leq N(k)+N(1)$.
\end{itemize}

By the properties of $N(n)$ listed above %, the boundedness of
%$\widetilde{I}'_{Y}(\widetilde{z})$ for any
%$\widetilde{z}\in\widetilde{M}$ (in the homology sense),
and Inequality \ref{homotopy finiteness}, we get for every $z\in
M\setminus\{a,b\}$, there exists $C'_2>0$ such that

\begin{equation}\label{ineq:N(n)}
\sharp\{\widetilde{z}\in\pi^{-1}(z)\mid\widetilde{I}'^{(n)}(\widetilde{z})
\wedge\widetilde{\gamma}\neq0\}\leq C'_2N(n).
\end{equation}

Write
\begin{equation*}
C\,'_{j,l}=\{\widetilde{z}\in\pi^{-1}(z_{j,l,N(n)})\mid\widetilde{I}'^{(n)}(\widetilde{z})
\wedge\widetilde{\gamma}\neq0\}.
\end{equation*}

Under the hypotheses of Lemma \ref{lem:zimmer2}, we want to improve
the value $N^2(n)$ to $N(n)$ in the inequality \ref{eq: key eq}.

%We first analyse the possible composing of the value
Based on the analyses above, we have
$$L_m(\widetilde{F}^n;\widetilde{a},\widetilde{b},z_*)=\widetilde{\gamma}\wedge
\widetilde{\Gamma}^m_{\widetilde{I}'^{(n)},z_*}=
\widetilde{\gamma}\wedge\widetilde{\Gamma}^m_{\left(\widetilde{I}\,'_Y\right)^{\,n},z_*}
=\pi(\widetilde{\gamma})\wedge\Gamma^m_{\left(I\,'_Y\right)^{\,n},z_*}.$$

To estimate the value
$L_m(\widetilde{F}^n;\widetilde{a},\widetilde{b},z_*)$, we need to
consider the isotopy $\widetilde{I}'^{(n)}$ (If we use the isotopy
$\left(\widetilde{I}\,'_Y\right)^{\,n}$, the difficulty is we do not
know how the isotopy $\left(I\,'_Y\right)^{\,n}$ rotates around the
points $\pi(\widetilde{a})$ and $\pi(\widetilde{b})$). If the
immersed squares $A_k$ are uniformly bounded for $n$, then by the
inequalities \ref{ineq:Cjlk1}$-$\ref{eq: key eq}, we have done. We
explain now the case where the squares $A_k$ are not bounded for $n$
is also true. The inequality \ref{ineq:N(n)} shows that, for every
fixed $j$ and $l$, it is sufficient to consider at most $C_2'N(n)$
elements of $C'_{j,l}$. For every $\widetilde{z}\in C'_{j,l}$, we
consider the value
$V(\widetilde{z})=|\widetilde{\gamma}\wedge\widetilde{I}'^{(n)}(\widetilde{z})|$
(Maybe it is no sense since the ends of
$\widetilde{I}'^{(n)}(\widetilde{z})$ are possibly on
$\widetilde{\gamma}$. In this case, we let
$V(\widetilde{z})=|\widetilde{\gamma}\wedge\mathrm{Int}(\widetilde{I}'^{(n)}(\widetilde{z}))|+2$
). We can write $\widetilde{I}'^{(n)}(\widetilde{z})$ as the
concatenation of $N(n)$ sub-paths $\widetilde{J}_k(\widetilde{z})\,
(k=1,\cdots,N(n))$. Obviously, we have \begin{eqnarray*}
% \nonumber to remove numbering (before each equation)
  \left|L_m(\widetilde{F}^n;\widetilde{a},\widetilde{b},z_*)
  \right|&\leq&\left|\widetilde{\gamma}\wedge\prod_{j,l,k}\left(\prod_{\widetilde{z}
\in
C'_{j,l,k}}\widetilde{J}_k(\widetilde{z})\right)\right|+\left|\widetilde{\gamma}\wedge
\prod_{j,l}\left(\prod_{\widetilde{z} \in
C'_{j,l}}\widetilde{I}'^{(n)}(\widetilde{z})\right)\right|.
\end{eqnarray*}

\begin{figure}[ht]
\begin{center}\scalebox{0.5}[0.5]{\includegraphics{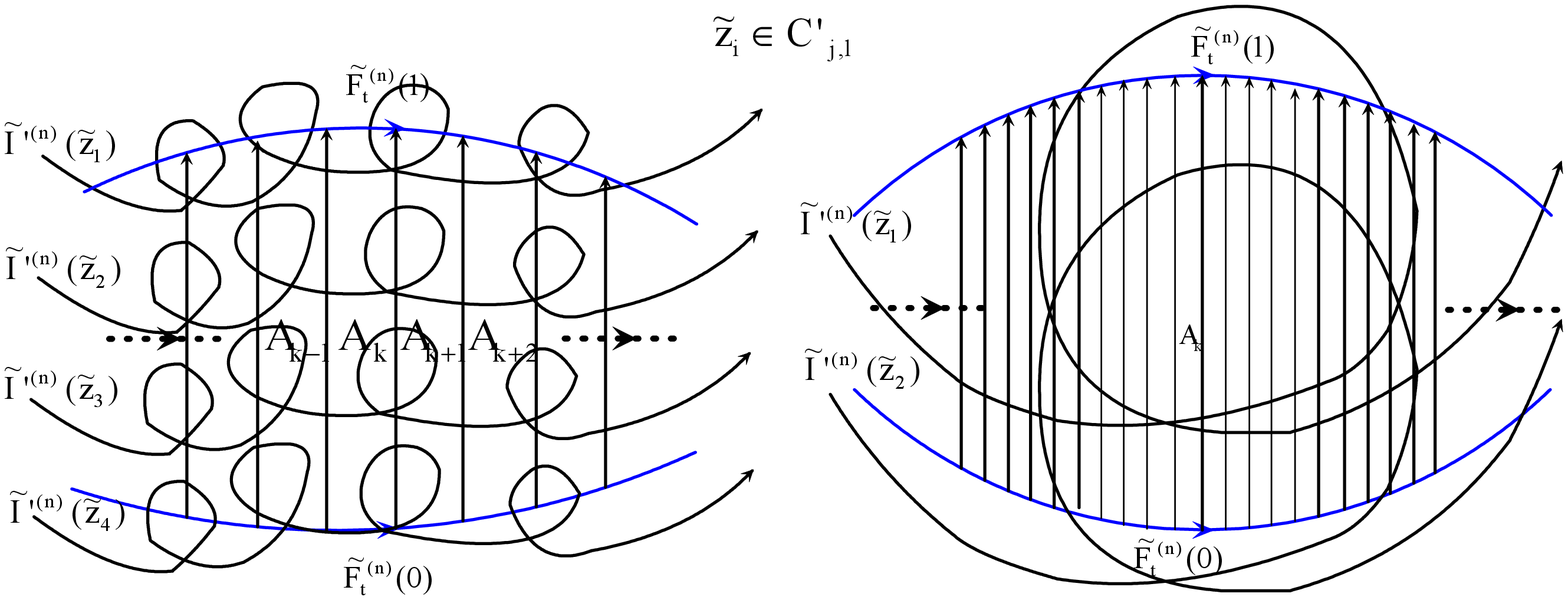}}\end{center}
\caption{The proof of Lemma
\ref{lem:zimmer2}\label{fig:Zimmer-conjecture}}
\end{figure}

We know that the value of the first part of the right hand of the
inequality above is less than $2C_1C_2N(n)\tau_m(n,z_*)$. Hence, to
explore the relation of the bound of
$\left|L_m(\widetilde{F}^n;\widetilde{a},\widetilde{b},z_*)
  \right|$ and the power of $N(n)$, we can suppose that the path
  $\widetilde{J}_k(\widetilde{z})$ never meets
  $A_k(\{r=0\}\cup\{r=1\})$ for every $k$ and $\widetilde{z}\in
  C'_{j,l}$. As the isotopies $\widetilde{I}_i$~($1\leq
i\leq s$) commutes with the transformations, we can observe that
(see Figure \ref{fig:Zimmer-conjecture} and refer to the proof of
Lemma \ref{lem:zimmer program}) if we demand the values
$V(\widetilde{z})$ for some $\widetilde{z}\in C'_{j,l}$ (in all of
the probabilities) as large as possible, then other values of
$V(\widetilde{z})$ will be as small as possible; if we demand most
of the values $V(\widetilde{z})$ increase, then the maximal of
$V(\widetilde{z})$ will decrease. It is easy to prove that
$\sum_{z\in C'_{j,l}}V(\widetilde{z})\leq C'_1 N(n)$ for some
$C'_1>0$. Hence,
\begin{equation*}
    \left|\widetilde{\gamma}\wedge
\left(\prod_{j,l}\prod_{\widetilde{z} \in
C'_{j,l}}\widetilde{I}'^{(n)}(\widetilde{z})\right)\right|=\left|\widetilde{\gamma}\wedge
\left(\prod_{j,l}\prod_{\widetilde{z} \in
C'_{j,l}}\prod_{k=1}^{N(n)}\widetilde{J}_k(\widetilde{z})\right)\right|\leq
C_1'N(n)\tau_m(n,z_*).
\end{equation*}

Therefore, there exists $C''_1>\{C_1,C'_1\}$ such that
$$\left|L_m(\widetilde{F}^n;\widetilde{a},\widetilde{b},z_*)
  \right|\leq C''_1(2C_2+C_2')N(n)\tau_m(n,z_*).$$

It implies that, for every $n\geq1$,
\begin{equation}
0<\left|i(\widetilde{F};\widetilde{a},\widetilde{b},z_*)\right|\leq
c'_0\frac{N(n)}{n},\end{equation}where $c_0'=C''_1(2C_2+C_2')$. This
implies that $$\|F^n\|_{\mathscr{G}}\succeq n.$$ Therefore,
$\|F^n\|_{\mathscr{G}}\thicksim n$, which completes the proof.
\end{proof}

As a consequence of Theorem \ref{thm:Fn thicksim n}, we have the
following theorem:

\begin{thm}\label{thm:zimmer1}
Let $\mathscr{G}$ be a finitely generated group with generators
$\{g_1,\ldots,g_s\}$ and $f\in\mathscr{G}$ be an element which is
distorted with respect to the word norm on $\mathscr{G}$. Then
$\phi(f)=\mathrm{Id}_{\mathbb{T}^2}$ (resp.
$\phi'(f)=\mathrm{Id}_{M_g}$ where $g>1$) for any homomorphism
$\phi: \mathscr{G}\rightarrow \mathrm{Ham}^1(\mathbb{T}^2,\mu)$
(resp. $\phi': \mathscr{G}\rightarrow \mathrm{Diff}^1_*(M_g,\mu)$
with $g>1$). In particular, if $\mathscr{G}$ is a finitely generated
subgroup of $\mathrm{Ham}^1(\mathbb{T}^2,\mu)$ (resp.
$\mathrm{Diff}^1_*(M_g,\mu)$ with $g>1$), every element of
$\mathscr{G}\setminus\{\mathrm{Id}_{M_g}\}$ ($g\geq1$) is
undistorted with respect to the word norm on $\mathscr{G}$.
\end{thm}

\begin{proof}We only prove the case where $\phi: \mathscr{G}\rightarrow
\mathrm{Ham}^1(\mathbb{T}^2,\mu)$ since other cases follow
similarly. Let $\mathscr{G}'$ be the finitely generated group
generated by $\{\phi(g_1),\ldots,\phi(g_s)\}$. As $f$ is a
distortion element of $\mathscr{G}$, there exists a subsequence
$\{n_i\}_{i\geq1}\subset\mathbb{N}$ such that
$$\lim_{i\rightarrow+\infty}\frac{\|\phi^{n_i}(f)\|_{\mathscr{G}'}}{n_i}=
\lim_{i\rightarrow+\infty}\frac{\|\phi(f^{n_i})\|_{\mathscr{G}'}}{n_i}=0
.$$ By Theorem \ref{thm:Fn thicksim n}, we have
$\phi(f)=\mathrm{Id}_{\mathbb{T}^2}$.
\end{proof}\smallskip

%设半单李群$\mathcal{G}=\mathrm{SL}(k_1,\mathbb{R})\times\cdots\times
%\mathrm{SL}(k_d,\mathbb{R})$，其中$k_i\geq2$和$d\in\mathbb{N}$，我们称$\sum_{i=1}^{d}(k_i-1)$为$\mathcal{G}$的秩。
%The Lubotzky-Mozes-Raghunathan theorem (see \cite{LMR}) states that
%every nonuniform lattice of rank $\geq2$ (for example,
%$\mathrm{SL}(n,\mathbb{Z})\,(n\geq3)$, see \cite{P2}) has a
%U-element (of cause, distortion element).

Let us recall some results about the irreducible lattice
$\mathrm{SL}(n,\mathbb{Z})$ with $n\geq3$. The lattice
$\mathrm{SL}(n,\mathbb{Z})$ and its any normal subgroup of finite
order have the following properties:
\begin{itemize}
 \item It contains a subgroup isomorphic to the group of upper
  triangular integer valued matrices of order $3$ with $1$ on the
  diagonal (the integer Heisenberg group),
  which tells us the existence of distortion element of every infinite norm subgroup
   of $\mathrm{SL}(n,\mathbb{Z})$
   (see \cite[Prop. 1.7]{P2});
   \item It is almost simple (every normal subgroup is finite or has a finite
  index) which is due to Margulis (Margulis finiteness theorem, see
  \cite{Mar}).
\end{itemize}

Applying these results above and Theorem \ref{thm:zimmer1}, we get
the following result:

\smallskip

\newenvironment{thm03}{\noindent\textbf{Theorem \ref{thm:zimmer2}}~\itshape}{\par}\smallskip

\begin{thm03}Every homomorphism from
$\mathrm{SL}(n,\mathbb{Z})$ ($n\geq3$) to
$\mathrm{Ham}^1(\mathbb{T}^2,\mu)$ or $\mathrm{Diff}^1_*(M_g,\mu)$
($g>1$) is trivial.
\end{thm03}

\begin{proof}Again, we only prove the case where $\phi: \mathscr{G}\rightarrow
\mathrm{Ham}^1(\mathbb{T}^2,\mu)$ since other cases follow
similarly. The following argument is due to Polterovich \cite[Proof
of Theorem 1.6]{P2}. By the first item of properties of
$\mathrm{SL}(n,\mathbb{Z})$, there is a distortion element $f$ in
$\mathrm{SL}(n,\mathbb{Z})$. Apply Theorem \ref{thm:zimmer1} to the
distortion element $f$ of infinite order of
$\mathrm{SL}(n,\mathbb{Z})$. We have that $f$ lies in the kernel of
$\phi$. Note that $\mathrm{Ker}(\phi)$ is an infinite normal
subgroup of $\mathrm{SL}(n,\mathbb{Z})$. By the second item of
properties of $\mathrm{SL}(n,\mathbb{Z})$, $\mathrm{Ker}(\phi)$ has
finite index in $\mathrm{SL}(n,\mathbb{Z})$. Hence the quotient
$\mathrm{SL}(n,\mathbb{Z})/\mathrm{Ker}(\phi)$ is finite. Therefore,
$\phi$ has finite images. Applying Corollary \ref{cor:no torsion},
we get $\phi$ is trivial.
\end{proof}

Finally, let us recall a classical result about the mapping class
group $\mathrm{Mod}(M)$, where
$\mathrm{Mod}(M)=\mathrm{Homeo}^+(M)/\mathrm{Homeo}_*(M)$ is the
isotopy classes of orientation-preserving homeomorphisms of $M$ (see
\cite{FM}): any homomorphism $\phi: \Gamma\rightarrow
\mathrm{Mod}(M)$ has finite images where $\Gamma$ is an irreducible
lattice in a semisimple lie group of $\mathbb{R}$-rank at least two.

Applying the result above and Theorem \ref{thm:zimmer2}, we get the
following general conjecture of Zimmer in the special case of
surfaces:

\smallskip

\newenvironment{thm04}{\noindent\textbf{Theorem \ref{thm:zimmer program}}~\itshape}{\par}\smallskip

\begin{thm04}Every  homomorphism from
$\mathrm{SL}(n,\mathbb{Z})$ ($n\geq3$) to $\mathrm{Diff}^1(M_g,\mu)$
($g>1$) has only finite images.
\end{thm04}

%Finally, modifying the argument of Section 4.7 in \cite{P}
%(replacing $\overline{\mathrm{Flux}}(\cdot)$ by
%$\rho_{M,\cdot}(\mu)$), we get the following result:
%
%\begin{prop}For
%every homomorphism $\phi: BS(q,p)\rightarrow
%\mathrm{Ham}^1(M_g,\mu)$ with $g\geq1$ (resp. $\phi':
%BS(q,p)\rightarrow \mathrm{Diff}^1_*(M_g,\mu)$ with $g>1$) has
%finite images.
%\end{prop}\bigskip

\bigskip

\section{Appendix}
\subsection*{Appendix A}
\begin{lem}\label{subsec:positively
recurrent}Let $(X,d)$ be a metric space and $f: X\rightarrow X$ be a
continuous map. A positively recurrent point of $f$ is also a
positively recurrent point of $f^q$ for all $q\in
\mathbb{N}$.\end{lem}

\begin{proof}
If $z\in \mathrm{Rec}^+(f)$, let $O_i=\{z'\in X\mid
d(z,z')<\frac{1}{i}\}$ for $i\in \mathbb{N}\setminus\{0\}$. We
suppose that $f^{n_k}(z)\rightarrow z$ when $k\rightarrow +\infty$.
Write $n_k=l_kq+p_k$ where $0\leq p_k <q$. If there are infinitely
many $k$ such that $p_k=0$, we are done. Otherwise, there are
infinitely many $k$ such that $p_k=p$ where $0< p <q$. We can
suppose that $f^{l_kq+p}(z)\rightarrow z$ when $k\rightarrow
+\infty$ by considering subsequence if necessary. We suppose that
$f^{l_{k_1}q+p}(z)\in O_{m_1}$, then there exists $O_{m_2}$ such
that $f^{l_{k_1}q+p}(O_{m_2})\subset O_{m_1}$. Similarly, there
exists $l_{k_2}$ and $O_{m_3}$ such that
$f^{l_{k_1}q+p}(O_{m_3})\subset O_{m_2}$. By induction, there is a
subsequence $(l_{k_j})_{j\geq1}$ of $(l_{k})_{k\geq1}$ and a
subsequence $\{O_{m_j}\}_{j\geq1}$ of $\{O_m\}_{m\geq1}$ such that
$f^{l_{k_j}q+p}(O_{m_{j+1}})\subset O_{m_{j}}$. Consider the
subsequence $\{f^{q(p+\sum_{j=(t-1)q}^{tq-1}l_{k_j})}(z)\}_{t\geq
1}$, we are done.\end{proof}\smallskip

\subsection*{Appendix B}\par\quad\par\quad

\vspace{-2mm}

 We fix a closed surface $M$ of genus $g\geq 1$
 and a topological closed disk $D$ on $M$ all examples will coincide with
 the identity outside of $D$ including isotopies. Up to a diffeomorphism,
 we may suppose that $D$ is the closed unit Euclidean disk. We will
 construct an identity isotopy  $I=(F_t)_{t\in[0,1]}$, we will write
  $F=F_1$ and $\widetilde{F}=\widetilde{F}_1$ the time-one map of $\widetilde{I}=(\widetilde {F}_t)_{t\in[0,1]}$ that is the
  lifted identity isotopy of $I$ on the universal covering space $\pi:\widetilde{M}\rightarrow M$.

\begin{exem}\label{exem:proposition 21 and mu integrable}\;\;
We construct an isotopy $I$ of $M$ and a measure $\mu\in\mathcal
{M}(F)$ such that
\begin{itemize}
  \item $F\notin\mathrm{Diff}(M)$;
  \item  $I$ satisfies the B-property;
  \item there are two
different fixed points $\widetilde{z}_0$ and $\widetilde{z}_1$ of
$\widetilde{F}$ such that the linking number
$i(\widetilde{F};\widetilde{z}_0,\widetilde{z}_1,z)$ is not bounded;
  \item there are two
different fixed points $\widetilde{z}_0$ and $\widetilde{z}_1$ of
$\widetilde{F}$ such that the linking number
$i(\widetilde{F};\widetilde{z}_0,\widetilde{z}_1,z)$ is not
$\mu$-integrable.
\end{itemize}

Use the polar coordinate for $D$ with the center $z_0=(0,0)$ and
suppose $z_1=(4/5,0)$. Let $D_{p/q}=\{(r,\theta)\mid
r\in]0,p/q[\,\}$ where $p/q\in ]0,1[\cap \,\mathbb{Q}$.

Consider a smooth decreasing function $\alpha:[0,3/4]\rightarrow
\mathbb{R}$ such that $\alpha|_{[0,1/2]}\equiv1$ and $\alpha=0$ on
neighborhood of $3/4$.

Consider a $C^\infty$-diffeomorphism $\rho(r)$ of $]0,3/4[$ as
follows
\begin{itemize}
  \item $\rho(r)$ fixes the point $1/k$ for every $k>1$ and $\rho(r)=r$ when $r\in [1/2,3/4[$;
  \item
  $\rho^n(r)\rightarrow1/(k+1)$ when $n\rightarrow -\infty$ for every $k>1$ and $r\in]1/(k+1),1/k[$;
  \item
  $\rho^n(r)\rightarrow1/k$ when $n\rightarrow +\infty$ for every $k>1$ and $r\in]1/(k+1),1/k[$.
\end{itemize}

Consider the following homeomorphism $F$ of $D$ defined on $D$ by
the formula:

\begin{equation}\label{exem:unbounded}F(r\mathrm{e}^{2i\pi\theta})=
\begin{cases}\rho(r)\mathrm{e}^{2i\pi\left(\theta+\alpha(r)(2^{\frac{1}{r}}+\frac{1}{2})\right)}& \textrm{on} \quad D_{3/4};
\\\mathrm{Id}& \textrm{on} \quad
D\setminus D_{3/4}.\end{cases}
\end{equation}

We construct an isotopy $I=(F_t)_{t\in[0,1]}$ on $D$ by replacing
$\alpha(r)(2^{\frac{1}{r}}+\frac{1}{2})$ with
$t\alpha(r)(2^{\frac{1}{r}}+\frac{1}{2})$ and $\rho(r)$ with
$(1-t)r+t\rho(r)$ in Formula (\ref{exem:unbounded}). It is easy to
see that $F$ is not differentiable at $z_0$.

Consider a finite measure $\mu$ on $M$ that is invariant by $F$ as
follows
$$\mu=\sum_{k\geq2}2^{-(k-1)}\mu_k$$ where $\mu_k$ is the Lebesgue
probability measure on $C_k$.

Let $B_k=\{(r,\theta)\mid r\in]1/(k+1),1/k[\,\}$ and $C_k=\{z\in
D\mid|z|=1/k\}$ $(k\geq2)$. Fix one point $z_k\in C_k$ for every
$k\geq2$. Let $\widetilde{z}_k$ $(k\geq0)$ be any lift of $z_k$
contained in a connected component of $\pi^{-1}(D)$. For any point
$z\in B_k$,  the $\omega$-limit set of $z$ is included in $C_k$ and
the $\alpha$-limit set of $z$ is included in $C_{k+1}$. When $z\in
C_k$, the angle of the trajectory of $I(z)$ rotating around $z_0$ is
$(2^{k+1}+1)\pi$. Hence $F$ has no contractible fixed points on
$D_{1/2}$. When $z\in D_{3/4}\setminus D_{1/2}$, the angle of the
trajectory of $I(z)$ rotating around $z_0$ is uniformly bounded.
Therefore, $I$ satisfies the B-property. However,
$i(\widetilde{F};\widetilde{z}_0,\widetilde{z}_1,z_k)=2^k+1/2$ and
$i(\widetilde{F};\widetilde{z}_0,\widetilde{z}_1,z)$ is not
$\mu$-integrable. Remark that %the measure $\mu$ is not ergodic and
the support of $\mu$ is not the whole space.
\end{exem}\smallskip

\begin{exem}\label{ex:the action of non C1-diffeo isnot bounded and continuous}
We construct an isotopy $I$ of $M$ and a measure $\mu\in\mathcal
{M}(F)$ with total support and no atoms on
$\mathrm{Fix}_{\mathrm{Cont},I}(F)$ such that
\begin{itemize}
  \item $\rho_{M,I}(\mu)=0$;
  \item $F\in\mathrm{Diff}(M)$ (and hence $I$ satisfies the WB-property);
  \item $I$ does not satisfy the B-property (and hence $F\notin\mathrm{Diff}^1_*(M)$);
\item there is a compact set
$\widetilde{P}\subset \widetilde{M}$ and
$\{(\widetilde{z}_k,\widetilde{z}\,'_k)\}_{k\geq1}\subset
\mathrm{Fix}(\widetilde{F})\times\mathrm{Fix}(\widetilde{F})\setminus\widetilde{\Delta}$
in $\widetilde{P}\times\widetilde{P}$, such that the linking numbers
$i(\widetilde{F};\widetilde{z}_k,\widetilde{z}\,'_k,z)$ are not
uniformly bounded;
\item the action $L_\mu$ (see
\ref{sec:definition of action function}) is not bounded;
  \item the action $L_\mu$ and $l_\mu$ are not continuous.
\end{itemize}

Use the Cartesian $(x,y)$-coordinate system in $D$ and suppose
$z_0=(0,0)$. On the $x$-axis, we suppose that $B_k$ $(k\geq1)$ is a
ball whose center is on $z_k=1/(k+1)+1/(2k(k+1))$ and whose radius
is $r_k= 1/2(k+1)^2$.

Consider a family of smooth functions $\alpha_k: [0,r_k]\rightarrow
\mathbb{R}$ such that $\alpha_k=0$ on neighborhoods of $0$ and
$r_k$, $\alpha_k(r_k/2)=2(-1)^k(k+1)^5$ and
$$2\pi\int_0^{\,r_k}\alpha_k(r)r\,\mathrm{d}r=(-1)^kk.$$

Consider the following diffeomorphism $F$ of $D$ which is defined by
the formula:

\begin{equation}\label{exem:WB not B}F(z_k+r\mathrm{e}^{2i\pi\theta})=
\begin{cases}z_k+r\mathrm{e}^{2i\pi(\theta+\alpha_k(r))}& \textrm{on} \quad B_k;
\\\mathrm{Id}& \textrm{on} \quad
D\setminus\bigcup_{k\geq1}B_k.\end{cases}
\end{equation}

We construct an isotopy $I=(F_t)_{t\in[0,1]}$ on $D$ by replacing
$\alpha_k(r)$ with $t\alpha_k(r)$ in Formula (\ref{exem:WB not B}).

Obviously, $z_k$ and $z_k' =z_k+r_k/2$ are fixed points of $F$ and
we have
$$i(\widetilde{F};\widetilde{z}_k,z_k')=2(-1)^k(k+1)^5$$ and $$i(\widetilde{F};\widetilde{z}_0,\widetilde{z}_k,z_k')=\rho_{A_{{\widetilde{z}_0},
{\widetilde{z}_k}},\widehat{F}_{{\widetilde{z}_0},{\widetilde{z}_k}}}
(\widetilde{z}_k\,')=2(-1)^{k+1}(k+1)^5$$ where
$\widetilde{z}_0,\widetilde{z}_k$ and $\widetilde{z}_k\,'$ are
contained in a connected component $\widetilde{D}$ of $\pi^{-1}(D)$.
Hence $I$ does not satisfy the B-property and there is a compact set
$\mathrm{Cl}(\widetilde{D})$ and
$\{\widetilde{z}_k\}_{k\geq1}\subset
\mathrm{Fix}(F)\setminus\{\widetilde{z}_0\}$ in
$\mathrm{Cl}(\widetilde{D})$, such that the linking numbers
$i(\widetilde{F};\widetilde{z}_0,\widetilde{z}_k,z)$ are not
uniformly bounded.

It is easy to prove that $F$ is a diffeomorphism of $M$ but it is
not a $C^1$-diffeomorphism of $M$: its differential $DF$ is not
continuous at $z_0$.\smallskip

Consider a finite measure $\mu$ on $M$ satisfying that
\begin{itemize}
  \item $\mu$ has total support;
  \item $\mu$ is non-atomic;
  \item $\mu$ restricted on $B_k$ is the Lebesgue measure
with $\mu(B_k)=\pi r_k^2$ for every $k\geq1$.
  %\item $\mu(M\setminus\bigcup_{k\geq1}B_k)=2-\pi\sum_{k\geq1}r_k^2>0$.
\end{itemize}

Obviously, $\mu\in\mathcal {M}(F)$ and $\rho_{M,I}(\mu)=0$.
Furthermore, we have
$$I_{\mu}(\widetilde{F};z_{k+1},z_k)=i_{\mu}(\widetilde{F};\widetilde{z}_{k+1},\widetilde{z}_k)
=(-1)^{k+1}(2k+1)$$ and
$$I_{\mu}(\widetilde{F};z_0,z_k)=i_{\mu}(\widetilde{F};\widetilde{z}_0,\widetilde{z}_k)
=(-1)^{k+1}k.$$ Therefore, the action $L_\mu$ is not bounded.
Observe that $z_k\rightarrow z_0$ and $\widetilde{z}_k\rightarrow
\widetilde{z}_0$ as $k\rightarrow+\infty$, so that $L_\mu$ and
$l_\mu$ are not continuous (at $z_0$ and $\widetilde{z}_0$).
\end{exem}
\end{document}